%% file: VMM.tex
\documentclass{memo-l}
\pdfoutput=1
\input{mymacros.tex}

\makeindex

\begin{document}

\frontmatter

\title{The Vanishing Moment Method for Fully Nonlinear Second Order Partial 
Differential Equations:  Formulation, Theory, and Numerical Analysis}

\author{Xiaobing Feng}
\address{Department of Mathematics, The University of Tennessee, Knoxville, 
TN 37996.}
\email{xfeng@math.utk.edu}
\thanks{The work of the first author was partially supported by the NSF grants DMS-0410266
and DMS-0710831.}

\author{Michael Neilan }
\address{Department of Mathematics,
University of Pittsburgh, Pittsburgh, PA 15260.}
\email{neilan@pitt.edu}
\thanks{The work of the second author was partially supported by the NSF grants DMS-1115421
and DMS-0902683.}

\date{}



\begin{abstract}
The vanishing moment method was introduced by the authors in \cite{Feng2} as 
a reliable methodology for computing viscosity solutions of fully nonlinear 
second order partial differential equations (PDEs),
in particular, using Galerkin-type numerical methods such as finite element 
methods, spectral methods, and discontinuous Galerkin methods, a task which has 
not been practicable in the past. The crux of the vanishing moment method
is the simple idea of approximating a fully nonlinear second order PDE by
a family (parametrized by a small parameter $\vepsi$) of quasilinear higher 
order (in particular, fourth order) PDEs. The primary objectives
of this book are to present a detailed convergent analysis for the method 
in the radial symmetric case and to carry out a comprehensive 
finite element numerical analysis for the vanishing moment equations 
(i.e., the regularized fourth order PDEs). Abstract methodological and 
convergence analysis frameworks of conforming finite element methods 
and mixed finite element methods are first developed for
fully nonlinear second order PDEs in general settings. The abstract
frameworks are then applied to three prototypical nonlinear equations,
namely, the Monge-Amp\`ere equation, the equation of prescribed Gauss curvature,
and the infinity-Laplacian equation. Numerical experiments are 
also presented for each problem to validate the theoretical error
estimate results and to gauge the efficiency of the proposed numerical
methods and the vanishing moment methodology.
\end{abstract}

\maketitle

\setcounter{page}{4}

\tableofcontents


\mainmatter

\include{chapter1}

\include{chapter2}

\include{chapter3}
\include{chapter4}

\include{chapter5}
\include{chapter6a}

\input{chapter6b}

\include{chapter6c}

\include{chapter7}

\appendix

\backmatter

\input ref.tex
\printindex

\end{document}

%% file: mymacros.tex
\usepackage{amsmath}
\usepackage{amssymb}
\usepackage{amsthm}
\usepackage{amsmath}
\usepackage{array}
\usepackage{xy}
\usepackage{graphicx}
\DeclareGraphicsExtensions{.jpg,.pdf,.png,.jpeg}

\newtheorem{thm}{Theorem}[chapter]
\newtheorem{lem}[thm]{Lemma}
\newtheorem{cor}[thm]{Corollary}
\newtheorem{prop}[thm]{Proposition}

\theoremstyle{definition}
\newtheorem{definition}[thm]{Definition}

\theoremstyle{remark}
\newtheorem{remark}[thm]{Remark}
\newtheorem{remarks}[thm]{Remark}
\numberwithin{section}{chapter}
\numberwithin{equation}{chapter}

\newcommand{\hu}{\hat{u}}

\newcommand{\vepsi}{{\varepsilon}}
\newcommand{\eps}{\varepsilon}
\newcommand{\Ome}{{\Omega}}
\newcommand{\ome}{{\omega}}
\newcommand{\p}{{\partial}}
\newcommand{\Del}{{\Delta}}
\newcommand{\del}{\delta}
\newcommand{\Div}{{\mbox{\rm div}}}
\newcommand{\nab}{\nabla}
\newcommand{\mck}{\mathcal{K}}
\newcommand{\cof}{{\rm cof}}
\newcommand{\nonum}{\nonumber}
\renewcommand{\(}{\bigl(}
\renewcommand{\)}{\bigr)}
\newcommand{\normd}[1]{\frac{\partial #1}{\partial \nu}} 
\newcommand{\util}{\mathcal{I}^h \ue} 
\newcommand{\vtil}{\mathcal{I}^h v}
\newcommand{\wtil}{\mathcal{I}^h w}
\newcommand{\ztil}{\mathcal{I}^h z}
\newcommand{\norm}[1]{\left\|#1\right\|} 
\newcommand{\bnorm}[1]{\bigl\|#1\bigr\|}

\newcommand{\Fp}{F^\prime}
\newcommand{\Gp}{G^\prime_\eps}
\newcommand{\lt}{{L^2}}
\newcommand{\ho}{{H^1}}
\newcommand{\htw}{{H^2}}
\newcommand{\hl}{{H^\ell}}

\newcommand{\ue}{u^\eps}
\newcommand{\ueh}{u^\eps_h}
\newcommand{\err}{e^\eps}
\newcommand{\serr}{\pi^\eps}
\newcommand{\bl}{\bigl\langle}
\newcommand{\br}{\bigr\rangle}
\newcommand{\Bl}{\Bigl\langle}
\newcommand{\Br}{\Bigr\rangle}
\newcommand{\wh}{\widehat}
\newcommand{\mci}{\mathcal{I}^h}

\newcommand{\se}{\sigma^\eps}
\newcommand{\seh}{\se_h}

\newcommand{\lnorm}{\|v\|_\hl}
\newcommand{\lnorms}{\|v\|^2_\hl}

\newcommand{\snorms}{\|\ue\|_\hl^2}
\newcommand{\snorm}{\|\ue\|_\hl}
\newcommand{\Mbf}{\mathbf{M}_h}
\newcommand{\Mo}{M^{(1)}_h}
\newcommand{\Mt}{M^{(2)}_h}
\newcommand{\Mone}{\Mo\(\stil,\util\)}
\newcommand{\Mtwo}{\Mt\(\stil,\util\)}
\newcommand{\tbar}[2]{\left|\hspace{-0.03cm}\left|\left(#1,#2\right)\right|\hspace{-0.03cm}\right|_\eps}
\newcommand{\ttbar}[2]{\left|\hspace{-0.03cm}\left|\hspace{-0.03cm}\left|\left(#1,#2\right)
\right|\hspace{-0.03cm}\right|\hspace{-0.03cm}\right|_\eps}

\newcommand{\stil}{\Pi^h \se}

\newcommand{\wt}{\widetilde}
\newcommand{\set}{\wt{\sigma}^\eps}
\newcommand{\chit}{\wt{\chi}}
\newcommand{\wb}{\wt{b}}
\newcommand{\wc}{\wt{c}}
\newcommand{\wW}{\wt{W}}
\newcommand{\wF}{\wt{F}}

\newcommand{\setnorm}{\|\ue\|_\hl}

\newcommand{\blank}{(-----)} 
\newcommand{\Blank}{---------------------}

\newcommand{\ttbart}[2]{\left|\hspace{-0.03cm}\left|\hspace{-0.03cm}\left|\left(#1,#2\right)
\right|\hspace{-0.03cm}\right|\hspace{-0.03cm}\right|_{\wt{\eps}}}

%% file: chapter1.tex
\chapter{Prelude}\label{chapter-1}
\section{Introduction}\label{chapter-1.1}

Fully nonlinear partial differential equations (PDEs) are those equations
which are nonlinear in the highest order derivative(s) of the unknown function(s).
In the case of the second order equations, the general form of fully nonlinear
PDEs is given by
\begin{equation}
\label{generalPDE} F(D^2 u, \nab u, u, x)=0,
\end{equation}
where $D^2 u(x)$ and $\nab u(x)$ denote respectively the Hessian
and the gradient of $u$ at $x\in \Ome\subset \mathbf{R}^n$.  Here, $F$ is
assumed to be a nonlinear function in at least one of its entries of $D^2 u$.
Fully nonlinear PDEs arise from many scientific and engineering
fields including differential geometry, optimal control, mass transportation,
geostrophic fluid, meteorology, and general relativity (cf. 
\cite{Caffarelli_Cabre95,Caffarelli_Milman99, Gilbarg_Trudinger01,
Fleming_Soner06,McCann_Oberman04} and the references therein).

Examples of such equations include (cf. \cite{Gilbarg_Trudinger01})
\begin{itemize}
\item {\em The Monge-Amp\`ere equation}
\begin{equation}
\label{ma}\det(D^2u)=f.
\end{equation}
\item {\em The equation of prescribed Gauss curvature}
\begin{equation}
\label{gausseqn}\det(D^2u)=\mathcal{K}(1+|\nab u|^2)^{\frac{n+2}{2}}.
\end{equation}
\item {\em The Bellman equation}
\begin{equation}
\label{bellman}\inf_{\theta \in V} (L_{\theta}u-f_\theta)=0.
\end{equation}
\end{itemize}
Here, $\det(D^2 u(x))$ denotes the determinant of the 
Hessian $D^2 u$ at $x$, and $\{L_{\theta}\}$ denotes a family of
second order linear differential operators.

Because of the full nonlinearity in \eqref{generalPDE}, the standard weak solution
theory based on the integration by parts approach does not work and other notions of 
weak solutions must be sought. Progress has been made in 
the latter half of the twentieth century concerning this issue after the 
introduction of viscosity solutions. In 1983, Crandall and 
Lions \cite{Crandall_Lions83} introduced the notion of viscosity solutions 
and used the vanishing viscosity method to 
show existence of a solution for the Hamilton-Jacobi equation:
\begin{alignat}{2}\label{HJeqn1}
u_t+H(\nab  u,u,x)&=0\qquad &&\mbox{in } \textbf{R}^n\times (0,\infty).
\end{alignat} 

The vanishing viscosity method approximates the Hamilton-Jacobi equation by the 
following regularized, second order quasilinear PDE:
\begin{alignat}{2} \label{HJeqn3}
u_t^\eps-\eps\Delta u^\eps+H(\nab u^\eps,u^\eps,x)&=0\qquad 
&&\mbox{in } \mathbf{R}^n \times (0,\infty).
\end{alignat}

It was shown that \cite{Crandall_Lions83} there exists a unique solution 
$u^\eps$ to the regularized Cauchy problem that 
converges locally and uniformly to a continuous 
function $u$ which is defined to be a viscosity 
solution of the Hamilton-Jacobi equation \eqref{HJeqn1}.  
However, to establish uniqueness, the following intrinsic 
definition of viscosity solutions was also proposed 
\cite{Crandall_Lions83,Crandall_Evans_Lions84}:
\begin{definition}
\label{IntrinsicDefinition}
Let $H:\mathbf{R}^n\times \mathbf{R}\times \Ome\to \mathbf{R}$ and 
$g:\p\Ome\to \mathbf{R}$ be continuous functions,
and consider the following problem:
\begin{alignat}{2} \label{gen1order1}
H(\nab u,u,x)&=0\qquad &&\text{in }\Ome,\\
\label{gen1order2}u&=g\qquad &&\text{on }\p\Ome.
\end{alignat}
\begin{enumerate}
\item[(i)]
$u\in C^0(\Ome)$ is called a {\em viscosity
subsolution} of \eqref{gen1order1}--\eqref{gen1order2}
if $u\big|_{\p\Ome}=g$, and for every $C^1$ function
$\varphi(x)$ such that $u-\varphi$ has a local maximum at $x_0\in
\Ome$, there holds 
\[
H(\nab  \varphi(x_0),u(x_0),x_0)\le 0.
\]
\item[(ii)]
$u\in C^0(\Ome)$ is called a {\em viscosity
supersolution} of \eqref{gen1order1}--\eqref{gen1order2}
 if $u\big|_{\p\Ome}=g$, and for every $C^1$ function
$\varphi(x)$ such that $u-\varphi$ has a local minimum at $x_0\in
\Ome$, there holds 
\[
H(\nab  \varphi(x_0),u(x_0),x_0)\ge 0.
\]
\item[(iii)]
$u\in C^0(\Ome)$ is called a {\em viscosity
solution} of \eqref{gen1order1}--\eqref{gen1order2} if it is both a viscosity
subsolution and supersolution.
\end{enumerate}
\end{definition} 

Clearly, the above definition is not variational, 
as it is based on a ``differentiation
by parts'' approach (a terminology introduced in 
\cite{Crandall_Lions83, Crandall_Evans_Lions84}).  In addition, the word 
``viscosity'' loses its original meaning
in the definition.  However, it was shown 
\cite{Crandall_Lions83, Crandall_Evans_Lions84} 
that every viscosity solution constructed by 
the vanishing viscosity method is an intrinsic 
viscosity solution (i.e., a solution that satisfies Definition \ref{IntrinsicDefinition}). 
Besides addressing the uniqueness
issue, another reason to favor the intrinsic 
differentiation by parts definition 
is that the definition and the notion of viscosity 
solutions can be readily extended 
to fully nonlinear second order PDEs as follows (cf. \cite{Caffarelli_Cabre95}): 
\begin{definition}
\label{viscosity}
Let $F:\mathbf{R}^{n\times n}\times\mathbf{R}^n\times 
\mathbf{R}\times \Ome\to \mathbf{R}$ and 
$g:\p\Ome\to \mathbf{R}$ be continuous functions, 
and consider the following problem:
 \begin{alignat}{2}
 \label{gen2order1}F(D^2u,\nab u,u,x)&=0\qquad &&\text{in }\Ome,\\
\label{gen2order2}u&=g\qquad &&\text{on }\p\Ome.
\end{alignat}
\begin{enumerate}
\item[(i)]
$u\in C^0(\Ome)$ is called a {\em viscosity
subsolution} of \eqref{gen2order1}--\eqref{gen2order2} if 
$u\big|_{\p\Ome}=g$, and
 for every $C^2$ function
$\varphi(x)$ such that $u-\varphi$ has a local maximum at $x_0\in
\Ome$, there holds 
\[
F(D^2\varphi(x_0),\nab \varphi(x_0),u(x_0),x_0)\le 0.
\]
\item[(ii)]
$u\in C^0(\Ome)$ is called a {\em viscosity
supersolution} of \eqref{gen2order1}--\eqref{gen2order2} 
if $u\big|_{\p\Ome}=g$,
and for every $C^2$ function
$\varphi(x)$ such that $u-\varphi$ has a local minimum at $x_0\in
\Ome$, there holds 
\[
F(D^2\varphi(x_0),\nab \varphi(x_0),u(x_0),x_0)\ge 0.
\]

\item[(iii)]
$u\in C^0(\Ome)$ is called a {\em viscosity
solution} of \eqref{gen2order1}--\eqref{gen2order2} if it is both a viscosity
subsolution and supersolution.\end{enumerate}
\end{definition}

\begin{remark}
Without loss of generality, we may assume that $u(x_0)=\varphi(x_0)$
whenever $u-\varphi$ achieves a local maximum or local minimum at $x_0\in \Ome$
in Definition \ref{viscosity}.  Therefore, in an informal setting, $u$ is a viscosity 
solution if for every smooth function  $\varphi$ that ``touches'' the graph of
$u$ from above at $x_0$ (see Figure \ref{viscosityfig}) there holds
\[
F(D^2\varphi(x_0),\nab \varphi(x_0),\varphi(x_0),x_0)\le 0,
\]
and if $\varphi$ ``touches'' the graph of
$u$ from below at $x_0$, then 
\[F(D^2\varphi(x_0),\nab \varphi(x_0),\varphi(x_0),x_0)\ge 0.
\]
\end{remark}

\begin{figure}
\begin{center}
\includegraphics[scale = 0.25]{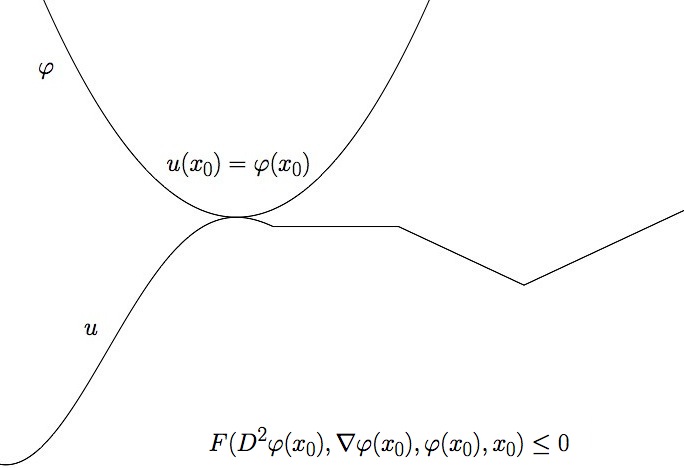}
\includegraphics[scale = 0.25]{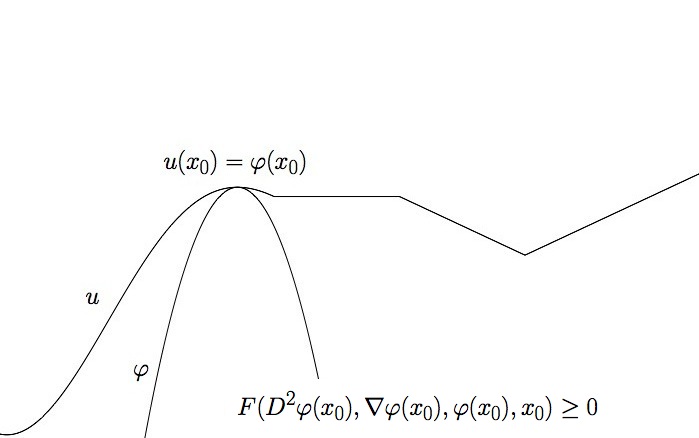}
\caption{A Geometric interpretation of viscosity solutions}\label{viscosityfig}
\end{center}
\end{figure}

In case of the fully nonlinear {\em first order} PDEs, 
tremendous progress has been made in the 
past three decades in terms of PDE analysis and numerical methods. 
A profound viscosity solution theory has been established (cf.
\cite{Crandall_Lions83,Crandall_Evans_Lions84,Crandall_Ishii_Lions92,Fleming_Soner06})
and a wealth of efficient and robust numerical methods and
algorithms have been developed and implemented (cf.
\cite{Barth,Bryson_Levy03,Cockburn03,Crandall_Lions96,
Lin_Tadmor00,Osher1,Osher2,Osher3,Sethianbook,Zhang_Shu02,Zhao05}).
However, in the case of fully nonlinear {\em second order} PDEs, 
the situation is strikingly different. On the one hand, 
there have been enormous advances
in PDE analysis in the past two decades after the introduction of
the notion of viscosity solutions by M. Crandall and  P. L. Lions in
1983 (cf. \cite{Caffarelli_Cabre95,Crandall_Ishii_Lions92,Gutierrez01}).
On the other hand, in contrast to the success of the PDE analysis,
numerical solutions for general fully nonlinear second order PDEs
is a relatively untouched area.  

There are several reasons for this lack of progress in numerical methods.
First, the most obvious difficulty is the full nonlinearity in the equation.  
Second, solutions to fully nonlinear second order equations are often only
unique in a certain class of functions, and this conditional uniqueness
is very difficult to handle numerically. Lastly and most importantly,
it is extremely difficult (if all possible) to mimic the
differentiation by parts approach at the discrete level.
As a consequence, there is little hope to develop a 
discrete viscosity solution theory.
Furthermore, it is impossible to directly compute
viscosity solutions using Galerkin-type numerical methods
including finite element methods, spectral Galerkin methods, 
and discontinuous Galerkin methods, since they are all based 
on variational formulations of PDEs.  In fact, this is clear 
from the definition of viscosity solutions,
which is not based on the traditional integration by parts
approach, but rather is defined by the differentiation by parts
approach.

To explain the above points, consider the Dirichlet problem for
the Monge-Amp\`ere equation as an example:
\begin{alignat}{2}
\label{MAeqn1}\text{det}(D^2 u)&=f\qquad &&\text{in}\ \Ome,\\
\label{MAeqn2}u&=g\qquad && \text{on}\ \p\Ome,
\end{alignat}
which corresponds to $F(D^2u,\nab u,u,x)=f(x)-\mbox{det}(D^2u)$.
It is known that for a non-strictly convex domain $\Ome$,
the above problem does not have classical solutions in general even
if $f, g$, and $\p\Ome$ are smooth \cite{Gilbarg_Trudinger01}.
Results of A. D. Aleksandrov state that the Dirichlet
problem with $f>0$ has a unique generalized solution (which is 
also the viscosity solution) in the class of
convex functions \cite{Aleksandrov61,Gutierrez01}.  
The reason to restrict the admissible set to be the set of convex
functions is that the Monge-Amp\`ere equation is only elliptic
in that set \cite{Gilbarg_Trudinger01,Gutierrez01}.  It should be noted
that in general, the Dirichlet problem \eqref{MAeqn1}--\eqref{MAeqn2}
may have other nonconvex solutions even when $f>0$.
It is easy to see that if one discretizes
\eqref{MAeqn1} directly using a standard finite 
difference method, not only would the resulting
algebraic system be difficult to solve, 
one immediately loses control
on which solution the numerical scheme 
approximates - and this is assuming that the nonlinear
discrete problem has solutions!  Furthermore, 
the situation is even worse if one tries to formulate a Galerkin-type
numerical method because there is not a variational or weak formulation
in which to start.

Nevertheless, a few recent numerical attempts and results have been known
in the literature. In \cite{Oliker_Prussner88} Oliker
and Prussner proposed a finite difference scheme for
computing Aleksandrov measure induced by $D^2u$ (and obtained
the solution $u$ of \eqref{ma} as a by-product) in two dimensions.
The scheme is extremely geometric and difficult
to generalize to other fully nonlinear second order PDEs.
In \cite{Barles_Souganidis91} Barles and Souganidis
showed that any monotone, stable, and consistent finite difference
scheme converges to the viscosity solution provided that there
exists a comparison principle for the limiting equation.
Their result provides a guideline for constructing
convergent finite difference methods, but they did
not address how to construct such a scheme. 
In \cite{Baginski_Whitaker96}, Baginski and
Whitaker proposed a finite difference scheme for 
the equation of prescribed Gauss curvature \eqref{gausseqn}
in two dimensions by mimicking the unique continuation
method (used to prove existence of the PDE) at the discrete level.
The method becomes very unstable when the homotopy is dominated 
by the fully nonlinear equation.
Oberman \cite{Oberman07} constructed a wide stencil finite difference scheme
for fully nonlinear elliptic PDEs which can be written
as functions of eigenvalues of the Hessian matrix
and proved that the scheme satisfies the 
convergence criterion established by Barles and Souganidis 
in \cite{Barles_Souganidis91}.
In a series of papers \cite{Dean_Glowinski_a, Dean_Glowinski_b,Dean_Glowinski_c}
Dean and Glowinski proposed an augmented Lagrange multiplier method
and a least squares method for problem \eqref{ma}
and Pucci's equation (cf. \cite{Caffarelli_Cabre95,Gilbarg_Trudinger01})
in two dimensions by treating the nonlinear PDEs
as a constraint and using a variational criterion to select
a particular solution. However, as the admissible set 
is contained in $H^2(\Ome)$, it could become empty if all solutions of
the underlying fully nonlinear PDE are not differentiable.
Finally, B\"ohmer \cite{Bohmer08} recently introduced a projection 
method using $C^1$ finite elements for approximating classical
solutions of a certain class of fully nonlinear second order elliptic PDEs. 
However, the issue of how to reliably compute a selected solution
(the resulting discrete problem often has multiple solutions) was not addressed and still remains
an open question.
%
%
Numerical experiments were reported in
\cite{Oliker_Prussner88,Baginski_Whitaker96,Oberman07,
Dean_Glowinski_a,Dean_Glowinski_b,Dean_Glowinski_c}, 
however, convergence analysis was not addressed
except in \cite{Oberman07}.

In addition, we like to remark that there is a considerable amount
of literature available on using finite difference methods to
approximate viscosity solutions of fully nonlinear second order
Bellman-type PDEs arising from stochastic optimal
control (cf. \cite{Barles_Souganidis91,Barles_Jakobsen05,Jakobsen03,Krylov05}).
However, due to the special structure of Bellman-type
PDEs, the approach used and the methods proposed in those papers
could not be extended to other types of fully nonlinear 
second order PDEs since the construction of those methods
critically relies on the linearity of the operators $L_\theta$.

The first goal of this book is to present a general
framework for the vanishing moment method and the notion
of moment solutions for fully nonlinear second order PDEs.
 The vanishing moment method is very much in 
the same spirit of the vanishing viscosity method introduced in
\cite{Crandall_Lions83}, and the notion of moment solutions
for fully nonlinear second order PDEs is a natural extension of 
the (original) notion of viscosity solutions for fully nonlinear 
first order PDEs.
This methodology was first introduced by the authors
in \cite{Feng2} as a reliable way for computing 
viscosity solutions of fully nonlinear second order PDEs,
in particular, using Galerkin-type numerical methods.
The crux of this new method is to approximate
a fully nonlinear second order PDE by a family
of quasilinear fourth order PDEs. The limit of the
solutions of the fourth order PDEs (if it exists)
is defined as a moment solution of the original fully nonlinear 
second order PDE. As moment solutions are defined constructively, 
they can be readily computed by existing numerical methods.
In the case of Monge-Amp\`ere-type equations, extensive numerical 
experiments in \cite{Feng2,Feng3,Feng4,Neilan09} suggest that the 
moment solution coincides with the viscosity solution 
as long as the latter exists. In this book, we shall
present a detailed convergence theory for the vanishing moment 
method in the radial symmetric case. This then provides a 
theoretical foundation for the method and for the
numerical results of \cite{Feng2,Feng3,Feng4,Neilan09}.

The second goal of this book, which is the bulk of
the book's content, is to carry out a 
comprehensive finite element numerical analysis
for the vanishing moment method. Two abstract
frameworks are developed for this purpose in 
a general setting. The first framework concerns
$C^1$ conforming finite element approximations of
the vanishing moment equations (i.e., the regularized
fourth order equations).  The second framework 
develops (Herman-Miyoshi) mixed finite element methods for 
the vanishing moment equations. Each of these two frameworks
consists of the formulation of the respective numerical 
methods, proving existence and uniqueness of numerical
solutions, and deriving error estimates for the 
numerical solutions. Due to the strong nonlinearity
of the PDEs, the standard numerical analysis techniques for 
finite element methods do not work here. To overcome
the difficulty, we combine a fixed point argument with a linearization technique.
After having completed both abstract frameworks, we 
apply them to three prototypical nonlinear 
equations, namely, the Monge-Amp\`ere equation, the equation 
of prescribed Gauss curvature, and the infinity-Laplacian equation.
The three equations are chosen because they present three different 
and interesting scenarios, that is, their linearizations
are respectively coercive, indefinite, and degenerate. 
It is shown that our abstract frameworks are rich enough to
cover all three scenarios.
 
The remainder of the book is organized as follows. Chapter \ref{chapter-2}
represents the formulation of the vanishing moment method and its 
informal insights. The material of this chapter has a large overlap with
that of \cite{Feng2}.  Chapter \ref{chapter-3} is devoted to the 
convergence analysis of the vanishing moment method  
for the Monge-Amp\`ere equation
in the radial symmetric case. The main tasks of the chapter are to analyze 
the vanishing moment equations and to derive uniform (in $\vepsi$) 
estimates for its solutions. The chapter also contains
a convergence rate estimate result for the regularized solutions 
in the case that the viscosity solution of the Monge-Amp\`ere equation
belongs to $W^{2,\infty}(\Ome)\cap H^3(\Ome)$. 
Chapter \ref{chapter-4} and \ref{chapter-5} 
develop, respectively, the abstract frameworks for the two types of 
finite element (i.e., conforming and mixed finite element)
approximations of the vanishing moment equations under some structure 
assumptions on the nonlinear differential operator $F$.
Chapter \ref{chapter-6} presents applications of 
the abstract frameworks of Chapter \ref{chapter-4} and \ref{chapter-5}
to three prototypical nonlinear equations: the Monge-Amp\`ere equation, 
the equation of prescribed Gauss curvature, and the infinity-Laplacian
equation. For each equation, we formulate its vanishing moment 
approximations, subsequent finite element and mixed finite element 
methods, and obtain their error estimates by fitting the equation 
into the abstract frameworks. For the Monge-Amp\`ere equation, 
besides some slight improvements, we essentially recover the early 
results reported in \cite{Feng3,Feng4}.  On the other hand,
the results for the equation of prescribed Gauss curvature and the 
infinity-Laplacian equation are new. In fact, to the best of our knowledge,
no comparable results are known in the literature.  
Numerical experiments are also presented for each problem to validate
the theoretical (error estimate) results, and to gauge the efficiency
of the proposed numerical methods and the vanishing moment methodology.
Finally, we end the book with a few concluding remarks in 
Chapter \ref{chapter-7}.

\section{Preliminaries}\label{chapter-1.2}
Standard space notation is adopted in this book, we refer the reader
to \cite{Brenner,Gilbarg_Trudinger01,Ciarlet78} for their exact
definitions. In addition, $\Ome$ denotes a bounded convex domain in
$\mathbf{R}^n$. $(\cdot,\cdot)$ and $\langle\cdot, \cdot\rangle_{\partial\Omega}$
denote the $L^2$-inner products on $\Ome$ and on $\p\Ome$, respectively.
The unlabeled constant $C$ is used to denote generic $\vepsi$- and $h-$independent
positive constants that may take on different values at different occurrences,
where as labeled constants denote $\eps$-dependent (but $h$-independent)
constants. Furthermore all constants, labeled and unlabeled,
are chapter-independent unless otherwise specified.

Throughout this book we assume that 
\[
F:\, \mathbf{R}^{n\times n}\times\mathbf{R}^n\times \mathbf{R}\times \Ome
\longrightarrow \mathbf{R}
\]
is a differentiable function in all its arguments. For a given
(small) constant $\vepsi>0$, we define
\[
G_\vepsi(r,p,z,x):=\vepsi \Del {\rm tr}(r)+ F(r,p,z,x)\quad\forall
r\in \mathbf{R}^{n\times n},\ p\in \mathbf{R}^n,\ z\in \mathbf{R},\ x\in \Ome.
\]
For a given scalar function $v$ and an $n\times n$ matrix-valued function 
$\mu=[\mu_{ij}]$ \footnote{In an effort to clarify notation, we mostly use 
Greek letters to represent matrix-valued functions, and Roman letters to 
represent scalar functions throughout the book}, we set 
\begin{alignat*}{1}
&F_r[r,p,z,x](\mu):=F_r:\mu
=\sum_{i,j=1}^n \frac{\p F}{\p r_{ij}}(r,p,z,x) \mu_{ij}(x),\\
\nonum&F_p[r,p,z,x](v):=F_p\cdot \nab v
=\sum_{i=1}^n \frac{\p F}{\p p_i}(r,p,z,x) \frac{\p v}{\p x_i}(x),\\
\nonum&F_z[r,p,z,x](v):=F_z\cdot v
=\frac{\p F}{\p z}(r,p,z,x)v(x),\\
\nonum&\Fp[r,p,z,x](\mu,v):=F_r[r,p,z,x](\mu)+F_p[r,p,z,x](v)+F_z[r,p,z,x](v),\\
\nonum&\Gp[r,p,z,x](\mu,v):=\eps \Del {\rm tr}(\mu)+\Fp[r,p,z,x](\mu,v).
\end{alignat*}

We also define, with a slight abuse of notation, for a scalar function $w$
and an $n\times n$ tensor function $\kappa=[\kappa_{ij}]$, the following
short-hand notation, which will be extensively used
when developing mixed finite element methods in Chapter \ref{chapter-5},
\begin{alignat}{1}
\label{mixedFdef}
&F(\kappa,w):=F(\kappa,\nab w,w,x),\\
&\nonum F_r[\kappa,w](\mu):=F_r[\kappa,\nab w,w,x](\mu),\\
\nonum &F_p[\kappa,w](v):=F_p[\kappa,\nab w,w,x](v),\\
\nonum &F_z[\kappa,w](v):=F_z[\kappa,\nab w,w,x](v),\\
\nonum &\Fp[\kappa,w](\mu,v):=F_r[\kappa,w](\mu)+F_p[\kappa,w](v)+F_z[\kappa,w](v),\\
&\nonum G_\vepsi(\kappa,w):=\vepsi \Del {\rm tr}(\kappa) + F(\kappa,w),\\
&\nonum \Gp[\kappa,w](\mu,v):=\eps\Del {\rm tr}(\mu)+\Fp[\kappa,w](\mu,v).
\end{alignat}

For notation used in Chapter \ref{chapter-4}, we overload the operators 
$F,G_\vepsi, \Fp,$ and $G^\prime_\vepsi$ once again and define the 
additional short-hand notation:
\begin{alignat}{1}\label{conformFdef}
&F(w):=F(D^2w,w), \\
\nonum &F_r[w](v):=F_r[D^2w,w](D^2v), \\
\nonum & F_p[w](v):=F_p[D^2w,w](v), \\
\nonum &F_z[w](v):=F_z[D^2w,w](v), \\
\nonum &\Fp[w](v):=F_r[w](v)+F_p[w](v)+F_z[w](v),\\
\nonum &G_\vepsi(w):=G_\vepsi(D^2w,w)=\vepsi \Del^2 w +F(w),\\
\nonum &\Gp[w](v):=\eps\Del^2 v+\Fp[w](v). 
\end{alignat}

We conclude this section and chapter by citing a divergence-free row property of 
the cofactor matrix of the gradient of a vector-valued smooth function (a special 
case of Piola's identity).
This property will be used many times in the later chapters of the book.
A proof of this property can be found in \cite[page 440]{evans}.

\begin{lem}\label{cofactor}
Given a vector-valued function $\mathbf{v}=(v_1,v_2,\cdots,v_n):
\Ome\rightarrow \mathbf{R}^n$. Assume $\mathbf{v}\in
[C^2(\Ome)]^n$. Then the cofactor matrix
$\cof (D\mathbf{v})$ of the gradient matrix $D\mathbf{v}$ of
$\mathbf{v}$ satisfies the following row divergence-free property:
\begin{equation*}
\Div (\text{\rm cof}(D\mathbf{v}))_i =\sum_{j=1}^n \frac{\p }{\p x_j}
(\text{\rm cof}(D\mathbf{v}))_{ij} =0 \qquad\text{\rm for }
i=1,2,\cdots, n,
\end{equation*}
where $(\cof(D\mathbf{v}))_i$ and $(\cof (D\mathbf{v}))_{ij}$ denote respectively the $i$th row and the
$(i,j)$-entry of $\cof (D\mathbf{v})$.
\end{lem}

%% file: chapter2.tex
\chapter{Formulation of the vanishing moment method}\label{chapter-2}
In this chapter we shall present the formulation of 
the vanishing moment method for fully nonlinear second order PDE
\eqref{generalPDE}. We also explain how the method was conceived 
and give some informal insights about the method. We note that
the material of this chapter has a large overlap with that of \cite{Feng2}. 

For the reasons and difficulties explained in Chapter \ref{chapter-1}, 
as far as we can see, it is unlikely (at least very difficult if at all 
possible) that one can directly approximate viscosity solutions of
general fully nonlinear second order PDEs using 
available numerical methodologies such as finite difference 
methods, finite element methods, spectral and discontinuous 
Galerkin methods, meshless methods, etc.  In particular, the robust 
and popular Galerkin-type methods (such as finite element methods, 
spectral, and discontinuous Galerkin methods) for solving linear and 
quasilinear PDEs become powerless when facing fully nonlinear 
second order PDEs.
From a computational point of view, the notion of viscosity
solutions is, in some sense, an ``inconvenient'' notion for fully 
nonlinear second order PDEs because it is neither constructive
nor variational. 
In searching for a ``better'' notion of weak solutions for fully nonlinear 
second order PDEs, we are inspired by the following simple but crucial 
observation: {\em the crux of the vanishing viscosity method for 
the Hamilton-Jacobi equation and the original notion of
viscosity solutions is to approximate a lower order
fully nonlinear PDE by a family of quasilinear higher order PDEs.}
  
It is exactly this observation which motivates us to 
apply the above quoted idea to fully nonlinear second order 
PDE \eqref{generalPDE} in \cite{Feng2}. To this end, 
we take one step further and approximate
fully nonlinear second order PDE \eqref{generalPDE}
by the following fourth order quasilinear PDEs \cite{Feng2}\footnote{Other 
higher order linear operators may be used 
in the place of $\Del^2 u^\vepsi$, we refer the reader to 
\cite{Feng2} for more discussions on the choices of the 
regularizing operators. Here, we implicitly assume that $-F$ is elliptic
in the sense of \cite[Chapter 17]{Gilbarg_Trudinger01}, otherwise, 
\eqref{fourthorder} needs to be replaced by 
$-\eps\Delta^2 u^\eps+F(D^2 u^\eps, \nabla u^\eps, u^\eps, x)=0.$}:
\begin{align}\label{fourthorder}
\eps\Delta^2 u^\eps+F(D^2 u^\eps, \nabla u^\eps, u^\eps, x)&=0
\qquad \text{in}\ \Omega,\quad \eps>0.
\end{align}
Here and for the continuation of the paper,
we only consider the Dirichlet problem for \eqref{generalPDE}, so 
we suppose that
\begin{align}
\label{dirichlet}u=g\qquad \text{on }\p\Ome.
\end{align}
It is then obvious that we need to impose
\begin{align}
\label{moment_dirichlet}u^\eps=g\qquad \text{on }\p\Ome.
\end{align}
However, the Dirichlet boundary condition \eqref{moment_dirichlet} 
is not sufficient for well-posedness, and therefore an additional boundary 
condition must be used. Several boundary conditions could be used for 
this purpose, but physically, any additional boundary condition will introduce 
a so-called ``boundary layer". A better choice would be one which minimizes 
the boundary layer. Based on some heuristic arguments and evidence of
numerical experiments, we propose to use one of the following 
additional boundary conditions:
\begin{equation}\label{additional}
\Delta u^\eps=\eps \quad\text{on }\p\Ome,
\end{equation}
or 
\begin{equation}\label{additional_1}
\frac{\p \Del u^\vepsi}{\p \mathbf{\nu}}=\eps \quad\text{on } \p\Ome,
\end{equation}
or 
\begin{equation} \label{additional_2}
 D^2u^\eps\mathbf{\nu}\cdot\mathbf{\nu}=\eps \quad\text{on } \p\Ome,
\end{equation}
where $\nu$ denotes the outward unit normal to $\p\Ome$.  

We note that another valid boundary condition is the following 
Neumann boundary condition:
\[
\frac{\p u^\eps}{\p \nu}=\eps \quad\text{on } \p\Ome.
\]
However, since this is an essential boundary condition,
it produces a larger boundary layer than the other three boundary 
conditions, and therefore, we do not recommend the use of this boundary condition.

The rationale for picking boundary condition \eqref{additional} 
is that we implicitly impose an extra boundary condition  
\[
\eps^m\Delta u^\eps+u^\eps=g+\eps^{m+1}\quad \text{on }\p\Ome,
\] 
which is a higher order perturbation of the original 
Dirichlet boundary condition \eqref{dirichlet}.
Intuitively, we expect 
that the extra boundary condition 
converges to the original Dirichlet boundary condition as 
$\eps$ tends to zero for sufficiently large positive integer $m$.

\begin{remark}
(a) We note that boundary conditions \eqref{additional} and \eqref{additional_1}, 
which are natural boundary conditions for equation \eqref{fourthorder}, have
an advantage in PDE convergence analysis. Also, both boundary conditions
are better suited for conforming and nonconforming finite element 
methods \cite{Feng4,Neilan09}, where as boundary condition in \eqref{additional_2} 
fits naturally with the mixed finite element formulation \cite{Feng3}.

(b) From the PDE analysis viewpoint, the reason why high order boundary conditions
such as 
\eqref{additional}--\eqref{additional_2} work
better may be explained as follows. Since viscosity solutions  
generally do not have second or higher order (weak) derivatives, 
we do not expect $u^\vepsi$ to converge to $u$ in $H^s(\Ome)$ for 
$s\geq 2$ in general. Therefore, it is possible that errors
in higher order derivatives, which could be big, 
would have small effects on the convergence of $u^\vepsi$ 
in the lower order norms if $u^\vepsi$ is constructed appropriately.
Also, as we shall see later, the reason we do not 
impose homogeneous boundary conditions in \eqref{additional}--\eqref{additional_2}
is that the regularized solution inherits favorable properties such as strict convexity.
\end{remark}

To summarize, the vanishing moment method consists of 
approximating the (given) nonlinear second order problem
\begin{alignat}{2}\label{generalPDEa}
F(D^2u,\nabla u,u,x)&=0\qquad &&\text{in }\Ome,\\
\label{generalPDEb}u&=g\qquad &&\text{on }\p\Ome,
\end{alignat}
by the following quasilinear fourth order boundary value problem:
\begin{alignat}{2}
\label{moment1}\eps\Delta^2u^\eps+
F(D^2u^\eps,\nabla u^\eps,u^\eps,x)&=0\qquad &&\text{in }\Ome,\\
\label{moment2}u^\eps&=g\qquad &&\text{on }\p\Ome,\\
\label{moment3}\Delta u^\eps=\eps,
\quad \text{or}\quad \frac{\partial \Delta u^\eps}{\partial \nu}=\eps,
\quad \text{or}\quad D^2u^\eps\mathbf{\nu}\cdot\mathbf{\nu}&=\eps
\quad&&\text{on}\ \p\Ome.
\end{alignat}


Since equation \eqref{moment1} is quasilinear, we can then
define the notion of a weak solution using the usual 
integration by parts approach. 
\begin{definition}\label{momentsolndef}
We define $u^\eps\in H^2(\Omega)$ with $u\big|_{\p\Ome}=g$ 
to be a solution of \eqref{moment1}--\eqref{moment3}$_1$ if
for all $v\in H^2(\Ome)\cap H^1_0(\Ome)$
\begin{align} \label{moment_var}
\eps(\Delta u^\eps,\Delta v)+(F(D^2u^\eps,\nabla u^\eps,u^\eps,x),v)
=\left\langle \eps^2,\normd{v}\right\rangle_{\partial\Omega}.
\end{align}
\end{definition}

We now are ready to define the notion of moment solutions for
\eqref{generalPDEa}--\eqref{generalPDEb}.
\begin{definition}
Suppose that $u^\eps$ solves problem \eqref{moment1}--\eqref{moment3}$_1$.
$\lim_{\eps\to 0^+} u^\eps$ is called a weak ({\it resp. strong}) moment 
solution to problem \eqref{generalPDEa}--\eqref{generalPDEb} if the convergence 
holds in $H^1$-weak ({\it resp. $H^2$-weak}) topology.
\end{definition}

\begin{remark}
(a) The terminologies ``moment solutions" and ``vanishing moment method"
were chosen due to the following consideration.
In two-dimensional mechanical applications, $u^\eps$ 
often stands for the vertical displacement of a plate, 
and $D^2u^\eps$ is the moment tensor. In the weak 
formulation, the biharmonic term becomes $\eps(D^2u^\eps,D^2v)$ 
which should vanish as $\eps\to 0^+$. This is the reason we 
call $\lim_{\eps\to 0^+} u^\eps$ (if it exists) a moment solution and 
call the limiting process the vanishing moment method.

(b) Since weak moment solutions do not have 
second order weak derivatives in general, they are difficult 
(if at all possible) to identify. On the other hand, 
since strong moment solutions do have second order weak derivatives, 
they are naturally expected to satisfy equation \eqref{generalPDEa}
almost everywhere and to fulfill the boundary condition \eqref{moment2}.
In the remainder of this book, moment solutions 
will always mean weak moment solutions.
\end{remark}

As problem \eqref{moment1}--\eqref{moment3} is a quasilinear
fourth order problem, one can compute its solutions using literally 
any well-known numerical methods, in particular, Galerkin-type 
methods such as finite element methods, spectral and discontinuous
Galerkin methods. We note that \eqref{moment_var} provides
a variational formulation for \eqref{moment1}--\eqref{moment3}$_1$.
Indeed, developing finite element numerical methods
is one of two main goals of this book.  
In Chapter \ref{chapter-4} and \ref{chapter-5} we shall
present comprehensive finite element and mixed finite element 
analysis for problem \eqref{generalPDEa}--\eqref{generalPDEb}.

However, a natural and larger question is whether 
the vanishing moment methodology will work. There are 
two ways to address this question. First, one can do 
many numerical experiments to see if the methodology
works in practice. We indeed have done so (and beyond) in
a series of papers \cite{Feng2,Feng3,Feng4,Feng5,Neilan09} 
(also see \cite{Neilan_thesis}) for the Monge-Amp\`ere
equation. All numerical experiments of these papers
show that the vanishing moment methodology works 
effectively.
Second, one can give a definitive answer to the question
by laying down its theoretical foundation, namely, 
proving the convergence (and rates of convergence
if it is possible) (cf. \cite{Feng1}) of the vanishing moment 
method. Partially accomplishing this goal is in fact the second 
main objective of this book. In the next chapter, we shall 
give a detailed convergence theory for the vanishing moment 
method applied to the Monge-Amp\'ere equation in the radial symmetric case.
We refer the interested reader to \cite{Feng1} for 
the convergence analysis in more general cases.

We conclude this chapter by mentioning another intriguing property 
of the vanishing moment method, which was reported in 
\cite{Feng2} and discovered numerically by 
accident. When constructing the vanishing moment approximation
\eqref{moment1}, we restrict the parameter $\vepsi$ to be positive
(and drive it to zero from the positive side). 
An interesting question is what happens if we allow $\vepsi$ to be 
negative (and drive it to zero from the negative side). In order words,
we want to know the limiting behaviour as $\vepsi\searrow 0^+$ of 
the following problem:
\begin{alignat}{2} \label{neg_moment1}
-\eps\Delta^2u^\eps+ F(D^2u^\eps,\nabla u^\eps,u^\eps,x)&=0\qquad &&\text{in }\Ome,\\
\label{neg_moment2}u^\eps&=g\qquad &&\text{on }\p\Ome,\\
\label{neg_moment3}\Delta u^\eps=-\eps,
\quad \text{or}\quad \frac{\partial \Delta u^\eps}{\partial \nu}=-\eps,
\quad \text{or}\quad D^2u^\eps\mathbf{\nu}\cdot\mathbf{\nu}&=-\eps
\quad&&\text{on}\ \p\Ome.
\end{alignat}

The numerical experiments of \cite{Feng2} (also see \cite{Neilan_thesis})
indicate that in the case of the two-dimensional Monge-Amp\`ere equation 
(cf. Chapter 6), that is,
\[
F(D^2v,\nabla v,v,x)=f-\mbox{det}(D^2v), \qquad f>0,
\]
$u^\vepsi$ converges to the {\em concave} solution of 
the Dirichlet problem \eqref{generalPDEa}--\eqref{generalPDEb}!
In the next chapter, we shall also give a proof for this 
numerical discovery in the radial symmetric case.

%% file: chapter3.tex
\chapter{Convergence of the vanishing moment method}\label{chapter-3}

The primary goal of this chapter is to present a detailed convergence 
analysis for the vanishing moment method applied to the Monge-Amp\`ere
equation in the $n$-dimensional radial symmetric case. Such a 
result then puts down the vanishing moment method on a solid footing and 
provides a (partial) theoretical foundation for the numerical 
work to be given in the remaining chapters.

\section{Preliminaries}\label{chapter-3.1}
Unless stated otherwise, throughout this chapter 
$\Ome=B_R(0)\subset \mathbf{R}^n$ ($n\geq 2$) stands for the ball centered at 
the origin with radius $R$. We do not assume $\Ome$ is the unit 
ball because many of our results will depend on the size of the radius $R$. 

Suppose that $f=f(r), f\not\equiv 0$ and $g=g(r)$ in \eqref{MAeqn1}--\eqref{MAeqn2}, that is, 
$f$ and $g$ are radial.  Then the solution $u$ of \eqref{MAeqn1}--\eqref{MAeqn2}
is expected to be radial, namely, $u(x)$ is a function of 
$r:=|x|=\sqrt{\sum_{j=1}^n x_j^2}$. We set $\hu(r):=\hu(|x|)=u(x)$, and
for the reader's convenience, we now compute $\Del u, \Del^2 u$ and
$\mbox{det}(D^2u)$ in terms of $\hu$ (cf. \cite{Monn86,Rios_Sawyer08}). 
Trivially,
\[
\frac{\p r}{\p x_j} = \frac{x_j}{r},\qquad 
\frac{\p }{\p x_j}\Bigl(\frac{1}{r}\Bigr) = -\frac{x_j}{r^3}.
\]
By the chain rule we have
\begin{align*}
\frac{\p u(x)}{\p x_j}&=\hu_r(r) \frac{\p r}{\p x_j}
=\hu_r(r) \frac{x_j}{r} ,\\
\frac{\p^2 u(x)}{\p x_j\p x_i}
&=\frac{x_j}{r} \frac{\p}{\p x_i} \hu_r(r)
  + \hu_r(r) \frac{\p}{\p x_i} \Bigl(\frac{x_j}{r}\Bigr)
=\frac{1}{r}  \Bigl(\frac{1}{r} \hu_r(r)\Bigr)_r x_ix_j 
  + \frac{\hu_r(r)}{r}\delta_{ij}.
\end{align*}
Here, the subscripts stand for the derivatives
with respect to the subscript variables.

On noting that $D^2u(x)$ is a diagonal perturbation of
a scaled rank-one matrix $xx^T$, and since the eigenvalues of $xx^T$ 
are $0$ (with multiplicity $n-1$) and $|x|^2=r^2$ (with multiplicity 
$1$ and corresponding eigenvector $x$),
then the eigenvalues of $D^2u(x)$ are
\begin{align*}
\lambda_1 &:=\frac{\hu_r(r)}{r} + r\Bigl(\frac{1}{r} \hu_r(r)\Bigr)_r =\hu_{rr}(r)
\quad\mbox{(with multiplicity $1$)},\\
\lambda_2 &:=\frac{\hu_r(r)}{r} \quad\mbox{(with multiplicity $n-1$)}.
\end{align*}
Thus,
\begin{align*}
&\Del u(x) = 
\lambda_1+ (n-1)\lambda_2 =\hu_{rr}(r) +\frac{n-1}{r} \hu_r(r)
=\frac{1}{r^{n-1}} \bigl(r^{n-1} \hu_r\big)_r,\\ 
&\Del^2 u(x) =\Del (\Del u) 
=\Del\Bigl( \frac{1}{r^{n-1}} \bigl(r^{n-1} \hu_r \big)_r \Bigr)
=\frac{1}{r^{n-1}} \Bigl(r^{n-1} \Bigl(\frac{1}{r^{n-1}} \bigl(r^{n-1} \hu_r \big)_r
\Bigr)_r \Bigr)_r,\\ 
&\mbox{det}(D^2u(x)) = \lambda_1 (\lambda_2)^{n-1}
=\hu_{rr}(r) \Bigl[\frac{\hu_r(r)}{r} \Bigr]^{n-1}
=\frac{1}{n r^{n-1}}  \bigl((\hu_r \big)^n)_r.
\end{align*}

Abusing the notion to denote $\hu(r)$ by $u(r)$, then 
problem \eqref{MAeqn1}--\eqref{MAeqn2} becomes seeking 
a function $u=u(r)$ such that
\begin{align}\label{MA1_radial}
\frac{1}{n r^{n-1}}  \bigl((u_r \big)^n)_r &= f  \qquad\mbox{in } (0,R),\\
u(R) &=g(R), \label{MA2_radial} \\
u_r(0) &=0. \label{MA3_radial}
\end{align}
We remark that boundary condition \eqref{MA3_radial} is due to the symmetry
of $u=u(r)$. 

\begin{lem}\label{lem3.1}
Suppose that $r^{n-1}f\in L^1((0,R))$ and $f\geq 0$ a.e. on $(0,R)$. Then there
exists exactly one real solution if $n$ is odd and there are exactly
two real solutions if $n$ is even, to the boundary value problem 
\eqref{MA1_radial}--\eqref{MA3_radial}. Moreover, the solutions are
given by the formula
\begin{equation}\label{solution}
u(r)=\left\{ 
\begin{array}{ll}
g(R) \pm \int_r^R \bigl(nL_f(s)\bigr)^{\frac{1}{n}}\,\, ds
&\qquad\mbox{if $n$ is even},\\
\\
g(R) - \int_r^R \bigl(nL_f(s)\bigr)^{\frac{1}{n}}\,\, ds
&\qquad\mbox{if $n$ is odd}
\end{array}\right.
\end{equation}
for $r\in (0,R)$. Where
\begin{equation}\label{L_f}
L_f(s):=\int_0^s t^{n-1} f(t)\, dt.
\end{equation}
\end{lem}

Since the proof is elementary (cf. \cite{Monn86,Rios_Sawyer08}), 
we omit it.  Clearly, when $n$ is even, the first solution (with ``$+$" sign)
is concave and the second solution (with ``$-$" sign) 
is convex because $u_{r}$ and $u_{rr}$ simultaneously positive 
and negative respectively in the two cases. When $n$ is odd, the 
real solution is convex.

\begin{remark}
The above theorem shows that $u$ is $C^2$ at a point $r_0\in (0,R)$ as long as
$f$ is $C^0$ at $r_0$ and $L_f(r_0)\neq 0$. Also, $u$ is smooth in $(0,R)$ if $f$ is 
smooth in $(0,R)$.  We refer the reader to \cite{Monn86,Rios_Sawyer08}
for the precise conditions on $f$ at $r=0$ to ensure the regularity
of $u$ at $r=0$, extensions to the complex Monge-Amp\`ere
equation, and generalized Monge-Amp\`ere equations in 
which $f=f(\nab u,u,x)$.
\end{remark}

Similarly, it is expected that $u^\vepsi=u^\vepsi(r)$ is also radial, and 
the vanishing moment approximation \eqref{moment1}--\eqref{moment3}$_1$
then becomes (cf. Chapter \ref{chapter-6})
\begin{align}\label{moment1_radial}
-\frac{\vepsi}{r^{n-1}} \Bigl(r^{n-1} \Bigl(\frac{1}{r^{n-1}} 
\bigl(r^{n-1} u^\vepsi_r \big)_r \Bigr)_r \Bigr)_r
+ \frac{1}{n r^{n-1}}  \bigl((u^\vepsi_r \big)^n)_r  &= f  \qquad\mbox{in } (0,R),\\
u^\vepsi(R) &=g(R), \label{moment2_radial} \\
u^\vepsi_r(0)=0,\quad |u^\vepsi_{rr}(0)|<\infty, 
\quad |u^\vepsi_{rrr}(r)| =o\bigl(\frac{1}{r^{n-1}}\bigr) &\quad\mbox{as }r\to 0^+,
\label{moment3_radial} \\
u^\vepsi_{rr}(R) + \frac{n-1}{R} u^\vepsi_r(R) &=\vepsi. \label{moment4_radial} 
\end{align}

Later in this chapter, we shall analyze problem 
\eqref{moment1_radial}--\eqref{moment4_radial}
which includes proving its existence and uniqueness as well as regularities.
After this is done, we then show that the solution
$u^\vepsi$ of \eqref{moment1_radial}--\eqref{moment4_radial}
converges to the unique convex solution of \eqref{MA1_radial}--\eqref{MA3_radial}.


Integrating over $(0,r)$ after multiplying \eqref{moment1_radial} by $r^{n-1}$,
using boundary condition \eqref{moment3_radial} and
\[
\lim_{r\to 0^+} r^{n-1}\Bigl( u^\vepsi_{rrr}+ \frac{n-1}{r} u^\vepsi_{rr} 
-\frac{n-1}{r^2} u^\vepsi_r \Bigr)=0
\]
we get
\begin{equation}\label{moment1b_radial}
-\vepsi r^{n-1} \Bigl(\frac{1}{r^{n-1}} 
\bigl(r^{n-1} u_r \big)_r \Bigr)_r 
+\frac1{n} (u^\vepsi_r)^n = L_f 
\qquad \mbox{in } (0,R).
\end{equation}

Introduce the new function $w^\vepsi(r):=r^{n-1} u^\vepsi_r(r)$.  
A direct calculation shows that $w^\vepsi$ satisfies
\begin{equation}\label{moment1c_radial}
-\vepsi r^{n-1} \Bigl(\frac{1}{r^{n-1}} w^\vepsi_r\Bigr)_r
+\frac{1}{nr^{n(n-1)}} (w^\vepsi)^n = L_f 
\qquad \mbox{in } (0,R).
\end{equation}
Converting the boundary conditions 
\eqref{moment3_radial}--\eqref{moment4_radial} to $w^\vepsi$ we have
\begin{align}
w^\vepsi(0)=w^\vepsi_r(0) &=0, \label{moment3a_radial} \\
w^\vepsi_r(R) &=\vepsi R^{n-1}. \label{moment4a_radial} 
\end{align}
In addition, since 
\begin{align*}
w^\vepsi_r&=r^{n-1} u^\vepsi_{rr} +(n-1)r^{n-2} u^\vepsi_r,\\
w^\vepsi_{rr}&=r^{n-1}u^\vepsi_{rrr} 
+ 2(n-1)r^{n-2}u^\vepsi_{rr} + (n-1)(n-2)r^{n-3}u^\vepsi_r,
\end{align*}
we have
\begin{align}\label{moment5a_radial}
\frac{\p^j w^\vepsi}{\p r^j}(0)=o\Bigl(\frac{1}{r^{n-1-j}}\Bigr)
\qquad\text{for } 0\leq j\leq \min\{2, n-1\}.
\end{align}

So we have derived from \eqref{moment1_radial} a reduced equation 
\eqref{moment1c_radial}, which is only of second order, hence, it is
easier to handle. After problem \eqref{moment1c_radial}--\eqref{moment4a_radial}
is fully understood, we then come back to analyze problem 
\eqref{moment1_radial}--\eqref{moment4_radial}.

\section{Existence, uniqueness, and regularity of vanishing moment approximations} 
\label{chapter-3.2}
We now prove that problem \eqref{moment1c_radial}--\eqref{moment4a_radial}
possesses a unique nonnegative classical solution. First, we state and
prove the following uniqueness result. 

\begin{thm}\label{uniqueness_thm}
Problem \eqref{moment1c_radial}--\eqref{moment4a_radial} has at most one
nonnegative classical solution.
\end{thm}

\begin{proof}
Suppose that $w^\vepsi_1$ and $w^\vepsi_2$ are two nonnegative 
classical solutions to \eqref{moment1c_radial}--\eqref{moment4a_radial}. 
Let 
\[
\phi^\vepsi:= w^\vepsi_1-w^\vepsi_2\qquad\text{and}\qquad 
\overline{w}^\vepsi:=
\left\{
\begin{array}{ll}
\displaystyle\mathop{\sum_{\alpha+\beta = n-1}}_{\alpha,\beta\ge 0} (w_1^\eps)^\alpha (w_2^\eps)^\beta &\text{if }w^\eps_1= w^\eps_2\medskip\\
\dfrac{(w^\vepsi_1)^n-(w^\vepsi_2)^n}{w^\vepsi_1-w^\vepsi_2} & \text{otherwise}.
\end{array}
\right.
\]
Subtracting the corresponding equations satisfied by 
$w^\vepsi_1$ and $w^\vepsi_2$ yields
\begin{align} \label{error1}
-\vepsi r^{n-1} \Bigl(\frac{1}{r^{n-1}} \phi^\vepsi_r\Bigr)_r
+ \frac{1}{nr^{n(n-1)}} \overline{w}^\vepsi \phi^\vepsi &= 0 
\qquad \mbox{in } (0,R), \\
\phi^\vepsi(0)=\phi^\vepsi_r(0) &=0,  \label{error2}\\
\phi^\vepsi_r(R) &=0. \label{error3}
\end{align}

On noting that $\overline{w}^\vepsi\geq 0$ in $[0,R]$, by the weak maximum 
principle \cite[Theorem 2, page 329]{evans} we conclude
\[
\max_{ [0,R]} |\phi^\vepsi (r)| =\max\{ |\phi^\vepsi(0)|, |\phi^\vepsi(R)| \}  
=\max\{ 0, |\phi^\vepsi(R)| \}.
\]
If $\phi^\vepsi(R)=0$, then $\phi^\vepsi\equiv 0$. If 
$\phi^\vepsi(R)\neq 0$, then $\phi^\vepsi$ takes its maximum
or minimum value at $r=R$. However, the strong maximum principle  
\cite[Theorem 4, page 7]{Protter_Weinberger67}  
implies that $\phi^\vepsi_r(R)\neq 0$, which contradicts with  
boundary condition $\phi^\vepsi_r(R)= 0$. Hence, $\phi^\vepsi\equiv 0$
or $w^\vepsi_1\equiv w^\vepsi_2$. The proof is complete.
\end{proof}

\begin{remark}
A more direct way to prove $\phi^\vepsi\equiv 0$ is given as follows.
Multiplying \eqref{error1} by $\phi^\vepsi$, integrating by parts, 
and using the boundary conditions \eqref{error2}--\eqref{error3} yield
\begin{align} \label{uniqueness_proof} 
\vepsi \int_0^R |\phi^\vepsi_r(r)|^2 \,dr 
&+ \vepsi \int_0^R \frac{n-1}{r} \phi^\vepsi_r(r) \phi^\vepsi(r)\, dr \\
&+\int_0^R \frac{1}{nr^{n(n-1)}} \overline{w}^\vepsi(r) |\phi^\vepsi(r)|^2\,dr=0.
\nonumber
\end{align}
On noting that 
\begin{align*}
\int_0^R \frac{n-1}{r} \phi^\vepsi_r(r) \phi^\vepsi(r)\, dr
&=\frac{n-1}{2r} \bigl(\phi^\vepsi(r)\bigr)^2\Bigr|_{r=0}^{r=R} 
 +\int_0^R\frac{n-1}{2r^2} \bigl(\phi^\vepsi(r)\bigr)^2 \,dr \\
&=\frac{n-1}{2R} \bigl(\phi^\vepsi(R)\bigr)^2 
+\int_0^R\frac{n-1}{2r^2} \bigl(\phi^\vepsi(r)\bigr)^2 \,dr.
\end{align*}
so each term on the left-hand side of \eqref{uniqueness_proof}
is nonnegative, hence, they all must be zero. The first term
then gives $\phi^\vepsi_r\equiv 0$. Then $\phi^\vepsi\equiv \mbox{const}$. 
Hence, $\phi^\vepsi\equiv 0$ by $\phi^\vepsi(0)=0$. 
\end{remark}

Next, we prove that the existence of nonnegative solutions to 
problem \eqref{moment1c_radial}--\eqref{moment4a_radial}.

\begin{thm}\label{existence_thm}
Suppose $r^{n-1}f\in L^1((0,R))$ and $f\geq 0$ a.e. in $(0,R)$, then 
there is a nonnegative classical solution to 
problem \eqref{moment1c_radial}--\eqref{moment4a_radial}.
\end{thm}

\begin{proof}
We divide the proof into three steps.

\medskip
{\em Step 1:}
Let $\psi^0\in C^2([0,R])$ be nonnegative and satisfy $\psi^0(0)=\psi^0_r(0)=0$
and $\psi^0_r(R)=\vepsi R^{n-1}$. One such an example is 
$\psi^0(r)=\frac{\vepsi}n r^n$.
We then define a sequence of functions $\{\psi^k\}_{k\geq 0}$ 
recursively by solving for $k=0,1,2,\cdots$
\begin{align}\label{existence1}
-\vepsi r^{n-1} \Bigl(\frac{1}{r^{n-1}} \psi^{k+1}_r\Bigr)_r
&+\frac{1}{nr^{n(n-1)}} (\psi^k)^{n-1} \psi^{k+1} \\
&= L_f(r) :=\int_0^r s^{n-1} f(s)\, ds \qquad \mbox{in } (0,R), \nonumber \\
\psi^{k+1}(0)=\psi^{k+1}_r(0) &=0, \label{existence2} \\
\psi^{k+1}_r(R) &=\vepsi R^{n-1}. \label{existence3} 
\end{align}

We first show by induction that for any such sequence satisfying \eqref{existence1}--\eqref{existence3}, 
there holds  $\psi^k\geq 0$ in $[0,R]$ for all $k\geq 0$.
Note that $\psi^0\geq 0$ by construction. Suppose that $\psi^k\geq 0$ 
in $[0,R]$.  Since $f\geq 0$ and $f\not \equiv 0$, then 
\[
-\vepsi r^{n-1}\Bigl(\frac{1}{r^{n-1}} \psi^{k+1}_r\Bigr)_r
+\frac{1}{nr^{n(n-1)}} (\psi^k)^{n-1} \psi^{k+1} > 0 \qquad \mbox{in } (0,R). 
\]
Hence, $\psi^{k+1}$ is a supersolution to the linear differential 
operator on the left-hand side of \eqref{existence1}.  By the weak maximum
principle \cite[Theorem 2, page 329]{evans} we have
\[
\min_{ [0,R]} \psi^{k+1} (r) \geq \min\{0, \psi^{k+1}(0), \psi^{k+1}(R)\}
=\min\{0, \psi^{k+1}(R)\}.
\]

If $\psi^{k+1}(R)< 0$, and since $\psi^{k+1}_r(R) =\vepsi R^{n-1}> 0$,
the strong maximum principle \cite[Theorem 4, page 7]{Protter_Weinberger67}
implies that $\psi^{k+1}\equiv\mbox{const}$ in $[0,R]$, which leads to 
a contradiction as $\psi^{k+1}(0)=0$. Thus, we must have
$\psi^{k+1}(R)\geq 0$, and therefore, $\psi^{k+1}\geq 0$ in $[0,R]$. 
By the induction argument, we conclude that $\psi^k\geq 0$ 
in $[0,R]$ for all $k\geq 0$.

It then follows from the standard theory for linear elliptic equations 
(cf. \cite{evans, Gilbarg_Trudinger01}) that \eqref{existence1}--\eqref{existence3}
has a unique classical solution $\psi^{k+1}$. Hence the $(k+1)$th iterate
$\psi^{k+1}$ is well defined, and therefore, so is the sequence $\{\psi^{k}\}_{k\geq 0}$.

\medskip
{\em Step 2:} Next, we shall derive some uniform (in $k$) estimates
for the sequence $\{\psi^k\}_{k\geq 0}$. To this end, we first prove that
$\psi^{k+1}(R)$ can be bounded from above uniformly in $k$.  
Multiplying \eqref{existence1} by $\psi^{k+1}$ and integrating by parts yield
\begin{align}\label{existence4}
&-\vepsi \psi^{k+1}_r(r) \psi^{k+1}(r) \Bigl|_{r=0}^{r=R}
+\vepsi \int_0^R |\psi^{k+1}_r(r)|^2 \,dr \\
&+\vepsi \int_0^R \frac{n-1}{r} \psi^{k+1}_r(r) \psi^{k+1}(r)\, dr 
+ \int_0^R \frac{1}{nr^{n(n-1)}} (\psi^k(r))^{n-1} |\psi^{k+1}(r)|^2\,dr \nonumber \\
&\hskip 1in 
= \int_0^R \psi^{k+1}(r) L_f(r) \,dr. \label{existence4} \nonumber
\end{align}

It follows from boundary conditions \eqref{existence2} and \eqref{existence3} that
\begin{equation}\label{existence5a}
-\vepsi \psi^{k+1}_r(r) \psi^{k+1}(r) \Bigl|_{r=0}^{r=R}=-\vepsi^2 R^{n-1}\, \psi^{k+1}(R).
\end{equation}
Integrating by parts gives
\begin{align}\label{existence5}
\int_0^R \frac{1}{r} \psi^{k+1}_r(r) \psi^{k+1}(r)\, dr
&=\frac{1}{2r} \bigl(\psi^{k+1}(r)\bigr)^2\Bigr|_{r=0}^{r=R} 
+\int_0^R \frac{1}{2r^2}\bigl(\psi^{k+1}(r)\bigr)^2\, dr \\
&=\frac{1}{2R} \bigl(\psi^{k+1}(R)\bigr)^2
+ \int_0^R \frac{1}{2r^2} \bigl(\psi^{k+1}(r)\bigr)^2\, dr. \nonumber
\end{align}
By Schwarz, Poincar\'e, and Young's inequalities, we get
\begin{align}\label{existence6}
&\int_0^R \psi^{k+1}(r) L_f(r) \,dr 
\leq \frac{\vepsi}2 \int_0^R |\psi^{k+1}_r(r)|^2 \,dr
+ \frac{C_1^2}{2\vepsi} \int_0^R \bigl(L_f(r)\bigr)^2\, dr
\end{align}
for some positive constant $C_1=C_1(R)$.

Combining \eqref{existence4}--\eqref{existence6} we obtain
\begin{align}\label{existence7a} 
&-2\vepsi^2 R^{n-1}\, \psi^{k+1}(R) + \vepsi \int_0^R |\psi^{k+1}_r(r)|^2 \,dr
+\frac{\vepsi(n-1)}{R}\bigl(\psi^{k+1}(R)\bigr)^2 \\
&+ \int_0^R \frac{\vepsi(n-1)}{r^2} \bigl(\psi^{k+1}(r)\bigr)^2\, dr
+\int_0^R \frac{2}{nr^{n(n-1)}} (\psi^k(r))^{n-1} |\psi^{k+1}(r)|^2\,dr 
\nonumber\\
&\hskip 1.8in
\leq \frac{C_1^2}{\vepsi} \int_0^R \bigl(L_f(r)\bigr)^2\, dr.
\nonumber
\end{align}
Let 
\[
z:=\psi^{k+1}(R),\quad 
b:=\frac{2\vepsi R^n}{n-1},\quad
c:=\frac{C_1^2R^2}{\vepsi^2(n-1)}\bigl(L_f(R)\bigr)^2.
\]
Then from \eqref{existence7a} we have
\[
z^2-bz-c\leq 0,
\]
which in turn implies that
\[
z_1\leq z \leq z_2,\quad\mbox{where}\quad  z_1=\frac{b-\sqrt{b^2+4c}}2,\quad
z_2=\frac{b+\sqrt{b^2+4c}}2.
\]
Since $-z_2< z_1$, the above inequality then infers that
$|z|\leq z_2$. Thus, there exists a positive constant
$C_2=C_2(R, L_f)$ such that
\begin{equation}\label{existence7}
\bigl|\psi^{k+1}(R)\bigr|\leq z_2\leq \frac{C_2}{\vepsi}.
\end{equation}
Substituting \eqref{existence7} into the first term on the
left-hand side of \eqref{existence7a} we also get
\begin{align}\label{existence8}
\vepsi \int_0^R |\psi^{k+1}_r(r)|^2 \,dr 
&+\int_0^R \frac{\vepsi(n-1)}{r^2} \bigl(\psi^{k+1}(r)\bigr)^2\, dr \\
&+\int_0^R \frac{2}{nr^{n(n-1)}} \bigl(\psi^k(r)\bigr)^{n-1}|\psi^{k+1}(r)|^2\,dr
\nonumber \\
&\leq C_3:=\frac{C_1^2R}{\vepsi} \bigl(L_f(R)\bigr)^2 + 2\vepsi^2 R^{n-1} C_2. 
\nonumber
\end{align}

Now using the pointwise estimate for linear elliptic equations
\cite[Theorem 3.7]{Gilbarg_Trudinger01} we have 
\begin{equation}\label{existence9}
\max_{[0,R]} |\psi^{k+1}(r)| \leq \Bigl(\frac{C_2}{\vepsi}  
+ \frac{L_f(R)}{\vepsi}\Bigr).
\end{equation}

Next, we show that $\psi^{k+1}_r$ is also uniformly bounded (in $k$) in $[0,R]$.
To this end, integrating \eqref{existence1} over $(0,r)$ after multiplying
it by $r^{n(n-1)}$, and integrating by parts twice in the first term yield
\begin{align}\label{existence10}
\psi^{k+1}_r(r)&= 
-\frac{(n^2-1)[n(n-1)-1]}{r^{n(n-1)}} \int_0^r s^{n(n-1)-2} \psi^{k+1}(s)\,ds 
\\
&\quad\, +\frac{n^2-1}{r} \psi^{k+1}(r)
+\frac{1}{\vepsi nr^{n(n-1)}} \int_0^r  \bigl(\psi^k(s)\bigr)^{n-1}
\psi^{k+1}(s)\,ds \nonumber  \\
&\quad\, -\frac{1}{\vepsi r^{n(n-1)}}  \int_0^r s^{n(n-1)} L_f(s) \,ds
\qquad\forall r\in (0,R). \nonumber
\end{align}
Using L'H\^opital's rule it is easy to check that the limit as
$r\to 0^+$ of each term on the right-hand side of \eqref{existence10}
is zero, hence, each term is bounded in a neighborhood of $r=0$. 
Moreover, on noting that $\psi^k\geq 0$, by Schwarz inequality, we have
\begin{align}\label{existence11}
&\int_0^r  \bigl(\psi^k(s)\bigr)^{n-1} \psi^{k+1}(s)\,ds \\
&\leq \Bigl(\int_0^r s^{n(n-1)}  \bigl(\psi^k(s)\bigr)^{n-1} \,ds\Bigr)^{\frac12} 
\Bigl(\int_0^r \frac{1}{s^{n(n-1)}} \bigl(\psi^k(s)\bigr)^{n-1} |\psi^{k+1}(s)|^2\,ds\Bigr)^{\frac12}.
\nonumber
\end{align}

Now in view of \eqref{existence8}--\eqref{existence11} we conclude 
that there exists a positive constant $C_4=C_4(R,L_f)$ such that
\begin{equation}\label{existence12}
\max_{[0,R]} |\psi^{k+1}_r(r)| \leq \frac{C_4}{\vepsi^{\frac{n+2}2}}.
\end{equation}

By \eqref{existence1} we get
\begin{align}\label{existence13}
\psi^{k+1}_{rr}(r)=\frac{1}{r} \psi^{k+1}_r(r) 
&+\frac{1}{\vepsi nr^{n(n-1)}} \bigl(\psi^k(r)\bigr)^{n-1} \psi^{k+1}(r) \\
&-\frac{1}{\vepsi} L_f(r) \qquad\forall r\in (0,R).  \nonumber
\end{align}
Again, using L'H\^opital's rule and \eqref{moment5a_radial}
it is easy to check that the limit as
$r\to 0^+$ of each term on the right-hand side of \eqref{existence13}
exists, and therefore, each term is bounded in a neighborhood of $r=0$.
Hence, it follows from \eqref{existence9} and \eqref{existence12} that
there exists a positive constant $C_5=C_5(R,L_f)$ such that
\begin{equation}\label{existence14}
\max_{[0,R]} |\psi^{k+1}_{rr}(r)| \leq \frac{C_5}{\vepsi^{n+1}}.
\end{equation}

To summarize, we have proved that 
$\|\psi^{k+1}\|_{C^j([0,R])}\leq C(\vepsi,R,n,L_f)$ for $j=0,1,2$ 
and the bounds are independent of $k$.  Clearly, by a
simple induction argument we conclude that these estimates hold for 
all $k\geq 0$.

\medskip
{\em Step 3:} 
Since $\|\psi^k\|_{C^2([0,R])}$ is uniformly bounded in $k$, then 
both $\{\psi^k\}_{k\geq 0}$ and $\{\psi^k_r\}_{k\geq 0}$
are uniformly equicontinuous. It follows from
Arzela-Ascoli compactness theorem (cf. \cite[page 635]{evans}) that
there is a subsequence of $\{\psi^k\}_{k\geq 0}$ (still denoted by the
same notation) and $\psi\in C^2([0,R])$ such that 
\begin{alignat*}{2}
\psi^k &\longrightarrow \psi 
&&\qquad\mbox{uniformly in every compact set $E\subset (0,R)$ as } k\to \infty,\\
\psi^k_r &\longrightarrow \psi_r
&&\qquad\mbox{uniformly in every compact set $E\subset (0,R)$ as } k\to \infty. 
\end{alignat*}

Testing equation \eqref{existence1} with an arbitrary function
$\chi\in C^1_0((0,R))$ yields
\begin{align*}
&\vepsi\int_0^R \psi^{k+1}_r(r) \chi_r(r) \,dr  
+\vepsi\int_0^R \frac{n-1}{r} \psi^{k+1}_r(r) \chi(r) \,dr \\
&\quad
+\int_0^R \frac{1}{nr^{n(n-1)}}\bigl(\psi^k(r)\bigr)^{n-1} \psi^{k+1}(r) \chi(r)\,dr
=\int_0^R L_f(r) \chi(r)\, dr. 
\end{align*}
Setting $k\to \infty$ and using the Lebesgue Dominated Convergence 
Theorem, we get
\begin{align}\label{existence15}
\vepsi\int_0^R \psi_r(r) \chi_r(r) \,dr  
&+\vepsi\int_0^R \frac{n-1}{r} \psi_r(r) \chi(r) \,dr \\
&+\int_0^R \frac{1}{nr^{n(n-1)}} \bigl(\psi(r)\bigr)^n \chi(r)\,dr
=\int_0^R L_f(r) \chi(r)\, dr. \nonumber
\end{align}
Since $\psi\in C^2([0,R])$, we are able to integrate by parts in the first term on the 
left-hand side of \eqref{existence15}, yielding 
\[
\int_0^R \Bigl[ -\vepsi \psi_{rr}(r) +\frac{\vepsi(n-1)}{r} \psi_r(r) 
+  \frac{1}{nr^{n(n-1)}} \bigl(\psi(r)\bigr)^n - L_f(r) \Bigr] \chi(r)\, dr=0
\]
for all $\chi\in C^1_0((0,R))$. This then implies that
\[
-\vepsi \psi_{rr}(r) +\frac{\vepsi(n-1)}{r} \psi_r(r) 
+ \frac{1}{nr^{n(n-1)}} \bigl(\psi(r)\bigr)^n
- L_f(r) =0 \qquad\forall r\in (0,R),
\]
that is,
\[
-\vepsi r^{n-1} \Bigl(\frac{1}{r^{n-1}} \psi_r(r)\Bigr)_r  
+ \frac{1}{nr^{n(n-1)}} \bigl(\psi(r)\bigr)^n = L_f(r)\qquad\forall r\in (0,R).
\]
Thus, $\psi$ satisfies \eqref{moment1c_radial} pointwise in $(0,R)$.

Finally, it is clear that $\psi \geq 0$ in $[0,R]$, and it follows easily 
from \eqref{existence2} and \eqref{existence3} that 
\[
\psi(0)=\psi_r(0) =0 \qquad\mbox{and}\qquad
\psi_r(R) =\vepsi R^{n-1}.
\]
So we have demonstrated that $\psi\in C^2([0,R])$ is a nonnegative
classical solution to problem \eqref{moment1c_radial}--\eqref{moment4a_radial}.  
The proof is complete.
\end{proof}

\begin{remark}
(a) The proof at the beginning of {\em Step 2} gives an estimate
for the Neumann to Dirichlet map: $\psi^{k+1}_r(R)\to \psi^{k+1}(R)$.

(b) We note that the a priori estimates derived in the proof are not sharp
in $\vepsi$. Better estimates will be obtained (and needed) in the
next section after the positivity of $\Del u^\vepsi$ is established.
\end{remark}

The above proof together with the uniqueness theorem, 
Theorem \ref{uniqueness_thm}, and the strong
maximum principle immediately give the following corollary.

\begin{cor}\label{existence_cor}
Suppose $r^{n-1}f\in L^1((0,R))$ and $f\geq 0$ a.e. in $(0,R)$, then 
there exists a unique nonnegative classical solution $w^\vepsi$ to
problem \eqref{moment1c_radial}--\eqref{moment4a_radial}. Moreover,
$w^\vepsi>0$ in $(0,R)$, $w^\vepsi \in C^3((0,R))$ if $f\in C^0((0,R))$, and 
$w^\vepsi$ is $C^\infty$ provided that $f$ is $C^\infty$.
\end{cor}

Recall that $w^\vepsi=r^{n-1}u^\vepsi_r$ where $u^\vepsi$ and $w^\vepsi$ are
solutions of \eqref{moment1_radial}--\eqref{moment4_radial} and
\eqref{moment1c_radial}--\eqref{moment4a_radial}. Let $w^\vepsi$ be the
unique solution to \eqref{moment1c_radial}--\eqref{moment4a_radial},
as stated in Corollary \ref{existence_cor}, define
\begin{equation}\label{existence19}
u^\vepsi(r):=g(R) - \int_r^R \frac{1}{s^{n-1}} w^\vepsi(s)\, ds
\qquad\forall r\in (0,R).
\end{equation}

We now show that $u^\vepsi$ is a unique monotone increasing classical 
solution of problem \eqref{moment1_radial}--\eqref{moment4_radial}. 

\begin{thm}\label{momnet1_existence}
Suppose $f\in C^0((0,R))$ and $f\geq 0$ in $(0,R)$, then problem 
\eqref{moment1_radial}--\eqref{moment4_radial} 
has a unique monotone increasing classical solution.  Moreover, 
$u^\vepsi$ is smooth provided that $f$ is smooth.
\end{thm}

\begin{proof}
By direct calculations one can easily show that $u^\vepsi$ defined
by \eqref{existence19} satisfies \eqref{moment1_radial}--\eqref{moment4_radial}.
Since $u^\vepsi_r> 0$ in $(0,R)$, then $u^\vepsi$ is a monotone increasing 
function. Hence, the existence is done. 

To show uniqueness, we notice that $u^\vepsi$ is 
a monotone increasing classical solution
of problem \eqref{moment1_radial}--\eqref{moment4_radial}
if and only if $w^\vepsi$ is a nonnegative classical solution of 
problem \eqref{moment1c_radial}--\eqref{moment4a_radial}.
Hence, the uniqueness of \eqref{moment1_radial}--\eqref{moment4_radial}
follows from the uniqueness of \eqref{moment1c_radial}--\eqref{moment4a_radial}.
The proof is complete.
\end{proof}

\section{Convexity of vanishing moment approximations} \label{chapter-3.3}
The goal of this section is to analyze the convexity of 
the solution $u^\vepsi$ whose existence is proved in 
Theorem \ref{momnet1_existence}. We shall prove that $u^\vepsi$
is strictly convex either in $(0,R)$ or in $(0,R-c_0\vepsi)$
for some $\vepsi$-independent positive constant $c_0$. 
From calculations of Section \ref{chapter-3.1} we know that
$D^2u^\vepsi$ only has two distinct eigenvalues $\lambda_1=u^\vepsi_{rr}$
(with multiplicity $1$) and $\lambda_2=\frac{1}{r}u^\vepsi_r$
(with multiplicity $n-1$), and we have proved that $\lambda_2\geq 0$ 
in $(0,R)$, so it is necessary to show $\lambda_1\geq 0$ in $(0,R)$ 
or in $(0,R-c_0\vepsi)$. In addition, in this section we   
derive some sharp uniform (in $\vepsi$) a priori estimates 
for the vanishing moment approximations $u^\vepsi$, which will play
an important role not only for establishing the convexity property for $u^\vepsi$
but also for proving the convergence of $u^\vepsi$ in the next section.

First, we have the following positivity result for $\Del u^\vepsi$.

\begin{thm}\label{convexity_thm1}
Let $u^\vepsi$ be the unique monotone increasing classical solution
of problem \eqref{moment1_radial}--\eqref{moment4_radial}
and define $w^\vepsi :=r^{n-1} u^\vepsi_r$. Then 
\begin{itemize}
\item[(i)] $w^\vepsi_r> 0$ in $(0,R)$, consequently, $\Del u^\vepsi >0$ 
in $(0,R)$, for all $\vepsi>0$.
\item[(ii)] For any $r_0\in (0,R)$, there exists an $\vepsi_0>0$ such that
$w^\vepsi_r>\vepsi r^{n-1}$ and $\Del u^\vepsi>\vepsi$ in $(r_0,R)$ for 
$\vepsi \in (0,\vepsi_0)$.
\end{itemize}
\end{thm}

\begin{proof} We split the proof into two steps.

{\em Step 1:} Since $u^\vepsi$ is monotone increasing and differentiable,
then $u^\vepsi_r\geq 0$ in $[0,R]$. From the derivation of 
Section \ref{chapter-3.1} we know that 
$w^\vepsi:=r^{n-1}u^\vepsi_r$ is the unique nonnegative 
classical solution of \eqref{moment1c_radial}--\eqref{moment4a_radial}.
Let $\varphi^\vepsi := w^\vepsi_r$.
By the definition of the Laplacian $\Del$ we have
\begin{equation}\label{convexity1}
\varphi^\vepsi= w^\vepsi_r =r^{n-1} u^\vepsi_{rr}
+ (n-1) r^{n-2} u^\vepsi_r =r^{n-1}\Del u^\vepsi.
\end{equation}
So $\varphi^\vepsi> 0$ in $(0,R)$ infers $\Del u^\vepsi > 0$ 
in $(0,R)$. 

To show $\varphi^\vepsi> 0$, we differentiate 
\eqref{moment1c_radial} with respect to $r$ to get
\begin{align*}
-\vepsi w^\vepsi_{rrr} 
&+\vepsi(n-1) r^{n-2} \Bigl(\frac{1}{r^{n-1}} w^\vepsi_r \Bigr)_r 
+\Bigl[ \frac{\vepsi(n-1)(n-2)}{r^2} + \frac{(w^\vepsi)^{n-1}}{r^{n(n-1)}}
\Bigr] w^\vepsi_r \\
&\quad -\frac{n-1}{r^{(n-1)^2+n}} (w^\vepsi)^n 
=r^{n-1} f(r) \qquad\mbox{in } (0,R).
\end{align*}
From \eqref{moment1c_radial}, we have
\[
\vepsi r^{n-2}\Bigl(\frac{1}{r^{n-1}} w^\vepsi_r \Bigr)_r
=\frac{1}{nr^{(n-1)^2+n}} \bigl(w^\vepsi\bigr)^n
-\frac{1}{r} L_f.
\]
Combining the above two equations yields
\begin{align}\label{convexity2a}
-\vepsi w^\vepsi_{rrr} 
+ \Bigl[ \frac{\vepsi(n-1)(n-2)}{r^2} &+ \frac{(w^\vepsi)^{n-1}}{r^{n(n-1)}}
\Bigr] w^\vepsi_r \\
& =r^{n-1}f + \frac{n-1}{r} L_f + \frac{(n-1)^2}{nr^{(n-1)^2+n}} 
\bigl(w^\vepsi\bigr)^n. \nonumber
\end{align}
Substituting $w^\vepsi_r = \varphi^\vepsi$ into the above equation we get 
\begin{align}\label{convexity2}
-\vepsi \varphi^\vepsi_{rr} 
&+\Bigl[ \frac{\vepsi(n-1)(n-2)}{r^2} + \frac{(w^\vepsi)^{n-1}}{r^{n(n-1)}}
\Bigr] \varphi^\vepsi \\
&\quad =r^{n-1}f + \frac{n-1}{r} L_f + \frac{(n-1)^2}{nr^{(n-1)^2+n}} 
\bigl(w^\vepsi\bigr)^n \geq 0
\qquad \mbox{in } (0,R), \nonumber
\end{align}
since $f, L_f, w^\vepsi\geq 0$ in $(0,R)$.
This means that $\varphi^\vepsi$ is a supersolution to a linear uniformly
elliptic differential operator. By the weak maximum principle 
we get (cf. \cite[page 329]{evans})
\[
\min_{[0,R]} \varphi^\vepsi (r) 
\geq \min\{0, \varphi^\vepsi(0), \varphi^\vepsi(R)\}
=\min\{0,0,R^{n-1} \vepsi\}=0.
\]
Here we have used the fact that 
$\varphi^\vepsi(R)=R^{n-1} \Del u^\vepsi(R)= R^{n-1} \vepsi$.
Hence, $\varphi^\vepsi\geq 0$ in $[0,R]$, so 
$\Del u^\vepsi\geq 0$ in $[0,R]$.

It follows from the strong maximum principle
(cf. \cite[Theorem 4, page 7]{Protter_Weinberger67}) that
$\varphi^\vepsi$ can not attain its nonpositive minimum value $0$ at
any point in $(0,R)$. Therefore, $\varphi^\vepsi > 0$ in $(0,R)$,  
which implies that $\Del u^\vepsi>0$ in $(0,R)$. So assertion (i) holds.

\medskip
{\em Step 2:} To show (ii), let $\psi^\vepsi:= w^\vepsi_r-\vepsi r^{n-1} 
=r^{n-1}(\Del u^\vepsi-\vepsi)$.  Using the identities 
\[
w^\vepsi_r = \psi^\vepsi + \vepsi r^{n-1},\qquad
w^\vepsi_{rrr}= \psi^\vepsi_{rr} + \vepsi (n-1)(n-2) r^{n-3},
\]
we rewrite \eqref{convexity2a} as 
\begin{align}\label{convexity2b}
-\vepsi \psi^\vepsi_{rr} 
&+\Bigl[ \frac{\vepsi(n-1)(n-2)}{r^2} + \frac{(w^\vepsi)^{n-1}}{r^{n(n-1)}}
\Bigr] \psi^\vepsi \\
&=r^{n-1}f + \frac{n-1}{r} L_f 
+ \frac{(w^\vepsi)^{n-1}[(n-1)^2 w^\vepsi -\vepsi n r^n ]}{nr^{(n-1)^2+n}}
\qquad \mbox{in } (0,R). \nonumber
\end{align}
Hence, $\psi^\vepsi$ satisfies a linear uniformly elliptic equation.

Now, on noting that $w^\vepsi\geq 0$ by (i), 
for any $r_0\in (0,R)$ (i.e., $r_0$ is away from $0$), it is easy to
see that there exists an $\vepsi_1>0$ such that the right-hand side 
of \eqref{convexity2b} is nonnegative in $(r_0,R)$ for all $\vepsi\in (0,\vepsi_1)$.
Hence, $\psi^\vepsi$ is a supersolution in $(r_0,R)$ to the uniformly elliptic 
operator on the right-hand side of \eqref{convexity2b}. By the weak maximum principle
we have (cf. \cite[page 329]{evans})
\[
\min_{[r_0,R]} \psi^\vepsi (r) 
\geq \min\{0, \psi^\vepsi(r_0), \psi^\vepsi(R)\}
=\min\{0,\Del u^\vepsi(r_0)-\vepsi, 0\}.
\]
Again, here we have used the fact that $\Del u^\vepsi(R)= \vepsi$.

Since $\Del u^\vepsi(r_0)>0$, choose 
$\vepsi_0=\min\{\vepsi_1, \frac12 \Del u^\vepsi(r_0)\}$,
then $\psi^\vepsi(r_0)=\Del u^\vepsi(r_0)-\vepsi \geq \frac12 \Del u^\vepsi(r_0)>0$
for $\vepsi\in (0,\vepsi_0)$. Thus, $\min_{[r_0,R]} \psi^\vepsi (r) \geq 0$
for $\vepsi\in (0,\vepsi_0)$. Therefore, $w^\vepsi_r\geq \vepsi r^{n-1}$,
consequently, $\Del u^\vepsi\geq \vepsi$ in $[r_0,R]$ for $\vepsi\in (0,\vepsi_0)$.

Finally, an application of the strong maximum principle
(cf. \cite[Theorem 4, page 7]{Protter_Weinberger67}) yields
that $w^\vepsi_r> \vepsi r^{n-1}$, hence $\Del u^\vepsi> \vepsi$, in $(r_0,R)$
for $\vepsi\in (0,\vepsi_0)$. The proof is complete.
\end{proof}

\begin{remark}
The proof also shows that $\vepsi_0$ decreases (resp. increases) as $r_0$ decreases
(resp. increases), and $v^\vepsi:=\Del u^\vepsi$ takes its minimum value
$\vepsi$ in $[r_0,R]$ at the right end of the interval $r=R$.
\end{remark}

With help of the positivity of $\Del u^\vepsi$, we can derive 
some better uniform estimates (in $\vepsi$) for $w^\vepsi$ and $u^\vepsi$.

\begin{thm}\label{fine_estimates}
Suppose $f\in C^0((0,R))$ and $f\geq 0$ in $(0,R)$. Let $u^\vepsi$ be
the unique monotone increasing classical solution to problem
\eqref{moment1_radial}--\eqref{moment4_radial}.  Define
$w^\vepsi=r^{n-1} u^\vepsi_r$ and $v^\vepsi= \Del u^\vepsi=u^\vepsi_{rr}
+\frac{n-1}{r} u^\vepsi_r$.  Then there holds the following estimates
(at least for sufficiently small $\vepsi>0$):
\begin{alignat*}{2}
&\mbox{\rm (i)}  &&\,\|u^\vepsi\|_{C^0([0,R])} 
+\int_0^R |u^\vepsi_r|^n\, dr \leq C_0, \\
&\mbox{\rm (ii)} &&\,\|u^\vepsi\|_{C^1([0,R])} 
+ \|w^\vepsi\|_{C^0([0,R])} \leq C_1, \\
&\mbox{\rm (iii)} &&\,\|w^\vepsi_r\|_{C^0([0,R])}\leq \frac{C_2}{\vepsi}, \\
&\mbox{\rm (iv)}  &&\|v^\vepsi\|_{C^0([r_0,R])}\leq \frac{C_3}{\vepsi r_0^{n-1}} 
\quad\forall 0<r_0\leq R, \\
&\mbox{\rm (v)} &&\,\|v^\vepsi_r\|_{C^0([r_0,R])}
\leq \frac{C_4}{\vepsi r_0^{(n-1)^2}} \quad\forall 0<r_0\leq R, \\
&\mbox{\rm (vi)} &&\, \int_0^R |w^\vepsi_r(r)|^2\, dr
+ \int_0^R r^{2(n-1)} |v^\vepsi(r)|^2  \,dr \leq \frac{C_5}{\vepsi}, \\
&\mbox{\rm (vii)} &&\,\vepsi \int_0^R r^{n-2-\alpha} |v^\vepsi(r)|^2 \, dr 
+ \int_0^R \frac{1}{r^\alpha}(u^\vepsi_r(r))^n v^\vepsi(r)\, dr 
\leq \frac{C_6}{\vepsi} \quad\forall \alpha< n-1, 
\end{alignat*}
\begin{alignat*}{2}
&\mbox{\rm (viii)} &&\,\vepsi \int_0^R r^{n-1} |v^\vepsi_r(r)|^2 \, dr 
+ \int_0^R (u^\vepsi_r(r))^{n-1} |v^\vepsi(r)|^2\, dr 
\leq  \frac{C_7}{\vepsi} \quad\mbox{for } n\geq 3, \\
&\mbox{\rm (ix)} &&\,\,\vepsi\int_0^R r^{2-\alpha} |v^\vepsi_r(r)|^2 \, dr 
+ \int_0^R r^{1-\alpha} u^\vepsi_r(r) |v^\vepsi(r)|^2\, dr\leq \frac{C_8}{\vepsi} 
\quad\mbox{for } n=2,\, \alpha<1,
\end{alignat*}
where $C_j=C_j(R,f,n)>0$ for $j=0,1,2,\cdots,8$ are $\vepsi$-independent
positive constants.
\end{thm}

\begin{proof}
We divide the proof into five steps

{\em Step 1}: Since $u^\vepsi$ is monotone increasing,
\begin{align}\label{v_eqn_1}
\max_{[0,R]} u^\vepsi(r) \leq u^\vepsi(R)=g(R).
\end{align}
We note that the above estimate also follows from 
$\Del u^\vepsi\geq 0$ and the maximum principle.

On noting that $w^\vepsi$ satisfies equation \eqref{moment1c_radial},
integrating \eqref{moment1c_radial} over $(0,R)$ and using integration by
parts on the first term on the left-hand side yield
\begin{align*} 
-\vepsi w^\vepsi_r(R) + \vepsi (n-1) \int_0^R \frac{1}{r} w^\vepsi(r)\, dr 
+ \frac{1}{n} \int_0^R \Bigl[\frac{w^\vepsi(r)}{r^{n-1}} \Bigr]^n\, dr
= \int_0^R L_f(r)\, dr.
\end{align*}

Because $w^\vepsi_r(R)=\vepsi R^{n-1}$ and $w^\vepsi\geq 0$, the above
equation and the relation $w^\vepsi=r^{n-1} u^\vepsi_r$ imply that
\begin{align} \label{v_eqn_7}
\int_0^R \Bigl|\frac{w^\vepsi(r)}{r^{n-1}} \Bigr|^n\, dr
= \int_0^R \bigl|u^\vepsi_r(r) \bigr|^n\, dr
\leq n R\bigl[L_f(R) + \vepsi^2 R^{n-2} \bigr]. 
\end{align}
It then follows from \eqref{v_eqn_1}, \eqref{v_eqn_7} and \eqref{existence19} that
\begin{align} \label{v_eqn_8}
g(R)-nR\bigl[ L_f(R) + \vepsi^2 R^{n-2} \bigr]^{\frac{1}{n}} 
\leq u^\vepsi(r) \leq g(R)
\qquad\forall r\in [0,R].
\end{align}
Hence, $u^\vepsi$ is uniformly bounded (in $\vepsi$) in $[0,R]$, and (i) holds.

\medskip
{\em Step 2}: Let 
\[
v^\vepsi:=\Del u^\vepsi = u^\vepsi_{rr} + \frac{n-1}{r} u^\vepsi_r
= \frac{1}{r^{n-1}} \bigl(r^{n-1} u^\vepsi_r\bigr)_r.
\]
By \eqref{moment1_radial} we have
\begin{align}\label{v_eqn_0}
-\vepsi \bigl( r^{n-1} v^\vepsi_r \bigr)_r  
+ \frac{1}{n} \bigl((u^\vepsi_r)^n\bigr)_r =r^{n-1} f \qquad \mbox{in } (0,R).
\end{align}
It was proved in the previous theorem that $v^\vepsi> \vepsi$ 
in $(\frac{R}2,R)$ for sufficiently small $\vepsi>0$
and it takes its minimum value $\vepsi$ at $r=R$.  Hence we have
$v^\vepsi_r(R)\leq 0$.\footnote{This is the only place in the proof 
where we may need to require $\vepsi$ to be sufficiently small.}

Integrating \eqref{v_eqn_0} over $(0,R)$ yields
\begin{align*} 
-\vepsi r^{n-1} v^\vepsi_r \Bigr|_{r=0}^{r=R}  
+ \frac{1}{n} (u^\vepsi_r)^n\Bigr|_{r=0}^{r=R} =L_f(R). 
\end{align*}
hence,
\begin{align*} 
\bigl(u^\vepsi_r(R)\bigr)^n
= n L_f(R) +\vepsi n R^{n-1} v^\vepsi_r(R) \leq  nL_f(R),
\end{align*}
therefore,
\begin{align} \label{v_eqn_2}
u^\vepsi_r(R)=
\bigl|u^\vepsi_r(R)\bigr|
\leq  \bigl(nL_f(R)\bigr)^{\frac{1}{n}}.
\end{align}
Here we have used boundary condition \eqref{moment3_radial} and
the fact that $v^\vepsi_r(R)\leq 0$ and $u^\vepsi_r\geq 0$.

By the definition of $w^\vepsi(r):=r^{n-1} u^\vepsi_r(r)$ we have
\begin{align} \label{v_eqn_3}
w^\vepsi(R)=
\bigl|w^\vepsi(R)\bigr|
\leq R^{n-1} \bigl|u^\vepsi_r(R)\bigr|
\leq R^{n-1} \bigl(nL_f(R)\bigr)^{\frac{1}{n}}.
\end{align}

Using the identity
\[
v^\vepsi(r)=\Del u^\vepsi(r)=u^\vepsi_{rr}(r) +\frac{n-1}{r} u^\vepsi_r(r),
\]
we get
\[
u^\vepsi_{rr}(R)=\Del u^\vepsi(R)-\frac{n-1}{R} u^\vepsi_r(R)
=\vepsi-\frac{n-1}{R} u^\vepsi_r(R). 
\]
Hence,
\begin{align} \label{v_eqn_4}
\bigl|u^\vepsi_{rr}(R)\bigr|
\leq \vepsi + \frac{n-1}{R} \bigl|u^\vepsi_r(R)\bigr|
\leq \vepsi + \frac{n-1}{R}  \bigl(nL_f(R)\bigr)^{\frac{1}{n}}.
\end{align}

\medskip
{\em Step 3}: From Theorem \ref{convexity_thm1} we have that 
$w^\vepsi_r(r)\geq 0$ in $(0,R)$, and hence,
$w^\vepsi$ is monotone increasing. Consequently,
\begin{align} \label{v_eqn_5}
\max_{[0,R]} w^\vepsi(r)
= \max_{[0,R]} \bigl|w^\vepsi(r)\bigr|
\leq w^\vepsi(R) \leq R^{n-1} \bigl(nL_f(R)\bigr)^{\frac{1}{n}}.
\end{align}

Evidently, \eqref{v_eqn_5} and the relation $w^\vepsi(r)=r^{n-1}u^\vepsi_r(r)$
as well as $\lim_{r\to 0^+} u^\vepsi_r(r)=0$ imply that there exists
$r_0>0$ such that
\begin{align} \label{v_eqn_6}
\max_{[0,R]} u^\vepsi_r(r)
= \max_{[0,R]} \bigl|u^\vepsi_r(r)\bigr|
\leq \frac12 + \Bigl(\frac{R}{r_0} \Bigr)^{n-1} \bigl(nL_f(R)\bigr)^{\frac{1}{n}}.
\end{align}
Hence, (ii) holds.

In addition, since $w^\vepsi_r$ satisfies the linear elliptic equation 
\eqref{convexity2a}, by the pointwise estimate for linear elliptic equations
\cite[Theorem 3.7]{Gilbarg_Trudinger01} we have
\begin{align}\label{v_eqn_6a}
\max_{[0,R]} |w^\vepsi_r(r)| \leq \vepsi R^{n-1} 
+\frac{1}{\vepsi} \Bigl( R^{n-1}\|f\|_{L^\infty} &+ (n-1) \|r^{-1}L_f\|_{\infty}\\
&+\frac{(n-1)^2}{n} \|r^{-1} (u^\vepsi_r)^n\|_{\infty} \Bigr).  \nonumber
\end{align}
Since $w^\vepsi_r=r^{n-1}\Del u^\vepsi=:r^{n-1}v^\vepsi$, it follows 
from \eqref{v_eqn_6a} that for any $r_0 >0$ there holds
\begin{align}\label{v_eqn_6b}
\max_{[r_0,R]} |v^\vepsi(r)|
&=\max_{[r_0,R]} |\Del u^\vepsi(r)| \\
&\leq \vepsi \Bigl(\frac{R}{r_0}\Bigr)^{n-1} 
+\frac{1}{\vepsi r_0^{n-1}} \Bigl(R^{n-1} \|f\|_{L^\infty} 
+ (n-1) \|r^{-1}L_f\|_{\infty} \nonumber \\
&\hskip 0.8in
+\frac{(n-1)^2}{n} \|r^{-1} (u^\vepsi_r)^n\|_{\infty} \Bigr).  \nonumber
\end{align}
Thus, (iii) and (iv) are true.

Integrating \eqref{v_eqn_0} over $(0,r)$ yields
\begin{align}\label{v_eqn_6c}
-\vepsi r^{n-1} v^\vepsi_r + \frac{1}{n} (u^\vepsi_r)^n
=L_f \qquad\mbox{in } (0,R).
\end{align}
By \eqref{v_eqn_6c} and \eqref{v_eqn_6} we conclude that
for any $r_0>0$ there holds
\begin{align} \label{v_eqn_6d}
\max_{[r_0,R]} \bigl|v^\vepsi_r(r)\bigr|
=\max_{[r_0,R]} \bigl| (\Del u^\vepsi(r))_r\bigr|
\leq \frac{1}{\vepsi}
\Bigl( 1+\Bigl(\frac{R}{r_0} \Bigr)^{n(n-1)} \Bigr) \frac{L_f(R)}{r_0^{n-1}}.  
\end{align}
So (v) holds.

\medskip
{\em Step 4}: Testing \eqref{moment1c_radial} with $w^\vepsi$ 
and integrating by parts twice on the first term on the left-hand side, we get
\begin{align*} 
&-\vepsi^2 R^{n-1} w^\vepsi(R) + \frac{\vepsi}2 \int_0^R |w^\vepsi_r(r)|^2\, dr
+\frac{\vepsi (n-1)}{2R} [w^\vepsi(R)]^2 \\
&+ \int_0^R \frac{\vepsi (n-1)}{2r^2} |w^\vepsi(r)|^2\, dr
+\int_0^R \frac{1}{nr^{n(n-1)}} |w^\vepsi(r)|^{n+1}\, dr
=\int_0^R L_f(r) w^\vepsi(r)\, dr.
\end{align*}
Combing the above equation and \eqref{v_eqn_5} we obtain
\begin{align}\label{v_eqn_9} 
\frac{\vepsi}2 \int_0^R |w^\vepsi_r(r)|^2\, dr
&+\int_0^R \frac{\vepsi (n-1)}{2r^2} |w^\vepsi(r)|^2\, dr
+\int_0^R \frac{1}{nr^{n(n-1)}} |w^\vepsi(r)|^{n+1}\, dr \\
&\leq R\bigl[\vepsi^2 R^{n-2}+ L_f(R)\bigr]\,\bigl(nL_f(R)\bigr)^{\frac{1}{n}}. 
\nonumber
\end{align}
Consequently,
\begin{align}\label{v_eqn_10} 
&\frac{\vepsi}2 \int_0^R |r^{n-1}\Del u^\vepsi(r)|^2\, dr
+\frac{\vepsi (n-1)}{2} \int_0^R r^{2(n-2)}  |u^\vepsi_r(r)|^2\, dr \\
&\qquad
+\frac{1}{n}\int_0^R r^{n-1} |u^\vepsi_r(r)|^{n+1}\, dr
\leq R\bigl[\vepsi^2 R^{n-2}+ L_f(R)\bigr]\,\bigl(nL_f(R)\bigr)^{\frac{1}{n}}. 
\nonumber
\end{align}
Hence, (vi) holds.

\medskip
{\em Step 5}: 
For any real number $\alpha < n-1$, testing \eqref{v_eqn_6c}
with $r^{-\alpha} v^\vepsi$ and using $ v^\vepsi(R)=\vepsi$ we get 
\begin{align} \label{v_eqn_11}
-\frac{\vepsi^3}{2} R^{n-1-\alpha}  
&+\frac{\vepsi (n-1-\alpha)}{2} \int_0^R r^{n-2-\alpha} |v^\vepsi(r)|^2 \, dr \\
&+ \int_0^R \frac{1}{nr^\alpha}(u^\vepsi_r(r))^n v^\vepsi(r)\, dr 
=\int_0^R \frac{1}{r^\alpha} L_f(r) v^\vepsi(r)\, dr.  \nonumber
\end{align} 
On noting that $v^\vepsi\geq 0$, $u^\vepsi_r\geq 0$, and 
\[
L_f(r) =\int_0^r s^{n-1} f(r)\, dr 
\leq \frac{r^n}{n}\|f\|_{L^\infty}, 
\]
it follows from \eqref{v_eqn_11} that
\begin{align} \label{v_eqn_12}
&\frac{\vepsi (n-1-\alpha)}{4} \int_0^R r^{n-2-\alpha} |v^\vepsi(r)|^2 \, dr 
+ \frac{1}{n}\int_0^R \frac{1}{r^\alpha}(u^\vepsi_r(r))^n v^\vepsi(r)\, dr  \\
&\hskip 0.7in \leq \frac{\vepsi^3}{2} R^{n-1-\alpha} 
+ \frac{R^{n+3-\alpha} \|f\|_{L^\infty}^2}{\vepsi n^2(n-1-\alpha)(n+3-\alpha)}\, dr
\qquad\forall  \alpha< n-1.
\nonumber
\end{align}
This gives (vii)

Recall that 
\[
v^\vepsi :=\Del u^\vepsi = u^\vepsi_{rr} + \frac{n-1}{r} u^\vepsi_r,
\]
and therefore, we can rewrite \eqref{v_eqn_0} as follows
\begin{align*}
-\vepsi \bigl( r^{n-1} v^\vepsi_r \bigr)_r  
+ (u^\vepsi_r)^{n-1} v^\vepsi 
=r^{n-1} f + \frac{n-1}{r} (u^\vepsi_r)^n
\qquad\mbox{in } (0,R).
\end{align*}
Testing the above equation with $r^\beta v^\vepsi$ for
$\beta >1-n$ and using $ v^\vepsi(R)=\vepsi$, we get
\begin{align*}
&-\vepsi^2 R^{n-1+\beta} v^\vepsi_r(R) 
+\vepsi \int_0^R r^{n-1+\beta} |v^\vepsi_r(r)|^2 \, dr 
+\frac{\vepsi^3\beta R^{n+\beta-2}}{2} \\
&\qquad
-\frac{\vepsi\beta (n+\beta-2)}{2}\int_0^R r^{n+\beta-3} |v^\vepsi(r)|^2 \, dr 
+ \int_0^R r^\beta(u^\vepsi_r(r))^{n-1} |v^\vepsi(r)|^2\, dr \\
&\qquad\qquad =\int_0^R \Bigl[ r^{n-1+\beta} f(r) 
+ \frac{n-1}{r^{1-\beta}} (u^\vepsi_r(r))^n\Bigr] v^\vepsi(r)\, dr.
\end{align*} 
Hence,
\begin{align}\label{v_eqn_13}
&-\vepsi^2 R^{n-1+\beta} v^\vepsi_r(R) 
+\vepsi \int_0^R r^{n-1+\beta} |v^\vepsi_r(r)|^2 \, dr 
+ \frac{\vepsi^3\beta R^{n+\beta-2}}{2} \\
&-\frac{\vepsi\beta (n+\beta-2)}{2}\int_0^R r^{n+\beta-3} |v^\vepsi(r)|^2 \, dr 
+ \int_0^R r^\beta(u^\vepsi_r(r))^{n-1} |v^\vepsi(r)|^2\, dr
\nonumber \\
&\qquad \leq \frac{\vepsi}{2} \int_0^R r^{n-1+\beta} |v^\vepsi_r(r)|^2 \, dr 
+\frac{1}{2\vepsi} \int_0^R r^{n-1+\beta} |f(r)|^2\, dr \nonumber \\
&\hskip 1in
+ (n-1) \int_0^R \frac{1}{r^{1-\beta}} (u^\vepsi_r(r))^n  v^\vepsi(r)\, dr.
\nonumber
\end{align}  

To continue, we consider the cases $n=2$ and $n>2$ separately.
First, for $n> 2$, since $v^\vepsi_r(R)\leq 0$, it follows 
from \eqref{v_eqn_12} with $\alpha=1$ and \eqref{v_eqn_13} with 
$\beta=0$ that 
\begin{align}\label{v_eqn_14}
&\frac{\vepsi}{2} \int_0^R r^{n-1} |v^\vepsi_r(r)|^2 \, dr 
+ \int_0^R (u^\vepsi_r(r))^{n-1} |v^\vepsi(r)|^2\, dr \\
&\qquad \leq \frac{1}{2\vepsi} \int_0^R r^{n-1} |f(r)|^2\, dr 
+ R^{n-2}\Bigl[\frac{\vepsi^3 n(n-1)}{2} 
+ \frac{R^{4} \|f\|_{L^\infty}^2}{\vepsi(n^2-4)}\, dr \Bigr]. 
\nonumber
\end{align}
When $n=2$, we note that $\alpha=1$ is not allowed in \eqref{v_eqn_12}.
Let $\alpha< 1$ be fixed in \eqref{v_eqn_12}, set $\beta=1-\alpha$
in \eqref{v_eqn_13} we get
\begin{align}\label{v_eqn_15}
&\frac{\vepsi}{2} \int_0^R r^{2-\alpha} |v^\vepsi_r(r)|^2 \, dr 
+ \int_0^R r^{1-\alpha} u^\vepsi_r(r) |v^\vepsi(r)|^2\, dr \\
&\qquad \leq \frac{1}{2\vepsi} \int_0^R r^{2-\alpha} |f(r)|^2\, dr 
+ 2 R^{1-\alpha} \Bigl[\vepsi^3 
+ \frac{R^4 \|f\|_{L^\infty}^2}{\vepsi(1-\alpha)(5-\alpha)}\, dr \Bigr]. 
\nonumber
\end{align}
Hence, (viii) and (ix) hold. The proof is complete.
\end{proof}

We now state and prove the following convexity result for 
the vanishing moment approximation $u^\vepsi$.

\begin{thm}\label{convexity_thm}
Suppose $f\in C^0((0,R))$ and there exists a 
positive constant $f_0$ such that $f\geq f_0$ on $[0,R]$.
Let $u^\vepsi$ denote the unique monotone increasing classical solution to 
problem \eqref{moment1_radial}--\eqref{moment4_radial}.  
\begin{enumerate}
\item[(i)] If $n=2,3$, then either $u^\vepsi$ is strictly 
convex in $(0,R)$ or there exists an $\vepsi$-independent positive constant 
$c_0$ such that $u^\vepsi$ is strictly convex in $(0,R-c_0\vepsi)$.
\item [(ii)] If $n>3$, then there exists a monotone decreasing sequence 
$\{s_j\}_{j\geq 0} \subset (0,R)$ and two corresponding sequences 
$\{\vepsi_j\}_{j\geq 0} \subset (0,1)$, which is also monotone deceasing,
and $\{r_j^*\}_{j\geq 0} \subset (0,R)$ satisfying $s_j\searrow 0^+$ 
as $j\to \infty$ and $u^\vepsi_{rr}(s_j)\geq 0$ and $R-r_j^*=O(\vepsi)$ 
such that for each $j\geq 0$, $u^\vepsi$ is strictly 
convex in $(s_j, r_j^*)$ for all $\vepsi\in (0,\vepsi_j)$.
\end{enumerate}
\end{thm}

\begin{proof}
We divide the proof into three steps.

\medskip
{\em Step 1}:
Let $w^\vepsi:=r^{n-1} u^\vepsi_r$ and 
$v^\vepsi:=\Del u^\vepsi=u^\vepsi_{rr} +\frac{n-1}{r} u^\vepsi_r=w^\vepsi_r$ be 
same as before, and define $\eta^\vepsi:=r^{n-1} u^\vepsi_{rr}$. On 
noting that 
\[
r^{n-1} v^\vepsi_r =\bigl(r^{n-1}u^\vepsi_{rr}\bigr)_r
-(n-1) r^{n-3} u^\vepsi_r
=\eta^\vepsi_r + \frac{1}{r} \eta^\vepsi -r^{n-2} v^\vepsi,
\]
\eqref{v_eqn_0} can be rewritten as
\begin{align}\label{v_eqn_16}
-\vepsi \eta^\vepsi_{rr} 
+\Bigl[ \frac{2\vepsi}{r^2} +\frac{(u^\vepsi_r)^{n-1}}{r^{n-1}} \Bigr]\eta^\vepsi 
=r^{n-1}f + \vepsi(3-n)r^{n-3} v^\vepsi \qquad\mbox{in } (0,R).
\end{align}
So $\eta^\vepsi$ satisfies a linear uniformly elliptic equation. 

Clearly, $\eta^\vepsi(0)=0$. We claim that there exists (at least one)
$r_1\in (0,R]$ such that $\eta^\vepsi(r_1)\geq 0$. If not, then 
$\eta^\vepsi <0$ in $(0,R]$, so is $u^\vepsi_{rr}$. This implies 
that $u^\vepsi_r$ is monotone decreasing in $(0,R]$. Since $u^\vepsi_r(0)=0$, 
hence, $u^\vepsi_r<0$ in $(0,R]$. But this contradicts with the fact 
that $u^\vepsi_r\geq 0$ in $(0,R]$. Therefore, the claim must be true.

Due to the factor $(3-n)$ in the second term on the right-hand side
of \eqref{v_eqn_16}, the situations for the cases $n\leq 3$ and 
$n>3$ are different, and need to be handled slightly different.

\medskip
{\em Step 2: The case $n=2,3$}. Since $v^\vepsi\geq 0$, hence,
\begin{equation}\label{v_eqn_17}
-\vepsi \eta^\vepsi_{rr} 
+\Bigl[ \frac{2\vepsi}{r^2} +\frac{(u^\vepsi_r)^{n-1}}{r^{n-1}} \Bigr]\eta^\vepsi 
\geq 0 \qquad\mbox{in } (0,R).
\end{equation}
Therefore, $\eta^\vepsi$ is a supersolution to a linear uniformly
elliptic differential operator.  By the weak maximum
principle (cf. \cite[page 329]{evans}) we have
\begin{equation*}
\min_{[0,r_1]} \eta^\vepsi (r) 
\geq \min\bigl\{0, \eta^\vepsi(0), \eta^\vepsi(r_1)\bigr\}
=\min\bigl\{0,\eta^\vepsi(r_1)\bigr\}=0.
\end{equation*}

Let $r_*=\max\{r_1\in (0,R]; \, \eta^\vepsi(r_1)\geq 0\}$. By the above 
argument and the definition of $r_*$ we have $\eta^\vepsi\geq 0$ 
in $[0,r_*]$, $\eta^\vepsi(r_*)=0$ if $r_*\neq R$, 
and $\eta^\vepsi< 0$ in $(r_*, R]$. If $r_*=R$, then $\eta^\vepsi\geq 0$ 
in $[0,R]$. An application of the strong maximum principle to 
conclude that $\eta^\vepsi> 0$ in $(0,R)$.  
Hence, $u^\vepsi_{rr}> 0$ in $(0,R)$. Thus, $u^\vepsi$ is strictly convex 
in $(0,R)$. So the first part of the theorem's assertion is proved.

On the other hand, if $r_*< R$, we only know that $u^\vepsi$ is strictly 
convex in $(0,r_*)$. We now prove that $R-r_*=O(\vepsi)$, which then
justifies the remaining part of the theorem's assertion.

By \eqref{v_eqn_16} and the above setup we have
\begin{equation*} 
-\vepsi \eta^\vepsi_{rr} 
\geq r^{n-1}f \geq f_0 r^{n-1} \qquad\mbox{in } (r_*,R).
\end{equation*}
Integrating the above inequality over $(r_*, r)$ for $r\leq R$
and noting that $\eta^\vepsi_r(r_*)\leq 0$ we get
\[
-\vepsi \eta^\vepsi_r \geq \frac{f_0}{n} (r^n-r_*^n) \qquad\mbox{in } (r_*,R).
\]
Integrating again over $(r_*,R)$ and using the fact that $\eta^\vepsi(r_*)=0$ and
the following algebraic inequality
\[
\frac{1}{n+1} \frac{R^{n+1}-r_*^{n+1}}{R-r_*}- r_*^n \geq \frac{1}{n+1} R^n
\]
we arrive at
\[
-\vepsi R^{n-1} u_{rr}(R)=-\vepsi \eta^\vepsi(R) 
\geq \frac{f_0 R^n(R-r_*)}{n(n+1)}.
\]
It follows from \eqref{v_eqn_4} that
\begin{align*}
R-r_* &\leq \frac{\vepsi n(n+1) |u_{rr}(R)|}{R f_0} \\
&\leq \frac{\vepsi n(n+1)}{R^2 f_0} \Bigl[ \vepsi R 
+(n-1)\bigl(nL_f(R)\bigr)^{\frac{1}{n}} \Bigr] =:c_0\vepsi. \nonumber
\end{align*}
Thus, 
\begin{align} \label{v_eqn_18}
R-r_*=O(\vepsi),
\end{align}
and $u^\vepsi$ is strictly convex in $(0,R-c_0\vepsi)$.

\medskip
{\em Step 3: The case $n>3$:} First, By the argument used in {\em Step 1}, it is
easy to show that $\eta^\vepsi$ can not be strictly negative in the whole of 
any neighborhood of $r=0$. Thus, there exists a monotone decreasing sequence
$\{s_j\}_{j\geq 0} \subset (0,R)$ such that $s_j\searrow 0^+$ as $j\to \infty$
and $\eta^\vepsi(s_j)\geq 0$.

Second, we note that
\begin{align*}
\eps(3-n)r^{n-3}v^\eps 
&= \eps(3-n)r^{n-3}\left(\ue_{rr}+\frac{n-1}r \ue_r\right)\\
&=\eps(3-n)\left(\frac{\eta^\eps}{r^2}+(n-1)r^{n-4}\ue_r\right)
\end{align*}
Using this identity in \eqref{v_eqn_16}, we have
\begin{align*}
-\vepsi \eta^\vepsi_{rr} 
+\Bigl[ \frac{(n-1)\vepsi}{r^2} +\frac{(u^\vepsi_r)^{n-1}}{r^{n-1}} \Bigr]\eta^\vepsi 
=r^{n-4} [r^3 f + \vepsi(3-n)(n-1)\ue_r] \quad\mbox{in } (0,R).
\end{align*}
By (ii) of Theorem \ref{fine_estimates} we know that $u^\vepsi_r$ is uniformly
bounded in $[0,R]$.  Then for each $s_j$ there exists an $\vepsi_j>0$ (without
loss of the generality, choose $\vepsi_j< \vepsi_{j-1}$) 
such that for $\vepsi\in (0, \vepsi_j)$ 
\[
[r^3 f + \vepsi(3-n)(n-1)\ue_r] \ge 0  \quad \mbox{in } (s_j,R).
\]
Hence, $\eta^\vepsi$ is a supersolution to a linear uniformly elliptic
operator on $(s_j,R)$ for $\vepsi<\vepsi_j$. 

Third, for each fixed $j\geq 1$, let 
$r_j^*=\max\{r\in (s_j,R]; \, \eta^\vepsi(r)\geq 0\}$. 
Trivially, by the construction, $r_j^*\geq s_{j-1} > s_j$. 
By the weak maximum principle (cf. \cite[page 329]{evans}) we have
\begin{equation*}
\min_{[s_j,r_j^*]} \eta^\vepsi (r) 
\geq \min\bigl\{0, \eta^\vepsi(s_j), \eta^\vepsi(r_j^*)\bigr\}\geq 0.
\end{equation*}

Finally, repeating the argument of {\em Step 2:}, we conclude that $\ue$ is either
strictly convex in $(s_j,R)$ or in $(s_j,r_j^*)$ with $R-r_j^*=O(\vepsi)$ 
for $\vepsi\in (0,\vepsi_j)$. The proof is now complete.
\end{proof}


\section{Convergence of vanishing moment approximations} \label{chapter-3.4}
The goal of this section is to show that the solution $u^\vepsi$ 
of problem \eqref{moment1_radial}--\eqref{moment4_radial} converges
to the convex solution $u$ of problem \eqref{MA1_radial}--\eqref{MA3_radial}.
We present two different proofs for the convergence.  The 
first proof is based on the variational formulations of both problems.
The second proof, which can be extended to more general non-radially 
symmetric case \cite{Feng1}, is done in the viscosity solution setting
\cite{Crandall_Ishii_Lions92}. Both proofs mainly rely on two key ingredients. 
The first is the solution estimates obtained in Theorem \ref{fine_estimates}, 
the second is the uniqueness of solutions to problem 
\eqref{MA1_radial}--\eqref{MA3_radial}.

\begin{thm}\label{convergence_thm1}
Suppose $f\in C^0((0,R))$ and there exists a
positive constant $f_0$ such that $f\geq f_0$ in $[0,R]$.
Let $u$ denote the convex (classical) solution to problem 
\eqref{MA1_radial}--\eqref{MA3_radial} and $u^\vepsi$ be the 
monotone increasing classical solution to 
problem \eqref{moment1_radial}--\eqref{moment4_radial}. Then

\begin{enumerate}
\item[{\rm (i)}] $u^0=\lim_{\vepsi\to 0^+} u^\vepsi$ exists pointwise
and $u^\vepsi$ converges to $u^0$ uniformly in every compact subset 
of $(0,R)$ as $\vepsi\to 0^+$. Moreover, $u^0$ is strictly convex in 
every compact subset, hence, it is strictly convex in $[0,R]$.

\item[{\rm (ii)}] $u^\vepsi_r$ converges to $u^0_r$ weakly $*$  
in $L^\infty((0,R))$ as $\vepsi\to 0^+$. 

\item[{\rm (iii)}] $u^0\equiv u$.
\end{enumerate}

\end{thm}

\begin{proof}
It follows from (ii) of Theorem \ref{fine_estimates} that
$\|u^\vepsi\|_{C^1([0,R])}$ is uniformly bounded in $\vepsi$, 
then $\{u^\vepsi\}_{\vepsi\geq 0}$ is uniformly equicontinuous. By
Arzela-Ascoli compactness theorem (cf. \cite[page 635]{evans}) 
we conclude that there exists a subsequence of $\{u^\vepsi\}_{\vepsi\geq 0}$ 
(still denoted by the same notation) and $u^0\in C^1([0,R])$ such that
\begin{alignat*}{2}
u^\vepsi &\longrightarrow u^0
&&\qquad\mbox{uniformly in every compact set $E\subset (0,R)$ as } \vepsi\to 0^+,\\
u^\vepsi_r &\longrightarrow u^0_r
&&\qquad\mbox{weakly $*$ in $L^\infty((0,R))$ as } \vepsi\to 0^+,
\end{alignat*}
and $u^\vepsi(R)=g(R)$ implies that $u^0(R)=g(R)$. 

In addition, by Theorem \ref{convexity_thm} we have that
for every compact subset $E\subset (0,R)$ there exists $\vepsi_0>0$ such that
$E\subset (0, R-c_0\vepsi)$ and $u^\vepsi$ is strictly convex
in $(0, R-c_0\vepsi)$ for $\vepsi<\vepsi_0$. 
It follows from a well-known property of convex functions
(cf. \cite{HUT01}) that $u^0$ must be strictly convex in $E$ and
$u^0\in C^{1,1}_{\mbox{\tiny loc}}((0,R))$.

Testing equation \eqref{v_eqn_0} with an arbitrary function
$\chi\in C^2_0((0,R))$ yields
\begin{align}\label{v_eqn_20}
\vepsi\int_0^R r^{n-1} v^\vepsi_r(r) \chi_r(r) \,dr  
-\frac{1}{n} \int_0^R \bigl(u^\vepsi_r(r)\bigr)^n \chi_r(r)\,dr
=\int_0^R r^{n-1} f(r) \chi(r)\, dr,
\end{align}
where as before $v^\vepsi=\Del u^\vepsi=u^\vepsi_{rr} +\frac{n-1}{r} u^\vepsi_r$.

By Schwartz inequality and (vi) of Theorem \ref{fine_estimates} we have
\begin{align*}
\vepsi\int_0^R r^{n-1} v^\vepsi_r(r) \chi_r(r) \,dr 
&=-\vepsi \int_0^R r^{n-1} v^\vepsi(r) \Bigl[ \chi_{rr}(r) 
+ \frac{n-1}{r} \chi_r(r)\Bigr]\, dr \\
&\leq \vepsi\Bigl(\int_0^R r^{2(n-1)}|v^\vepsi(r)|^2\,dr\Bigr)^{\frac12}
\Bigl(\int_0^R |\Del \chi(r)|^2 \,dr\Bigr)^{\frac12} \\
&\leq \sqrt{\vepsi\, C_5} \Bigl(\int_0^R |\Del \chi(r)|^2 \,dr\Bigr)^{\frac12}
\to 0 \quad\mbox{as } \vepsi\to 0^+.
\end{align*}

Setting $\vepsi \to 0^+$ in \eqref{v_eqn_20} and using the Lebesgue 
Dominated Convergence Theorem yield
\begin{align}\label{v_eqn_21}
-\frac{1}{n} \int_0^R \bigl(u^0_r(r)\bigr)^n \chi_r(r)\,dr
=\int_0^R r^{n-1} f(r) \chi(r)\, dr \qquad\forall \chi\in C^1_0((0,R)).
\end{align}
It also follows from a standard test function argument that
\[
u^0_r(0)=0.
\]
This means that $u^0\in C^1([0,R])\cap C^{1,1}_{\mbox{\tiny loc}}((0,R))$ 
is a convex weak solution to problem \eqref{MA1_radial}--\eqref{MA3_radial}.  
By the uniqueness of convex solutions of problem 
\eqref{MA1_radial}--\eqref{MA3_radial}, there must hold $u^0\equiv u$.

Finally, since we have proved that every convergent subsequence
of $\{u^\vepsi\}_{\vepsi\geq 0}$ converges to the unique convex 
classical solution $u$ of  problem 
\eqref{MA1_radial}--\eqref{MA3_radial}, the whole sequence 
$\{u^\vepsi\}_{\vepsi\geq 0}$ must converge to $u$.
The proof is complete.
\end{proof}

Next, we state and prove a different version of Theorem \ref{convergence_thm1}.
The difference is that we now only assume problem
\eqref{MA1_radial}--\eqref{MA3_radial} has a unique 
strictly convex viscosity solution and so the proof must be
adapted to the viscosity solution framework. 

\begin{thm}\label{convergence_thm2}
Suppose $f\in C^0((0,R))$ and there exists a
positive constant $f_0$ such that $f\geq f_0$ on $[0,R]$.
Let $u$ denote the strictly convex viscosity solution to problem
\eqref{MA1_radial}--\eqref{MA3_radial} and $u^\vepsi$ be the
monotone increasing classical solution to
problem \eqref{moment1_radial}--\eqref{moment4_radial}. Then
the statements of Theorem \ref{convergence_thm1} still hold.
\end{thm}

\begin{proof}
Except the step of proving the variational formulation \eqref{v_eqn_21}, 
all other parts of the proof of Theorem \ref{convergence_thm1} are
still valid.  So we only need to show that $u^0$ is 
a viscosity solution of problem \eqref{MA1_radial}--\eqref{MA3_radial},
which is verified below by the definition of viscosity solutions. 

Let $\phi\in C^2([0,R])$ be strictly convex\footnote{In fact,
$\phi$ can be taken as a convex quadratic polynomial
(cf. \cite{Caffarelli_Cabre95,Gutierrez01}).}, and suppose that $u^0-\phi$ 
has a local maximum at a point $r_0\in (0,R)$, that is, 
there exists a (small) number $\delta_0>0$ such that
$(r_0-\delta_0, r_0+\delta_0)\subset\subset (0,R)$ and 
\[
u^0(r)-\phi(r) \leq u^0(r_0)-\phi(r_0)
\qquad\forall r\in (r_0-\delta_0, r_0+\delta_0).
\]

Since $u^\vepsi$ (which still denotes a subsequence) converges 
to $u^0$ uniformly in $[r_0-\delta_0, r_0+\delta_0]$,
then for sufficiently small $\vepsi>0$, there exists
$r_\vepsi\in (0,R)$ such that $r_\vepsi\to r_0$ 
as $\vepsi\to 0^+$ and $u^\vepsi-\phi$ has a local 
maximum at $r_\vepsi$ (see \cite[Chapter 10]{evans}
for a proof of the claim). By elementary calculus, we have
\[
u^\vepsi_r(r_\vepsi)=\phi_r(r_\vepsi),\qquad
u^\vepsi_{rr}(r_\vepsi)\leq \phi_{rr}(r_\vepsi).
\]
Because both $u^\vepsi$ and $\phi$ are strictly convex,
there exists a (small) constant $\rho_0>0$ such that
for sufficiently small $\vepsi>0$ 
\begin{align*}
u^\vepsi_{rr}(r)\leq \phi_{rr}(r) \qquad\forall r\in (r_0-\rho,r_0+\rho),\,\,
\rho<\rho_0,
\end{align*}
which together with an application of Taylor's formula implies that
\begin{align*}
u^\vepsi_r(r)=\phi_r(r)+O(|r-r_\vepsi|)\qquad\forall r\in (r_0-\rho,r_0+\rho),\,\,
\rho<\rho_0.
\end{align*}

Let $\chi\in C^2_0((r_0-\rho,r_0+\rho))$ with $\chi \geq 0$ and
$\chi(r_0)>0$.  Testing \eqref{v_eqn_0} with $\chi$ yields
\begin{align}\label{v_eqn_22}
&\frac{1}{2n\rho}\int_{r_0-\rho}^{r_0+\rho} 
\bigl((\phi_r(r))^n\bigr)_r \chi(r) \,dr  
=\frac{1}{2\rho}\int_{r_0-\rho}^{r_0+\rho} 
(\phi_r(r))^{n-1} \phi_{rr}(r) \chi(r) \,dr \\
&\qquad \geq \frac{1}{2\rho }\int_{r_0-\rho}^{r_0+\rho} 
\bigl[(u^\vepsi_r(r))^{n-1} + O(|r-r_\vepsi|^{n-1})\bigr] 
u^\vepsi_{rr}(r) \chi(r)\,dr \nonumber \\
&\qquad =\frac{1}{2n\rho}\int_{r_0-\rho}^{r_0+\rho} 
\Bigl[ \bigl((u^\vepsi_r(r))^n\bigr)_r +O(|r-r_\vepsi|^{n-1})u^\vepsi_{rr}(r)
\Bigr] \chi(r) \,dr \nonumber \\
&\qquad \geq \frac{1}{2\rho }\int_{r_0-\rho}^{r_0+\rho} 
r^{n-1} \bigl[f(r) \chi(r)+\vepsi v^\vepsi_r(r) \chi_r(r) \bigr]\,dr \nonumber \\
&\hskip 1in
-C_9\rho^{n-2}\int_{r_0-\rho}^{r_0+\rho} u^\vepsi_r(r)
\chi_r(r)\,dr \nonumber 
\end{align}
for some positive $\rho$-independent constant $C_9$. Here we have
used the fact that $u^\vepsi_{rr}\geq 0, \chi\geq 0$ in 
$[r_0-\rho, r_0+\rho]$ to get the last inequality.

From (vi) of Theorem \ref{fine_estimates}, we have
\begin{align}\label{v_eqn_23}
&\vepsi  \int_{r_0-\rho}^{r_0+\rho} r^{n-1} v^\vepsi_r(r) \chi_r(r) \,dr \\
&\qquad = -\vepsi  \int_{r_0-\rho}^{r_0+\rho} r^{n-1} v^\vepsi(r)\Bigl[
\chi_{rr}(r) +\frac{n-1}{r} \chi_r(r) \Bigr]\,dr \nonumber \\
&\qquad \leq \vepsi \Bigl(\int_0^R r^{2(n-1)} |v^\vepsi(r)|^2\, dr\Bigr)^{\frac12}
\Bigl(\int_0^R |\Del \chi(r)|^2\, dr\Bigr)^{\frac12} \nonumber\\
&\qquad 
\leq \sqrt{\vepsi\, C_5} \Bigl(\int_0^R |\Del \chi(r)|^2\, dr\Bigr)^{\frac12}. 
\nonumber
\end{align}

Setting $\vepsi\to 0^+$ in \eqref{v_eqn_22} and using 
\eqref{v_eqn_23} we get
\begin{align}\label{v_eqn_24}
&\frac{1}{2\rho}\int_{r_0-\rho}^{r_0+\rho} 
(\phi_r(r))^{n-1} \phi_{rr}(r) \chi(r) \,dr  \\
&\qquad \geq \frac{1}{2\rho }\int_{r_0-\rho}^{r_0+\rho} r^{n-1}f(r)\chi(r)\,dr
-C_9\rho^{n-2}\int_{r_0-\rho}^{r_0+\rho} u^0_r(r)
\chi_r(r)\,dr.  \nonumber
\end{align}
Where we have used the fact that $u^\vepsi_r$ converges to $u^0_r$ 
weakly $*$ in $L^\infty((0,R))$ to pass to the limit in the
last term on the right-hand side.

Finally, letting $\rho\to 0^+$ in \eqref{v_eqn_24} and 
using the Lebesgue-Besicovitch Differentiation Theorem
(cf. \cite{evans}) we have
\[
(\phi_r(r_0))^{n-1} \phi_{rr}(r_0) \chi(r_0) 
\geq  r_0^{n-1} f(r_0) \chi(r_0).
\]
Hence,
\[
\Bigl[\frac{\phi_r(r_0)}{r_0}\Bigr]^{n-1} \phi_{rr}(r_0) \geq f(r_0),
\]
so $u^0$ is a viscosity subsolution to equation \eqref{MA1_radial}.

Similarly, we can show that if $u^0-\phi$ assumes a local minimum
at $r_0\in (0,R)$ for a strictly convex function $\phi\in C^2_0((0,R))$,
there holds
\[
\Bigl[\frac{\phi_r(r_0)}{r_0}\Bigr]^{n-1} \phi_{rr}(r_0) \leq f(r_0),
\]
so $u^0$ is also a viscosity supersolution to equation \eqref{MA1_radial}.
Thus, it is a viscosity solution. The proof is complete.
\end{proof}

\section{Rates of convergence} \label{chapter-3.5}
In this section,  we derive rates of convergence for 
$u^\vepsi$ in various norms. Here we consider two cases,
namely, the $n$-dimensional radially symmetric case and 
the general $n$-dimensional case, under different assumptions. 
In both cases, the linearization of the Monge-Amp\`ere operator 
is explicitly exploited, and it plays a key role in our proofs. 

\begin{thm}\label{convergence_rate_thm1}
Let $u$ denote the strictly convex classical solution to problem
\eqref{MA1_radial}--\eqref{MA3_radial} and $u^\vepsi$ be the
monotone increasing classical solution to
problem \eqref{moment1_radial}--\eqref{moment4_radial}. Then
there holds the following estimates:
\begin{align}\label{rate_1}
\Bigl(\int_0^R \theta^\vepsi(r) |u_r(r)-u^\vepsi_r(r)|^2\, dr\Bigr)^{\frac12} 
&\leq \vepsi^{\frac34}\, C_{10},\\
\Bigl(\int_0^R r^{n-1}|\Del u(r)-\Del u^\vepsi(r)|^2\, dr\Bigr)^{\frac12} 
&\leq \vepsi^{\frac14}\,C_{11}, \label{rate_2} 
\end{align}
where $C_j=C_j(\|r^{n-1} \Del u_r\|_{L^2})$ for $j=10,11$ are two 
positive $\vepsi$-independent constants, and 
\begin{equation}\label{rate_6}
\theta^\vepsi(r): =\frac{(u_r)^n -(u^\vepsi_r)^n}{u-u^\vepsi}
=\sum_{j=0}^{n-1} (u_r(r))^j (u^\vepsi_r(r))^{n-1-j} > 0
\quad\mbox{\rm in } (0,R].
\end{equation}

\end{thm}

\begin{proof}
Let
\begin{alignat*}{2}
v:=\Del u=u_{rr}-\frac{n-1}{r} u_r, \quad 
v^\vepsi:=\Del u^\vepsi=u^\vepsi_{rr}-\frac{n-1}{r} u^\vepsi_r,\quad
e^\vepsi:=u-u^\vepsi. 
\end{alignat*}
On noting that \eqref{moment1_radial} can be written into \eqref{v_eqn_0},
multiplying \eqref{MA1_radial} by $r^{n-1}$ and subtracting the resulting 
equation from \eqref{v_eqn_0} yield the following error equation:
\begin{align}\label{rate_4}
\vepsi \bigl(r^{n-1} v^\vepsi_r\bigr)_r 
+\frac{1}{n} \bigl[ (u_r)^n -(u^\vepsi_r)^n \bigr]_r =0
\qquad\mbox{in } (0,R).
\end{align}

Testing \eqref{rate_4} with $e^\vepsi$ using boundary condition
$e^\vepsi_r(0)=e^\vepsi(R)=0$ we get 
\begin{align}\label{rate_5}
\vepsi \int_0^R r^{n-1} v^\vepsi_r(r) e^\vepsi_r(r)\,dr 
+\frac{1}{n} \int_0^R \theta^\vepsi(r) |e^\vepsi_r(r)|^2\,dr=0,
\end{align}
where $\theta^\vepsi$ is defined by \eqref{rate_6}.

Integrating by parts on the first term of \eqref{rate_5}
and rearranging terms we get
\begin{align}\label{rate_7}
\vepsi \int_0^R r^{n-1} |\Del e^\vepsi(r)|^2\, dr
&+\frac{1}{n} \int_0^R \theta^\vepsi(r) |e^\vepsi_r(r)|^2\, dr \\
&=\vepsi R^{n-1} \Del e^\vepsi(R) e^\vepsi_r(R) 
-\vepsi \int_0^R r^{n-1} v_r(r) e^\vepsi_r(r)\,dr.
\nonumber
\end{align}

We now bound the two terms on the right-hand side as follows. 
First, for the second term, a simple application of the Schwarz 
and Young's inequalities gives
\begin{align}\label{rate_8}
\vepsi \int_0^R r^{n-1} v_r(r) e^\vepsi_r(r)\,dr
&\leq \frac{1}{4n} \int_0^R \theta^\vepsi(r) |e^\vepsi_r(r)|^2\,dr \\
&\hskip 1in
+\vepsi^2 n \int_0^R \frac{r^{2(n-1)}}{\theta^\vepsi(r)} |v_r(r)|^2\,dr.
\nonumber
\end{align}
Second, to bound the first term on the right-hand side of
\eqref{rate_7}, we use the boundary condition $v^\vepsi(R)=\vepsi$ to get 
\[
|\Del e^\vepsi(R)| =|v(R)-v^\vepsi(R)|=|v(R)-\vepsi|\leq |v(R)|+1 =:M,
\]
and
\begin{align*}
\bigl|R^{n-1} e^\vepsi_r(R)\bigr|^2 
&= \int_0^R \bigl((r^{n-1} e^\vepsi_r(r))^2\bigr)_r\, dr
= 2\int_0^R  r^{2(n-1)} e^\vepsi_r(r)\, \Del e^\vepsi(r) \, dr \\
&\leq 2\Bigl(\int_0^R  r^{n-1} |\Del e^\vepsi(r)|^2\, dr\Bigr)^{\frac12}
 \Bigl(\int_0^R  r^{3(n-1)} |e^\vepsi_r(r)|^2\, dr\Bigr)^{\frac12}. 
\end{align*}
Hence by Young's inequality, we get
\begin{align}\label{rate_9}
&|\vepsi R^{n-1} \Del e^\vepsi(R) e^\vepsi_r(R)| \\
&\qquad \leq \sqrt{2}\vepsi M 
\Bigl(\int_0^R  r^{n-1} |\Del e^\vepsi(r)|^2\, dr\Bigr)^{\frac14}
\Bigl(\int_0^R  r^{3(n-1)} |e^\vepsi_r(r)|^2\, dr\Bigr)^{\frac14} \nonumber \\
&\qquad \leq \frac{\vepsi}{2} \int_0^R  r^{n-1} |\Del e^\vepsi(r)|^2\, dr
+ 2\vepsi M^{\frac43} \Bigl(\int_0^R r^{3(n-1)}|e^\vepsi_r(r)|^2\,dr\Bigr)^{\frac13} \nonumber\\
&\qquad\leq \frac{\vepsi}{2} \int_0^R  r^{n-1} |\Del e^\vepsi(r)|^2\, dr
 + \frac{1}{4n} \int_0^R \theta^\vepsi(r) |e^\vepsi_r(r)|^2\,dr 
 + \vepsi^{\frac32} n M^2 C \nonumber
\end{align}
for some $\vepsi$-independent constant $C=C(f,R,n)>0$.

Combining \eqref{rate_7}--\eqref{rate_9} yields
\begin{align}\label{rate_10}
\vepsi \int_0^R r^{n-1} |\Del e^\vepsi(r)|^2\, dr
&+\frac{1}{n} \int_0^R \theta^\vepsi(r) |e^\vepsi_r(r)|^2\, dr \\
&\leq 2\vepsi^2 n \int_0^R \frac{r^{2(n-1)}}{\theta^\vepsi(r)} |v_r(r)|^2\,dr
+\vepsi^{\frac32} n M^2 C. \nonumber
\end{align}
Thus, \eqref{rate_1} and \eqref{rate_2} follow from the fact that 
$\|r^{n-1} (\theta^\vepsi)^{-1}\|_{L^\infty}<\infty$.
\end{proof}

\begin{cor}\label{convergence_rate_thm1a}
Inequality \eqref{rate_1} implies that there exists an $\vepsi$-independent
constant $C>0$ such that
\begin{align}\label{rate_1a}
\Bigl(\int_0^R r^{n-1}|u_r(r)-u^\vepsi_r(r)|^2\, dr\Bigr)^{\frac12} 
\leq \vepsi^{\frac34}\, C C_{10}.
\end{align}
\end{cor}

Since the proof is simple, we omit it.

\begin{thm}\label{convergence_rate_thm2}
Under the assumptions of Theorem \ref{convergence_rate_thm1},
there also holds the following estimate:
\begin{align}\label{rate_3}
\Bigl(\int_0^R r^{n-1}|u(r)-u^\vepsi(r)|^2\, dr\Bigr)^{\frac12} \leq \vepsi\, C_{12}
\end{align}
for some positive $\vepsi$-independent constant 
$C_{12}=C_{12}(R,n,u,C_{11})$.

\end{thm}

\begin{proof}
Let $\theta^\vepsi$ be defined by \eqref{rate_6}, and  $e^\vepsi$, $v$ and
$v^\vepsi$ be same as in Theorem \ref{convergence_rate_thm1}. Consider the 
following auxiliary problem:
\begin{align}\label{MA1_radial_lin}
\bigl( \theta^\vepsi \phi_r \bigr)_r 
&= n r^{n-1} e^\vepsi  \qquad\mbox{in } (0,R),\\
\phi(R) &=0, \label{MA2_radial_lin} \\
\phi_r(0) &=0. \label{MA3_radial_lin}
\end{align}
We note that the left-hand side of \eqref{MA1_radial_lin} is 
the linearization of \eqref{MA1_radial} at $\theta^\vepsi$.

Since $\theta^\vepsi>0$ in $(0,R]$, then \eqref{MA1_radial_lin} is a
linear elliptic equation. Using the fact that  
$c_1\geq r^{n-1}(\theta^\vepsi)^{-1} \geq c_0>0$ in $[0,R]$
for some $\vepsi$-independent positive constants $c_0$ and $c_1$,
it is easy to check that problem \eqref{MA1_radial_lin}--\eqref{MA3_radial_lin} 
has a unique classical solution $\phi$. Moreover,
\begin{align} \label{rate_3a}
\int_0^R r^{n-1} |\phi_{rr}(r)|^2\, dr 
+\int_0^R r^{n-1} |\phi_r(r)|^2\, dr
\leq \hat{C} \int_0^R r^{n-1} |e^\vepsi(r)|^2 \,dr 
\end{align}
for some $\vepsi$-independent constant $\hat{C}=\hat{C}(f,R,n,c_0,c_1)>0$.

Testing \eqref{MA1_radial_lin} by $e^\vepsi$, 
using the facts that $\phi_r(0)=\phi(R)=0$, $e^\vepsi(R)=0$ and 
$v^\vepsi(R)=\vepsi$ as well as error equation  \eqref{rate_4} we get
\begin{align}\label{rate_3b}
\int_0^R r^{n-1} |e^\vepsi(r)|^2 \, dr 
&=-\frac{1}{n}\int_0^R \theta^\vepsi(r) \phi_r(r) e^\vepsi_r(r)\, dr \\
&=\vepsi \int_0^R r^{n-1} v^\vepsi_r(r) \phi_r(r) \, dr \nonumber \\
&=\vepsi  R^{n-1} v^\vepsi(R) \phi_r(R) 
  - \vepsi \int_0^R  r^{n-1} v^\vepsi(r) \Del \phi(r) \, dr \nonumber \\
&=\vepsi^2  R^{n-1} \phi_r(R) 
  +\vepsi \int_0^R  r^{n-1} [v(r)-v^\vepsi(r)] \Del \phi(r) \, dr \nonumber \\
&\hskip 1.0in 
  -\vepsi \int_0^R  r^{n-1} v(r) \Del \phi(r) \, dr, \nonumber
\end{align}
where we have used the short-hand notation $\Del\phi=r^{n-1}[ \phi_{rr} 
+(n-1) r^{-1} \phi_r ]$.

For each term on the right-hand side of \eqref{rate_3b} we have 
the following estimates:
\begin{align*}
&\vepsi^2 R^{n-1} \phi_r(R) 
= \frac{\vepsi^2}{R} \int_0^R (r^n \phi_r(r))_r \,dr \\
&\hskip 0.5in 
\leq \vepsi^2 R^{\frac{n}{2}} 
\Bigl(\int_0^R r^{n-1} |\phi_{rr}(r)|^2\, dr\Bigr)^{\frac12}
+\vepsi^2 \sqrt{n} R^{\frac{n-2}2}\Bigl(\int_0^R r^{n-1} |\phi_r(r)|^2\, dr\Bigr)^{\frac12},\\
&\vepsi \int_0^R  r^{n-1} [v(r)-v^\vepsi(r)] \Del \phi(r) \, dr \\
&\hskip 0.5in
\leq \vepsi \Bigl( \int_0^R r^{n-1} |v(r)-v^\vepsi(r)|^2\, dr\Bigr)^{\frac12}
\Bigl(\int_0^R r^{n-1} |\Del\phi(r)|^2\, dr\Bigr)^{\frac12}\\
&\hskip 0.5in
\leq \vepsi^{\frac54} C_{11} 
\Bigl(\int_0^R r^{n-1} |\Del\phi(r)|^2\, dr\Bigr)^{\frac12},\\
&-\vepsi \int_0^R  r^{n-1} v(r) \Del \phi(r) \, dr
\leq \vepsi\Bigl(\int_0^R r^{n-1} |v(r)|^2\, dr \Bigr)^{\frac12}
 \Bigl(\int_0^R r^{n-1} |\Del\phi(r)|^2\, dr \Bigr)^{\frac12}.  
\end{align*}
Substituting the above estimates into \eqref{rate_3b} and using
\eqref{rate_3a} we get
\begin{align}\label{rate_3c}
&\int_0^R r^{n-1} |e^\vepsi(r)|^2 \, dr \\
&\qquad
\leq \vepsi^2 R^{\frac{n}2} \left\{ 
\Bigl( \int_0^R r^{n-1} |\phi_{rr}(r)|^2\, dr\Bigr)^{\frac12} 
+\frac{\sqrt{n}}{R} \Bigl( \int_0^R r^{n-1} |\phi_r(r)|^2\, dr\Bigr)^{\frac12} \right\}
\nonumber \\
&\hskip 1in
+ \vepsi \bigl(\vepsi^{\frac14} C_{11}+C_u\bigr)
 \Bigl(\int_0^R r^{n-1} |\Del\phi(r)|^2\, dr \Bigr)^{\frac12} 
\nonumber \\
&\qquad \leq 4\vepsi \bigl(\vepsi R^{\frac{n}2} +\vepsi \sqrt{n} R^{-1} 
+\vepsi^{\frac14} C_{11} + C_u\bigr) \hat{C} 
\Bigl( \int_0^R r^{n-1} |e^\vepsi(r)|^2 \, dr \Bigr)^{\frac12} \nonumber
\end{align}
for some $\vepsi$-independent constant $C_u=C(u)>0$.

Hence, by \eqref{rate_3c} we conclude that \eqref{rate_3} holds 
with $C_{12}= 4\bigl(\vepsi R^{\frac{n}2} + \vepsi \sqrt{n} R^{-1}
+\vepsi^{\frac14} C_{11} + C_u\bigr) \hat{C}$.  The proof is complete.
\end{proof}

\begin{remark}
The argument used in the above proof is so-called duality argument,
which has been the main technique and used extensively in the finite 
element error analysis to derive error bounds in norms lower than the 
energy norm of the underlying PDE problem (cf. \cite{Brenner,Ciarlet78} 
and the references therein).  However, as far as we know, 
the duality argument is rarely (maybe has never been) used to derive 
error estimates in a PDE setting as done in the above proof.
\end{remark}

Since the proofs of Theorem \ref{convergence_rate_thm1} and
\ref{convergence_rate_thm2} only rely on the ellipticity of 
the linearization of the Monge-Amp\`ere operator,  hence,
the results of both theorems can be easily extended to
the general Monge-Amp\`ere problem \eqref{MAeqn1}--\eqref{MAeqn2}
and its moment approximation \eqref{moment1}--\eqref{moment3}$_1$
\footnote{This observation was pointed out to the first author by 
Professor Haijun Wu of Nanjing University, China, and the proof
for \eqref{rate_14} and \eqref{rate_15} is essentially due to him.}.

\begin{thm}\label{convergence_rate_thm3}
Let $u$ denote the strictly convex viscosity solution to problem
\eqref{MAeqn1}--\eqref{MAeqn2} and $u^\vepsi$ be a 
classical solution to problem \eqref{moment1}--\eqref{moment3}$_1$.
Assume $u\in W^{2,\infty}(\Ome)\cap H^3(\Ome)$ and $u^\vepsi$ is either convex
or ``almost convex\footnote{``Almost convex" means that $u^\vepsi$ is
convex in $\Ome$ minus an $\vepsi$-neighborhood of $\p\Ome$, 
see Theorem \ref{convexity_thm} for a precise description.}" in $\Ome$.  
Then there holds the following estimates:
\begin{align}\label{rate_14}
\Bigl(\int_\Ome |\nab u- \nab u^\vepsi|^2\, dx\Bigr)^{\frac12} 
&\leq \vepsi^{\frac34}\, C_{13},\\
\Bigl(\int_\Ome |\Del u-\Del u^\vepsi|^2\, dx\Bigr)^{\frac12} 
&\leq \vepsi^{\frac14}\,C_{14}, \label{rate_15}\\
\Bigl(\int_\Ome |u- u^\vepsi|^2\, dx\Bigr)^{\frac12} 
&\leq \vepsi\, C_{15},  \label{rate_16}
\end{align}
where $C_j=C_j(\|\nab \Del u\|_{L^2})$ for $j=13,14, 15$ are 
positive $\vepsi$-independent constants.

\end{thm}

\begin{proof}
Since the proof follows the exact same lines as those 
for Theorem \ref{convergence_rate_thm1}, we just briefly 
highlight the main steps.

First, the error equation \eqref{rate_4} is replaced by 
\begin{align}\label{rate_20}
\vepsi \Del v^\vepsi + \mbox{det}(D^2 u) -\mbox{det}(D^2 u^\vepsi) =0
\qquad\mbox{in } \Ome,
\end{align}
where $v^\vepsi=\Del u^\vepsi$.  

Next, equation \eqref{rate_6} becomes 
\begin{align}\label{rate_21}
\vepsi \int_\Ome |\Del e^\vepsi|^2\, dx
&+\int_\Ome \theta^\vepsi\nab e^\vepsi\cdot \nab e^\vepsi\, dx \\
&=\int_{\p \Ome} \Del e^\vepsi \frac{\p e^\vepsi}{\p \nu}\, dS
-\vepsi \int_\Ome \nab v \cdot\nab e^\vepsi\,dx,  \nonumber
\end{align}
where
\begin{equation}\label{rate_22}
\theta^\vepsi= \Phi^\vepsi:=\mbox{\rm cof} (tD^2u+(1-t) D^2 u^\vepsi) 
\qquad\mbox{for some } t\in [0,1],
\end{equation}
now stands for the cofactor matrix of $tD^2u+(1-t) D^2 u^\vepsi$.
Since $u$ is assumed to be strictly convex and $u^\vepsi$ is 
``almost convex", then there exists a positive constant $\theta_0$
such that (see Chapter \ref{chapter-4})
\[
\theta^\vepsi\nab e^\vepsi\cdot \nab e^\vepsi\geq \theta_0 |\nab e^\vepsi|^2.
\]

It remains to derive a boundary estimate that is analogous to \eqref{rate_9}.
To the end, by the boundary condition $v^\vepsi|_{\p\Ome}=\vepsi$ and
the trace inequality we have
\begin{align}\label{rate_23}
\int_{\p \Ome} \Del e^\vepsi \frac{\p e^\vepsi}{\p \nu}\, dS
&\leq \vepsi \bigl(\vepsi |\p\Ome| + \|\Del u\|_{L^2(\p\Ome)}^2 \bigr)^{\frac12} 
\Bigl\|\frac{\p e^\vepsi}{\p \nu}\Bigr\|_{L^2(\p\Ome)}\\
&\leq \vepsi M \|\nab e^\vepsi\|_{L^2(\Ome)}^{\frac12} 
\|\Del e^\vepsi\|_{L^2(\Ome)}^{\frac12}\nonumber  \\
&\leq \frac{\vepsi}{2} \|\Del e^\vepsi\|_{L^2(\Ome)}^2 
+ M^{\frac43} \vepsi \|\nab e^\vepsi\|_{L^2(\Ome)}^{\frac23} \nonumber  \\
&\leq \frac{\vepsi}{2} \|\Del e^\vepsi\|_{L^2(\Ome)}^2 
+\frac{\theta_0}{4} \|\nab e^\vepsi\|_{L^2(\Ome)}^2 
+\frac{\vepsi^{\frac32} M^2}{\theta_0}. \nonumber
\end{align}
The desired estimates \eqref{rate_14} and \eqref{rate_15}
follow from combining \eqref{rate_21} and \eqref{rate_23}.

Finally, \eqref{rate_16} can be derived by using the same duality argument 
as that used in the proof of Theorem \ref{convergence_rate_thm2}.
We leave the details to the interested reader.
\end{proof}

\begin{remark}
The convergence rates proved in Theorem 
\ref{convergence_rate_thm1}--\ref{convergence_rate_thm3}
have been observed in numerical experiments.  We refer the reader to
Chapter \ref{chapter-6} for details. 
\end{remark}

\section{Epilogue} \label{chapter-3.6}

We like to comment that the analysis of Section \ref{chapter-3.1}--\ref{chapter-3.5}
can be easily extended to the cases of the other two boundary conditions 
in \eqref{moment3}.  We note that in the case \eqref{moment3}$_2$ boundary condition
\eqref{moment4_radial} should be replaced by
\[
u^\vepsi_{rrr}(R) + \frac{n-1}{R} u^\vepsi_{rr}(R)=\vepsi,
\]
and \eqref{moment4a_radial} should be replaced by
\[
w^\vepsi_{rr}(R)- \frac{n-1}{R} w^\vepsi_r(R)=R^{n-1}\vepsi.
\]

We also reiterate an interesting property of the vanishing moment 
method which was briefly touched on at the end of Chapter \ref{chapter-2}. 
That is, the ability of the vanishing moment method to approximate
the {\em concave} solution of the Monge-Amp\`ere problem 
\eqref{MAeqn1}--\eqref{MAeqn2}. This can be achieved simply 
by letting $\vepsi\nearrow 0^-$ in \eqref{moment1}--\eqref{moment3}$_1$.
This property can be easily proved as follows in the radially symmetric case.

Before giving the proof, we note that for a given $f> 0$ in $\Ome$,
equation \eqref{MAeqn1} does not a have concave solution in odd dimensions
(i.e., $n$ is odd) because $\mbox{det}(D^2 u)=f$ does not hold for any concave
function $u$ as all $n$ eigenvalues of Hessian $D^2u$ of a concave 
function $u$ must be nonpositive. On the other hand, in even dimensions (i.e., $n$ is
even), it is trivial to check that if $u$ is a convex solution of 
problem \eqref{MAeqn1}--\eqref{MAeqn2} with $g=0$, then $-u$, which is a concave function,
must also be a solution of problem \eqref{MAeqn1}--\eqref{MAeqn2}.

Next, by the same token, it is easy to prove that if $u^\vepsi$ is a 
convex or ``almost convex" solution to problem \eqref{moment1}--\eqref{moment3}$_1$,
then $-u^\vepsi$, which is concave or ``almost concave\footnote{A function $\varphi^\vepsi$ is
said to be ``almost concave" in $\Ome$ if it is concave in $\Ome$ minus 
an $O(\vepsi)$-neighborhood of the boundary $\p\Ome$ of $\Ome$.}", must also be a 
solution of \eqref{moment1}--\eqref{moment3}$_1$. 

Finally, let $n$ be a positive even integer,
it is easy to see that changing $u^\vepsi$ to $-u^\vepsi$ in 
\eqref{moment1}--\eqref{moment3}$_1$ is equivalent to changing 
$\vepsi$ to $-\vepsi$ in \eqref{moment1}--\eqref{moment3}$_1$. 
For $\vepsi<0$, let $\delta:=-\vepsi$.  After replacing $\vepsi$ by $-\delta$
and $u^\vepsi$ by $\hat{u}^\delta:=-u^\vepsi$
in \eqref{moment1_radial}--\eqref{moment4_radial}, we see that 
$\hat{u}^\delta$ satisfies the same set of equations 
\eqref{moment1_radial}--\eqref{moment4_radial} with $\delta (>0)$
in place of $\vepsi$. Hence, by the analysis of 
Section \ref{chapter-3.2}--\ref{chapter-3.5} we know that
there exists a monotone increasing solution $\hat{u}^\delta$ 
to problem \eqref{moment1_radial}--\eqref{moment4_radial} with
$\vepsi$ being replaced by $\delta$, which satisfies all the properties 
proved in Section \ref{chapter-3.2}--\ref{chapter-3.5}.
Translating all these to $u^\vepsi=-\hat{u}^\delta$ we conclude that 
problem \eqref{moment1_radial}--\eqref{moment4_radial} for $\vepsi<0$ has
a monotone decreasing solution which is either concave or ``almost concave"
in $(0,R)$ and converges to the unique concave solution of 
problem \eqref{MAeqn1}--\eqref{MAeqn2} as $\vepsi\nearrow 0^-$. In addition,
$u^\vepsi$ satisfies the error estimates stated in Theorem \ref{convergence_rate_thm1}
and \ref{convergence_rate_thm2}.

The final comment we like to make is about the possible but well-behaved 
boundary layer generated by the vanishing moment solution $u^\vepsi$. 
In the worst case scenario, the boundary layer, where $u^\vepsi$ may cease
to be convex, is confined in an $O(\vepsi)$-neighborhood of the boundary 
$\p\Ome$. This nice behavior of the boundary layer can be 
exploited in numerical computations. Indeed, in Chapter \ref{chapter-7}
we propose an iterative surgical procedure to take advantage of
this property of the (possible) boundary layer. 
We refer the reader to Chapter \ref{chapter-7} for
the detailed description of the procedure and numerical experiments 
which show the effectiveness of the proposed iterative surgical procedure.

%% file: chapter4.tex
\chapter{Conforming finite element approximations}\label{chapter-4}

The goal of this chapter is to construct and analyze $C^1$
finite element approximations for the general fully nonlinear second
order Dirichlet problem \eqref{generalPDEa}--\eqref{generalPDEb} 
based upon the vanishing moment methodology introduced in Chapter
\ref{chapter-2} and further analyzed in Chapter \ref{chapter-3}.
Letting $\ue$ be the solution to problem \eqref{moment1}--\eqref{moment3}$_1$, 
we construct and analyze conforming finite element methods
to approximate $\ue$ using a class of $C^1$ finite elements
such as Argyris, Bell, Bogner-Fox-Schmit, and Hsieh-Clough-Tocher elements (cf.
\cite{Ciarlet78}).  As a result, we obtain convergent
numerical methods for fully nonlinear second order PDEs.

We note that finite element approximations of fourth order PDEs, in
particular, the biharmonic equation, were carried out extensively in
the seventies for the two-dimensional case \cite{Ciarlet78}, and have
attracted renewed interests lately for generalizing the well-known
two-dimensional finite elements to the three-dimensional case
(cf. \cite{Tai_Wagner00,Wang_Shi_Xu07,Wang_Xu07}). 
Although all of these methods can 
be readily adapted to discretize problem
\eqref{moment1}--\eqref{moment3}$_1$, the
convergence analysis does not come easy 
due to the strong nonlinearity of the PDE \eqref{moment1}.
For example, to use the standard perturbation technique for deriving error 
estimates (a technique successfully used for linear and mildly 
nonlinear problems), we would have to assume very stringent conditions 
on the nonlinear differential operator $F$, which would rule out many 
interesting application problems, and hence, should be avoided.
Instead, we assume very mild conditions on the operator
(see Section \ref{chapter-4-sec-2} for details), and use a 
combined fixed-point and linearization technique to 
simultaneously prove existence and uniqueness for 
the numerical solution, and also derive error estimates.

The remainder of the chapter is organized as follows.
First in Section \ref{chapter-4-sec-2}, we give additional 
notation, and then define the
finite element method based upon the variational formulation 
\eqref{moment_var}. Next, we make certain structure assumptions about the
nonlinear differential operator $F$ which will play an important
role in our analysis.  In Section \ref{chapter-4-sec-3}, we show existence of
solutions of the linearized problem and prove stability and
convergence results of its finite element approximations.  The main results
 of the chapter are found in Section \ref{chapter-4-sec-4},
where we use a fixed point argument to simultaneously show
existence, uniqueness, and convergence of 
the finite element approximation of
\eqref{moment1}--\eqref{moment3}$_1$.

\section{Formulation of conforming finite element methods} 
\label{chapter-4-sec-2}

First, we introduce the following function space notation:
\[
V:=H^2(\Ome),\qquad V_0:=H^2(\Ome)\cap H^1_0(\Ome),\qquad
V_g:=\{v\in V;\ v|_{\p\Ome}=g\}.
\]

Let $\mathcal{T}_h$ be a quasiuniform triangular 
or rectangular partition of $\Ome$, and let $V^h\subset V$ be 
a conforming finite element space consisting of piecewise polynomials
of degree $k>4$ such that for any $v\in V\cap H^s(\Ome)$, we have
\begin{align}
\label{interp}\inf_{v_h\in V^h} \|v-v_h\|_{H^j}\le
Ch^{\ell-j}\|v\|_\hl\qquad j=0,1,2,\qquad \ell={\rm min}\{s,k+1\}.
\end{align}
Let
\begin{align}
\label{Vhzgdef}&V^h_0:=\left\{v_h\in V^h;\ v_h\big|_{\p \Ome}=0\right\},
\quad V^h_g:=\left\{v_h\in V^h;\ v_h\big|_{\p\Ome}=g\right\}.
\end{align}

Based on \eqref{moment_var}, we define the finite element
formulation of \eqref{moment1}--\eqref{moment3} as to find
$\ueh\in V^h_g$ such that
\begin{align}
\label{momentfem}
\eps(\Del \ue_h,\Delta v_h)+\bigl( F(D^2\ueh,\nab \ueh,\ueh,x),v_h \bigr)
=\left\langle\eps^2,\normd{v_h}\right\rangle_{\p\Ome}\qquad \forall v_h\in V^h_0.
\end{align}

Let $\ue$ be the solution to \eqref{moment_var} and let
$\ueh$ be a solution to \eqref{momentfem}.  
The primary goal of this chapter is to derive error 
estimates of $\ue-\ueh$, which then means we need to first
prove that there exists $\ueh\in V^h_g$ solving \eqref{momentfem},
and that $\ueh$ is unique. Clearly, we must assume some structure 
conditions on the nonlinear differential operator
$F$ to achieve any of these goals.  Indeed, the assumptions that we make
will play an important role in our results and in the techniques 
to derive them. We refer to Section \ref{chapter-1.2} for the notation used 
in this chapter.

\medskip
{\bf Assumption (A)}
\smallskip

\begin{enumerate}
\item[{\bf[A1]}]\label{abstracthyp1} 
There exists $\eps_0\in (0,1)$ such that
for all $\eps\in (0,\eps_0]$, there exists a locally unique
solution to \eqref{moment1}--\eqref{moment3}$_1$ 
with $\ue\in H^s(\Ome)\ (s\ge 3).$\smallskip

\item[{\bf[A2]}]
For $\eps\in (0,\eps_0]$, 
the operator $\(G_\eps^\prime[\ue]\)^*$ 
(the adjoint of $G_\eps^\prime[\ue]$) is
an isomorphism from $V_0$ to $V_0^*$.  
That is for all $\varphi\in V_0^*$ (the dual space of $V^*_0$), 
there
exists $v\in V_0$ such that
\begin{align}
\label{a2again}\left\langle \(\Gp[\ue]\)^*(v),w\right\rangle
=\langle \varphi,w\rangle\qquad \forall w\in V_0.
\end{align}
Here, $\langle \cdot,\cdot\rangle$ denotes the dual pairing 
between $V_0$ and $V^*_0$.
Furthermore, there exists positive constants $C_0=C_0(\eps),\ C_1=C_1(\eps)$ 
such that the following G\aa rding inequality holds:
\begin{align} \label{abGarding}
\bl \Gp[\ue](v),v\br \ge C_1\|v\|_\htw^2-C_0\|v\|_\lt^2\qquad \forall v\in V_0,
\end{align}
and there exists $C_2=C_2(\eps)>0$ such that
\begin{align*}
\bnorm{\Fp[\ue]}_{VV^*}&\le C_2,
\end{align*}
where
\begin{align*}
\bnorm{F^\prime[\ue]}_{VV^*}
:=\sup_{v\in V_0} \frac{\bnorm{\Fp[\ue](v)}_{H^{-2}}}{\|v\|_\htw}
:=\sup_{v\in V_0}\sup_{w\in V_0} \frac{\bl \Fp[\ue](v),w\br}{\|v\|_\htw\|w\|_\htw}.
\end{align*}
Moreover, there exists $p>2$ and $C_R=C_R(\eps)>0$ such that if $\varphi\in L^2(\Ome)$
and $v\in V_0$ satisfies \eqref{a2again}, then $v\in H^p(\Ome)$ and
\begin{align*}
\|v\|_{H^p}\le C_R\|\varphi\|_\lt.
\end{align*}

\item[{\bf[A3]}] 
There exists a Banach space $Y$ 
with norm $\|\cdot\|_Y$ that is
well-defined and finite on $V^h$,
and a constant $C>0$, independent of $\eps$,
such that
\begin{align*}
\sup_{y\in Y} \frac{\bnorm{\Fp[y]}_{VV^*}}{\|y\|_{Y}}\le C.
\end{align*}

\item[{\bf[A4]}]
There exists a constant $C>0$ independent of
$\eps$ such that 
\begin{align*}
\|\util\|_Y&\le C\|\ue\|_Y,
\end{align*}
where $\util\in V^h_g$ denotes the finite element interpolant 
of $\ue$.

\item[{\bf[A5]}]
There exists a constant
$\del=\del(\eps)\in (0,1)$, such that for any $w_h\in V^h_g$ with
$\|\util-w_h\|_\htw\le \del$, there holds
\begin{align*}
\bnorm{\Fp[\ue]-\Fp[w_h]}_{VV^*}\le L(h)\|\ue-w_h\|_\htw,
\end{align*}
where $L(h)=L(\eps,h)$ may depend on both $h$ and $\eps$
and satisfies $L(h) = o(h^{2-\ell})$.
\end{enumerate}

\begin{remarks} 
(a) \label{Remark41i}
Conditions {\rm [A1]--[A5]} are fairly mild, 
and a very large class of fully nonlinear
second order differential operators satisfy these requirements
(cf. Chapter \ref{chapter-6}).
Clearly, we must assume {\rm [A1]} in order
for the finite element method \eqref{momentfem}
to have any significance, and the regularity requirements
of $\ue$ are needed to obtain any meaningful error estimates.

(b) \label{Remark41ii}
Condition {\rm [A2]} is naturally satisfied if $-F$ is 
elliptic at $\ue$ (cf. \cite[Chapter 17]{Gilbarg_Trudinger01}),
and the regularity requirements are expected to hold provided
that $\ue$ and $\p\Ome$ are sufficiently regular.

(c) \label{Remark41iii}
By standard interpolation theory \cite{Ciarlet78,Brenner}, there holds
\begin{align}
\label{uinterp}\|\ue-\util\|_{H^j}\le Ch^{\ell-j}\|\ue\|_\hl\qquad j=0,1,2,\quad \ell={\rm min}\{s,k+1\}.
\end{align}
%

(d) \label{Remark41iv}
Condition {\rm [A5]}, which states that $\Fp$ is locally Lipschitz 
near $\ue$, is the strongest requirement among the five listed, and it
is the authors' experience that this is the most difficult property to verify.
As one may expect, this assumption plays an important role in the
fixed point argument, which is needed in our analysis to obtain existence, uniqueness, 
and error estimates of the finite element method \eqref{momentfem}.
\end{remarks}

\section{Linearization and its finite element approximations}\label{chapter-4-sec-3}
To construct the necessary tools to analyze the
finite element method \eqref{momentfem},
we first study finite element approximation 
of the linearization of \eqref{moment1}.
We note that the materials of this section
have an independent interest within themselves.
To the best of our knowledge, finite element
error estimates for non-coercive linear fourth order 
problems have not been explicitly given in the literature before.
  
\subsection{Linearization}\label{chapter-4-sec-3-1}
For given $\varphi\in V_0^*$
and $\psi\in H^{-\frac12}(\p\Ome)$, 
we consider the following linear problem:
\begin{alignat}{2}
\label{abstractlin1}\Gp[\ue](v)
&=\varphi\qquad &&\text{in }\Ome,\\
\label{abstractlin2}v&=0\qquad &&\text{on }\p\Ome,\\
\label{abstractlin3}\Del v
&=\psi\qquad &&\text{on }\p\Ome.
\end{alignat}

Multiplying the equation \eqref{abstractlin1}
by $w\in V_0$,  integrating over
$\Ome$, and integrating by parts, we obtain
\begin{align*}
\bl \Gp[\ue](v),w\br
=\eps(\Del v,\Del w) +
\bl \Fp[\ue](v),w\br-\eps\left\langle \Del
v,\normd{w}\right\rangle_{\p\Ome}.
\end{align*}
Based on this calculation, we define the weak formulation of 
\eqref{abstractlin1}--\eqref{abstractlin3} as to find $v\in V_0$ such
that 
\begin{align}\label{abstractvar}
a^{\eps}(v,w)=\langle \varphi,w\rangle+
\eps \left \langle \psi,\normd{w}\right\rangle_{\p \Ome}
\qquad \forall w\in V_0,
\end{align} where
\begin{align*}
a^{\eps}(v,w):=\eps(\Delta v,\Delta w)+\bl F^\prime[\ue](v),w\br.
\end{align*}
In view of assumptions {\rm [A1]--[A2]}, we immediately have the following theorem.

\begin{thm}\label{linexistencethm}
Suppose assumptions {\rm [A1]--[A2]} hold.  
Then there exists a unique solution $v\in V_0$ to \eqref{abstractvar}.
Furthermore, there exists $C_3=C_3(\eps)>0$ such that
\begin{align}\label{linexistence}
\|v\|_\htw
\le C_3\Bigl(\|\varphi\|_{H^{-2}}+\eps \|\psi\|_{H^{-\frac12}(\p \Ome)}\Bigr).
\end{align}
\end{thm}

\begin{proof}
From the G\aa rding inequality \eqref{abGarding}
and the fact $\left(G^{\prime}_\eps[\ue]\right)^*$ is injective 
on $V_0$, it follows that $G^\prime_\eps[\ue]$ is
an isomorphism from $V_0$ to $V_0^*$
using a Fredholm alternative argument \cite[Theorem 8.5]{Agmon}.

We now claim that there exists $C_S=C_S(\eps)$ 
such that $\|v\|_\lt \le C_S\bigl(\|\varphi\|_{H^{-2}}+\eps \|\psi\|_{H^{-\frac12}(\p\Ome)}\bigr)$.  
If not, there would exist sequences 
$\{\varphi_m\}_{m=1}^\infty\subset V^*_0,\ \{\psi_m\}_{m=1}^\infty \subset H^{-\frac12}(\p\Ome)$,
and $\{v_m\}_{m=1}^\infty\subset V_0$ such that
\begin{align*}
\bl G^\prime_\eps[\ue](v_m),w\br
=\langle \varphi_m,w\rangle+\eps\left\langle \psi_m,\normd{w}\right\rangle_{\p\Ome}\qquad w\in V_0,
\end{align*}
but
\begin{align*}
\|v_m\|_\lt > m\bigl(\|\varphi_m\|_{H^{-2}}+\eps \|\psi_m\|_{H^{-\frac12}(\p\Ome)}\bigr).
\end{align*}
Without loss of generality, we may as well suppose 
$\|v_m\|_\lt=1$ 
(and therefore $\|\varphi_m\|_{H^{-2}}+\eps \|\psi_m\|_{H^{-\frac12}(\p\Ome)}\to 0$
as $m\to \infty$).  
In light of \eqref{abGarding}, $\{v_m\}_{m=1}^\infty$ is bounded in $V_0$,
and hence by a compactness argument, there exists a subsequence 
$\{v_{m_j}\}_{m=1}^\infty$ and $v\in V_0$ such that
\begin{alignat}{2}
v_{m_j}&\rightharpoonup v\quad &&\text{weakly in }V_0,\\
\label{abstractH1}v_{m_j}&\to v\quad && \text{in }H^1_0(\Ome).
\end{alignat}
Therefore, 
\[
\bl G^\prime_\eps[\ue](v),w\br=0\quad \forall w\in V_0.
\]
Since $G^\prime_\eps[\ue]$ is an isomorphism, $v\equiv 0$.  
However \eqref{abstractH1} implies that $\|v\|_\lt=1$, a contradiction. 

Hence there exists $C_S$ such that 
\begin{align*}
\|v\|_\lt \le C_S\bigl(\|\varphi\|_{H^{-2}}+\eps \|\psi\|_{H^{-\frac12}(\p\Ome)}\bigr),
\end{align*}
and therefore by \eqref{abGarding} and a trace inequality, we have
\begin{align*}
C_1\|v\|_\htw^2&\le \eps(\Del v,\Del v)+
\bl F^\prime[\ue](v),v\br+C_0\|v\|_\lt^2\\
&=a^\eps(v,v)+C_0\|v\|_\lt^2\\
&=\langle \varphi,v\rangle+
\eps\left\langle \psi,\normd{v}\right\rangle_{\p \Ome}+C_0\|v\|_\lt^2\\
&\le C\Bigl(\|\varphi\|_{H^{-2}}+\eps \|\psi\|_{H^{-\frac12}(\p\Ome)}
+C_0\|v\|_\lt\Bigr)\|v\|_\htw\\
&\le C(1+C_0C_S)\Bigl(\|\varphi\|_{H^{-2}}+\eps \|\psi\|_{H^{-\frac12}(\p\Ome)}
\Bigr)\|v\|_\htw.
\end{align*}
Dividing by $C_1\|v\|_\htw$, we obtain \eqref{linexistence} with 
$C_3=CC_1^{-1}(1+C_0C_S)$.
\end{proof}

\subsection{Finite element approximation}\label{chapter-4-sec-3-2}
Let $V^h_0\subset V_0$ be one of the finite dimensional 
subspaces of degree $k>4$
defined in Section \ref{chapter-4-sec-2}.  
Based on the variational formulation \eqref{abstractvar}, 
we define the finite element method for 
\eqref{abstractlin1}--\eqref{abstractlin3} 
as to find $v_h\in V^h_0$ such that
\begin{align}
\label{abstractlinfem}
a^{\eps}(v_h,w_h)=\langle \varphi,w_h\rangle
+\eps\left\langle \psi,\normd{w_h}\right\rangle_{\p\Ome}
\qquad \forall w_h\in V^h_0.
\end{align}

Using a modification of the well-known Schatz's argument 
(cf.~\cite[Theorem 5.7.6]{Brenner}), we 
obtain the following result.
\begin{thm}
\label{abstractbound1thm}
Let assumptions {\rm [A1]--[A2]} hold
and suppose that $v\in H^s(\Ome)\ (s\ge 3)$ is the 
unique solution to \eqref{abstractvar}.  Then for $h\le h_0(\eps)$,
there exists a unique solution $v_h\in V_0^h$ to 
\eqref{abstractlinfem}, where
\begin{align}
\label{h0def}
h_0=\left\{
\begin{array}{ll}
C\left(C_0C_1^{-1}C_2^2C_R^2\right)^{\frac{1}{4-2r}} &\mbox{if } C_0\neq 0,\\
1 &\mbox{if } C_0=0,
\end{array}\right.
\qquad r={\rm min}\{p,k+1\}.
\end{align}

Furthermore, there holds the following inequalities:
\begin{align}
\label{abstractyut}\|v-v_h\|_\htw&\le C_4h^{\ell-2}\|v\|_\hl,\\
\|v-v_h\|_\lt &\le C_5h^{\ell+r-4}\|v\|_\hl,
\end{align}
where 
\begin{align*}
C_4=C_4(\eps)=CC_1^{-1}C_2,
\qquad C_5=C_5(\eps)=CC_1^{-1}C_2^2C_R,\qquad
\ell={\rm min}\{s,k+1\}.
\end{align*}
\end{thm}

\begin{proof}
To show existence, we begin by deriving 
estimates for a solution $v_h$ to \eqref{abstractlinfem} 
that may exist.  We start with the error equation:
\begin{align*}
a^\eps(v-v_h,w_h)=0\qquad \forall w_h\in V^h_0.
\end{align*}
Then using \eqref{abGarding} and {\rm [A2]}, we have for any $w_h\in V^h_0$
\begin{align*}
&C_1\|v-v_h\|_\htw^2\\
&\qquad\nonumber=a^\eps(v-v_h,v-v_h)+C_0\|v-v_h\|_\lt^2\\
&\qquad\nonumber= a^\eps(v-v_h,v-w_h)+C_0\|v-v_h\|_\lt^2\\
&\qquad\nonumber\le \eps\|\Delta(v-v_h)\|_\lt\|\Delta(v-w_h)\|_\lt\\
&\qquad\nonumber\qquad+\bigl\|F^\prime[\ue]\bigr\|_{VV^*}\|v-v_h\|_\htw\|v-w_h\|_\htw
+C_0\|v-v_h\|_\lt^2\\
&\qquad\nonumber\le CC_2\|v-v_h\|_\htw\|v-w_h\|_\htw
+C_0\|v-v_h\|_\lt^2.
\end{align*}
Thus, by \eqref{interp}
\begin{align}
\label{abstractyu2}
C_1\|v-v_h\|_\htw^2\le 
CC_1^{-1}C^2_2h^{2\ell-4}\|v\|^2_\hl+C_0\|v-v_h\|_\lt^2.
\end{align}

Next, we let $w\in V_0\cap H^p(\Ome) \,(p>2)$ be the solution to 
the following auxiliary problem:
\begin{align*}
\bl \left(G_\eps^\prime[\ue]\right)^*(w),z\br
=(v-v_h,z)\qquad \forall z\in V_0.
\end{align*}
By assumption {\rm [A2]}, there exists such a solution $w$ with 
\begin{align}\label{abstractadjoint}
\|w\|_{H^p}\le C_R\|v-v_h\|_\lt.
\end{align}

We then have for any $w_h\in V_0^h$
\begin{align*}
\|v-v_h\|_\lt^2&=
\bl (G_\eps^\prime[\ue])^*(w),(v-v_h)\br\\
&=\bl G_\eps^\prime[\ue](v-v_h),w\br\\
&=a^\eps(v-v_h,w)\\
&=a^\eps(v-v_h,w-w_h)\\
&\le CC_2\|v-v_h\|_\htw\|w-w_h\|_\htw.
\end{align*}

Consequently from \eqref{interp} and \eqref{abstractadjoint}
\begin{align*}
\|v-v_h\|_\lt^2&\le CC_2h^{r-2}\|v-v_h\|_\htw\|w\|_{H^p}\\
&\le CC_2C_Rh^{r-2}\|v-v_h\|_\htw\|v-v_h\|_\lt,
\end{align*}
and thus,
\begin{align}\label{lineabove1}
\|v-v_h\|_\lt\le CC_2C_Rh^{r-2}\|v-v_h\|_\htw.
\end{align}

Applying the inequality \eqref{lineabove1} into \eqref{abstractyu2} gives us
\begin{align*}
C_1\|v-v_h\|_\htw^2&\le 
CC^{-1}_1C^2_2h^{2\ell-4}\|v\|^2_\hl+C_0\|v-v_h\|_\lt^2\\
&\le CC^{-1}_1C^2_2h^{2\ell-4}\|v\|^2_\hl
+CC_0C^2_2C_R^2h^{2r-4}\|v-v_h\|_\htw^2.
\end{align*}
Thus, for $h\le h_0$
\begin{align*}
C_1\|v-v_h\|_\htw^2\le CC^{-1}_1C^2_2h^{2\ell-4}\|v\|^2_\hl,
\end{align*}
and therefore
\begin{align*}
\|v-v_h\|_\htw&\le CC^{-1}_1C_2h^{\ell-2}\|v\|_\hl,\\
\|v-v_h\|_\lt&\le CC^{-1}_1C^2_2C_Rh^{\ell+r-4}\|v\|_\hl.
\end{align*}

So far, we have been under the assumption that 
there exists a solution $v_h$. We now consider 
the question of existence and uniqueness.  First,
since the problem under consideration is linear and in a finite dimensional 
setting, existence and uniqueness are equivalent.  
Now suppose $\varphi\equiv 0,\ \psi\equiv 0$.  
In light of \eqref{linexistence}, we have
$v\equiv 0$, and therefore, \eqref{abstractyut} 
implies $v_h\equiv 0$ as well provided
that $h$ is sufficiently small.  In particular, 
this means that \eqref{abstractlinfem}
has a unique solution for $h\le h_0$.  
\end{proof}

\begin{remarks}
(a) \label{Remark44i}
Because  \eqref{a2again} is a fourth order problem, we expect 
$p\geq 3$. Therefore, since the polynomial degree $k$ is 
strictly greater than four, we expect $r=p$ in Theorem \ref{abstractbound1thm}.

(b) \label{Remark44ii}
In many cases, it is possible to get a relatively good idea
of how the constant $C_R$ depends on $\eps$.  To see this,
suppose that there exists a constant $\wh{C}_2>0$ such that 
if $v$ solves \eqref{a2again}, then $\(\Fp[\ue](v)\)^*\in H^{-1}(\Ome)$ and
\begin{align*}
\bigl\|\(\Fp[\ue]\)^*(v)\bigr\|_{H^{-1}}\le \wh{C}_2\|v\|_\ho.
\end{align*}
Here, $\(\Fp[\ue]\)^*$ denotes the adjoint operator of $\Fp[\ue]$.

Now if $p =3$ in {\rm [A2]}, then
\begin{align*}
\left\langle \(\Gp[\ue]\)^*(v),\Del v\right\rangle &=(\varphi,\Del v),
\end{align*}
where $\langle \cdot,\cdot\rangle$ now denotes the dual pairing of
$H^1_0(\Ome)$ and $H^{-1}(\Ome)$.
Therefore, after integrating by parts
\begin{align*}
\eps \|\nab \Del v\|_\lt^2
&=\left\langle \(\Fp[\ue]\)^*(v),\Del v\right\rangle -(\varphi,\Del v)\\
&\le \Bigl(\wh{C}_2\|v\|_\ho+\|\varphi\|_{H^{-1}}\Bigr)\|\Del v\|_\ho.
\end{align*}
Hence, by Poincare's inequality
\begin{align}
\label{chapter4remarkline1}
\|\nab \Del v\|_\lt\le C\eps^{-1}\Bigl(\wh{C}_2\|v\|_\ho+\|\varphi\|_\lt\Bigr).
\end{align}
By the proof of Theorem \ref{linexistencethm}, it is apparent that
\[
\|v\|_\htw\le CC_3\|\varphi\|_\lt,
\]
and therefore
\begin{align}
\label{chapter4remarkline2}
\|\nab \Del v\|_\lt\le C\eps^{-1}\bigl(\wh{C}_2C_3+1\bigr)\|\varphi\|_\lt
\le C\eps^{-1}\wh{C}_2C_3\|\varphi\|_\lt.
\end{align}
Furthermore, if $\(\Fp[\ue]\)^*$ is coercive on $H^1_0(\Ome)$, that is,
there exists a constant $\wh{C}_1>0$ such that 
\begin{align}\label{remarkcoercive}
\left\langle \(\Fp[\ue]\)^*(v),v\right\rangle \ge \wh{C}_1\|v\|_\ho^2\qquad \forall v\in H^1_0(\Ome),
\end{align}
then by \eqref{chapter4remarkline1} 
\begin{align}
\label{chapter4remarkline3}
\|\nab \Del v\|_\lt\le C\eps^{-1}\wh{C}^{-1}_1\wh{C}_2\|\varphi\|_\lt.
\end{align}
In view of \eqref{chapter4remarkline2} or \eqref{chapter4remarkline3}, we
can expect that in the general case
\begin{align*}
\|v\|_{H^3}\le C\eps^{-1}\wh{C}_2C_3\|\varphi\|_\lt,
\end{align*}
and if \eqref{remarkcoercive} holds
\begin{align*}
\|v\|_{H^3}\le C\eps^{-1}\wh{C}_1^{-1}\wh{C}_2\|\varphi\|_\lt.
\end{align*}
Hence, for $p=3$ we have $C_R=C\eps^{-1}\wh{C}_2C_3$ in the general case
and $C_R=C\eps^{-1}\wh{C}_1^{-1}\wh{C}_2$ if \eqref{remarkcoercive} holds.

Now we consider the case $p = 4$, and for simplicity, we assume
$\Fp[\ue]$ is self-adjoint.  We then have

\begin{align*}
\left\langle \Gp[\ue](v),\Del^2 v\right\rangle &=(\varphi,\Del^2 v),
\end{align*}
and therefore
\begin{align*}
\eps\|\Del^2v\|_\lt^2
&=(\varphi,\Del^2v)-\left\langle \Fp[\ue](v),\Del^2v\right\rangle\\
&\le C\Bigl(\|\varphi\|_\lt+\|\Fp[\ue]\|_\infty \|v\|_\htw\Bigr)\|\Del^2v\|_\lt\\
&\le C\Bigl(1+C_3\bigl\|\Fp[\ue]\bigr\|_\infty \Bigr)\|\varphi\|_\lt\|\Del^2v\|_\lt,
\end{align*}
where we define
\begin{align*}
\bigl\|\Fp[\ue]\bigr\|_\infty 
:= \max_{1\le i,j\le n} \left\|\frac{\p F(\ue)}{\p r_{ij}}\right\|_{L^\infty}
+ \max_{1\le i\le n} \left\|\frac{\p F(\ue)}{\p p_{i}}\right\|_{L^\infty}
+\left\|\frac{\p F(\ue)}{\p z}\right\|_{L^\infty}.
\end{align*}
We then expect that in this case that
\begin{align*}
\|v\|_{H^4}\le CC_3\eps^{-1}\left\|\Fp[\ue]\right\|_{\infty}\|\varphi\|_\lt.
\end{align*}
Hence, for $p=4$ we have $C_R=CC_3\eps^{-1}\left\|\Fp[\ue]\right\|_{\infty}$.

\end{remarks}

\section{Convergence analysis of finite element approximation}\label{chapter-4-sec-4}

In this section, we give the main results of this chapter,
where we establish existence and uniqueness, and derive 
error estimates for the finite element method \eqref{momentfem}.
First, we define an operator $T_h:V^h_g\mapsto V^h_g$ such that for a given
$v_h\in V^h_g$, $T_h(v_h)$ is the solution to the
following linear problem:
\begin{align}
\label{abstractTdef}&a^{\eps}\bigl(v_h-T_h(v_h),w_h\bigr)\\
&\nonum\hspace{0.5in}=\eps(\Delta v_h,\Delta w_h)+\bl 
F(v_h),w_h\br-\left\langle
\eps^2,\normd{w_h}\right\rangle_{\p\Ome}\quad \forall w_h\in V_0^h.
\end{align} 
In view of Theorem
\ref{abstractbound1thm}, $T_h$ is well-defined provided that
assumptions {\rm [A1]--[A2]} hold and $h\le h_0$.   
We note that the right-hand side of \eqref{abstractTdef}
is the residual of the finite element method \eqref{momentfem},
and therefore, any fixed point of $T_h$ (i.e. $T(v_h)=v_h$)
is a solution to \eqref{momentfem} and vice-versa.
Our goal is to show that indeed, $T_h$ has a unique fixed 
point in a small neighborhood of $\ue$. To this end, we define the following ball:
\begin{align*}
\mathbb{B}_h(\rho):=\bigl\{v_h\in V^h_g;\ \|\util-v_h\|_\htw\le \rho\bigr\},
\end{align*}
where the center of the ball $\util$ is the finite element interpolant of $\ue$.

For the continuation of this chapter, we let $\ell={\rm min}\{s,k+1\}$, 
where we recall that $k$ is the polynomial
degree of the finite element space $V^h$ and $s$ is defined in {\rm [A1]}.  
The following lemma shows that the distance between 
the center of $\mathbb{B}_h$ and its image under $T_h$ is small.

\begin{lem}\label{lem51}
Suppose assumptions {\rm [A1]--[A4]} hold.  
Then for $h\le h_0(\eps)$,
\begin{align}\label{ablem51bound}
\bigl\|\util-T_h(\util)\bigr\|_\htw\le C_6h^{\ell-2}\|\ue\|_\hl,
\end{align}
where 
\[
C_6=C_6(\eps)=CC_1^{-\frac12}\|\ue\|_Y
{\rm max}\{C_1^{-\frac12},C_0^\frac12C_R\}.
\]
\end{lem}

\begin{proof}
To ease notation, set $r^\eps=\util-\ue$.  
Using the definition of $T_h(\cdot)$ and the mean value theorem, 
we have for any $z_h\in V^h_0$
\begin{align}\label{lem51line1}
&a^{\eps}\bigl(\util-T_h(\util),z_h\bigr)\\
&\qquad\nonum=\eps(\Del \util,\Del z_h)+(F(\util),z_h)
-\left\langle \eps^2,\normd{z_h}\right\rangle_{\p\Ome}\\
&\nonum\qquad=\eps(\Del r^\eps,\Del z_h)+(F(\util)-F(\ue),z_h)\\
&\nonum\qquad=\eps(\Del r^\eps,\Del z_h)+
\bl \Fp[y_h](r^\eps),z_h\br,
\end{align} 
where $y_h=\util-\gamma r^\eps$ for some $\gamma \in [0,1]$.

Setting $z_h=\util-T_h(\util)$ and making use of {\rm [A2]--[A4]}, we have
\begin{align*}
&C_1\bigl\|\util-T_h(\util)\bigr\|_\htw^2
\le \eps\|r^\eps\|_\htw\bnorm{\util-T_h(\util)}_\htw\\
&\qquad \quad+C\|\ue\|_Y\|r^\eps\|_\htw \bigl\|\util-T_h(\util)\bigr\|_\htw
+C_0\bigl\|\util-T_h(\util)\bigr\|_\lt^2,
\end{align*}
and so by the Cauchy-Schwarz inequality,
\begin{align} \label{lem51line1-new}
&C_1\bigl\|\util-T_h(\util)\bigr\|_\htw^2\\
\nonum&\quad \le C_1^{-1}\eps^2\|r^\eps\|^2_\htw
+CC_1^{-1}\|\ue\|^2_Y\|r^\eps\|^2_\htw
+C_0\bigl\|\util-T_h(\util)\bigr\|_\lt^2\\
\nonum &\quad \le CC_1^{-1}h^{2\ell-4}\|\ue\|_Y^2
\|\ue\|_\hl^2
+C_0\bigl\|\util-T_h(\util)\bigr\|_\lt^2.
\end{align}

Next, we let $w\in V_0\cap H^p(\Ome) \,(p>2)$ be the solution
to the following auxiliary problem:
\begin{align*}
\bl \left(G^\prime_\eps[\ue]\right)^*(w),z\br =\bigl(\util-T_h(\util),z\bigr)
\qquad \forall z\in V_0,
\end{align*}
with 
\begin{align}
\label{lem51reg}
\|w\|_{H^p}\le C_R\bigl\|\util-T_h(\util)\bigr\|_\lt.
\end{align}
Then for any $z_h\in V_0^h$ we get
\begin{align*}
&\bigl\|\util-T_h(\util)\bigr\|_\lt^2\\
&\qquad\quad=a\bigl(\util-T_h(\util),w\bigr)\\
&\qquad\quad=a\bigl(\util-T_h(\util),w-z_h\bigr)+\eps(\Del r^\eps,\Del z_h)
+\bl F^\prime[y_h](r^\eps),z_h\br\\
&\qquad\quad\le CC_2\bigl\|\util-T_h(\util)\bigr\|_\htw \|w-z_h\|_\htw
+\eps\|\Del r^\eps\|_\lt\|\Del z_h\|_\lt\\
&\qquad\quad\quad +C\|\ue\|_Y\|r^\eps\|_\htw\|z_h\|_\htw.
\end{align*}
Taking $z_h=\mathcal{I}^h w$, we have from \eqref{uinterp} and \eqref{lem51reg}
\begin{align*}
\bigl\|\util &-T_h(\util)\bigr\|_\lt^2 \\
&\le CC_R^2\Bigl(C_2^2h^{2r-4}\bigl\|\util-T_h(\util)\bigr\|_\htw^2
+h^{2\ell-4}\|\ue\|_Y^2\|\ue\|_\hl^2\Bigr).
\end{align*}
Substituting this last bound into the inequality 
\eqref{lem51line1-new} we have
\begin{align*}
C_1\bigl\|\util-T_h(\util)\bigr\|_\htw^2&\le C\Bigl(
\bigl(C_1^{-1}+C_0C_R^2 \bigr)h^{2\ell-4}\|\ue\|_Y^2\|\ue\|_\hl^2\\
&\qquad +C_0C_2^2C_R^2 h^{2r-4}\bigl\|\util -T_h(\util)\bigr\|_\htw^2\Bigr).
\end{align*}
It then follows that for $h\le h_0$
\begin{align*}
\bigl\|\util-T_h(\util)\bigr\|_\htw
&\le CC_1^{-\frac12}\bigl(C_1^{-\frac12}+C_0^\frac12C_R \bigr)
h^{\ell-2} \|\ue\|_Y\|\ue\|_\hl,
\end{align*}
which is the inequality \eqref{ablem51bound}. The proof is complete.
\end{proof}

\begin{lem}\label{lem52}
Suppose assumptions {\rm [A1]--[A5]} hold.  
Then there exists an $h_1=h_1(\eps)>0$ such that
for $h\le {\rm min}\{h_0,h_1\}$, the operator $T_h$ is a
contracting mapping in the ball $\mathbb{B}_h(\rho_0)$ with
a contraction factor $\frac12$, that is
\begin{align*}
\bnorm{T_h(v_h)-T_h(w_h)}_\htw\le \frac12 \|v_h-w_h\|_\htw \qquad \forall 
v_h,w_h\in \mathbb{B}_h(\rho_0),
\end{align*}
where
\begin{align*}
\rho_0&={\rm min}\left\{\delta, CC_1^{\frac12}L^{-1}(h){\rm min}
\{C_1^\frac12,C_0^{-\frac12}C_R^{-1}\}\right\},
\end{align*}
and $h_1$ is chosen such that
\begin{align*}
h_1&=C\left(C_1^{-\frac12}L(h_1)
{\rm max}\{C_1^{-\frac12},C_0^\frac12C_R\}\right)^{\frac{1}{2-\ell}}.
\end{align*}
\end{lem}
\begin{proof}
By the definition of $T_h$, we have for any
$v_h,w_h\in \mathbb{B}_h(\rho_0),\ z_h\in V^h_0$,
\begin{align*}
a^{\eps}\bigl(T_h(v_h)-T_h(w_h),z_h\bigr)
&=a^{\eps}(v_h,z_h)-a^{\eps}(w_h,z_h)
+\eps(\Del(w_h-v_h),\Del z_h)\\
&\hspace{0.4in}+\bigl(F(w_h)-F(v_h),z_h\bigr)\\
&=\bl F^\prime[\ue](v_h-w_h),z_h\br+\bigl(F(w_h)-F(v_h),z_h\bigr).
\end{align*} 
Using
the mean value theorem, we obtain
\begin{align*}
a^{\eps}\(T_h(v_h)-T_h(w_h),z_h\)
&=\bl \Fp[\ue](v_h-w_h),z_h\br
+\(F(w_h)-F(v_h),z_h\)\\
&=\bl (\Fp[\ue]-\Fp[y_h])(v_h-w_h),z_h\br,
\end{align*}
where $y_h=w_h+\gamma (v_h-w_h)$ for some $\gamma\in [0,1]$.
Here, we have abused the notation of $y_h$, defining it differently in
two different proofs in this section.

Using {\rm [A2], [A5]}, and the triangle inequality yields
\begin{align*}
&C_1\bnorm{T_h(v_h)-T_h(w_h)}^2_\htw\\
&\le \bnorm{\Fp[\ue]-\Fp[y_h]}_{VV^*}\|v_h-w_h\|_\htw\bnorm{T_h(v_h)-T_h(w_h)}_\htw\\
&\qquad+C_0\bnorm{T_h(v_h)-T_h(w_h)}^2_\lt\\
&\le L(h)\|\ue-y_h\|_\htw\|v_h-w_h\|_\htw\bnorm{T_h(v_h)-T_h(w_h)}_\htw\\
&\qquad+C_0\bnorm{T_h(v_h)-T_h(w_h)}_\lt^2\\
&\le CL(h)\bigl(h^{\ell-2}\|\ue\|_\hl+\rho_0\bigr)\|v_h-w_h\|_\htw\bnorm{T_h(v_h)-T_h(w_h)}_\htw\\
&\qquad +C_0 \bnorm{T_h(v_h)-T_h(w_h)}_\lt^2.
\end{align*}

Thus,
\begin{align}
\label{lem52line1}C_1\bigl\|T_h(v_h)-T_h(w_h)\bigr\|^2_\htw&\le C_0 \bigl\|T_h(v_h)-T_h(w_h)\bigr\|_\lt^2\\
&\nonum \hspace{-0.5cm}+CC_1^{-1}L^2(h)\bigl(h^{2\ell-4}\|\ue\|_\hl^2+\rho_0^2\bigr)\|v_h-w_h\|_\htw^2.
\end{align}

Next, employing a duality argument similar to the one
used in Lemma \ref{lem51}, we let $w\in V_0\cap H^p(\Ome) \,(p>2)$
satisfy
\begin{align*}
\bl \left(G^\prime_\eps[\ue]\right)^*(w),z\br 
=\bigl(T_h(v_h)-T_h(w_h),z\bigr)\qquad \forall z\in V_0,
\end{align*}
with 
\begin{align}
\label{lem52reg}\|w\|_{H^p}\le C_R\bigl\|T_h(v_h)-T_h(w_h)\bigr\|_\lt.
\end{align}
Then using the same methods as in Lemma \ref{lem51}, we conclude
\begin{align*}
\bigl\|T_h(v_h)-T_h(w_h)\bigr\|_\lt^2
&\le C\Bigl(L(h)\bigl(h^{\ell-2}\|\ue\|_\hl+\rho_0\bigr)\|v_h-w_h\|_\htw\\
&\quad +C_2h^{r-2}\bigl\|T_h(v_h)-T_h(w_h)\bigr\|_\htw\Bigr)\|w\|_{H^p}\\
&\le CC_R\Bigl(L(h)\bigl(h^{\ell-2}\|\ue\|_\hl+\rho_0\bigr)\|v_h-w_h\|_\htw\\
&\quad +C_2h^{r-2}\bigl\|T_h(v_h)-T_h(w_h)\bigr\|_\htw\Bigr)\bigl\|T_h(v_h)-T_h(w_h)\bigr\|_\lt,
\end{align*}
and therefore
\begin{align*}
\bigl\|T_h(v_h)-T_h(w_h)\bigr\|_\lt^2
&\le C\Bigl(C_R^2L^2(h)\bigl(h^{2\ell-4}\|\ue\|^2_\hl+\rho_0^2\bigr)\|v_h-w_h\|_\htw^2\\
&\quad +C_2^2C_R^2h^{2r-4}\bigl\|T_h(v_h)-T_h(w_h)\bigr\|_\htw^2\Bigr).
\end{align*}

Using this last inequality in \eqref{lem52line1} gives us
\begin{align*}
&C_1\bigl\|T_h(v_h)-T_h(w_h)\bigr\|_\htw^2\\
&\le C\Bigl(L^2(h)\bigl(C_1^{-1}+C_0C_R^2\bigr)
\bigl(h^{2\ell-4}\|\ue\|^2_\hl+\rho_0^2\bigr)\|v_h-w_h\|^2_\htw\\
&\qquad +C_0C_2^2C_R^2 h^{2r-4}\bigl\|T_h(v_h)-T_h(w_h)\bigr\|_\htw^2.
\end{align*}
Therefore, for $h\le h_0$
\begin{align*}
&\bigl\|T_h(v_h)-T_h(w_h)\bigr\|_\htw\\
&\le CC_1^{-\frac12}L(h)\bigl(C_1^{-\frac12}+C_0^\frac12C_R)
\bigl(h^{\ell-2}\|\ue\|_\hl+\rho_0\bigr)\|v_h-w_h\|_\htw.
\end{align*}
It then follows from the definition of $\rho_0$ and $h_1$
that for $h\le {\rm min}\{h_0,h_1\}$,
\begin{align*}
&\bigl\|T_h(v_h)-T_h(w_h)\bigr\|_\htw\le \frac12\|v_h-w_h\|_\htw.
\end{align*}
\end{proof}

With these two lemmas in hand, we can now derive the main
results of this chapter.

\begin{thm}\label{abstractmainthm}
Under the same
hypotheses of Lemma \ref{lem52}, there exists
$h_2=h_2(\eps)>0$ such that for $h\le {\rm min}\{h_0,h_2\}$, 
there exists a locally unique solution to \eqref{momentfem}, 
where $h_2$ is chosen such that
\[
h_2=C\left(C_6\|\ue\|_\hl{\rm max}\left\{\delta^{-1},C_1^{-\frac12}L(h_2)
{\rm max}\bigl\{C_1^{-\frac12},C_0^\frac12C_R\bigr\}
\right\}\right)^{\frac{1}{2-\ell}}.
\]
Furthermore, there holds the following error estimate:
\begin{align}
\label{h2boundabstract}\|\ue-\ueh\|_\htw&\le C_{7}h^{\ell-2}\|\ue\|_\hl,
\end{align}
with 
\[
C_{7}=C_{7}(\eps)=CC_1^{-\frac12}
\|\ue\|_Y{\rm max}\{C_1^{-\frac12},C_0^\frac12C_R\}.\]
Moreover, there exists $h_3=h_3(\eps)>0$ such that for 
$h\le {\rm min}\{h_0,h_2,h_3\}$
\begin{align}\label{l2boundabstract}
\|\ue-\ueh\|_\lt\le C_{8}\Big(C_2h^{\ell+r-4}\|\ue\|_\hl
+L(h)C_{7}h^{2\ell-4}\|\ue \|^2_\hl\Big),
\end{align}
where 
\begin{align*}
h_3&=C\left(C_7\delta^{-1}\|\ue\|_\hl\right)^{\frac{1}{2-\ell}},
\quad
C_8=CC_{7}C_R,\quad
r={\rm min}\{p,k+1\}.
\end{align*}
\end{thm}

\begin{proof}
Let $\rho_1:=2C_{6}h^{\ell-2}\|\ue\|_\hl$, and note that
for $h\le h_2$, there holds $\rho_1\le \rho_0$.  
Thus for $h\le {\rm min}\{h_0,h_2\}$ and noting $h_2\le h_1$,
we use Lemmas \ref{lem51} and \ref{lem52} to conclude 
that for any  $v_h\in\mathbb{B}_h(\rho_1)$,
\begin{align*}
\bnorm{\util-T_h(v_h)}_\htw&\le
\bnorm{\util-T_h(\util)}_\htw+\bnorm{T_h(\util)-T_h(v_h)}_\htw\\
&\le C_{6}h^{\ell-2}\|\ue\|_\hl+\frac12\|\util-v_h\|_\htw\\
&\le \frac{\rho_1}{2}+\frac{\rho_1}{2}=\rho_1.
\end{align*}

Hence, $T_h$ maps $\mathbb{B}_h(\rho_1)$ into
$\mathbb{B}_h(\rho_1)$. Since $T_h$ is continuous and a 
contraction mapping in $\mathbb{B}_h(\rho_1)$, 
by Banach's Fixed Point Theorem \cite{Gilbarg_Trudinger01}
$T_h$ has a unique fixed point $u_h^\vepsi\in \mathbb{B}_h(\rho_1)$, which
is the unique solution to \eqref{momentfem}.  To derive the error estimate
\eqref{h2boundabstract}, we use the triangle inequality to obtain
\begin{align*}
\|\ue-\ueh\|_\htw
&\le \|\ue-\util\|_\htw+\|\util-\ueh\|_\htw\\
&\le Ch^{\ell-2}\|\ue\|_\hl+\rho_1
\le C_{7}h^{\ell-2}\|\ue\|_\hl.
\end{align*}

To obtain the $L^2$ error estimate \eqref{l2boundabstract}, 
we start with the error equation:
\begin{align*}
(\Del \err,\Del v_h)+
\bl F(\ue)-F(\ueh),v_h\br
=0\quad \forall v_h\in V_0^h,
\end{align*}
where $\err:=\ue-\ueh$.  Using the mean value theorem, we obtain
\begin{align}\label{abstracterreqn}
(\Del \err,\Del v_h)+\bl \Fp[y_h](\err),v_h\br
=0\quad \forall v_h\in V_0^h,
\end{align}
where $y_h=\ue-\gamma \err$ for some $\gamma\in [0,1]$.
Again, we have abused the notation of $y_h$, defining it differently in
different proofs.

Next, let $w\in H^p(\Ome)\cap V_0 \,(p>2)$ be the 
solution to the following auxiliary problem:
\begin{align*}
\bl (\Gp[\ue])^*(w),z\br=(\err,z)\qquad \forall z\in V_0,
\end{align*} 
with
\begin{align}\label{abstractmhp}
\|w\|_{H^p}\le C_R\|\err\|_\lt.
\end{align}

Using \eqref{abstracterreqn}, we then have for any $w_h\in V^h_0$
\begin{align}\label{ch4mainline1}
\|\err\|_\lt^2&=\bl (\Gp[\ue])^*(w),\err\br\\
&\nonum=\bl \Gp[\ue](\err),w\br\\
&\nonum=a^\eps(\err,w)\\
&\nonum=a^\eps(\err,w-w_h)
+\eps(\Del \err,\Del w_h)
+\bl \Fp[\ue](\err),w_h\br\\
&\nonum=a^\eps(\err,w-w_h)
+\bl \big(\Fp[\ue]-\Fp[y_h]\big)(\err),w_h\br\\
&\nonum\le CC_2\|\err\|_\htw\|w-w_h\|_\htw
+\bnorm{\Fp[\ue]-\Fp[y_h]}_{VV^*}\|\err\|_\htw\|w_h\|_\htw.
\end{align}

Then by \eqref{h2boundabstract}
for $h\le h_3$ 
\begin{align*}\|\ue-y_h\|_\htw=\gamma \|\err\|_\htw\le \del.
\end{align*}

Therefore, setting $w_h=\mathcal{I}_h w$ in \eqref{ch4mainline1}, 
we have for $h\le {\rm min}\{h_0,h_2,h_3\}$, 
\begin{align*}
\|\err\|_\lt^2&\le C\Big(C_2h^{r-2}\|\err\|_\htw
+L(h)\|\err\|_\htw^2\Big)\|w\|_{H^p}\\
&\le CC_R\Big(C_2h^{r-2}\|\err\|_\htw
+L(h)\|\err\|_\htw^2\Big)\|\err\|_\lt.
\end{align*}

Thus,
\begin{align*}
\|\err\|_\lt&
\le CC_R\Big(C_2h^{r-2}\|\err\|_\htw+L(h)\|\err\|_\htw^2\Big)\\
&\le C C_{7}C_R\Big(C_2h^{\ell+r-4}\|\ue\|_\hl
+L(h)C_{7}h^{2\ell-4}\|\ue\|^2_\hl\Big).
\end{align*}
\end{proof}

\begin{remarks}
(a) \label{Remark48i}
Noting $2\ell-4>\ell$ for $\ell\ge 4$ and $k>4$,
Theorem \ref{abstractmainthm} requires $p\ge 4$ 
to obtain optimal order error estimates in the $L^2$-norm.
This regularity condition is expected provided that 
the domain $\Ome$ is smooth and solution $\ue$ is sufficiently regular.

(b) \label{Remark48ii}
If $\(\Gp(v)\)^*$ is coercive on $V_0$, that is $C_0=0$ in 
the inequality \eqref{abGarding}, then $C_7=CC_1^{-1}\|\ue\|_Y$ 
in the error bound \eqref{h2boundabstract}.  Furthermore, 
it is expected that $C_1=O(\eps)$ in such cases, and therefore
\eqref{h2boundabstract} reads
\[
\|\ue-\ueh\|_\htw\le C\eps^{-1}h^{\ell-2}\|\ue\|_Y\|\ue\|_\hl.
\]

(c) We note that the constants $C_2,C_7,C_8$ appeared in the error bounds 
of Theorem \ref{abstractmainthm} all depend on some negative powers of $\eps$,
which is expected. The dependence of $C_2,C_7,C_8$ on $\eps^{-1}$ we derived 
are the worst-case scenarios, they are far from being sharp (in particular, in the
$3$-D case) althrough the proved convergence rates in $h$ are optimal.
In Section \ref{chapter-6} we shall present a detailed numerical study 
about the sharpness of the dependence of the error bounds on $\eps^{-1}$. 
Our numerical experiments suggest that the error bounds only grow
in $\eps^{-1}$ in some small power orders, which are considerably better than 
the theoretical estimates indicate.


\end{remarks}

%% file: chapter5.tex
\chapter{Mixed finite element approximations}\label{chapter-5}

The goal of this chapter is to construct and analyze a family of 
Hermann-Miyoshi mixed finite element methods for general fully nonlinear 
second order problem \eqref{generalPDEa}--\eqref{generalPDEb} based on the 
vanishing moment method \eqref{moment1}--\eqref{moment3}$_3$.  
The mixed formulation is based on rewriting \eqref{moment1} 
as a system of two second order PDEs by introducing an additional variable.  
By decoupling \eqref{moment1} as a system, we are able to approximate
\eqref{moment1}--\eqref{moment3}$_3$ using only $C^0$ finite elements, 
opposed to $C^1$ finite elements used in Chapter \ref{chapter-4}, 
which can be computational expensive and complicated.

We note that the theory of mixed finite element methods,
such as Hermann-Miyoshi mixed methods, has been extensively developed 
in the seventies and eighties for biharmonic problems
in two dimensions (cf. \cite{Ciarlet78, Brenner}).
It is straightforward to formulate these methods for the fourth order 
quasilinear PDE \eqref{moment1} in two and three dimensions.
Although it is a simple task to define mixed finite element methods for
problem \eqref{moment1}--\eqref{moment3}$_3$, proving existence
of solutions and obtaining convergence rates are quite difficult.
As is now well-known, proving existence and deriving error estimates
for mixed methods relies heavily on the so-called inf-sup condition, 
and naturally, this is the starting point in our analysis.  
However, due to the strong nonlinearity in \eqref{moment1},
the inf-sup condition is not sufficient for our purposes,
and therefore, we must look for other techniques to obtain
existence, uniqueness, and error estimates.  To this end,
we use a combined fixed-point and linearization 
technique that is in the same spirit as in the previous chapter.

The chapter is organized as follows.
In Section \ref{chapter-5-sec-2}, we define the mixed
formulation of \eqref{moment1}--\eqref{moment3}$_3$, and
then define the Hermann-Miyoshi mixed finite element method
based upon this formulation.  We then make certain
structure assumptions on the nonlinear differential operator $F$, which
will be used frequently in the analysis of the mixed finite element
method.  The assumptions are generally mild and are
very similar to those in Chapter \ref{chapter-4}.  In Section
\ref{chapter-5-sec-3}, we prove convergence results of 
the mixed finite element method for the linearized problem
\eqref{abstractlin1}--\eqref{abstractlin3}.  In Section \ref{chapter-5-sec-4}
we obtain our main results, where we obtain existence and uniqueness
for the proposed Hermann-Miyoshi mixed finite element method 
and also derive error estimates.

\section{Formulation of mixed finite element methods} 
\label{chapter-5-sec-2}

There are several popular mixed formulations for 
fourth order problems.  However, since the Hessian
matrix appears in \eqref{moment1} in a nonlinear
fashion, we cannot use $\Delta \ue$ as an additional 
variable.  This observation then rules out the
family of Ciarlet-Raviart mixed finite element methods.
On the other hand, this observation motivates us to try
Hermann-Miyoshi mixed elements which use
$\se:=D^2\ue$ as an additional unknown,
and so, in this chapter, we will only focus on developing
Hermann-Miyoshi type mixed methods for problem
\eqref{moment1}--\eqref{moment3}$_3$.

In addition to the notation introduced in Section \ref{chapter-1.2}, 
we also define the following space notation:
\begin{alignat*}{2}
&Q:=H^1(\Ome),\quad &&Q_0:=H^1_0(\Ome),\\
&Q_g:=\bigl\{v\in Q;\ v\big|_{\p\Ome}=g\bigr\},\quad 
&&W:=\bigl\{\mu\in Q^{n\times n};\ \mu_{ij}=\mu_{ji}\bigr\},\\
&W_0:=\bigl\{\mu\in W;\ \mu\nu\cdot\nu\big|_{\p\Ome}=0\bigr\},\quad 
&&W_\eps:=\bigl\{\mu\in W;\ \mu\nu\cdot\nu\big|_{\p\Ome}=\eps\bigr\}.
\end{alignat*}
Recall that we use Greek letters to represent tensor functions
and Roman letters to represent scalar functions throughout the paper.

To define the mixed variational formulation for problem
\eqref{moment1}--\eqref{moment3}$_3$, we rewrite the PDE
into a system of two second order equations as follows:
\begin{align}
\label{mixedmoment1}\se-D^2\ue&=0,\\
\label{mixedmoment2}\eps\Delta {\rm tr}(\se)+F(\se,\ue)&=0,
\end{align}
where $F(\se,\ue)$ is defined in \eqref{mixedFdef}.

Testing \eqref{mixedmoment1} with $\mu\in W_0$, we get
\begin{align}
\label{mixedline1}(\se,\mu)+\bigl(\Div(\mu),\nabla \ue\bigr)
&=\sum_{i=1}^{n-1}\left\langle \mu\nu \cdot \tau_i,\frac{\p g}{\p \tau_i}\right\rangle_{\p\Ome},
\end{align}
where $\{\tau_1(x),...,\tau_{n-1}(x)\}$ denotes the standard basis of the tangent
space to $\p\Omega$ at $x$, and 
\[
(\se,\mu)=\int_\Ome \se:\mu\: dx=\sum_{i,j=1}^n \int_{\Ome} \se_{ij}\mu_{ij}\: dx.
\]
Next, multiplying \eqref{mixedmoment2} with $w\in Q_0$ 
and integrating over $\Ome$ gives us
\begin{align}
\label{mixedline2}-\eps\(\Div (\se),\nabla w\)
+\(F(\se,\ue),w\)&=0.
\end{align}

Based on \eqref{mixedline1}--\eqref{mixedline2}, we 
define the mixed formulation of \eqref{moment1}--\eqref{moment3}$_3$
as follows:  find $(\se,\ue)\in W_\eps\times Q_g$ such that
\begin{alignat}{2}
\label{mixedvar1}
(\se,\kappa)+b(\kappa,\ue)&=G(\kappa)
\quad && \forall \kappa\in W_0,\\
\label{mixedvar2}b(\se,v)-\eps^{-1}c(\se,\ue,v)&=0
\quad && \forall v\in Q_0,
\end{alignat}
where for $\mu\in W,\ v,w\in Q$
\begin{alignat}{2}
\nonum b(\mu,v):&=\bigl(\Div (\mu),\nab v\bigr),\qquad c(\mu,w,v):
&&=\bigl( F(\mu,w),v\bigr),\\
\label{Gdef}G(\mu):&=\sum_{i=1}^{n-1}
\left\langle \mu\nu \cdot \tau_i,\frac{\p g}{\p \tau_i}\right\rangle_{\p\Ome}.&&
\end{alignat}

Next, let $\mathcal{T}_h$ be a quasiuniform triangular 
or quadrilateral partition of $\Omega$ if $n=2$,
and  tetrahedral or hexahedra mesh if $n=3$ parameterized by $h\in (0,1)$.  
Let $Q^h\subset Q$ be the Lagrange finite element
space consisting of globally continuous, piecewise polynomials
of degree $k\ (\ge 2)$ associated with the mesh $\mathcal{T}_h$.

We then define the following finite element spaces:
\begin{alignat*}{2}
Q^h_0:&=Q^h\cap Q_0,\qquad &&Q^h_g:=Q^h\cap Q_g,\\
W^h_0&:=\left[Q^h\right]^{n\times n}\cap W_0,
\qquad &&W^h_\eps:=\left[Q^h\right]^{n\times n}\cap W_\eps,
\end{alignat*}
and define the norms $\tbar{\cdot}{\cdot},\ \ttbar{\cdot}{\cdot}:
W\times Q\mapsto \mathbf{R}^+$
such that for any $(\mu,v)\in W\times Q$,
\begin{align*}
\tbar{\mu}{v}:&=\norm{\mu}_\lt+K^\frac12_1\eps^{-\frac12}\norm{v}_\ho,\\
\ttbar{\mu}{v}:&=h\norm{\mu}_\ho+\tbar{\mu}{v},
\end{align*}
and $K_1$ is defined by {\rm [B2]} below.

Based on \eqref{mixedvar1}--\eqref{mixedvar2}, we define the 
Herman-Miyoshi-type mixed finite element method as follows:
find $(\se_h,\ue_h)\in W^h_\eps\times Q^h_g$ 
such that
\begin{alignat}{2}
\label{mixedfem1}
(\se_h,\kappa_h)+b(\kappa_h,\ue_h)
&=G(\kappa_h)\quad && \forall \kappa_h\in W_0^h,\\
\label{mixedfem2}b(\se_h,z_h)-c(\se_h,\ue_h,z_h)&=0
\quad && \forall z_h\in Q^h_0.
\end{alignat}

The main goal of this chapter is to prove
existence and uniqueness for problem 
\eqref{mixedfem1}--\eqref{mixedfem2}
and to also derive error estimates for 
$\se-\seh$ and $\ue-\ueh$.  As a first step,
we state the following inf-sup condition for the finite element
pair $(W^h_0,Q^h_0)$.  The proof can be found in \cite{Feng3,Neilan_thesis}.
\begin{lem}\label{infsuplem}
For every $w_h\in Q^h_0$, there exists $C>0$ independent of $h$, such that
\begin{align}
\label{infsup}\sup_{\mu_h\in W^h_0} \frac{b(\mu_h,w_h)}{\norm{\mu_h}_\ho}\ge C\norm{w_h}_\ho.
\end{align}
\end{lem}

\begin{remark}\label{Remark52}
By \cite[Proposition 1]{Falk_Osborn_1980},
Lemma \ref{infsuplem} implies that there exists a linear operator $\Pi^h:W\mapsto W^h$
such that
\begin{align}
\label{Piop}b\bigl(\mu-\Pi^h\mu,w_h\bigr)&=0\qquad \forall w_h\in Q^h_0,
\end{align}
and for $\mu\in W\cap \left[H^s(\Ome)\right]^{n\times n},\ s\ge 1$, there holds
\begin{align}\label{Piapprox}
\norm{\mu-\Pi^h\mu}_{H^j}
\le Ch^{\ell-j}|\mu |_\hl\qquad j=0,1,\quad 1\le \ell\le {\rm min}\{s,k+1\}.
\end{align}
\end{remark}

\medskip
Next, we assume the following structure conditions
on the nonlinear differential operator $F$, which play 
an important role in our analysis.

\medskip
{\bf Assumption (B)}
\smallskip

\begin{enumerate}
\item[{\bf [B1]}]
There exists $\eps_0>0$ such that
for all $\eps\in (0,\eps_0]$, there exists a locally unique
solution to \eqref{moment1}--\eqref{moment3}$_3$ 
with $\ue\in H^{s}(\Omega)\ (s\ge 3).$\smallskip

\item[{\bf [B2]}]
The operator $\(G_\eps^\prime[\se,\ue]\)^*$ 
(the adjoint of $G_\eps^\prime[\se,\ue]$) 
is an isomorphism 
from $H^2(\Ome)\cap H^1_0(\Ome)$ to 
$\(H^2(\Ome)\cap H^1_0(\Ome)\)^*$.  
That is for any $\varphi\in \(H^2(\Ome)\cap H^1_0(\Ome)\)^*$, there
exists $v\in H^2(\Ome)\cap H^1_0(\Ome)$ such that
\begin{align}
\label{a2pagain}\left\langle \(\Gp[\se,\ue]\)^*(D^2v,v),w\right\rangle
=\langle \varphi,w\rangle\qquad \forall w\in H^2(\Ome)\cap H^1_0(\Ome).
\end{align}
Furthermore,
there exists positive constants $K_0=K_0(\eps),\ K_1=K_1(\eps),$ such that
the following G\aa rding inequality holds\footnote{See Remark \ref{Remark53iv}(d) 
for an interpretation.}:
\begin{align}\label{mixedGarding1}
\bigl\langle \Fp[\se,\ue](D^2v,v),v\bigr\rangle
&\ge  K_1\|v\|_\ho^2-K_0\|v\|_\lt^2\qquad \forall v\in Q_0,
\end{align}
where $\bl\cdot,\cdot\br$ denotes the dual pairing of $Q_0$ and $Q_0^*$.
Also, there exists $K_2=K_2(\eps)>0$ such that
\begin{align*}
\bnorm{\Fp[\se,\ue]}_{QQ^*}&\le K_2,
\end{align*}
where
\begin{align*}
\bnorm{F^\prime[\se,\ue]}_{QQ^*}
:&=\sup_{v\in Q_0} \frac{\bnorm{\Fp[\se,\ue](D^2v,v)}_{H^{-1}}}{\|v\|_\ho}\\
:&=\sup_{v\in Q_0}\sup_{w\in Q_0} 
\frac{\bl \Fp[\se,\ue](D^2v,v),w\br}{\|v\|_\ho\|w\|_\ho}.
\end{align*}
Moreover, there exists $p\ge 3$ and 
$K_{R_m}>0,\ (m=0,1)$ 
such that if $\varphi\in H^{-m}(\Ome)$
and $v\in V_0$ satisfies \eqref{a2pagain}, then $v\in H^{p-m}(\Ome)$ and
\begin{align*}
\|v\|_{H^{p-m}}\le K_{R_m}\|\varphi\|_{H^{-m}}.
\end{align*}

\item[{\bf [B3]}] 
There exists Banach spaces $X,\ Y$ with a functional
\[
\|(\cdot,\cdot)\|_{X\times Y}:X\times Y \mapsto \mathbf{R}^+,
\]
and a constant $C>0$ such that
for all $\omega\in X,\ y\in Y,\ \chi\in W,\ v\in Q$
\begin{align*}
\bnorm{\Fp[\ome,y](\chi,v)}_{H^{-1}}
\le C\big\|(\ome,y)\bigr\|_{X\times Y} \(\|\chi\|_\lt+\|v\|_\ho\).
\end{align*}
Furthermore, $\big\|(\cdot,\cdot)\bigr\|_{X\times Y}$
is well-defined and finite on $W^h\times Q^h$.\medskip

\item[{\bf [B4]}]
There exists a constant $K_3=K_3(\eps)>0$
such that
\begin{align*}
&\Bigl\|\bigl(\stil-\gamma\se,\util-\gamma \ue\bigr)\Bigr\|_{X\times Y}\le K_3(\eps)\qquad \forall \gamma\in [0,1]. 
\end{align*}
where 
$\util\in Q^h_g$ is the finite element
interpolant of $\ue$.\medskip

\item[{\bf [B5]}]
There exists a constant
$\delta=\delta(\eps)\in (0,1)$, such that for any 
$(\mu_h,v_h)\in W^h_\eps \times Q_g^h$
with $\ttbar{\stil-\mu_h}{\util-v_h}\le \delta$, 
there holds $\forall (\kappa_h,z_h)\in W^h\times Q^h$
\begin{align*}
&\sup_{w_h\in Q^h}
\frac{\Bigl\langle \bigl(F^\prime[\se,\ue]-F^\prime[\mu_h,v_h]\bigr)
\bigl(\kappa_h,z_h\bigr),w_h\Bigr\rangle}{\|w_h\|_\ho}\\
&\qquad\nonum\qquad\le R(h)\bigl(\|\se-\mu_h\|_\lt+\|\ue-v_h\|_\ho\bigr)
\ttbar{\kappa_h}{z_h},
\end{align*}
where $R(h)=R(\eps,h)$ may depend on $\eps$ and $h$ 
and $R(h)=o(h^{2-\ell})$.
\medskip
\item[{\bf [B6]}]
There exists $K_G=K_G(\eps)>0$ and
$\alpha>0$ such that
for any 
\begin{align*}
&(\chi_h,v_h)\in \mathbb{T}_h
:=\bigl\{(\kappa_h,z_h)\in W^h_0\times Q_0;
\ (\kappa_h,\chi_h)+b(\chi_h,z_h)=0\ \forall \chi_h\in W^h_0\bigr\},
\end{align*}
\hspace{-0.6cm}there holds
\begin{align*}
\norm{F^\prime[\se,\ue](\chi_h-D^2v_h,0)}_{H^{-1}}
\le K_Gh^\alpha \ttbar{\chi_h}{v_h}.
\end{align*}
\end{enumerate}

\medskip
\begin{remarks}
(a) \label{Remark53i}
We made an effort in our presentation to state
assumptions in this section that resemble those 
in the previous chapter, where conforming
finite element methods for \eqref{moment1}--\eqref{moment3}$_1$ were studied.
It is clear that conditions {\rm [B1]--[B6]} are similar, 
but slightly stronger than conditions {\rm [A1]--[A5]}.
For example, the inequality \eqref{mixedGarding1}
suggests that the operator $-\Fp[\ue]$ is {\em uniformly}
elliptic, which rules out degenerate problems.
However, assumptions {\rm [B1]--[B6]} are still not very restrictive, 
and we will show in Chapter \ref{chapter-6}
that many well-known fully nonlinear second order differential  operators 
satisfy these requirements.  We also show a simple trick
at the end of the chapter which makes it possible to 
incorporate degenerate elliptic PDEs (i.e.\ $K_1$=0) into the theory.

(b) \label{Remark53ii}
We note that by definition of $\se,\ \Gp,$ and $\Fp$
(see Section \ref{chapter-1.2})
\begin{alignat*}{1}
\Gp[\se,\ue](D^2v,v)&=\Gp[D^2\ue,\ue](D^2v,v)=\Gp[\ue](v),\\
\Fp[\se,\ue](D^2v,v)&=\Fp[D^2\ue,\ue](D^2v,v)=\Fp[\ue](v).
\end{alignat*}
It then seems redundant to write $\Gp[\se,\ue](D^2v,v)$ and
$\Fp[\se,\ue](D^2v,v)$ instead of $\Gp[\ue](v)$ and $\Fp[\ue](v)$.
However, this (longer) short-hand notation naturally fits into 
the mixed method framework, and makes the subsequent 
analysis easier to follow.

(c) \label{Remark53iii}
It is obvious that assumption {\rm [B1]} is needed,
and this assumption is actually the same as {\rm [A1]}; we include
it again for consistency and standardization.

(d) \label{Remark53iv}
Assumption {\rm [B2]} is a natural extension of {\rm [A2]}, and $F$ 
is expected to satisfy these conditions provided that $-F$ 
is uniformly elliptic at $\ue$, and $\p\Ome$ is sufficiently regular.
We note that \eqref{mixedGarding1} needs to be understood with care
because of the special notation we use. 
The left-hand side should be understood in the distributional 
sense. To derive the inequality, an integration by parts 
must be used on the second order derivative term. 
Also, Remark \ref{Remark48ii} gives heuristic estimates for 
the constants $K_{R_0}$ and $K_{R_1}$ in terms of $\eps$.

(e) \label{Remark53v}
By the standard interpolation theory and \eqref{Piapprox}, we have
\begin{align}
\label{suinterp}h^{-1}\norm{\se-\stil}_\lt+\norm{\se-\stil}_\ho 
&\le Ch^{\ell-1}\|\se\|_\hl.\\ 
\label{suinterp2}h^{-1}\norm{\ue-\util}_\lt+\norm{\ue-\util}_\ho
&\le Ch^{\ell-1}\|\ue\|_\hl.
\end{align}

(f) \label{Remark53vi}
Condition {\rm [B5]}, which is used in the fixed-point argument,
states that $F^\prime$ is in some sense locally Lipschitz near $(\se,\ue)$. 

(g) Clearly, if $(\kappa,z)\in W_0\times Q_0$ satisfy
\begin{align*}
(\kappa,\chi)+b(\chi,z)=0\quad \forall \chi \in W_0,
\end{align*}
then $D^2z=\kappa$ in a weak sense.  However, if 
$(\kappa_h,z_h)\in\mathbb{T}_h$, the analogous equality 
$D^2z_h=\kappa_h$ is not necessarily true.
Assumption {\rm [B6]} indicates that the discrepancy between 
$\kappa_h$ and $D^2z_h$ under the image of $F^\prime[\se,\ue]$ 
is small. However, in what follows, we show
that assumption {\rm [B6]} holds with $\alpha=1$
 if $F$ is sufficiently smooth at the solution $(\se,\ue)$.
\end{remarks}

\begin{prop}
\label{B6prop}
Suppose 
\begin{align*}
\frac{\p F(\se,\ue)}{\p r_{ij}}\in L^\infty(\Ome)\cap W^{1,\frac65}(\Ome)\qquad i,j=1,2,...,n.
\end{align*}
Then assumption {\rm [B6]} holds with $\alpha=1$ and 
\begin{align*}
K_G=
C\left(\max_{1\le i,j\le n} 
\left\|\frac{\p F(\se,\ue)}{\p r_{ij}}\right\|_{L^\infty}
+\max_{1\le i,j\le n} 
\left\|\frac{\p F(\se,\ue)}{\p r_{ij}}\right\|_{W^{1,\frac65}}\right).
\end{align*}

\end{prop}
\begin{proof}

For any $z\in Q_0$, define $\lambda^\eps$ such that 
\[
\lambda_{ij}^\eps=\displaystyle\frac{\p F(\se,\ue)}{\p r_{ij}}z.
\]
Then using the property \eqref{Piop}, 
we have for any $(\chi_h,v_h)\in\mathbb{T}_h$
\begin{align*}
\bl \Fp[\se,\ue](\chi_h-D^2v_h,0),z\br
&=\bl \chi_h-D^2v_h,\lambda^\eps\br
=(\chi_h,\lambda^\eps)+b(\lambda^\eps,v_h)\\
&=(\chi_h,\lambda^\eps)+b\(\Pi^h \lambda^\eps,v_h\)
=\(\chi_h,\lambda^\eps-\Pi^h\lambda^\eps\)\\
&\le \ttbar{\chi_h}{v_h}\bigl\|\lambda^\eps-\Pi^h\lambda^\eps\|_\lt.
\end{align*}

Next, by \eqref{Piapprox}, the definition of $\lambda^\eps$, 
and the product rule, we have
\begin{align*}
&\bigl\|\lambda^\eps-\Pi^h\lambda^\eps\|_\lt\\
&\le Ch\left(\|\nab z\|_\lt \max_{1\le i,j\le n} 
\left\|\frac{\p F(\se,\ue)}{\p r_{ij}}\right\|_{L^\infty}
+\|z\|_{L^6}
\max_{1\le i,j\le n} 
\left\|\frac{\p F(\se,\ue)}{\p r_{ij}}\right\|_{W^{1,\frac65}}\right).
\end{align*}
Therefore, by Poincar\'e's inequality and a Sobolev inequality
\begin{align*}
&\bigl\|\lambda^\eps-\Pi^h\lambda^\eps\|_\lt\\
&\le Ch\left(\max_{1\le i,j\le n} 
\left\|\frac{\p F(\se,\ue)}{\p r_{ij}}\right\|_{L^\infty}
+\max_{1\le i,j\le n} 
\left\|\frac{\p F(\se,\ue)}{\p r_{ij}}\right\|_{W^{1,\frac65}}\right)
\|z\|_\ho.
\end{align*}
The result follows from the above inequality.
\end{proof}

\section{Linearization and its mixed finite element approximations}
\label{chapter-5-sec-3}

To derive existence, uniqueness, and the desired error 
estimates for the mixed finite element method \eqref{mixedfem1}--\eqref{mixedfem2},
we must first study the mixed finite element approximations of
\eqref{abstractlin1}--\eqref{abstractlin3}, but with an alternative boundary condition:
\begin{alignat}{2}
\label{mixedlin1}G^\prime_\eps[\ue](v)&=\varphi\qquad \text{in }\Ome,\\
\label{mixedlin2}v&=0\qquad \text{on }\p\Ome,\\
\label{mixedlin3}D^2v\nu\cdot\nu&=0\qquad \text{on }\p\Ome,
\end{alignat}
where $\varphi\in Q^*_0$ is some given function.
Using arguments similar to the proof of 
Theorem \ref{linexistencethm}, we conclude that there exists 
a unique solution $v\in H^2(\Ome)\cap H^1_0(\Ome)$ 
to \eqref{mixedlin1}--\eqref{mixedlin3}.

To introduce a mixed formulation for \eqref{mixedlin1}--\eqref{mixedlin3}, 
we rewrite the fourth order PDE \eqref{mixedlin1} as the following system 
of two second order PDEs:
\begin{align}
\label{mixedform1}\chi-D^2v&=0,\\
\label{mixedform2}\eps\Del {\rm tr}(\chi)+F^\prime[\se,\ue](D^2v,v)&=\varphi,
\end{align}
where ${\rm tr}(\chi)$ denotes the trace of $\chi$.

The mixed variational formulation of \eqref{mixedlin1}--\eqref{mixedlin3}
is then defined as follows:
find $(\chi,v)\in W_0\times Q_0$ such that
\begin{alignat}{2}
\label{mixedlinvar1}(\chi,\mu)+b(\mu,v)&=0\qquad &&\forall \mu\in W_0,\\
\label{mixedlinvar2}b(\chi,w)
-\eps^{-1}d(\ue;v,w)&=-\eps^{-1}\langle \varphi,w\rangle\qquad &&\forall w\in Q_0,
\end{alignat}
where for $v,w\in Q$
\begin{align*}
d(\ue;v,w):=
\bigl\langle F^\prime[\se,\ue](D^2v,v),w\bigr\rangle
\end{align*}

\begin{remark}
We note again that the right-hand side of $d(\cdot;\cdot,\cdot)$ 
should be understood in the distributional sense.

\end{remark}

\subsection{Mixed finite element approximation of linearized problem}
Based on the variational formulation
\eqref{mixedlinvar1}--\eqref{mixedlinvar2}, we define the 
mixed finite element method for \eqref{mixedlin1}--\eqref{mixedlin3}
as seeking $(\chi_h,w_h)\in W^h_0\times Q^h_0$ such that
\begin{alignat}{2}
\label{mixedlinfem1}(\chi_h,\mu_h)+b(\mu_h,v_h)&=0\qquad &&\forall \mu_h\in W^h_0,\\
\label{mixedlinfem2}b(\chi_h,w_h)
-\eps^{-1}d(\ue;v_h,w_h)&=-\eps^{-1}\langle \varphi,w_h\rangle\quad 
&&\forall w_h\in Q_0^h.
\end{alignat}

Our objective in this section is to prove existence and uniqueness for problem
\eqref{mixedlinfem1}--\eqref{mixedlinfem2} and then to derive error estimates in
various norms.

\begin{thm}\label{mixedlinthm}
Suppose assumptions {\rm [B1]--[B2]}
hold.  Let $v\in H^{s}(\Ome)\ (s\ge 3)$ be the unique
solution to \eqref{mixedlin1}--\eqref{mixedlin3} and $\chi=D^2v$.
Then there exists $h_0=h_0(\eps)>0$
such that for $h\le h_0$, there exists a unique solution
$(\chi_h,w_h)\in W^h_0\times Q^h_0$ to problem 
\eqref{mixedlinfem1}--\eqref{mixedlinfem2}, where
\begin{align*}
h_0&=\left\{
\begin{array}{ll}
C\left({\rm min}\left\{\left(K_0K_1^{-1}K_2^2K_{R_0}^2\right)^{\frac{1}{2-2r}},
\left(K_0 K_{R_0}^2\eps\right)^{\frac{1}{4-2r}}\right\}\right) &\mbox{if } K_0\neq 0,\\
1 &\mbox{if } K_0=0,
\end{array}\right.\\
r&={\rm min}\{p,k+1\}.
\end{align*}
Furthermore, there hold the following error estimates:
\begin{align}\label{mixedlinbound1}
&\ttbar{\chi-\chi_h}{v-v_h}\le Ch^{\ell-2}\bigl(K_4 h+1\bigr)\lnorm,\\
&\|v-v_h\|_\lt
\le K_5h^{\ell+r-4}\bigl(K_4h+1\bigr)\lnorm. \label{mixedlinbound2}
\end{align}
where 
\begin{align*}
K_4&=C{\rm max}\{K_1^{-\frac12}K_2 \eps^{-\frac12},K_0^\frac12 K_{R_0}\eps^\frac12 \},
\quad K_5=CK_1^{-\frac12}K_2K_{R_0}\eps^\frac12, \\
\quad \ell&={\rm min}\{s,k+1\}.
\end{align*}
\end{thm}

\begin{proof}
We first start by showing that the error estimates 
\eqref{mixedlinbound1}--\eqref{mixedlinbound2}
 hold in the case that there does exist a solution to 
\eqref{mixedlinfem1}--\eqref{mixedlinfem2}.

Let $\vtil$ denote the standard finite element interpolant of $v$ in $Q^h_0$.  
Then using \eqref{Piop}, we have for all $(\mu_h,w_h)\in W^h_0\times Q^h_0$,
\begin{align}
&\label{mixedext1}\(\chi_h-\Pi^h\chi,\mu_h\)+b\(\mu_h,v_h-\vtil\)
=\(\chi-\Pi^h\chi,\mu_h\)+b\(\mu_h,v-\vtil\), \\
&\label{mixedext2}b\(\chi_h-\Pi^h\chi,w_h\)-\eps^{-1}d\(\ue;v_h-\vtil,w_h\)\\
&\nonum\hspace{2in}=\eps^{-1}d\(\ue;\vtil-v,w_h\).
\end{align}

Setting $\mu_h=\chi_h-\Pi^h\chi$ and $w_h=v_h-\vtil$ and subtracting
\eqref{mixedext2} from \eqref{mixedext1} yields
\begin{align}
&\(\chi_h-\Pi^h\chi,\chi_h-\Pi^h\chi\)
+\eps^{-1}d\(\ue;v_h-\vtil,v_h-\vtil\)\\
&\nonum\qquad=\(\chi-\Pi^h\chi,\chi_h-\Pi^h\chi\)+b\(\chi_h-\Pi^h\chi,v-\vtil\)\\
&\nonum\hspace{2in}+\eps^{-1}d\(\ue;v-\vtil,v_h-\vtil\).
\end{align}

Thus, by assumption {\rm [B2]},
\begin{align*}
&\tbar{\chi-\Pi^h\chi}{v_h-\vtil}^2\\
&\qquad \le 
\norm{\chi-\Pi^h\chi}_\lt\norm{\chi_h-\Pi^h\chi}_\lt
+\bnorm{\Div(\chi_h-\Pi^h\chi)}_\lt\norm{\nabla (v-\vtil)}_\lt\\
&\qquad \qquad +\eps^{-1}\bnorm{F^\prime[\se,\ue]}_{QQ^*}\norm{v-\vtil}_\ho\bnorm{w_h}_\ho
+K_0\eps^{-1}\norm{v_h-\vtil}_\lt^2\\
&\qquad \le 
\bnorm{\chi-\Pi^h\chi}_\lt\bnorm{\chi_h-\Pi^h\chi}_\lt
+Ch^{-1}\bnorm{\chi_h-\Pi^h\chi}_\lt\bnorm{\nabla (v-\vtil)}_\lt\\
&\qquad \qquad +K_2\eps^{-1}\bnorm{v-\vtil}_\ho\bnorm{v_h-\vtil}_\ho
+K_0\eps^{-1}\bnorm{v_h-\vtil}_\lt^2,
\end{align*}
where we have used the inverse inequality in the last expression.

Using the Schwarz inequality, standard
interpolation estimates, and rearranging terms, we have
\begin{align*}
&\tbar{\chi_h-\Pi^h\chi}{v_h-\vtil}^2\\
&\qquad \le 
C\Bigl(\bnorm{\chi-\Pi^h\chi}_\lt^2+h^{-2}\bnorm{\nabla (v-\vtil)}_\lt^2\\
&\qquad \qquad +K_1^{-1}K^2_2\eps^{-1}\bnorm{v-\vtil}^2_\ho
+K_0\eps^{-1}\bnorm{v_h-\vtil}_\lt^2\Bigr)\\
&\qquad \le C\Bigl(h^{2\ell-4}\bigl(K_1^{-1}K^2_2\eps^{-1}h^2+1\bigr) \lnorms 
+K_0\eps^{-1}\bnorm{v_h-\vtil}^2_\lt\Bigr),
\end{align*}
which by an application of the triangle and inverse inequalities yields
\begin{align}
\label{mixedlinlemline1}
&\ttbar{\chi-\chi_h}{v-v_h}^2\\
&\nonum\qquad\le C \Bigl(h^{2\ell-4}\bigl(K^2_3h^2+1\bigr)\lnorms
+K_0\eps^{-1}\norm{v-v_h}^2_\lt\Bigr).
\end{align}

Continuing, we let $w\in Q_0\cap H^p(\Ome)\, (p\ge 3)$  be the solution
to the following auxiliary problem:
\begin{alignat*}{2}
\(\Gp[\ue]\)^*(w)&=v-v_h\qquad &&\text{in }\Ome,\\
D^2w\nu\cdot\nu&=0\qquad &&\text{on }\p\Ome.
\end{alignat*}
By assumption {\rm [B2]}, there exists such a solution and
\begin{align}
\label{mixedlinreg}\norm{w}_{H^p}\le K_{R_0}\norm{v-v_h}_\lt.
\end{align}

Setting $\kappa=D^2w\in \left[H^{p-2}(\Ome)\right]^{n\times n}$,
it is easy to verify that $(\kappa,w)$ satisfy
\begin{alignat*}{2}
(\kappa,\mu)+b(\mu,z)&=0\qquad &&\forall \mu\in W_0,\\
b(\kappa,z)-\eps^{-1}d^*(\ue;w,z)&=\eps^{-1}
(v-v_h,z)\qquad &&\forall z\in Q_0,
\end{alignat*}
where $d^*(\ue;\cdot,\cdot)$ denotes the adjoint of $d(\ue;\cdot,\cdot)$, 
that is,
\[
d^*(\ue;v,w)=d(\ue;w,v)\qquad \forall v,w\in Q_0.
\]

We also note that there hold the following Galerkin orthogonality:
\begin{alignat*}{2}
(\chi-\chi_h,\mu_h)+b(\mu_h,v-v_h)&=0\qquad &&\forall \mu_h\in W^h_0,\\
b(\chi-\chi_h,w_h)-\eps^{-1}d(\ue;v-v_h,w_h)&=0
\qquad &&\forall w_h\in Q_0^h.
\end{alignat*}

Thus, choosing $z=v-v_h$ we get
\begin{align*}
\eps^{-1}\norm{v-v_h}_\lt^2
&=-b(\kappa,v-v_h)
+\eps^{-1}d^*(\ue;w,v-v_h)\\
&=-b\(\kappa-\Pi^h\kappa ,v-v_h\)
+\eps^{-1}d(\ue;v-v_h,w)\\
&\qquad -b\(\Pi^h\kappa,v-v_h\)\\
&=-b\(\kappa-\Pi^h\kappa ,v-\vtil\)
+\eps^{-1}d(\ue;v-v_h,w)\\
&\qquad +\(\chi-\chi_h,\Pi^h\kappa\)\\
&=-b\(\kappa-\Pi^h\kappa,v-\vtil\)
+\eps^{-1}d(\ue;v-v_h,w)\\
&\qquad +(\chi-\chi_h,\kappa)+\(\chi-\chi_h,\Pi^h\kappa-\kappa\)\\
&=-b\(\kappa-\Pi^h\kappa,v-\vtil\)
+\eps^{-1}d(\ue;v-v_h,w)\\
&\qquad -b(\chi-\chi_h,w)+\(\chi-\chi_h,\Pi^h\kappa-\kappa\)\\
&=-b\(\kappa-\Pi^h\kappa,v-\vtil\)
+\eps^{-1}d\(\ue;v-v_h,w-\wtil\)\\
&\qquad -b\(\chi-\chi_h,w-\wtil\)+\(\chi-\chi_h,\Pi^h\kappa-\kappa\).
\end{align*}

Therefore, using \eqref{mixedlinreg},
\begin{align*}
&\eps^{-1}\norm{v-v_h}_\lt^2\\
&\quad \le \bnorm{\Div(\kappa-\Pi^h\kappa)}_\lt\bnorm{\nabla (v-\vtil)}_\lt
+K_2\eps^{-1}\bnorm{v-v_h}_\ho\bnorm{w-\wtil}_\ho\\
&\quad\qquad +\bnorm{\Div(\chi-\chi_h)}_\lt\bnorm{\nabla (w-\wtil)}_\lt
+\bnorm{\chi-\chi_h}_\lt \bnorm{\Pi^h\kappa-\kappa}_\lt\\
&\quad \le C\Big(h^{\ell+r-4}\norm{\kappa}_{H^{r-2}}\norm{v}_\hl
+K_2\eps^{-1}h^{r-1} \norm{w}_{H^r}\norm{\nabla (v-v_h)}_\lt\\
&\quad \qquad +h^{r-1}\bnorm{\Div(\chi-\chi_h)}_\lt\bnorm{w}_{H^r}
+h^{r-2}\norm{\chi-\chi_h}_\lt \norm{\kappa}_{H^{r-2}}\Bigr)\\
&\quad \le CK_{R_0}\Bigl(
h^{\ell+r-4}\norm{v}_\hl+K_2\eps^{-1}h^{r-1}\norm{\nab (v-v_h)}_\lt\\
&\quad \qquad +h^{r-1}\norm{\chi-\chi_h}_\ho+h^{r-2}\norm{\chi-\chi_h}_\lt\Bigr)
\norm{v-v_h}_\lt,
\end{align*}
and hence
\begin{align}
\label{mixedlinl2}
\|v-v_h\|_\lt^2&\le
C K^2_{R_0} \eps^2 \Bigl(h^{2\ell+2r-8}\|v\|^2_\hl
+K^2_2 \eps^{-2}h^{2r-2}\|\nab (v-v_h)\|^2_\lt\\
&\nonum\ \ +h^{2r-2} \|\chi-\chi_h\|^2_\ho
+h^{2r-4}\|\chi-\chi_h\|^2_\lt\Bigr).
\end{align}

Using estimate \eqref{mixedlinl2} in \eqref{mixedlinlemline1} yields
\begin{align*}
&\ttbar{\chi-\chi_h}{v-v_h}^2\\
&\qquad \le C\Bigl(h^{2\ell-4}\bigl(K_3^2h^2+1\bigr)\lnorms
+\eps^{-1}K_0\|v-v_h\|_\lt^2\Big)\\
&\qquad \le C\Bigl(h^{2\ell-4}\bigl(K_4^2h^2+1\bigr)\lnorms\\
&\qquad\quad+K_0 K_{R_0}^2\eps  \Bigl[h^{2\ell+2r-8}\|v\|^2_{H^\ell}
+K_2^2\eps^{-2}h^{2r-2}\norm{\nab (v-v_h)}_\lt^2\\
&\qquad\quad+h^{2r-2}\norm{\chi-\chi_h}_\ho^2+h^{2r-4}\norm{\chi-\chi_h}_\lt^2
\Bigr]\Big).
\end{align*}

It then follows that for $h\le h_0$,
\begin{align*}
&\ttbar{\chi-\chi_h}{v-v_h}^2\\
&\quad\le C\Bigl(h^{2\ell-4}\bigl(K_4^2h^2+1\bigr)\lnorms
+K_0K_{R_1}^2\eps  h^{2\ell+2r-8}\norm{v}^2_\hl\Bigr),
\end{align*}
and therefore
\begin{align*}
&\ttbar{\chi-\chi_h}{v-v_h}\\
&\qquad\le C\Bigl\{h^{\ell-2}\bigl(K_4h+1\bigr)\lnorm
+K_0^\frac12K_{R_0}\eps^\frac12
h^{\ell+r-4}\norm{v}_\hl
\Big\}\\
&\qquad \le Ch^{\ell-2}\bigl(K_4h+1\bigr)\lnorm,
\end{align*}
where we have used the fact that $r\ge 3$.  Finally, 
\eqref{mixedlinbound2} is obtained from 
\eqref{mixedlinbound1} and \eqref{mixedlinl2}.

So far, we have been working under the assumption
that there exists a solution $(\chi_h,v_h)$.
However, using the Schatz's argument 
similar to the end of Theorem \ref{abstractbound1thm}, 
we can conclude from \eqref{mixedlinbound1}--\eqref{mixedlinbound2}
 that \eqref{mixedlinfem1}--\eqref{mixedlinfem2} has a unique
solution for $h\le h_0$.
\end{proof}

\begin{remarks}
(a) \label{Remark56i}
To obtain optimal order error estimates in 
the $L^2$-norm \eqref{mixedlinbound2}, 
we require $p\ge 4$ and $k\ge 3$ in 
the proof of Theorem \ref{mixedlinthm}.

(b) \label{Remark56ii}
It is natural to ask why we use \eqref{mixedform2} instead of
the alternative formulation
\begin{align}
\label{mixedformwithchi}
\eps \Del {\rm tr}(\chi)+\Fp[\se,\ue](\chi,v)&=\varphi.
\end{align}
As it turns out, it is advantageous to use \eqref{mixedform2} 
opposed to \eqref{mixedformwithchi}, as we now explain.

If we based the mixed finite element method on \eqref{mixedformwithchi},
the method would be similar to \eqref{mixedlinfem1}--\eqref{mixedlinfem2},
but with $d(\ue; v_h,w_h)$ replaced by $\hat{d}(\ue;\chi_h,v_h,w_h)$, where
\begin{align*}
\hat{d}(\ue;\chi_h,v_h,w_h):=\bigl\langle \Fp[\se,\ue](\chi_h,v_h),w_h\bigr\rangle.
\end{align*}
Notice that by assumption {\rm [B2]} (cf. \eqref{mixedGarding1}) there holds
\begin{align*}
d(\ue;v_h,v_h)\ge K_1\|v_h\|_\ho^2-K_0\|v_h\|_\lt^2\qquad \forall v_h\in Q^h_0.
\end{align*}
However, an inequality such as this one does not hold for $\hat{d}(\ue;\chi_h,v_h,v_h)$ 
even if $(\chi_h,v_h)\in \mathbb{T}_h$, where $\mathbb{T}_h$ is defined in {\rm [B6]}.

However, if $(\chi_h,v_h)\in\mathbb{T}_h$, and if we 
define $\lambda^\eps\in W_0$ such that
%
\[
\lambda^\eps_{ij}=\frac{\p F}{\p r_{ij}}(\se,\ue)v_h\qquad i,j=1,2,...n,
\]
then
\begin{align*}
\hat{d}(\ue;\chi_h,v_h,v_h)
&=(\chi_h,\lambda^\eps)
+\bl F_p[\se,\ue](\chi_h,v_h),v_h\br+\bl F_z[\se,\ue](\chi_h,v_h),v_h\br\\
&=(\chi_h,\lambda^\eps-\Pi^h\lambda^\eps)-b(\lambda^\eps,v_h)\\
&\quad +\bl F_p[\se,\ue](\chi_h,v_h),v_h\br+\bl F_z[\se,\ue](\chi_h,v_h),v_h\br,
\end{align*}
and after integrating by parts
\begin{align*}
\hat{d}(\ue;\chi_h,v_h,v_h)&=(\chi_h,\lambda^\eps-\Pi^h\lambda^\eps)+d(\ue;v_h,v_h).
\end{align*}
Thus, to obtain any coercivity from the alternative bilinear 
form $\hat{d}(\ue;\cdot,\cdot,\cdot)$,
we need to obtain bounds for $\|\lambda^\eps-\Pi^h\lambda^\eps\|_\lt$,
adding more complexity to the proof of Theorem \ref{mixedlinthm}.  
We also note the similarities
of this derivation and the proof of Proposition \ref{B6prop}.
\end{remarks}

\section{Convergence analysis of mixed finite element methods}\label{chapter-5-sec-4}

In this section, we give the main results of the chapter by proving
there exists a unique solution to \eqref{mixedfem1}--\eqref{mixedfem2}
and deriving error estimates of the numerical solution.
First, we define the bilinear operator 
$\Mbf:W^h_\eps\times Q^h_g\mapsto W^h_\eps\times Q^h_g$
such that for given $(\mu_h,w_h)\in W^h_\eps\times Q^h_g$,
$\Mbf(\mu_h,w_h):=\bigl(\Mo(\mu_h,w_h),\Mt(\mu_h,w_h)\bigr)\in W^h_\eps\times Q^h_g$
is the solution to the following problem:
\begin{alignat}{2}
\label{mixedTdef}
&\bigl(\mu_h-\Mo(\mu_h,w_h),\kappa_h\bigr)+b\bigl(\kappa_h,w_h-\Mt(\mu_h,w_h)\bigr)\\
&\hspace{1.2in}\nonumber=(\mu_h,\kappa_h)+b(\kappa_h,v_h)
-G(\kappa_h)\qquad &&\forall \kappa_h\in W^h_0,\\
&b\bigl(\mu_h-\Mo(\mu_h,w_h),z_h\bigr)
-\eps^{-1}d\bigl(\ue;w_h-\Mt(\mu_h,w_h),z_h\bigr)\\
&\hspace{1.2in}\nonumber=b(\mu_h,z_h)
-\eps^{-1}c(\mu_h,w_h,z_h)\qquad &&\forall z_h\in Q_0^h.
\end{alignat}

By Theorem \ref{mixedlinthm}, $\Mbf$ is well-defined provided assumptions 
{\rm [B1]--[B2]} hold and $h\le h_0$.  Clearly any fixed point of the 
mapping $\Mbf$ (i.e. $\Mbf(\mu_h,w_h)=(\mu_h,w_h)$) is
a solution to problem \eqref{mixedfem1}--\eqref{mixedfem2} and vice-versa.  
The goal of this section is to show that the mapping $\Mbf$ has a unique 
fixed point in a small neighborhood of $(\stil, \util)$.
To this end, we define the following sets:
\begin{align}
\label{Sdef}\mathbb{S}_h(\rho):&=\bigl\{(\mu_h,v_h)\in W^h_\eps\times Q^h_g;\ 
\ttbar{\mu_h-\stil}{v_h-\util}\le \rho\bigr\},\\
\label{Zdef}\mathbb{Z}_h:&=\bigl\{(\mu_h,w_h)\in W^h_\eps\times Q^h_g;\ 
(\mu_h,\kappa_h)+b(\kappa_h,w_h)\\
&\nonum\hspace{2in}=G(\kappa_h)
\ \forall \kappa_h\in W^h_0\bigr\},\\
\label{Bdef}\mathbb{B}_h(\rho):&=\mathbb{S}_h(\rho)\cap \mathbb{Z}_h.
\end{align}

For the continuation of the chapter, we set 
$\ell={\min}\{s,k+1\}$, 
where $k$ is
the polynomial degree of the finite element 
spaces $W^h$ and $Q^h$, and $s$ is defined in {\rm [B1]}.
The next lemma shows that the distance
between the center of $\mathbb{B}_h(\rho)$
and its image under the mapping $\Mbf$ is small.

\begin{lem}\label{mixedlem51}
Suppose assumptions {\rm [B1]--[B4]} hold.  Then for $h\le h_0$, 
there hold the following estimate:
\begin{align}\label{mixedlem51line}
&\ttbar{\stil-\Mone}{\util-\Mtwo} \\
&\hskip 2.2in \le K_6h^{\ell-2}\snorm, \nonumber
\end{align}
where
\begin{align*}
K_6&=CK_3\eps^{-\frac12}\(K_1^{-\frac12}+K_0^\frac12 K_{R_0}\).
\end{align*}
\end{lem}
\begin{proof}
To ease notation set $\ome_h=\stil- \Mone$, $s_h=\util-\Mtwo$,
$r^\eps=\util-\ue,$ and $\theta^\eps=\stil-\se$.
By the definition of $\Mbf$, we have for any 
$(\kappa_h,z_h)\in W^h_0\times Q_0^h$
\begin{align*}
(\ome_h,\kappa_h)+b(\kappa_h,s_h)
&=\(\stil,\kappa_h\)+b\(\kappa_h,\util\)-G(\kappa_h), \\
b(\ome_h,z_h)-\eps^{-1}d(\ue;s_h,z_h)
&=b\(\stil,z_h\)-\eps^{-1}c\(\stil,\util,z_h\),
\end{align*}
and therefore by \eqref{mixedvar1}--\eqref{mixedvar2}, \eqref{Piop},
and by employing the mean value theorem,
\begin{align} \label{mixedlem51line1}
&(\omega_h,\kappa_h)+b(\kappa_h,s_h)
=(\theta^\eps,\kappa_h)+b(\kappa_h,r^\eps), \\
\label{mixedlem51line2}&b(\ome_h,z_h)-\eps^{-1}d(\ue;s_h,z_h)\\
&\nonum\hspace{0.5in}=b(\theta^\eps,z_h)
-\eps^{-1}\Bigl(c\(\stil,\util,z_h\)-c\(\se,\ue,z_h\)\Bigr)\\
&\nonum\hspace{0.5in}=-\eps^{-1}\bl \Fp[\xi_h,y_h](\theta^\eps,r^\eps),z_h\br,
\end{align}
where $\xi_h=\stil-\gamma\theta^\eps$ and 
$y_h=\util-\gamma r^\eps$ for some $\gamma\in [0,1]$.

Setting $\kappa_h=\ome_h$ and $z_h=s_h$, and subtracting
\eqref{mixedlem51line2} from \eqref{mixedlem51line1} yield
\begin{align*}
(\omega_h,\omega_h)+\eps^{-1}d(\ue;s_h,s_h)
&=(\theta^\eps,\omega_h)+b(\omega_h,r^\eps)\\
&\ \
+\eps^{-1}\bigl\langle F^\prime[\xi_h,y_h](\theta^\eps,r^\eps),s_h\bigr\rangle.
\end{align*}
Consequently, by {\rm [B2]--[B4]}, and the inverse inequality,
\begin{align*}
&\tbar{\omega_h}{s_h}^2
\le \|\theta^\eps\|_\lt\|\ome_h\|_\lt
+\bnorm{\Div(\ome_h)}_\lt\bnorm{\nab r^\eps}_\lt\\
&\quad\quad+\eps^{-1}\bnorm{F^\prime[\xi_h,y_h](\theta^\eps,r^\eps)}_{H^{-1}}
\norm{s_h}_\ho+K_0\eps^{-1}\norm{s_h}_\lt^2\\
&\quad\le \norm{\theta^\eps}_\lt\norm{\omega_h}_\lt
+h^{-1}\norm{\omega_h}_\lt\norm{\nab r^\eps}_\lt\\
&\quad\quad+C\eps^{-1}\bigl\|\(\xi_h,y_h\)\bigr\|_{X\times Y}
\(\norm{\theta^\eps}_\lt+\norm{r^\eps}_\ho\)\norm{s_h}_\ho
+K_0\eps^{-1}\|s_h\|_\lt^2\\
&\quad\le \norm{\theta^\eps}_\lt\norm{\omega_h}_\lt
+h^{-1}\norm{\omega_h}_\lt\norm{\nab r^\eps}_\lt\\
&\quad\quad+CK_3\eps^{-1}\bigl(\|\theta^\eps\|_\lt+\|r^\eps\|_\ho\bigr)\|s_h\|_\ho
+K_0\eps^{-1}\|s_h\|_\lt^2.
\end{align*}
Using the Cauchy-Schwarz and inverse inequalities, and rearranging terms, give us
\begin{align}\label{mixedlem52line1}
&\ttbar{\omega_h}{s_h}^2 \le C\Bigl(\norm{\theta^\eps}_\lt^2+
h^{-2}\norm{\nab r^\eps}^2_\lt\\ 
&\nonum \qquad +K_1^{-1}K_3^2\eps^{-1}
\bigl(\norm{\theta^\eps}_\lt^2+\norm{r^\eps}_\ho^2\bigr)
+K_0\eps^{-1}\norm{s_h}_\lt^2\Bigr)\\
&\nonum\le C\Bigl(h^{2\ell-4}\norm{\se}^2_{H^{\ell-2}}
+h^{2\ell-4}\norm{\ue}^2_\hl+K_0\eps^{-1}\norm{s_h}_\lt^2\\
&\nonum\quad+K_1^{-1}K_3^2\eps^{-1}h^{2\ell-4}\snorms\Bigr)\\
&\nonum\le C\Bigl(K_1^{-1}K_3^2\eps^{-1}h^{2\ell-4}\snorms+K_0\eps^{-1}\|s_h\|_\lt^2\Bigr).
%
\end{align}

Next, we let $w\in Q_0\cap H^p(\Ome)\, (p\ge 3)$ be the solution to 
the following auxiliary problem:
\begin{alignat*}{2}
\left(G^\prime_\eps[\ue]\right)^*(w)&=s_h\qquad &&\text{in }\Ome,\\
D^2w\nu\cdot\nu&=0\qquad &&\text{on }\p\Ome,
\end{alignat*}
with
\begin{align}
\label{mixedlem51reg}
\|w\|_{H^p}\le K_{R_0}\|s_h\|_\lt.
\end{align}

Setting $\kappa=D^2w\in \left[H^{p-2}(\Ome)\right]^{n\times n}$, we have
\begin{alignat*}{2}
(\kappa,\mu)+b(\mu,z)&=0\qquad &&\forall \mu\in W_0,\\
b(\kappa,z)-\eps^{-1}d^*(\ue;w,z)&=-\eps^{-1}(s_h,z)\qquad &&\forall z\in Q_0.
\end{alignat*}

Thus, by \eqref{mixedlem51line1}--\eqref{mixedlem51line2},
\begin{align*}
\eps^{-1}\|s_h\|_\lt^2
&=-b(\kappa,s_h)+\eps^{-1}d^*(\ue;w,s_h)\\
&=-b\(\Pi^h\kappa,s_h\)+\eps^{-1}d(\ue;s_h,w)\\
&=\(\omega_h,\Pi^h\kappa\)-\(\theta^\eps,\Pi^h\kappa\)
-b\(\Pi^h\kappa,r^\eps\)+\eps^{-1}d(\ue;s_h,w)\\
&=(\omega_h,\kappa)+\(\omega_h,\Pi^h\kappa-\kappa\)
-\(\theta^\eps,\Pi^h\kappa\)\\
&\qquad-b\(\Pi^h\kappa,r^\eps\)+\eps^{-1}d(\ue;s_h,w)\\
&=-b(\omega_h,w)+\(\omega_h,\Pi^h\kappa-\kappa\)\\
&\qquad-\(\theta^\eps,\Pi^h\kappa\)
-b\(\Pi^h\kappa,r^\eps\)+\eps^{-1}d(\ue;s_h,w)\\
&=-b\(\omega_h,w-\wtil\)+\(\omega_h,\Pi^h\kappa-\kappa\)\\
&\qquad-\(\theta^\eps,\Pi^h\kappa\)-b\(\Pi^h\kappa,r^\eps\)
+\eps^{-1}d\(\ue;s_h,w-\wtil\)\\
&\qquad 
+\eps^{-1}\left\langle F^\prime[\xi_h,y_h](\theta^\eps,r^\eps),\wtil\right\rangle\\
&\le \bnorm{\Div(\omega_h)}_\lt\bnorm{\nab (w-\wtil)}_\lt
+\|\omega_h\|_\lt\bnorm{\Pi^h\kappa-\kappa}_\lt\\
&\qquad+\|\theta^\eps\|_\lt\bnorm{\Pi^h\kappa}_\lt+
\bnorm{\Div(\Pi^h\kappa)}_\lt\norm{\nabla r^\eps}_\lt\\
&\qquad+K_2\eps^{-1}\|s_h\|_\ho\bnorm{w-\wtil}_\ho\\
&\qquad+K_3\eps^{-1} \bigl(\|\theta^\eps\|_\lt+\|r^\eps\|_\ho\bigr)\bnorm{\wtil}_\ho\\
&\le C\Big(h^{r-2}\|\omega_h\|_\lt+K_2\eps^{-1}h^{r-1}\|s_h\|_\ho
 +K_3\eps^{-1}h^{\ell-2}\snorm\Bigr)\|w\|_{H^p}.
\end{align*}
Therefore, using \eqref{mixedlem51reg},
\begin{align*}
\|s_h\|_\lt^2
&\le CK_{R_0}^2\eps^2\Bigl(h^{2r-4}\|\omega_h\|^2_\lt
+K_2^2\eps^{-2}h^{2r-2}\|s_h\|_\ho^2
+K_3^2\eps^{-2}h^{2\ell-4}\snorms\Bigr).
%
%
\end{align*}

Using this bound in \eqref{mixedlem52line1}, we have
\begin{align*}
\ttbar{\omega_h}{s_h}^2
& \le C\Bigl(K_3^2\eps^{-1}\(K_1^{-1}+K_0K_{R_0}^2\)h^{2\ell-4}\|\ue\|_\hl^2\\
&\qquad +K_0K_{R_0}^2\eps\(h^{2r-4}\|\omega_h\|_\lt^2+K_2^2\eps^{-2}h^{2r-2}\|s_h\|_\ho^2\)\Bigr).
%
\end{align*}
It then follows that for $h\le h_0$,
\begin{align*}
\ttbar{\omega_h}{s_h}
&\le CK_3\eps^{-\frac12}\(K_1^{-\frac12}+K_0^\frac12 K_{R_0}\)h^{\ell-2} \|\ue\|_\hl.
\end{align*}
which is the inequality \eqref{mixedlem51line}. The proof is complete.
\end{proof}

\begin{lem}\label{mixedlem52}
Let {\rm [B1]--[B6]} hold and suppose that $u^\vepsi\in H^s(\Ome)\,(s\geq 3)$.
Then there exists an $h_1=h_1(\eps)>0$ such that for 
$h\le {\rm min}\{h_0,h_1\}$, the mapping $\Mbf$
is a contracting mapping with a contracting factor 
of $\frac12$ in the ball $\mathbb{B}_h(\rho_0)$, where
\begin{align*}
\rho_0:&=\left(K_7R(h)\right)^{-1}\\
h_1:&={\rm min}\left\{\left(K_7K_G\right)^{-\frac{1}{\alpha}},
\left(K_7R(h_1)\snorm\right)^{\frac{1}{2-\ell}}\right\},\\
K_7:&=C\eps^{-\frac12}\bigl(K_1^{-\frac12}+K_0^\frac12K_{R_0}\bigr),
\end{align*}
and $\alpha>0$ is defined in {\rm [B6]}. 
That is, for all $(\mu_h,v_h),(\kappa_h,w_h)\in \mathbb{B}_h(\rho_0)$
\begin{align*}
&
\bigl|\hspace{-0.03cm}\bigl|\hspace{-0.03cm}\bigl|
\Mbf(\mu_h-\kappa_h,v_h-w_h)
\bigr|\hspace{-0.03cm}\bigr|\hspace{-0.03cm}\bigr|_\eps
\le \frac12 \ttbar{\mu_h-\kappa_h}{v_h-w_h}.
\end{align*}
\end{lem}
\begin{proof}
Let $(\mu_h,v_h),(\kappa_h,w_h)\in \mathbb{B}_h(\rho_0)$, and to 
ease notation we set
\[
\Mo=\Mo(\mu_h,v_h)-\Mo(\kappa_h,w_h),\quad 
\Mt=\Mt(\mu_h,v_h)-\Mt(\kappa_h,w_h).
\]

Using the definition of $\Mbf$ and employing
the mean value theorem, we conclude that for all $(\chi_h,z_h)\in W^h_0\times Q_0^h$,
\begin{align}
\label{mixedlem52a}&\(\Mo,\chi_h\)+b\(\chi_h,\Mt\)=0,\\
\label{mixedlem52b}&b\(\Mo,z_h\)-\eps^{-1}d\(\ue;\Mt,z_h\)\\
&\nonum\qquad=\eps^{-1}\Bigl(d(\ue;v_h-w_h,z_h)-
\bigl(c(\mu_h,v_h,z_h)-c(\kappa_h,w_h,z_h)\bigr)\Bigr)\\
&\nonum\qquad =\eps^{-1}\Bigl(d(\ue;v_h-w_h,z_h)-
\bl \Fp[\xi_h,y_h]\(\mu_h-\kappa_h,v_h-w_h\),z_h\br\Bigr),
\end{align}
where 
$\xi_h=\mu_h+\gamma(\kappa_h-\mu_h)$ and $y_h=v_h+\gamma(w_h-v_h)$ 
for some $\gamma\in [0,1]$.
Here, we have abused the notation of $\xi_h$ and $y_h$, defining them differently in two 
different proofs in this section.

Setting $\chi_h=\Mo$ and $z_h=\Mt$,
subtracting \eqref{mixedlem52a} from \eqref{mixedlem52b},
using assumptions {\rm [B2]} and {\rm [B5]}, and the 
inverse inequality yields
\begin{align*}
&\ttbar{\Mo}{\Mt}^2\\
&\le \eps^{-1}\Bigl(d\(\ue;w_h-v_h,\Mt\)
-\Bigl\langle \Fp[\xi_h,y_h]\(\kappa_h-\mu_h,w_h-v_h\),\Mt\Bigr\rangle\Bigr)\\
&\quad +K_0\eps^{-1}\bnorm{\Mt}_\lt^2\\
&=\eps^{-1}\Bigl(\Bl \Fp[\se,\ue](D^2w_h-D^2v_h,w_h-v_h)
-\Fp[\se,\ue]\(\kappa_h-\mu_h,w_h-v_h\),\Mt\Br\Bigr)\\
&\quad+\eps^{-1}\bl \(\Fp[\se,\ue]
-\Fp[\xi_h,y_h]\)(\kappa_h-\mu_h,w_h-v_h),\Mt\br+K_0\eps^{-1}\bnorm{\Mt}_\lt^2\\
&\le\eps^{-1}\Bigl(\Bl \Fp[\se,\ue]\(D^2(w_h-v_h)-(\kappa_h-\mu_h),0\),\Mt\Br\Bigr)
+K_0\eps^{-1}\bnorm{\Mt}_\lt^2\\
&\quad +\eps^{-1}R(h)\(\|\se-\xi_h\|_\lt+\|\ue-y_h\|_\ho\)
\ttbar{\kappa_h-\mu_h}{w_h-v_h}\bnorm{\Mt}_\ho\\
%
&\le C\eps^{-1}\Bigl(K_Gh^\alpha+R(h)\(h^{\ell-2}\snorm+\rho_0\)\Bigr)\\
&\qquad \times \ttbar{\kappa_h-\mu_h}{w_h-v_h}\bnorm{\Mt}_\ho
+K_0\eps^{-1}\bnorm{\Mt}_\lt^2,
\end{align*}
and therefore
\begin{align}
\label{mixedlem52line2}\ttbar{\Mo}{\Mt}^2
&\le CK_1^{-1}\eps^{-1}\Bigl(K_G^2h^{2\alpha}+R^2(h)\(h^{2\ell-4}\snorms+\rho^2_0\)\Bigr)\\
&\nonum\qquad \times \ttbar{\kappa_h-\mu_h}{w_h-v_h}^2+K_0\eps^{-1}\bnorm{\Mt}_\lt^2.
\end{align}

Next, we let $z\in Q_0\cap H^p(\Ome)\, (p\ge 3)$ 
be the solution to the following auxiliary problem:
\begin{alignat*}{2}
\left(\Gp[\ue]\right)^*(z)&=\Mt\qquad &&\text{in }\Ome,\\
D^2z\nu\cdot\nu&=0\qquad &&\text{on }\p\Ome,
\end{alignat*}
with
\begin{align*}
\|z\|_{H^p}\le K_{R_0}\bnorm{\Mt}_\lt.
\end{align*}

Letting $\chi=D^2z$, we have
\begin{alignat*}{2}
(\chi,\lambda)+b(\lambda,z)&=0\qquad &&\forall \lambda\in W_0,\\
b(\chi,y)-\eps^{-1}d^*(\ue;z,y)&=-\eps^{-1}\(\Mt,y\)\qquad &&\forall y\in Q_0,
\end{alignat*}
and hence by \eqref{mixedlem52a}--\eqref{mixedlem52b},
\begin{align*}
&\eps^{-1}\bnorm{\Mt}_\lt^2
=-b\(\chi,\Mt\)+\eps^{-1}d^*\(\ue;z,\Mt\)\\
&=\(\Mo,\Pi^h\chi\)+\eps^{-1}d\(\ue;\Mt,z\)\\
&=\(\Mo,\chi\)+\eps^{-1}d\(\ue;\Mt,z\)+\(\Mo,\Pi^h\chi-\chi\)\\
&=-b\(\Mo,z-\ztil\)+\eps^{-1}d\(\ue;\Mt,z-\ztil\)
 +\(\Mo,\Pi^h\chi-\chi\)\\
&\quad +\eps^{-1}\Bigl(\Bigl\langle F^\prime[\xi_h,y_h]\bigl(\kappa_h-\mu_h,w_h-v_h\bigr),\ztil\Bigr\rangle-
d\(\ue;w_h-v_h,\ztil\)\Bigr)\\
&\le \bnorm{\Div(\Mo)}_\lt\norm{\nab (z-\ztil)}_\lt+K_2\eps^{-1}\bnorm{\Mt}_\ho\bnorm{z-\ztil}_\ho\\
&\quad +\bnorm{\Mo}_\lt\bnorm{\Pi^h\chi-\chi}_\lt
 +C\eps^{-1}\Bigl(K_G h^\alpha +R(h)\bigl(h^{\ell-2}\snorm
+\rho_0\bigr)\Bigr)\\
&\quad \times \ttbar{\kappa_h-\mu_h}{w_h-v_h}\bnorm{\ztil}_\ho\\
&\le CK_{R_0}\Bigl(K_2\eps^{-1}h^{r-1}\bnorm{\Mt}_\ho
+h^{r-2}\bnorm{\Mo}_\lt\\
&\quad +C\eps^{-1}\Bigl[K_Gh^\alpha +R(h)\bigl(h^{\ell-2}\snorm
+\rho_0\bigr)\Bigr]\\
&\quad \times \ttbar{\kappa_h-\mu_h}{w_h-v_h}\Bigr)\bnorm{\Mt}_\lt.
\end{align*}
Thus,
\begin{align*}
&\bnorm{\Mt}_\lt^2
\le CK^2_{R_0}\eps^2\Bigl(
+K^2_2\eps^{-2}h^{2r-2}\bnorm{\Mt}^2_\ho+h^{2r-4}\bnorm{\Mo}^2_\lt\\
&\quad+\eps^{-2}\Bigl[K_G^2h^{2\alpha}
+R^2(h)\bigl(h^{2\ell-4}\snorms+\rho_0^2\bigr)\Bigr]
\ttbar{\kappa_h-\mu_h}{w_h-v_h}^2\Bigr).
\end{align*}

Using the above bound in inequality \eqref{mixedlem52line2} yields for $h\le h_0$
\begin{align*}
&\ttbar{\Mo}{\Mt}\\
&\nonum\le CK_1^{-1}\eps^{-1}\Bigl(K_G^2h^{2\alpha}+R^2(h)\bigl(h^{2\ell-4}\snorms+\rho^2_0\bigr)\Bigr)\\
&\nonum\qquad \times \ttbar{\kappa_h-\mu_h}{w_h-v_h}^2+K_0\eps^{-1}\bnorm{\Mt}_\lt^2\\
&\nonum\le C\eps^{-1}\bigl(K_1^{-1}+K_0K_{R_0}^2\bigr)
\Bigl(K_G^2h^{2\alpha}+R^2(h)\bigl(h^{2\ell-4}\snorms+\rho^2_0\bigr)\Bigr)\\
&\nonum\qquad \times \ttbar{\kappa_h-\mu_h}{w_h-v_h}^2.
\end{align*}

It then follows from the definition of $\rho_0$
that for $h\le {\rm min}\{h_0,h_1\}$
\begin{align*}
&\ttbar{\Mo}{\Mt}\\
%
&\le K_6
\Bigl(K_Gh^\alpha+R(h)\bigl(h^{\ell-2}\snorm+\rho_0\bigr)\Bigr)
\ttbar{\kappa_h-\mu_h}{w_h-v_h}\\
&\le \frac12\ttbar{\kappa_h-\mu_h}{w_h-v_h}.
\end{align*}
\end{proof}

\begin{thm}\label{mixedmainthm}
Under the same assumptions of Lemma \ref{mixedlem52}, there
 exists an $h_2=h_2(\eps)>0$ such that for $h\le {\rm min}\{h_0,h_2\}$
\eqref{mixedfem1}--\eqref{mixedfem2} has a locally unique solution, where $h_2$ is 
chosen such that
\begin{align*}
h_2&={\rm min}\left\{\{\left(K_7K_G\right)^{-\frac{1}{\alpha}},
\Bigl(2K_6K_7R(h_2)\snorm\Bigr)^{\frac{1}{2-\ell}}\right\}.
\end{align*}
Furthermore, there holds the following error estimate:
\begin{align}
\label{mixedmainbound}
\ttbar{\se-\se_h}{\ue-\ue_h} \le h^{\ell-2} K_8\snorm,
\end{align}
where 
\[
K_8=CK_6=CK_3\eps^{-\frac12}\bigl(K_1^{-\frac12}+K_0^\frac12K_{R_0}\bigr).
\]
\end{thm}

\begin{proof}
Let 
\begin{align*}
\rho_1&=2K_6 h^{\ell-2} \snorm.
\end{align*}
Then for $h\le {\rm min}\{h_0,h_2\}$, there
holds $\rho_1\le \rho_0$.

Thus noting $h_2\le h_1$, 
for any $(\mu_h,v_h)\in \mathbb{B}_h(\rho_1)$, 
we use Lemmas \ref{mixedlem51} and \ref{mixedlem52}
to conclude that
\begin{align*}
&\ttbar{\stil-\Mo(\mu_h,v_h)}{\util-\Mt(\mu_h,v_h)}\\
&\le \ttbar{\stil-\Mone}{\util-\Mtwo}\\
&\quad +\ttbar{\Mone-\Mo(\mu_h,v_h)}{\Mtwo-\Mt(\mu_h,v_h)}\\
&\le K_6 h^{\ell-2}\snorm+\frac12\ttbar{\stil-\mu_h}{\util-v_h}\\
&\le \frac{\rho_1}{2}+\frac{\rho_1}{2}=\rho_1,
\end{align*}
and so $\Mbf(\mu_h,v_h)\in \mathbb{B}_h(\rho_1)$.  It is clear that $\Mbf$
is a continuous mapping. It follows from Banach's Fixed Point Theorem
\cite{Gilbarg_Trudinger01} that $\Mbf$ has a unique fixed point 
$(\seh,\ueh)$ in the ball $\mathbb{B}_h(\rho_1)$, which is the unique
solution to \eqref{mixedfem1}--\eqref{mixedfem2}.

To obtain the error estimate \eqref{mixedmainbound}, we use
the triangle inequality to conclude
\begin{align*}
&\ttbar{\se-\seh}{\ue-\ueh}\\
&\qquad \le \ttbar{\se-\stil}{\ue-\util}
+\ttbar{\stil-\seh}{\util-\ueh}\\
&\qquad \le Ch^{\ell-2}\snorm+C\rho_1
 \le CK_6h^{\ell-2}\snorm.
\end{align*}
\end{proof}

Note that the error estimates of $\|\ue-\ue_h\|_\ho$
in Theorem \ref{mixedmainthm} are sub-optimal.  In the next theorem,
we employ a duality argument to improve the above error 
estimates and to also obtain $L^2$ error estimates.

\begin{thm}\label{mixedmainthmimp}
In addition to the hypotheses of 
Theorem \ref{mixedmainthm}, suppose
that $p\ge 4$ in assumption {\rm [B2]}.
Then 
there hold the following error estimates:
\begin{align*}
\|\ue-\ueh\|_\lt&\le
K_{R_0}\Bigl(K_{9}h^{\ell-2+{\rm min}\{2,\alpha\}}\snorm
 + K^2_{8}R(h)h^{2\ell-4}\snorms\Bigr),\\
\|\ue-\ue_h\|_\ho&\le K_{R_1}\Bigl( K_{9}h^{\ell-2+{\rm min}\{1,\alpha\}}\snorm
 + K^2_{8}R(h)h^{2\ell-4}\snorms\Bigr),
\end{align*}
where 
\begin{alignat*}{2}
&K_{9}=CK_8{\rm max}\{K_2,K_G\}.
\end{alignat*}
\end{thm}

\begin{proof}
To ease notation, we set 
\[
\serr:=\se-\seh,\qquad \err:=\ue-\ueh.
\]
We note that by using the mean value theorem, there hold the
following error equations:
\begin{alignat}{2}
\label{mixederr1}(\serr,\mu_h)+b(\mu_h,\err)
&=0\qquad &&\forall \mu_h\in W_0^h,\\
\label{mixederr2}b(\serr,v_h)
-\bl \Fp[\xi_h,y_h](\serr,\err),v_h\br
&=0\qquad &&\forall v_h\in Q_0^h,
\end{alignat}
where $\xi_h=\se-\gamma \serr,\ y_h=\ue-\gamma \err$
for some $\gamma\in [0,1]$.  Again, we have abused the notation of
$\xi_h$ and $y_h$, defining them differently in two separate proofs.

Next, let $w_m\in H^{p-m}(\Ome)\cap Q_0\ (m=0,1; p\ge 4)$ be the solution 
to the following auxiliary problem:
\begin{alignat*}{2}
\left(\Gp[\ue]\right)^*(w_m)&=(-1)^m\Del^m\err\qquad &&\text{in }\Ome,\\
D^2w_m\nu\cdot\nu&=0\qquad &&\text{on }\p\Ome,
\end{alignat*}
with
\begin{align}
\|w_m\|_{H^{p-m}}\le K_{R_{m}}\norm{\nabla^m \err}_\lt.
\end{align}

Here, we have used the notation $\Del^1=\Del$, $\nab^1=\nab$,
and $\Del^0,\ \nab^0$ are the identity operators on $Q$.
Setting $\kappa_m=D^2w_m\in \left[H^{p-m-2}(\Ome)\right]^{n\times n}$, 
we then have
\begin{alignat*}{2}
(\kappa_m,\mu)+b(\mu,w_m)&=0\qquad &&\forall \mu\in W_0,\\
b(\kappa_m,v)-\eps^{-1}d^*(\ue;w_m,v)&=-\eps^{-1}\(\nab^m \err,\nab^m v\)
\qquad &&\forall v\in Q_0.
\end{alignat*}

Therefore, 
\begin{align*}
\eps^{-1}\bnorm{\nab^m \err}_\lt^2
&= -b(\kappa_m,\err)+\eps^{-1}d^*(\ue;w_m,\err)\\
&=\(\pi_h^\eps,\Pi^h\kappa_m\)+\eps^{-1}d(\ue;\err,w_m)
-b\(\kappa_m-\Pi^h\kappa_m,\err\)\\
&=(\serr,\kappa_m)+\eps^{-1}d(\ue;\err,w_m)\\
&\quad -b\(\kappa_m-\Pi^h\kappa_m,\ue-\util\)+\(\serr,\Pi^h\kappa_m-\kappa_m\)\\
&=-b(\serr,w_m)+\eps^{-1}d(\ue;\err,w_m)\\
&\quad-b\(\kappa_m-\Pi^h\kappa_m,\ue-\util\)+\(\serr,\Pi^h\kappa_m-\kappa_m\)\\
&=-b\(\serr,w_m-\wtil_m\)+\eps^{-1}d\(\ue;\err,w_m-\wtil_m\)\\
&\quad -b\(\kappa_m-\Pi^h\kappa_m,\ue-\util\)
+\(\serr,\Pi^h\kappa_m-\kappa_m\)\\
&\quad+\eps^{-1}d\(\ue;\err,\wtil_m\)
-\eps^{-1}\bl \Fp[\xi_h,y_h]\(\serr,\err\),\wtil_m\br\\
&=-b\(\serr,w_m-\wtil_m\)+\eps^{-1}d\(\ue;\err,w_m-\wtil_m\)\\
&\quad -b\(\kappa_m-\Pi^h\kappa_m,\ue-\util\)
+\(\serr,\Pi^h\kappa_m-\kappa_m\)\\
&\quad+\eps^{-1}\Bl \Fp[\se,\ue]\(D^2\err-\serr,0\),\wtil_m\Br\\
&\quad+\eps^{-1}\Bl \(\Fp[\se,\ue]-\Fp[\xi_h,y_h]\)
\(\serr,\err\),\wtil_m\Br.
\end{align*}

Bounding the right-hand side in the last expression, we have
\begin{align*}
&\eps^{-1}\bnorm{\nab^m \err}^2_\lt\\
&\le \bnorm{\Div(\serr)}_\lt\bnorm{\nab(w_m-\wtil_m)}_\lt
+K_2\eps^{-1}\bnorm{\err}_\ho\bnorm{w_m-\wtil_m}_\ho\\
&\quad +\bnorm{\Div(\kappa_m-\Pi^h\kappa_m)}_\lt\bnorm{\nab(\ue-\util)}_\lt
+\bnorm{\serr}_\lt\bnorm{\Pi^h\kappa_m-\kappa_m}_\lt\\
&\quad+\eps^{-1}\Bigl\langle \Fp[\ue,\se]\bigl(D^2\err
-\serr,0\bigr),\wtil_m\Bigr\rangle\\
&\quad+\eps^{-1}\Bigl\langle \(\Fp[\xi_h,y_h]-\Fp[\se,\ue]\)
\bigl(\serr,\err\bigr),\wtil_m\Bigr\rangle\\  
&\le C\Bigl(h^{3-m}\|\serr\|_\ho+K_2\eps^{-1}h^{2-m}\|\err\|_\ho
+\bnorm{\nab(\ue-\util)}_\lt\\
&\quad+h\|\serr\|_\lt+\eps^{-1}K_G h^\alpha\ttbar{\serr}{\err}\\
&\quad +\eps^{-1}R(h)\bigl(\|\xi_h-\se\|_\lt+\|y_h-\ue\|_\ho\bigr)\ttbar{\serr}{\err}
\Bigr)\|w_m\|_{H^{p-m}}\\
&\le CK_{R_{m}}\Bigl(h^{3-m}\|\serr\|_\ho+K_2\eps^{-1}h^{2-m}\|\err\|_\ho\\
&\quad +h^{\ell-1}\|\ue\|_\hl+h\|\serr\|_\lt
+\eps^{-1}K_G h^\alpha\ttbar{\serr}{\err}\\
&\quad +\eps^{-1}R(h)\ttbar{\serr}{\err}^2
\Bigr)\bnorm{\nab^m \err}_\lt\\
%
&\le CK_{R_{m}}\eps^{-1}\Bigl(\(K_2h^{2-m}+K_Gh^\alpha\)
\ttbar{\serr}{\err}
+R(h)\ttbar{\serr}{\err}^2\Bigr)\bnorm{\nab^m \err}_\lt.
\end{align*}
Therefore,
\begin{align*}
\bnorm{\nab^m \err}_\lt
&\le 
CK_{R_{m}}\Bigl(\bigl(K_2h^{2-m}+K_Gh^\alpha\bigr)
\ttbar{\pi^\eps_h}{\err}
+R(h)\ttbar{\serr}{\err}^2\Bigr)\\
&\le
C K_8 K_{R_{m}}\Bigl(\bigl(K_2h^{2-m} +K_Gh^\alpha\bigr)
h^{\ell-2} \snorm
 +R(h)h^{2\ell-4}K_8\snorms
\Bigr).
\end{align*}

The proof is complete.
\end{proof}

\section{Generalizations: the case of degenerate equations}\label{chapter-5-sec-5}
In this section, we generalize the analysis of 
the preceding sections to handle cases in which condition {\rm [B2]} 
fails to hold, namely when the inequality
\begin{align}\label{mixedGaragain}
\langle \Fp[\se,\ue](\chi,v),v\rangle \ge K_1\|v\|_\ho^2-K_0\|v\|_\lt^2\qquad \forall v\in Q_0
\end{align}
does not hold for any positive constant $K_1$. 
Thus, in this section we consider cases in which 
the operator $F$ may become degenerate 
(i.e. has vanishing smallest eigenvalue) at the solution
$\ue$.  An instance of such a case arises when studying
mixed finite element approximations of the infinity-Laplacian
equation (cf. Section \ref{chapter-6-sec-4}).

Here, we introduce a more flexible mixed finite element formulation 
to overcome this difficulty.  
To this end, we rewrite \eqref{moment1}--\eqref{moment3}$_3$
into the following system of second order equations:
\begin{align}
\label{altform1}&\set-D^2\ue-\tau I_{n\times n} \ue=0,\\
\label{altform2}&\eps \Div\(\Div(\set)\) +\eps \tau {\rm tr}(\set)+\wF(\set,\ue)=0,
\end{align}
where 
\[
\wF(\set,\ue):=-2\eps \tau \Del \ue-n\eps \tau^2 \ue+F(\set-\tau I_{n\times n}\ue,\ue),
\]
$I_{n\times n}$ denotes the $n\times n$ identity matrix and $\tau$ is
a nonnegative constant that is independent of $\eps$.  
Clearly, \eqref{altform1}--\eqref{altform2} is the same as 
\eqref{mixedmoment1}--\eqref{mixedmoment2}
 with $\set=\se+\tau I_{n\times n}\ue$.

The variational formulation of \eqref{altform1}--\eqref{altform2}
is then defined as seeking $(\set,\ue)\in \wW_\eps\times Q_g$
such that
\begin{alignat}{2}
\label{altvar1}
(\set,\mu)+\wb(\mu,\ue)&=G(\mu)\qquad &&\forall \mu\in W_0,\\
\label{altvar2}\wb(\set,v)-\eps^{-1}\wc(\set,\ue,v)&=0\qquad &&\forall v\in Q_0,
\end{alignat}
where 
\begin{align*}
\wW_\eps&:=\{\mu\in W;\ \mu\nu\cdot \nu\big|_{\p\Ome}=\eps+\tau g\},\\
\wb(\mu,v)&:=\(\Div(\mu),\nab v\)-\tau ({\rm tr}(\mu),v),\\
\wc(\mu,v,w)&:=\bigl\langle \wF(\mu,v),w\big\rangle=2\eps \tau (\nab v,\nab w)-\eps n \tau^2(v,w)
+\(F(\mu-\tau I_{n\times n}v,v),w\),
\end{align*}
and $G(\mu)$ is defined by \eqref{Gdef}.  We note that \eqref{altvar1}--\eqref{altvar2} 
is the same as \eqref{mixedvar1}--\eqref{mixedvar2} for the case $\tau=0$.

Based on the variational formulation \eqref{altform1}--\eqref{altform2}, we define
our mixed finite element method of \eqref{moment1}--\eqref{moment3}$_3$ as seeking
$(\set_h,\ueh)\in \wW^h_\eps\times Q^h_g$ (where
$\wW^h_\eps:=W^h\cap \wW_\eps$) such that
\begin{alignat}{2}
\label{altfem1}
(\set_h,\mu_h)+\wb(\mu_h,\ueh)&=G(\mu_h)\qquad &&\forall \mu_h\in W_0,\\
\label{altfem2}\wb(\set_h,v_h)-\eps^{-1}\wc(\set_h,\ue_h,v_h)&=0\qquad &&\forall v_h\in Q_0.
\end{alignat}

The specific goal of this section is to analyze 
the finite element method
\eqref{altfem1}--\eqref{altfem2} and to determine what conditions
are sufficient to show existence, uniqueness, and error estimates
of the solution.  Clearly, the finite element method 
and \eqref{mixedfem1}--\eqref{mixedfem2}
have a similar structure, and therefore, one would expect that 
most of the analysis in the previous sections can 
be inherited in the present case.   However, one issue of concern
is that we have changed the bilinear form $b(\cdot,\cdot)$
in the new formulation, leading to question whether
the inf-sup condition (cf. Lemma \ref{infsuplem}) still holds.  
As is now well-known, this is a crucial ingredient in mixed finite element
analysis, and we have used it copiously in the analysis
above (albeit, indirectly).  We appease these worries in the next lemma, 
showing that the inf-sup condition still holds provided $\tau$
is small enough. The reason for using the new bilinear form
$\wb(\cdot,\cdot)$ will become clear later (see \eqref{altGardingremark}).

\begin{lem}\label{altinfsup}
There exists positive constants $\tau_0,C$ depending only on $n$ 
and $\Ome$ such that for $\tau\le \tau_0$ 
there holds the following inequality for any $v_h\in Q^h_0$:
\begin{align}\label{altinfsupline}
\sup_{\mu_h\in W^h_0} \frac{\wb(\mu_h,v_h)}{\|\mu_h\|_\ho}\ge C\|v_h\|_\ho.
\end{align}
\end{lem}

\begin{proof}
By Poincar\'e's inequality there exists a positive constant $C_p$ that
depends only on $\Ome$ and $n$ such that for all $v\in H^1_0(\Ome)$
\begin{align*}
\|v\|_\lt\le C_p \|\nab v\|_\lt.
\end{align*}

For $v_h\in Q^h_0\subset H^1_0(\Ome)$, 
set $\kappa_h=I_{n\times n} v_h\in W^h_0$.
Then
\begin{align*}
\sup_{\mu_h\in W^h_0} \frac{\wb(\mu_h,v_h)}{\|\mu_h\|_\ho}
&\ge \frac{\wb(\kappa_h,v_h)}{\|\kappa_h\|_\ho}
=\frac{\(\Div(\kappa_h),\nab v_h\)
-\tau ({\rm tr}(\kappa_h),v_h)}{\sqrt{n}\|v_h\|_\ho}\\
&=\frac{\|\nab v_h\|_\lt^2-n\tau \|v_h\|_\lt^2}{\sqrt{n}\|v_h\|_\ho}
\ge \frac{(1-C^2_p n \tau)\|\nab v_h\|_\lt^2}{\sqrt{n}\|v_h\|_\ho}\\
&\ge\frac{1}{2\sqrt{n}}{\rm min}\left\{(1-C^2_pn\tau),C_p\right\}\|v_h\|_\ho.
\end{align*} 
Choosing $\tau_0=\frac{1}{2}C_p^{-2}n$, we obtain the desired 
inequality \eqref{altinfsupline}.
\end{proof}

Next, we introduce the analogous linearization problem
and mixed formulation to \eqref{altform1}--\eqref{altform2}.  
That is, instead of \eqref{mixedform1}--\eqref{mixedform2},
we write
\begin{alignat}{2}
&\label{altmixedform1}\chit-D^2v -\tau I_{n\times n}v=0\qquad &&\text{in }\Ome,\\
&\label{altmixedform2}\eps \Div\(\Div(\chit)\)+\eps \tau {\rm tr}(\chit)
+\wF^\prime[\set,\ue](D^2v,v)
=\varphi\qquad &&\text{in }\Ome,\\
&\label{altmixedform3}\wt{\chi}\nu\cdot\nu =0,\ v=0\qquad &&\text{on }\p\Ome,
\end{alignat}
where we define
\begin{align*}
\wF^\prime[\ome,y](\mu,w):&=-2\eps \tau \Del w-\eps n \tau^2 w+
\Fp[\ome-\tau I_{n\times n} y,y](\mu,w),
\end{align*}
and $\Fp[\cdot,\cdot](\cdot,\cdot)$ is defined by \eqref{mixedFdef}.
We note that (recall $\se=D^2\ue$)
\begin{align*}
\wF^\prime[\set,\ue](\mu,w)&=-2\eps \tau \Del w-\eps  n \tau^2w 
+F^\prime[\se,\ue](\mu,w).
\end{align*}

The variational formulation of \eqref{altmixedform1}--\eqref{altmixedform2}
is then defined as seeking $(\chit,v)\in W_0\times Q_0$ such that
\begin{alignat*}{2}
(\chit,\mu)+\wb(\mu,v)&=0\qquad &&\forall \mu\in W_0,\\
\wb(\chit,w)-\eps^{-1}\wt{d}(\ue;v,w)
&=-\eps^{-1}\langle \varphi,w\rangle\qquad &&\forall w\in Q_0,
\end{alignat*}
where
\begin{align*}
\wt{d}(\ue;v,w):&=\left\langle \wF^\prime[\set,\ue](D^2 v,v),w\right\rangle\\
&=2\eps \tau (\nab v,\nab w)-\eps n \tau^2 (v,w)
+\left\langle \Fp[\set-I_{n\times n} \ue,\ue](D^2v,v),w\right\rangle.
\end{align*}
It then follows that the corresponding finite element method for the 
linearized problem is to find $(\chit_h,v_h)\in W^h_0\times Q^h_0$ such that 
\begin{alignat}{2}
\label{altlinearizedfem1}(\chit_h,\mu_h)+\wb(\mu_h,v_h)&=0
\qquad &&\forall \mu_h\in W^h_0,\\
\label{altlinearizedfem2}\wb(\chit_h,w_h)-\eps^{-1}\wt{d}(\ue;v_h,w_h)
&=-\eps^{-1}\langle \varphi,w_h\rangle\qquad &&\forall w_h\in Q_0.
\end{alignat}

We now address what conditions are sufficient 
to show that the finite element methods \eqref{altfem1}--\eqref{altfem2}
and \eqref{altlinearizedfem1}--\eqref{altlinearizedfem2} are well-posed.
As it turns out, we are able to obtain results with weaker conditions
than imposed in the previous section.  Specifically, we are
able to replace assumption {\rm [B2]} by the following less-strict condition.

\begin{enumerate}
\item[$\widetilde{{\rm\bf [B2]}}$]
The operator $(G_\eps^\prime[\ue])^*$ 
(the adjoint of $G_\eps^\prime[\ue]$ defined 
in Chapter \ref{chapter-4}) is an isomorphism 
from $H^2(\Ome)\cap H^1_0(\Ome)$ to $\(H^2(\Ome)\cap H^1_0(\Ome)\)^*$.  
That is for all $\varphi\in \(H^2(\Ome)\cap H^1_0(\Ome)\)^*$, there
exists $v\in H^2(\Ome)\cap H^1_0(\Ome)$ such that
\begin{align*}
\bigl\langle (G_\eps^\prime[\ue])^*(v),w\bigr\rangle
=\langle \varphi,w\rangle\qquad \forall w\in H^2(\Ome)\cap H^1_0(\Ome).
\end{align*}
Furthermore, there exists a positive constant 
$K_0=K_0(\eps)$ such that the following inequality holds:
\begin{align}\label{altGarding}
\bigl\langle \Fp[\se,\ue](D^2v,v),v\bigr\rangle\ge -K_0\|v\|_\lt^2. 
\end{align}
and there exists $K_2>0$ such that
\begin{align*}
\bnorm{\Fp [\se,\ue]}_{QQ^*}&\le K_2.
\end{align*}
Moreover, there exists $p\ge 3$ and $K_{R_0}>0,\ K_{R_1}>0$ 
such that if $\varphi\in H^{-m}(\Ome)\ (m=0,1)$
and $v\in V_0$ satisfies \eqref{a2pagain}, then $v\in H^{p-m}(\Ome)$ and
\begin{align*}
\|v\|_{H^{p-m}}\le K_{R_m}\|\varphi\|_{H^{-m}}.
\end{align*}
\end{enumerate}

\begin{remark}
We note that the only difference between {\rm [B2]} 
and [$\wt{\rm B2}$] are the inequalities \eqref{altGarding}
and \eqref{mixedGarding1}.  Clearly if \eqref{mixedGarding1}
holds, then \eqref{altGarding} holds as well, but not vice-versa.
\end{remark}

We now address the well-posedness of the finite element method
for the linearized problem \eqref{altlinearizedfem1}--\eqref{altlinearizedfem2}.  

\begin{thm}\label{altmixedlinthm}
Suppose assumptions {\rm [B1]} and {\rm [}$\wt{\rm B2}${\rm ]}
hold, $\tau\in (0,\tau_0)$, $v\in H^s(\Ome)\ (s\ge 3)$ is the unique 
solution to \eqref{mixedlin1}--\eqref{mixedlin3}
and $\wt{\chi}=D^2v+\tau I_{n\times n} v$.  Then there exists
an $\wt{h}_0=\wt{h}_0(\eps)>0$ such that for $h\le \wt{h}_0$, there exists
a unique solution $(\wt{\chi_h},w_h)\in W^h_0\times Q^h_0$
to problem \eqref{altlinearizedfem1}--\eqref{altlinearizedfem2}, where
\begin{align*}
\wt{h}_0
&=O\left({\rm min}\left\{\left(K_0K_2^2K_{R_1}^2\eps^{-1}
\tau^{-1}\right)^{\frac{1}{2-2r}},
\left(K_0 K_{R_1}^2\eps\right)^{\frac{1}{4-2r}}\right\}\right),
\quad r={\rm min}\{p,k+1\}.
\end{align*}
Here, $k$ is the degree of the polynomial space of $Q^h$ and $W^h$, and
$p$ is defined in {\rm [}$\wt{\rm B2}${\rm ]}.
Furthermore, there hold the following error estimates:
\begin{align}\label{mixedlinbound1a}
&\ttbar{\chi-\chi_h}{v-v_h}
\le Ch^{\ell-2}\bigl(\wt{K}_4 h+1\bigr) \snorm,  \\ 
&\|v-v_h\|_\lt
\le \wt{K}_5h^{\ell+r-4}\(\wt{K}_4h+1\) \snorm, 
\label{mixedlinbound2a}
\end{align}
where
\begin{align*}
\wt{K}_4&=O\left({\rm max}\{K_2 \eps^{-1}\tau^{-\frac12},
K_0^\frac12 K_{R_1}\eps^\frac12 \}\right),\quad
\wt{K}_5=O\left(K_2K_{R_1}\tau^{-\frac12}\right),\\
\ell&={\rm min}\{s,k+1\},
\end{align*}
and 
\begin{align*}
\tbar{\mu}{v}:&=\|\mu\|_\lt+\tau^{\frac12}\|v\|_\ho,\\
\ttbar{\mu}{v}:&=h\|\mu\|_\ho+\tbar{\mu}{v}.
\end{align*}
\end{thm}

\begin{proof}
It is clear from the proof of Theorem \ref{mixedlinthm}
that we only need to verify that condition {\rm [B2]} holds,
but with $\Fp[\se,\ue]$ replaced by $\wF^\prime[\set,\ue]$.

By the definition of $\wF^\prime[\set,\ue]$, we have
\begin{align*}
\bl\wF^\prime[\set,\ue](D^2 v,v),v\br
&=2\eps \tau \|\nab v\|_\lt^2-\eps n\tau^2 \|v\|_\lt^2
+\bl \Fp[\se,\ue](D^2v,v),v\br. 
\end{align*}
Thus, if [$\wt{\rm B2}$] holds, then
\begin{align}\label{altGardingremark}
\bigl\langle \wF^\prime[\set,\ue](D^2v,v),v\big\rangle
&\ge \wt{K}_1\|v\|_\ho^2-\wt{K}_0\|v\|_\lt^2,
\end{align}
with
\begin{align*}
\wt{K}_0:=\eps n \tau^2  +K_0,\qquad \wt{K}_1:=2\eps \tau.
\end{align*}

We also notice that
\begin{align*}
\bnorm{\wF^\prime[\set,\ue]}_{QQ^*}
&=\sup_{v\in Q_0}\sup_{w\in Q_0} 
\frac{\bigl\langle \wF^\prime[\set,\ue](D^2v,v),w\bigr\rangle}{\|v\|_\ho\|w\|_\ho}\\
&=\sup_{v\in Q_0}\sup_{w\in Q_0} 
\frac{2\eps \tau (\nab v,\nab w)-\eps n \tau^2 (v,w)+\bigl\langle 
\Fp[\se,\ue](D^2v,v),w\bigr\rangle}{\|v\|_\ho\|w\|_\ho}\\
&\le \eps \tau(2 +n \tau) +\bnorm{\Fp[\se,\ue](D^2v,v)}_{QQ^*}\\
&\le \eps \tau(2+n\tau)+K_2=:\wt{K}_2.
\end{align*}

It then follows that {\rm [B2]} holds but with $\Fp[\se,\ue]$ replaced by
$\wF^\prime[\set,\ue]$, and the assertions of the theorem immediately follow.
\end{proof}

With the well-posedness results for the linear problem established,
we can now state and prove the main result of
this section (compare to Theorem \ref{mixedmainthmimp}).

\begin{thm}\label{altmainmixedthm}
Suppose assumptions {\rm [B1]},{\rm [}$\wt{\rm B2}${\rm ]},{\rm [B3]--[B6]}
hold, $u^\vepsi\in H^3(\Ome)\,(s\geq3)$, 
$R(h)=o(h^{2-\ell})$, 
$\tau\in(0, \tau_0)$, and there exists $\wt{K}_3=\wt{K}_3(\eps)$ such that
\eqref{mixeddextra} holds. Then there exists $\wt{h}_1=\wt{h}_1(\eps)>0$
such that for $h\leq \min\{\wt{h}_0,\wt{h}_1\}$, there hold
the following error estimates:
\begin{align}\label{alterrorest0}
&\ttbar{\set-\set_h}{\ue-\ueh} \le \wt{K}_8h^{\ell-2}\setnorm,\\
\label{alterrorest1}
&\|\ue-\ueh\|_\lt\le 
K_{R_0}\Bigl(\wt{K}_9h^{\ell-2+{\rm min}\{2,\alpha\}}\|\ue\|_\hl
+\wt{K}_8^2R(h)h^{2\ell-4}\|\ue\|_\hl^2\Bigr),\\
\label{alterrorest2} &\|\ue-\ue_h\|_\ho
\le K_{R_1}\Bigl(\wt{K}_9h^{\ell-2+{\rm min}\{1,\alpha\}}\|\ue\|_\hl
+\wt{K}_8^2R(h)h^{2\ell-4}\|\ue\|_\hl^2\Bigr),
\end{align}
where 
\begin{alignat*}{1}
&\wt{K}_8=C\eps^{-\frac12}=C\wt{K}_3\eps^{-\frac12}\Bigl(\tau^{-\frac12}+K_0^\frac12K_{R_0}\Bigr),\\
&\wt{K}_9=C\wt{K}_8{\rm max}\{K_2,K_G\},\\
%
&\ell={\rm min}\{s,k+1\},
\end{alignat*}
and $s$ is defined in {\rm [B1]}.
\end{thm}

\begin{proof}
The idea of the proof is to show that
{\rm [B2]--[B6]} hold for $\wF^\prime$ (and $\wF^\prime[\set,\ue]$)
if {\rm [}$\wt{\rm B2}${\rm ]},{\rm [B3]--[B6]}
hold for $\Fp$ (and $\Fp[\se,\ue]$).  The result then follows using 
the same techniques as those employed in the proof of Theorem \ref{mixedmainthmimp}.

First, from the proof of Theorem \ref{altmixedlinthm}, we know
that {\rm [B2]} holds for $\wF^\prime[\set,\ue]$.
Next, if assumption {\rm [B3]} holds then
\begin{align*}
&\bnorm{\wF^\prime[\ome,y](\chi,v)}_{H^{-1}}\\
&\qquad=\sup_{z\in Q_0} \frac{2\eps \tau (\nab v,\nab z)-\eps n \tau^2(v,z)
+\bigl\langle \Fp[\ome-\tau I_{n\times n} y,y](\chi,v),z\bigr\rangle}{\|z\|_\ho}\\
&\qquad\le \eps \tau (2+n \tau)\|v\|_\ho+C 
\big\|(\ome-\tau I_{n\times n}y,y)\bigr\|_{X\times Y}
\(\|\chi\|_\lt+\|v\|_\ho\).
\end{align*}

If we define
\begin{align}\label{wtYdef}
\bigr\|\(\ome,y\)\bigl\|_{\wt{X}\times \wt{Y}}
:=\bigl\|\(\ome-\tau I_{n\times n} y,y\)\bigr\|_{X\times Y},
\end{align}
then
\begin{align*}
\bnorm{\wF^\prime[\ome,y](\chi,v)}_{H^{-1}}
&\le \eps \tau (2+n\tau)\|v\|_\ho
+C\bigl\|\(\ome,y\)\|_{\wt{X}\times \wt{Y}}\(\|\chi\|_\lt+\|v\|_\ho\).
\end{align*}

From the definitions of $Q^h$ and $W^h$,
$\bigl\|\(\cdot,\cdot\)\|_{\wt{X}\times \wt{Y}}$ is well-defined
on $W^h\times Q^h$ and if 
\begin{align}\label{mixeddextra}
\bigl\|\bigl(\Pi^h\se -\gamma \se,\util
-\gamma \ue\bigr)\bigr\|_{\wt{X}\times \wt{Y}}\le \wt{K}_3 
\quad\forall \gamma\in [0,1],
\end{align}
then it follows that conditions {\rm [B3]--[B4]} hold for $\wF^\prime$ 
with $\big\|\(\cdot,\cdot\)\bigr\|_{\wt{X}\times \wt{Y}}$ in 
place of $\bigl\|\(\cdot,\cdot\)\bigr\|_{X\times Y}$.

Next, for any $(\mu_h,v_h)\in \wW_\eps^h\times Q_g^h$, 
$(\kappa_h,z_h)\in W^h\times Q^h,$ and $w_h\in Q^h_0$
\begin{align*}
&\Bl \(\wF^\prime[\set,\ue]-\wF^\prime[\mu_h,v_h]\)\(\kappa_h,z_h\),w_h\Br\\
&\qquad=\Bl \(F^\prime[\se,\ue]-F^\prime[\mu_h-\tau I_{n\times n} v_h,v_h]\)\(\kappa_h,z_h\),w_h\Br.
\end{align*}
Define $\mu_\tau\in W^h_\eps$ such that
\[
\mu_\tau:=\mu_h-\tau I_{n\times n}v_h,
\]
and notice that if 
\[
\ttbar{\mathcal{I}^h\set-\mu_h}{\util-v_h}\le \del,
\]
 then
\[\ttbar{\stil-\mu_\tau}{\util-v_h}\le C\del.\]
Therefore, redefining $\del$ if necessary,
we have
\begin{align*}
&\Bl \(\wF^\prime[\set,\ue]-\wF^\prime[\mu_h,v_h]\)\(\kappa_h,z_h\),w_h\Br\\
&\qquad=\Bl \(F^\prime[\se,\ue]-F^\prime[\mu_\tau,v_h]\)\(\kappa_h,z_h\),w_h\Br\\
&\qquad\le R(h)\(\|\se-\mu_\tau\|_\lt+\|\ue-v_h\|_\ho\)\ttbar{\kappa_h}{z_h}\|w_h\|_\ho\\
&\qquad\le CR(h)\(\|\set-\mu_h\|_\lt+\|\ue-v_h\|_\ho\)\ttbar{\kappa_h}{z_h}\|w_h\|_\ho.
\end{align*}
Hence, {\rm [B5]} holds for $\wF^\prime[\set,\ue]$.

Finally, we show that condition {\rm [B6]} holds for $\wF^\prime[\set,\ue]$.
Suppose that
\begin{align}
\label{Bt5}(\chi_h,\kappa_h)+\wb(\kappa_h,v_h)&=0\qquad \forall \kappa_h\in W^h_0,
\end{align}
where $(\chi_h,v_h)\in W^h_0\times Q^h_0$.  It then follows that
\begin{align*}
(\chi_h-\tau I_{n\times n} v_h,\kappa_h)+b(\kappa_h,v_h)&=0,
\end{align*}
that is $(\chi_h-\tau I _{n\times n} v_h,v_h)\in\mathbb{T}_h$,
where $\mathbb{T}_h$ is defined in {\rm [B6]}.  Thus, if 
{\rm [B6]} holds (with $K_1=\eps \tau$ in definition of
$\ttbar{\cdot}{\cdot}$) and $(\chi_h,v_h)$ satisfies \eqref{Bt5} then
\begin{align*}
&\norm{F^\prime[\se,\ue](\chi_h-\tau I_{n\times n} v_h-D^2v_h,0)}_{H^{-1}}\\
&\qquad \qquad\le K_Gh^{\alpha} \ttbar{\chi_h-\tau I_{n\times n}v_h}{v_h}\\
&\qquad \qquad\le K_Gh^{\alpha} \Bigr(\ttbar{\chi_h}{v_h}+\sqrt{n} \tau \(h\|v_h\|_\ho+\|v_h\|_\lt\)\Bigr)\\
&\qquad \qquad\le CK_Gh^{\alpha} \ttbar{\chi_h}{v_h}.
\end{align*}

Hence, $\wF^\prime$ fulfills all {\rm [B2]}--{\rm [B6]}. The proof is
complete.
\end{proof}

%% file: chapter6a.tex
\chapter{Applications}\label{chapter-6}

In the previous two chapters we have developed two abstract frameworks for 
conforming and mixed finite element approximations of the
vanishing moment equation \eqref{moment1} under some (mild) structure 
conditions on the nonlinear differential operator $F$. The goal of this
chapter is to apply the two abstract frameworks to three 
specific nonlinear PDEs, namely, the Monge-Amp\`ere equation,
the equation of the prescribed Gauss curvature, and the
infinity-Laplacian equation. These three equations are
chosen because they represent three different scenarios
categorized by their linearizations, which are respectively,
coercive, indefinite, and degenerate. It is shown that
the abstract frameworks of Chapter \ref{chapter-4} and 
\ref{chapter-5} are broad enough to cover all three scenarios.

\section{The Monge-Amp\`ere equation}\label{chapter-6-sec-2}
The Monge-Amp\`ere equation \eqref{MAeqn1} is without
question the best known fully nonlinear second order PDE. 
It is to fully nonlinear second order PDEs as
the Poisson equation is to linear second order PDEs.
The Monge-Amp\`ere equation arises from applications in differential 
geometry, optimal transportation, geophysics, antenna design, and astrophysics.
We refer the reader to \cite{Caffarelli_Milman99,Gilbarg_Trudinger01,
Gutierrez01} and the references therein for more discussions about 
applications and PDE analysis of the Monge-Amp\`ere equation.

In this section, we consider finite element approximations
of the Monge-Amp\`ere equation with Dirichlet boundary condition:
\begin{alignat}{2}
\label{ch6MAeqn1}
\det(D^2u)&=f\ (>0)\qquad &&\text{in }\Ome,\\
\label{ch6MAeqn2}
u&=g\qquad &&\text{on }\p\Ome.
\end{alignat}

A detailed analysis of conforming finite elements for the 
Monge-Amp\`ere equation was carried out in
\cite{Feng4} (also see \cite{Neilan_thesis}), where the authors
proved optimal error estimates in the energy norm.  
The authors also studied mixed finite element 
methods for the Monge-Amp\`ere equation 
in \cite{Feng3} (also see \cite{Neilan_thesis}) and obtained
optimal error estimates for the scalar variable. However,
we note that the results to be given below are sharper than
those obtained in \cite{Feng3,Feng4} in the sense that 
weaker regularities of the solution $u^\vepsi$ are required in
the error estimates and the dependence on $\vepsi^{-1}$ of the error
bounds is less stringent.

In the case of the Monge-Amp\`ere equation, we have
\begin{align*}
F(D^2u,\nab u,u,x)&=f-\det(D^2u),\\
\Fp[v](w)&=-\cof(D^2v):D^2w,\\
\Fp[\mu,v](\kappa,w)&=-\cof(\mu):\kappa
\end{align*}
\begin{remark}
The inequality \eqref{abGarding} implies that
$F(D^2u,\nab u,u,x)=f-\det(D^2u)$ instead of 
$\det(D^2 u)-f$, which is used in most PDE
literature \cite{Gilbarg_Trudinger01}. Recall that
we assume $-F$ is elliptic in the sense of 
\cite[Chapter 17]{Gilbarg_Trudinger01} in this book.
\end{remark}

The vanishing moment approximation 
\eqref{moment1}--\eqref{moment3} becomes
\begin{alignat}{3} \label{MAagain1}
-\vepsi \Del^2 \ue+\det(D^2\ue)&=f\qquad &&\text{in }\Ome,\\
\label{MAagain2}\ue&=g \qquad &&\text{on }\p\Ome,\\
\label{MAagain3}\Del \ue&=\eps\qquad &&\text{on }\p\Ome,
\end{alignat}
and the linearization of 
\[
G_\eps(\ue)=\eps\Del^2\ue -\det(D^2\ue)+f
\]
at the solution $\ue$ is
\[
G^\prime_\eps[\ue](v)=\eps\Del^2 v-\Phi^\eps:D^2v
=\eps\Del^2 v-{\rm div}(\Phi^\eps \nab v),
\]
where $\Phi^\eps=\cof(D^2\ue)$, the cofactor matrix of the Hessian $D^2\ue$, 
and we have used Lemma \ref{cofactor} 
to obtain the last equality.

\subsection{Conforming finite element methods for the Monge-Amp\`ere 
equation}\label{chapter-6-1-2}

The finite element method for \eqref{MAagain1}--\eqref{MAagain3}
is defined as follows (cf. \eqref{momentfem}):
find $\ueh\in V^h_g$ such that
\begin{align}
\label{MAagainfem1}&-\eps(\Del \ueh,\Del v_h)+(\det(D^2\ueh),v_h)
=(f,v_h)-\left\langle\eps^2,\normd{v_h}\right\rangle_{\p\Ome}
\qquad \forall v_h\in V^h_0.
\end{align}
Recall $V=H^2(\Ome)$, and $V^h_0$ and $V^h_g$
 are the $C^1$ finite element spaces
of degree $k>4$ defined by \eqref{Vhzgdef}.

The goal of this section is to apply the abstract framework
of Chapter \ref{chapter-4} toward the finite element method
\eqref{MAagainfem1} in two and three dimensions.  
Namely, we verify {\rm [A1]--[A5]} and determine
how the constants, $C_i$, $\delta$, and $L(h)$, 
depend on $\eps$.  We summarize these results in the following theorem.

\begin{thm}\label{MAmainthm}
Let $\ue\in H^s(\Ome)$ be the solution 
to \eqref{MAagain1}--\eqref{MAagain3} with
$s\ge 3$ when $n=2$ and $s>3$ when $n=3$.  
Then for $h\le h_2$, there exists a unique 
solution $\ueh \in V^h_g$ to \eqref{MAagainfem1}.  
Furthermore, there hold the following error estimates:
\begin{align}
\label{MAmainthmline1}\|\ue-\ueh\|_\htw
&\le C_{7}h^{\ell-2}\|\ue\|_{H^\ell},\\
\label{MAmainthmline2}\|\ue-\ueh\|_\lt
&\le C_{8}\Bigl(\eps^{-\frac12}h^{\ell}\|\ue\|_\hl
+C_{7}L(h)h^{2\ell-4}\|\ue\|_\hl^2\Bigr),
\end{align}
where 
\begin{alignat*}{2}
&C_{7}=O\bigl(\eps^{\frac12(1-2n)}\bigr),\qquad &&C_{8}
=O\bigl(\eps^{-\frac12(5+2n)}\bigr),\\
&L(h)=O\bigl(\eps^{\frac56(2-n)}+h^{(2-n)}\bigr),\qquad 
&&\ell={\rm min}\{s,k+1\},
\end{alignat*}
and $h_2$ is chosen such that
\begin{align*}
h_2&\le C\left(\eps^{-\frac12 (1+2n)}\|u\|_\hl L(h_2)\right)^{\frac{1}{2-\ell}},\\
\end{align*}
%
\end{thm}

\begin{proof}
We first state the a priori bounds shown in Chapter \ref{chapter-3}
(also see \cite{Feng1}):
\begin{alignat}{2}
\label{absbounds}&\|\ue\|_{H^j}=O\bigl(\eps^{\frac{1-j}{2}}\bigr)\
(j=1,2,3),\quad &&\|\ue\|_{W^{j,\infty}}=O\bigl(\eps^{1-j}\bigr)\ (j=1,2),\\
&\nonum\bnorm{\Phi^\eps}_{L^\infty}=O\bigl(\eps^{-1}\bigr),\quad
&&\bnorm{\Phi^\eps}_\lt=O\bigl(\eps^{-\frac12}\bigr).
\end{alignat}
We also note that by interpolation between $L^p$ spaces, we have for $p\in [2,\infty]$
\begin{align}\label{absboundsinterp}
\|D^2 \ue\|_{L^p}\le \|D^2 \ue\|_\lt^{\frac{2}p}\|D^2 \ue\|_{L^\infty}^{\frac{p-2}{p}}=O\bigl(\eps^{\frac{1-p}p}\bigr).
\end{align}
Next, since 
\begin{align*}
\Del^2 \ue=\eps^{-1}\bigl(\det(D^2\ue)-f\bigr),
\end{align*}
by standard theory for the biharmonic equation, if $\p\Ome$ 
is sufficiently smooth, then $\ue\in H^4(\Ome)$ with
\begin{align*}
\|u\|_{H^4}
&\le \eps^{-1}\bigl( \|\det(D^2 \ue)\|_\lt+\|f\|_\lt\bigr)\\
&\le \eps^{-1}\bigl(\|D^2 \ue\|_{L^{2n}}^n+\|f\|_\lt\bigr).
\end{align*}
Therefore, in view of \eqref{absboundsinterp}, we have
\begin{align}\label{absH4}
\|u\|_{H^4}=O\bigl(\eps^{-\frac{(1+2n)}2}\bigr)
\end{align}
Thus, by \eqref{absbounds}, \eqref{absH4}, 
and interpolation of Sobolev spaces, we have
\begin{align}\label{absboundsh}
\|u\|_{H^m}=O\bigl(\eps^{\frac{n(2-m)-1}2}\bigr)\qquad \forall m\in [2,4].
\end{align}

In addition, $\ue$ is strictly convex. Hence, $\Phi^\eps$ is positive definite,
and therefore, there exists $C>0$ such that
\begin{align*}
\bigl\langle \Phi^\eps\nab w,\nab w\bigr\rangle 
\ge C\|\nab w\|^2_\lt\qquad \forall w\in V_0.
\end{align*}
It then follows that
\begin{align}
\label{MAconformingGarding}
\bigl\langle \Gp[\ue](w),w\bigr\rangle \ge C\eps \|w\|_\htw^2\qquad \forall w\in V_0.
\end{align}

Next, using a Sobolev inequality
\begin{align}
\label{MAVVbound}
\bnorm{\Fp[\ue]}_{VV^*}
&=\sup_{v\in V_0}\sup_{w\in V_0} 
\frac{\bigl\langle \Fp[\ue](v),w\bigr\rangle}{\|v\|_\htw\|w\|_\htw}\\
&\nonum=\sup_{v\in V_0}\sup_{w\in V_0} 
\frac{(\Phi^\eps\nab v,\nab w)}{\|v\|_\htw\|w\|_\htw}\\
&\nonum\le \sup_{v\in V_0}\sup_{w\in V_0} 
\frac{\|\Phi^\eps\|_\lt\|\nab v\|_{L^4}\|\nab w\|_{L^4}}{\|v\|_\htw\|w\|_\htw}\\
&\nonum\le C\|\Phi^\eps\|_\lt\le C\eps^{-\frac12}.
\end{align}

If $\p\Ome$ is sufficiently smooth and $v\in V_0$ solves
\begin{align*}
\bigl\langle \Gp[\ue](v),w\bigr\rangle = (\varphi,w)\qquad \forall w\in V_0,
\end{align*}
where $\varphi$ is some $L^2(\Ome)$-function, then by standard elliptic PDE theory 
\cite{evans,Gilbarg_Trudinger01}, $v\in H^p(\Ome)$ for $p\geq 2$.  
Furthermore, in view of Remark \ref{Remark44ii} and the estimate
\begin{align*}
\left\|\frac{\p F}{\p r_{ij}}(\ue)\right\|_{L^\infty} 
=\bnorm{\Phi^\eps_{ij}}_{L^\infty}\le C\eps^{-1},
\end{align*}
we have
\begin{align} \label{MACR}
\|v\|_{H^3}\le C\eps^{-2}\|\varphi\|_\lt,\qquad
\|v\|_{H^4}\le C\eps^{-3}\|\varphi\|_\lt.
\end{align}

Thus, by \eqref{MAconformingGarding}-- \eqref{MACR},
condition {\rm [A2]} holds with
\begin{alignat}{3}
\label{MAconformingconstants1}
&C_0\equiv 0,\quad\qquad &&C_1=O(\eps),\qquad &&C_2=O(\eps^{-\frac12}),\\
\nonum&p=4,\quad\qquad &&C_R=O(\eps^{-3}),
\end{alignat}
and therefore (cf. Theorems \ref{linexistencethm} and \ref{abstractbound1thm})
\begin{alignat}{2}
\label{MAconformingconstants2}
&C_3=O(\eps^{-1}),\qquad\quad
&&C_4
=O(\eps^{-\frac32}),\\ 
\nonum &C_5
=O(\eps^{-5}),\qquad\quad
&&h_0
=1.
\end{alignat}

To confirm {\rm [A3]--[A4]}, we choose
\begin{align*}
Y=W^{2,2(n-1)}(\Ome),\qquad \|\cdot\|_Y=\|\cdot\|_{W^{2,2(n-1)}}^{n-1}.
\end{align*}
%
For a smooth function $y$, we
use Lemma \ref{cofactor} and 
a Sobolev inequality to conclude
\begin{align*}
\frac{\bnorm{\Fp[y]}_{VV^*}}{\|y\|_Y}
&=\sup_{v\in V_0}\sup_{w\in V_0} 
\frac{\bigl\langle\cof(D^2y):D^2v,w\bigr\rangle}{\|y\|_Y\|v\|_\htw\|w\|_\htw}\\
&=\sup_{v\in V_0}\sup_{w\in V_0} 
\frac{\(\cof(D^2y)\nab v,\nab w\)}{\|y\|_Y\|v\|_\htw\|w\|_\htw}\\
&\le C\left(\frac{\bnorm{\cof(D^2y)}_\lt}{\|y\|_Y}\right)\le  C\left(\frac{\bnorm{D^2y}_{L^{2(n-1)}}^{n-1}}{\|y\|_Y}\right)
\le C.
\end{align*}
It then follows from a simple density argument
that
\begin{align*}
\sup_{y\in Y}
\frac{\bnorm{\Fp[y]}_{VV^*}}{\|y\|_Y}\le C,
\end{align*}
and therefore condition {\rm [A3]} holds,
and by standard interpolation theory \cite{Ciarlet78,Brenner}
condition {\rm [A4]} holds as well.

We also note that by \eqref{absbounds}--\eqref{absboundsinterp}
and Lemma \ref{lem51}
\begin{align}\label{absMAY}
\|\ue\|_Y\le C\eps^{\frac12 (3-2n)},
\qquad\quad
C_6
=O\left(\eps^{\frac12(1-2n)}\right).
\end{align}

To verify {\rm [A5]}, we derive the following identity for any
$v_h\in V^h_g$:
\begin{align*}
\bnorm{\Fp[\ue]-\Fp[v_h]}_{VV^*}
&=\sup_{w\in V_0}\sup_{z\in V_0}
\frac{\Bigl(\(\cof(D^2\ue)-\cof(D^2v_h)\)\nab w,\nab z\Bigr)}{\|w\|_\htw\|z\|_\htw}\\
&\le C\bnorm{\cof(D^2\ue)-\cof(D^2v_h)}_{L^\frac32}
\end{align*}

It follows that for $n=2$,
\begin{align*}
\bnorm{\Fp[\ue]-\Fp[v_h]}_{VV^*}
&\le C\|\ue-v_h\|_{W^{2,\frac32}}\le C\|\ue-v_h\|_\htw.
\end{align*}
Hence, {\rm [A5]} holds with $L(h)=C$.

For the case $n=3$, we conclude by the mean value theorem 
that for any $i,j=1,2,3$,
\begin{align*}
\bnorm{\cof(D^2\ue)_{ij}-\cof(D^2v_h)_{ij}}_{L^\frac32}
&=\bnorm{\det(D^2\ue|_{ij})-\det(D^2v_h|_{ij})}_{L^\frac32}\\
&\le \bnorm{\Lambda^{ij}}_{L^6}\bnorm{D^2\ue|_{ij}-D^2v_h|_{ij}}_{L^2}\\
&\le \bnorm{\Lambda^{ij}}_{L^6}\bnorm{\ue-v_h}_\htw,
\end{align*}
where $D^2\ue|_{ij}$ denotes the resulting $2\times 2$ 
matrix after deleting the $i^{th}$ row and $j^{th}$ column of $D^2\ue$, and 
$\Lambda^{ij}=\cof(D^2\ue|_{ij}+\gamma(D^2v_h|_{ij}-D^2\ue|_{ij}))$
for some $\gamma\in [0,1]$.  
Noting $\Lambda^{ij}\in
\mathbf{R}^{2\times 2}$, we have 
\[
\bnorm{\Lambda^{ij}}_{L^6}\le \|\ue\|_{W^{2,6}}+\|v_h\|_{W^{2,6}}.
\] 
Thus, for any $\del\in (0,1)$ and
$v_h\in V^h_g$ with $\|\util-v_h\|_\htw \le \del$, we have using
the triangle inequality, the inverse inequality, and \eqref{absboundsinterp}
\begin{align*}
\bnorm{\Fp[\ue]-\Fp[v_h]}_{VV^*}
&\le \(\|\ue\|_{W^{2,6}}+\|v_h\|_{W^{2,6}}\)\|\ue-v_h\|_\htw\\
& \le C\(\|\ue\|_{W^{2,6}}+h^{-1}\bnorm{v_h-\util}_\htw\)\|\ue-v_h\|_\htw\\
&\le C\(\eps^{-\frac56}+h^{-1}\delta\)\|\ue-v_h\|_\htw\\
&\le C\(\eps^{-\frac56}+h^{-1}\)\|\ue-v_h\|_\htw\\
&=L(h)\|\ue-v_h\|_\htw.
\end{align*}
Thus, in the three-dimensional case {\rm [A5]} holds with
$L(h)=C\(\eps^{-\frac56}+h^{-1}\)$.

Gathering all of our results, existence of a unique
solution to \eqref{MAagainfem1} and the error estimates
\eqref{MAmainthmline1}--\eqref{MAmainthmline2} follow 
from Theorem \ref{abstractmainthm} and the estimates
\eqref{MAconformingconstants1}--\eqref{absMAY}.
\end{proof}

\begin{remarks}
(a) Estimates \eqref{MAmainthmline1} and 
\eqref{MAmainthmline2} give the same asymptotic
rates in $h$ as those obtained in \cite{Feng4}.
However, they provide an improvement to these previous results
in the sense that the constants $C_7$ and $h_2$ have
a better order dependence in terms of $\eps$.

(b) We require stronger regularity in the three-dimensional case
to ensure $L(h)=o(h^{2-\ell})$ (cf. Theorem \ref{abstractmainthm}).
\end{remarks}

\subsection{Mixed finite element methods for the Monge-Amp\`ere equation}\label{chapter-6-1-3}
The mixed finite element method for \eqref{MAagain1}--\eqref{MAagain3}
is defined as follows (cf. \eqref{mixedfem1}--\eqref{mixedfem2}):
find $(\se,\ue)\in W_\eps^h\times Q^h_g$ such that
\begin{alignat}{2}
\label{mixedfemagain1}
(\seh,\kappa_h)+b(\kappa_h,\ueh)
&=G(\kappa_h)\quad && \forall \kappa_h\in W_0^h,\\
\label{mixedfemagain2}b(\seh,z_h)-\eps^{-1}c(\seh,\ueh,z_h)&=0
\quad && \forall z_h\in Q^h_0,
\end{alignat}
where
\begin{align*}
b(\kappa_h,\ueh)&=\(\Div(\kappa_h),\nab \ueh\),
\qquad c(\seh,\ueh,z_h)=\(f-\det(\seh),z_h\),
\end{align*}
$G(\kappa_h)$ is defined by \eqref{Gdef},
$Q=H^1(\Ome),\ W=\left[H^1(\Ome)\right]^{n\times n}$, and 
$W^h_0,\ W^h_\eps,\ Q^h_0$, and $Q^h_g$ are the Lagrange finite element 
spaces of degree $k\ge 2$ defined in Section \ref{chapter-5-sec-2}.

We now apply the abstract theory developed in Chapter \ref{chapter-5}
to the mixed finite element method \eqref{mixedfemagain1}--\eqref{mixedfemagain2}.
Similar to the previous subsection, our goal is to show that 
assumptions {\rm [B1]--[B6]} hold, and to explicitly derive how 
the constants, $K_i,\ \delta,$ and $R(h)$ depend on the parameter $\eps$.
We summarize our findings in the following theorem.

\begin{thm}\label{mixedMAmain}
Suppose $\ue\in H^s(\Ome)$ is the solution 
of \eqref{MAagain1}--\eqref{MAagain3} with $s\ge 3$
when $n=2$ and $s>4$ when $n=3$.  Furthermore, assume
that $k\ge 4$ when $n=3$.
Then for $h\le h_2$, there exists a unique solution 
$(\seh,\ueh)\in W^h_\eps\times Q^h_g$
to \eqref{mixedfemagain1}--\eqref{mixedfemagain2}.  
Furthermore,  there hold the following error estimates:
\begin{align} \label{mixedMAmainline1}
&\ttbar{\se-\seh}{\ue-\ueh}
\le K_8h^{\ell-2}\snorm\\
\label{mixedMAmainline2}
&\|\ue-\ueh\|_\ho 
\le K_{R_1}\Bigl( K_{9} h^{\ell-1}\snorm+K_{8}^2R(h)h^{2\ell-4}\snorms\Bigr),
\end{align}
where
\begin{alignat*}{1}
&\ttbar{\mu}{v}=h\|\mu\|_\ho+\|\mu\|_\lt+\eps^{-\frac12}\|v\|_\ho,\\
&K_8=O\bigl(\eps^{\frac14(22-15n)}\bigr),\qquad
K_{9}
=O\bigl(\eps^{\frac{70-53n}{12}}\bigr),\\
&R(h)=O\left(|\log h|^\frac{n-3}{2}+(n-2)\(\eps^{-1}h^{-1}+h^{-2}\)\right),\\
&\ell={\rm min}\{s,k+1\},
\end{alignat*}
and $h_2$ is chosen such that
\begin{align*}
h_2
&\approx {\rm min}\Bigl\{\eps^{\frac{1+4n}6},
\left(\eps^{\frac54(4-3n)}R(h_2)\snorm\right)^{\frac{1}{1-\ell}},
 \left(\eps^{-\frac12}R(h_2)\snorm\right)^{\frac{1}{2-\ell}}\Bigr\}.
\end{align*}
\end{thm}

\begin{proof}
First, using the same arguments as those used to show 
condition {\rm [A2]} in Theorem \ref{MAmainthm},
we can also conclude that {\rm [B2]} holds with
\begin{alignat}{3}
\label{MAmixedline1}
&K_0\equiv 0,\qquad &&K_1=O(1),\qquad &&K_2=O\bigl(\eps^{-1}\bigr),\\
\nonum&p=4,\qquad &&K_{R_0}=O\bigl(\eps^{-3}\bigr),
\qquad &&K_{R_{1}}=O\bigl(\eps^{-2}\bigr),
\end{alignat}
and therefore (cf.~Theorem \ref{mixedlinthm} and Lemma \ref{mixedlem52})
\begin{align}
\label{MAmixedline2}
K_4&
=O\bigl(\eps^{-\frac32}\bigr),\qquad
K_5
=O\bigl(\eps^{-\frac72}\bigr),\quad
K_7
=O\bigl(\eps^{-\frac12}\bigr).
\end{align}

To confirm {\rm [B3]--[B4]}, on noting that $F'[\mu,v](\kappa,w)$ is
independent of $v$ and $w$, we choose the spaces $X$ and $Y$ as follows:
\begin{alignat*}{1}
X&=\left[L^{(n-1)(n+\vepsi(3-n))}(\Ome)\right]^{n\times n},
\qquad Y=\emptyset,\\
\bigl\|(\ome,y)\bigr\|_{X\times Y}&=\|\ome\|_{L^{(n-1)(n+\vepsi(3-n))}}^{n-1}\ 
\quad\forall \ome\in X,\ y\in Y.
\end{alignat*}
Then using a Sobolev inequality, we have
for all $\omega\in X,\ y\in Y,\ \chi\in W,\ v\in Q$
\begin{align*}
\bnorm{\Fp[\omega,y](\chi,v)}_{H^{-1}}
&=\sup_{w\in Q_0} \frac{\(\cof(\omega):\chi,w\)}{\|w\|_\ho}\\
&\le C \norm{\cof(\omega)}_{L^{n+\vepsi(3-n)}}\norm{\chi}_\lt\\
&\le C \norm{\ome}_{L^{(n-1)(n+\vepsi(3-n))}}^{n-1}\norm{\chi}_{L^2}\\
&\le C\norm{\(\ome,y\)}_{X\times Y}\(\|\chi\|_\lt+\|v\|_\ho\).
\end{align*}
Thus condition {\rm [B3]} holds.

To confirm {\rm [B4]}, we note that by the inverse inequality,
standard stability results for the interpolation operator, and \eqref{Piapprox}
to conclude that if $\se\in \left[H^{s-2}(\Ome)\right]^{n\times n}$ then
 for any $p\in [2,\infty]$ and $\ell \in [3,{\rm min}\{s,k+1\}]$
\begin{align}
\label{mixedmess}
\|\Pi^h\se\|_{L^p}
&\le \|\Pi^h\se-\mci \se\|_{L^p}+\|\mci \se\|_{L^p}\\
&\nonum\le C\Bigl(h^{\frac{n}p-\frac{n}2}\|\Pi^h\se-\mci \se\|_{L^2}+\|\mci \se\|_{L^p}\Bigr)\\
&\nonum\le C\Bigl(h^{\frac{n}p-\frac{n}2+\ell-2}\|\se\|_{H^{\ell-2}}+\|\se\|_{L^p}\Bigr).
\end{align}

Therefore, for any $\gamma\in [0,1]$
\begin{align*}
&\bigl\|\bigl(\stil-\gamma\se,\util-\gamma \ue\bigr)\bigr\|_{X\times Y}\\
&\hspace{2.5cm}=\bigl\|\stil-\gamma\se\bigr\|_{L^{6(n-1)}}^{n-1}\\
&\hspace{2.5cm} \le C\Bigl(h^{\frac{n}{6(n-1)}-\frac{n}2+\ell-2}\|\se\|_{H^{\ell-2}} +\|\se\|_{L^{6(n-1)}}\Bigr)^{n-1}.
\end{align*}
For the two-dimensional case, we set $\ell=3$ and use \eqref{absbounds}--\eqref{absboundsinterp} to get
\begin{align*}
\bigl\|\bigl(\stil-\gamma\se,\util-\gamma \ue\bigr)\bigr\|_{X\times Y}
&\le C\bigl(h^{\frac13}\|\se\|_{H^1} +\|\se\|_{L^{6}}\bigr)\\
&\le C\bigl(h^\frac13\eps^{-1}+\eps^{-\frac56}\bigr)=O\bigl(\eps^{-1}\bigr).
\end{align*}
For the three-dimensional case, we set $\ell=\frac{13}4$ and use \eqref{absboundsinterp}--\eqref{absboundsh}
to get
\begin{align*}
\bigl\|\bigl(\stil-\gamma\se,\util-\gamma \ue\bigr)\bigr\|_{X\times Y}
&\le C\bigl(\|\se\|_{H^\frac54}+\|\se\|_{L^{12}}\bigr)^2\\
&\le C\bigl(\eps^{-\frac{19}4}+\eps^{-\frac{11}6}\bigr)
=O\bigl(\eps^{-\frac{19}4}\bigr).
\end{align*}
Therefore by Lemma \ref{mixedlem51}
\begin{align}
\label{MAmixedXY}
K_3=O\Bigl(\eps^{\frac14(26-15n)}\Bigr),
\qquad K_6=O\Bigl(\eps^{\frac14(22-15n)}\Bigr).
\end{align}

To confirm {\rm [B5]}, we have for any $(\mu_h,v_h)\in W^h_\eps\times Q^h_g$,
$(\kappa_h,z_h)\in W^h\times Q^h$, and $w_h\in Q^h$
\begin{align*}
&\Bigl\langle \(\Fp[\se,\ue]-
\Fp[\mu_h,v_h]\)\(\kappa_h,z_h\),w_h\Bigr\rangle
=
\Bigl(\(\cof(\se)-\cof(\mu_h)\):\kappa_h,w_h\Bigr)\\
%
%
&\qquad\quad\le C|\log h|^{\frac{3-n}{2}} h^{1-\frac{n}2}
\bnorm{\cof(\se)-\cof(\mu_h)}_\lt\ttbar{\kappa_h}{z_h}\|w_h\|_\ho,
\end{align*}
where we have used the inverse inequality \cite[Lemma 4.9.1]{Brenner}.  
%

If $n=2$, then 
$\bnorm{\cof(\se)-\cof(\mu_h)}_\lt=\|\se-\mu_h\|_\lt$,
and so condition {\rm [B5]} holds with $R(h)=C|\log h|^{\frac12}$.
%
%
For $n=3$,
\begin{align*}
\bnorm{(\cof(\se)-\cof(\mu_h))_{ij}}_\lt
&=\bnorm{\det(\se|_{ij})-\det(\mu_h|_{ij})}_\lt\\
&=\bnorm{\Lambda^{ij}:(\se|_{ij}-\mu_h|_{ij})}_\lt\\
&\le \bnorm{\Lambda^{ij}}_{L^\infty}\bnorm{\se|_{ij}-\mu_h|_{ij}}_{L^2}\\
&\le C\bnorm{\Lambda^{ij}}_{L^\infty}\bnorm{\se-\mu_h}_\lt,
\end{align*}
where $\Lambda^{ij}=\cof(\se\big|_{ij}+\gamma(\mu_h\big|_{ij}-\se\big|_{ij}))$
for some $\gamma\in [0,1]$, and we have used the same notation as in Section 
\ref{chapter-6-1-2}.
Since $\Lambda^{ij}\in\mathbf{R}^{2\times 2}$, we have 
for $\|\stil-\mu_h\|_\lt \le \delta\in (0,1)$
\begin{align*}
\bnorm{\Lambda^{ij}}_{L^\infty}
&\le C\(\bnorm{\se+\stil}_{L^\infty}+\bnorm{\stil-\mu_h}_{L^\infty}\)\\
&\le C\(\eps^{-1}+h^{-\frac32} \delta\)
\le C\(\eps^{-1}+h^{-\frac32}\).
\end{align*}
It then follows that {\rm [B5]} holds in the case $n=3$ with
$R(h)=C\(\eps^{-1}h^{-\frac12}+h^{-2}\).$
We note that for the hypotheses in Theorems 
\ref{mixedmainthm}--\ref{mixedmainthmimp} to hold, 
we require $R(h)=o(h^{2-\ell})$ as $h\to 0^+$ for fixed $\eps$.
This requirement is satisfied if $\ell>2$ in two dimension, and 
this bound is true provided $\ell>4$ in three dimensions.

Next, to verify condition {\rm [B6]}, we first use Holder's inequality and 
\eqref{absbounds}--\eqref{absboundsinterp} to conclude that for $p\in [1,2]$ 
and any $i,j,k=1,2,...,n$ $(n=2,3)$
\begin{align*}
\left\|\frac{\p \Phi^\eps_{ij}}{\p x_k} \right\|_{L^{p}}
&\le C\|D^2 \ue\|_{L^{\frac{2p}{2-p}}}^{n-2}\|D^2 \ue\|_\ho=O\left(\eps^{\frac{(2-3p)(n-2)-2p}{2p}}\right).
\end{align*}
Therefore, in view of Proposition \ref{B6prop}
and the estimates
\begin{align*}
\left\|\frac{\p F}{\p r_{ij}}\right\|_{L^\infty}
&= \bnorm{\Phi^\eps_{ij}}_{L^\infty}=O\left(\eps^{-1}\right),\\
\left\|\frac{\p F}{\p r_{ij}}\right\|_{W^{1,\frac65}}
&=\bigl\|\Phi^\eps_{ij}\bigr\|_{W^{1,\frac65}}=O\left(\eps^{\frac{1-2n}{3}}\right),
\end{align*}
to conclude that {\rm [B6]} holds with
\begin{align}
\label{MAmixedB6}
\alpha=1,\quad\qquad K_G=\eps^{\frac{1-2n}{3}}.
\end{align}

Thus, the existence of a unique solution $(\seh,\ueh)$ to 
\eqref{mixedfemagain1}--\eqref{mixedfemagain2}
and the error estimates \eqref{mixedMAmainline1}--\eqref{mixedMAmainline2}
follows from Theorem \ref{mixedmainthm} and the estimates
 \eqref{MAmixedline1}--\eqref{MAmixedB6}.
\end{proof}

\begin{remark}

\label{Remarks65ii}
The error estimates in Theorem \ref{mixedMAmain}
have the same order of convergence in $h$ as the estimates derived in \cite{Feng3},
but the constants' dependence on $\eps$ in 
\eqref{mixedMAmainline1}--\eqref{mixedMAmainline2}
are sharper than these previous results.
\end{remark}

\subsection{Numerical experiments and rates of convergence}

Extensive numerical experiments for the finite element methods
\eqref{MAagainfem1} and
\eqref{mixedfemagain1}--\eqref{mixedfemagain2}
in the two-dimensional setting have already been reported in
\cite{Feng4} and \cite{Feng3}, respectively.  
These tests confirmed the error estimates 
\eqref{MAmainthmline1}--\eqref{MAmainthmline2}
and \eqref{mixedMAmainline1}--\eqref{mixedMAmainline2},
and indicate that these estimates are sharp.  Furthermore, 
the tests confirm the following rates of convergence:
\begin{align*}
\|u-\ue\|_\lt=O(\eps),
\quad \|u-\ue\|_\ho=O\bigl(\eps^\frac34\bigr),
\quad \|u-\ue\|_\htw=O\bigl(\eps^\frac14\bigr),
\end{align*}
which are proved in Theorem \ref{convergence_rate_thm3}
when the viscosity solution $u$ belongs to the space 
$W^{2,\infty}(\Ome)\cap H^3(\Ome)$. 

In this section, we expand on these earlier results,
performing two and three-dimensional numerical experiments and 
comparing the results with these earlier findings.
We also show that for certain problems, 
one must choose an appropriate $h-\eps$ relation
in order for the method to converge.
The tests below are done
on the unit square $\Ome=(0,1)^n\ (n=1,2,3)$.

\subsubsection*{Test 6.1.1}

For this test, we calculate $\|u-\ueh\|$ for fixed $h=0.02$, 
while varying $\eps$ in order
to estimate $\|u-\ue\|$.  We solve the mixed finite element method 
\eqref{mixedfemagain1}--\eqref{mixedfemagain2}
using the quadratic Lagrange finite element $(k=2)$
with the following test functions:
\begin{alignat*}{2}
&{\rm (a)}\ u=e^{(x^2_1+x_2^2+x_3^2)/2},\qquad 
&&f=\(1+x_1^2+x_2^2+x_3^2\)e^{3(x^2_1+x_2^2+x_3^2)/2},\\
&{\rm (b)}\ u=x_1^2+x_2^2+x_3^2,\qquad &&f=8.
\end{alignat*}

After having computed the solution, we list the 
errors in various norms in Table \ref{Test611Table} and
plot the results in Figures \ref{Test611Figure1}--\ref{Test611Figure2}.  
The figures indicate that 
\begin{alignat*}{2}
&\|u-\ueh\|_{L^\infty}=O(\eps),\qquad
&&\|u-\ueh\|_\lt=O\bigl(\eps\bigr),\\
&\|u-\ueh\|_\ho=O\bigl(\eps^\frac34\bigr),\qquad
&&\|\se-\seh\|_\lt=O\bigl(\eps^\frac14\bigr).
\end{alignat*}
Therefore, since $h$ is small, we expect 
\begin{alignat*}{2}
&\|u-\ue\|_{L^\infty}\approx O(\eps),\qquad
&&\|u-\ue\|_\lt\approx O(\eps),\\
&\|u-\ue\|_\ho\approx O\bigl(\eps^\frac34\bigr),\qquad
&&\|u-\ue\|_\htw\approx O\bigl(\eps^\frac14\bigr).
\end{alignat*}
We note that these are
the same rates of convergence found in \cite{Feng4} and \cite{Feng3}.

\begin{figure}[htbp]
\centering
\includegraphics[scale=0.3]{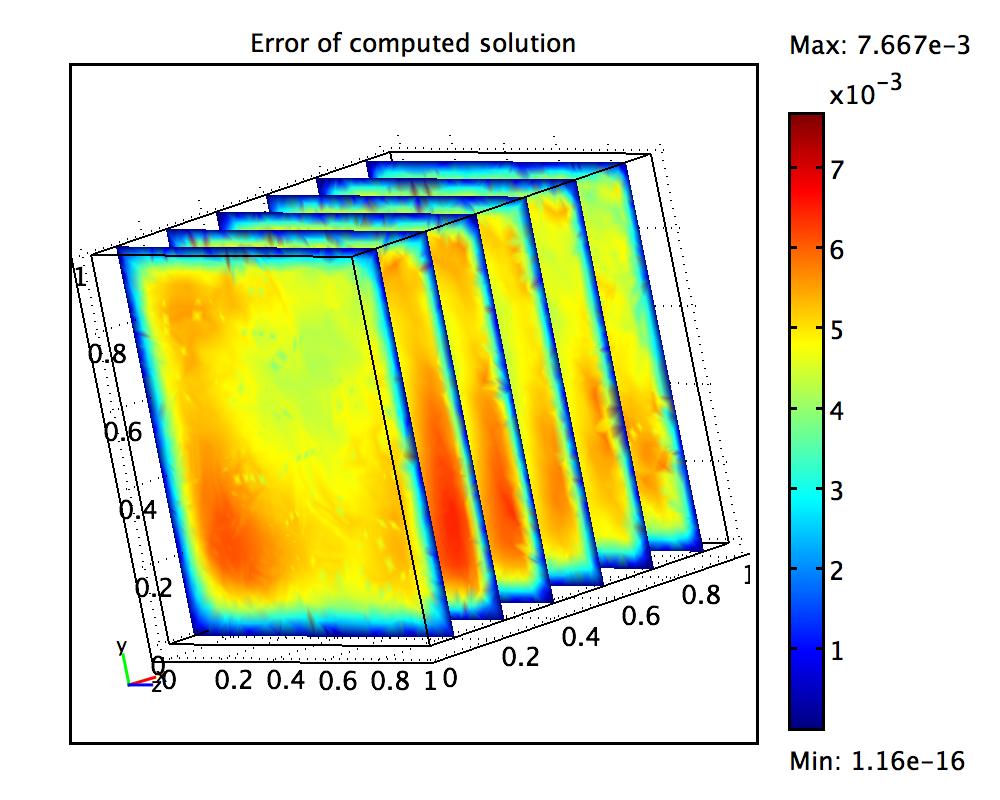}\\
\caption{Test 6.1.1a.  Error of computed solution with $\eps=0.01$ and $h=0.02$.}
\label{Test611Picture}
\end{figure}

\begin{table}[htbp]
\caption{Test 6.1.1.  Error of $\|u-\ueh\|$ w.r.t. $\eps$ ($h=0.02$)}
\label{Test611Table}
\centering
\begin{tabular}{lcllll}
&{\scriptsize$\eps$} & {\scriptsize$\|u-\ueh\|_{L^\infty}$(rate)} & 
{\scriptsize $\|u-\ueh\|_\lt$(rate)} &  {\scriptsize $\|u-\ueh\|_\ho$(rate)} & 
{\scriptsize $\|\sigma-\seh\|_\lt$(rate)}\\
\noalign{\smallskip}\hline\noalign{\smallskip}
Test 6.1.1a
&5.0E--01	&1.19E--01\blank	&5.71E--02\blank	&3.47E--01\blank&3.34E+00\blank\\
&2.5E--01	&8.91E--02(0.42)	&4.63E--02(0.30)	&2.88E--01(0.27)&3.08E+00(0.12)\\
&1.0E--01	&5.36E--02(0.55)	&3.19E--02(0.41)	&2.09E--01(0.35)&2.72E+00(0.14)\\
&5.0E--02	&2.35E--02(1.19)	&1.59E--02(1.00)	&1.21E--01(0.79)&2.29E+00(0.25)\\
&2.5E--02	&1.18E--02(0.99)	&8.95E--03(0.83)	&7.35E--02(0.72)&1.99E+00(0.20)\\
&1.0E--02	&5.57E--03(0.82)	&4.25E--03(0.81)	&3.91E--02(0.69)&1.66E+00(0.20)\\
\noalign{\smallskip}\hline\noalign{\smallskip}
Test 6.1.1b
& 5.0E--01	&1.61E--01\blank&7.47E--02\blank 	&4.27E--01\blank  	&3.12E+00\blank\\	
&2.5E--01	&1.36E--01(0.24)	&6.48E--02(0.21)	&3.75E--01(0.19)	&2.91E+00(0.10)\\
&1.0E--01	&7.94E--02(0.59)	&4.17E--02(0.48)	&2.52E--01(0.43)	&2.36E+00(0.23)\\
&5.0E--02	&4.20E--02(0.92)	&2.49E--02(0.74)	&1.61E--01(0.64)	&1.92E+00(0.29)\\
&2.5E--02	&1.99E--02(1.08)	&1.36E--02(0.88)	&9.70E--02(0.73)	&1.57E+00(0.29)\\
&1.0E--02	&7.36E--03(1.09)	&5.76E--03(0.94)	&4.85E--02(0.76) 	&1.26E+00(0.24)\\
&5.0E--03	&3.79E--03(0.96)	&3.10E--03(0.89)	&2.97E--02(0.71)	&1.11E+00(0.17)
\end{tabular}
\end{table}

\begin{figure}[htbp]
\centering
\includegraphics[scale=0.5]{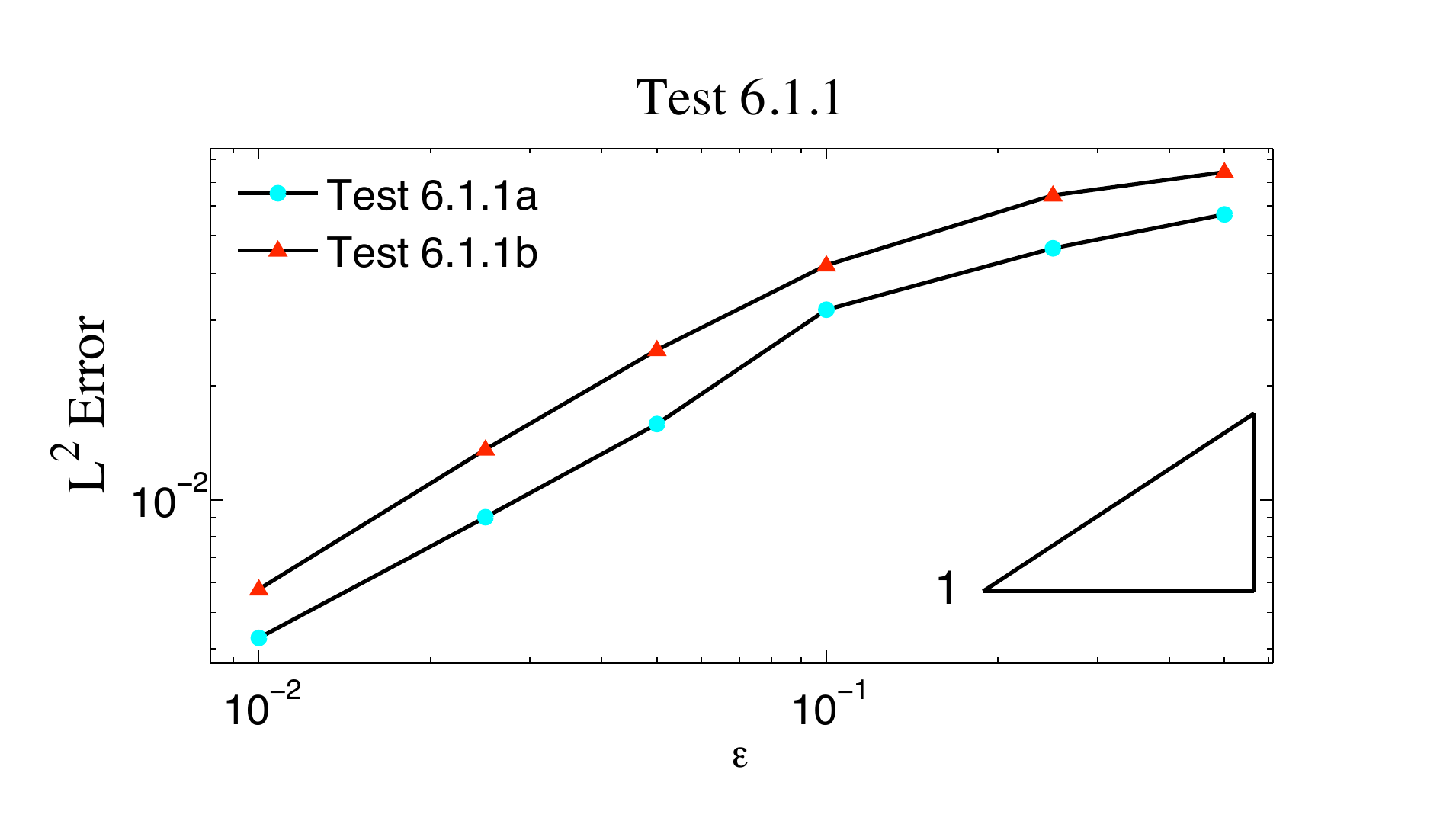}\\
\includegraphics[scale=0.5]{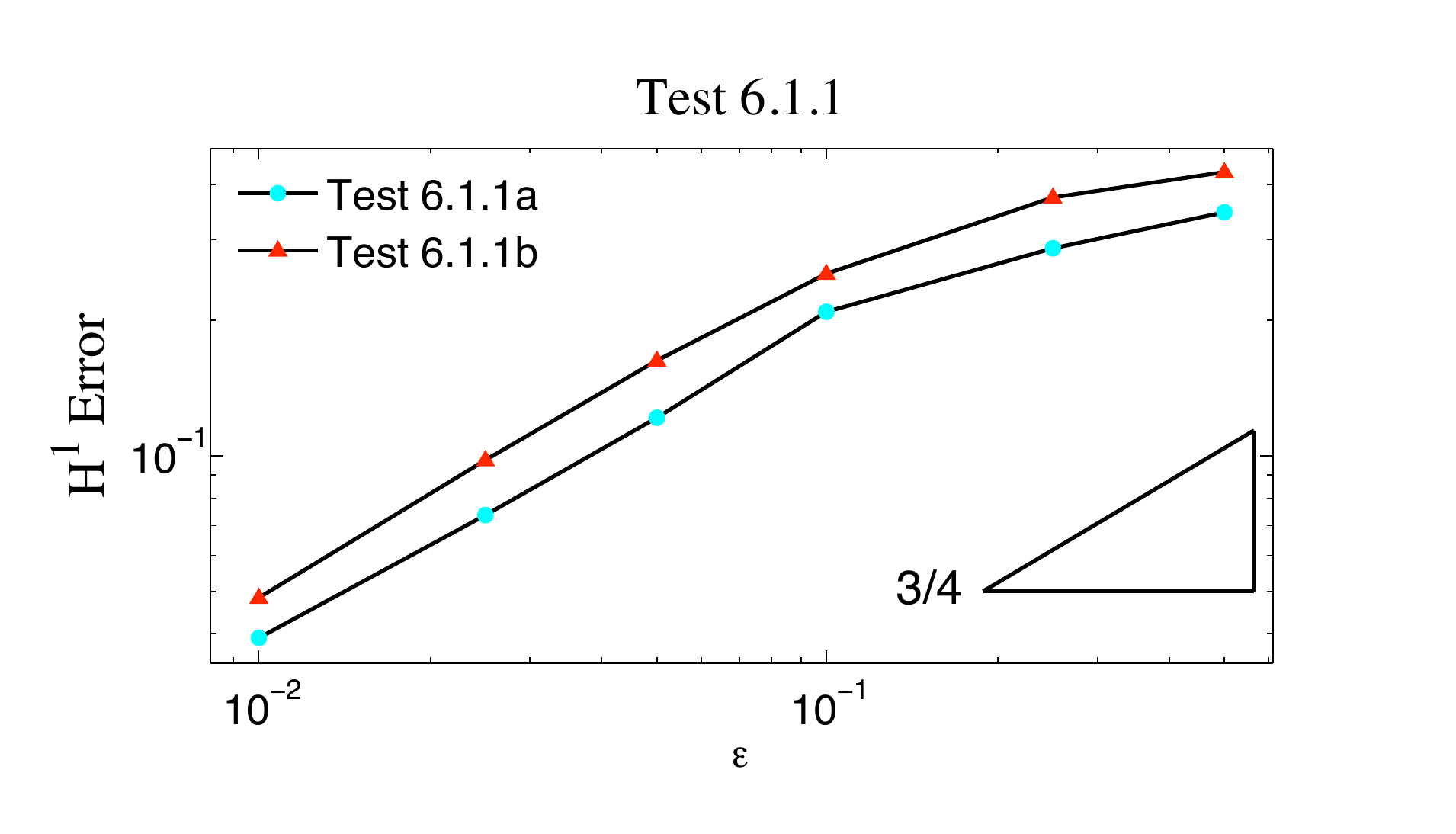} 
\caption{Test 6.1.1.  Error $\|u-\ueh\|_{L^\infty}$ (top) and 
$\|u-\ueh\|_\lt$ (bottom) w.r.t. $\eps$ ($h=0.02$).}
\label{Test611Figure1}
\end{figure}

\begin{figure}[htbp]
\centering
\includegraphics[scale=0.5]{Test611H1err}\\
\includegraphics[scale=0.5]{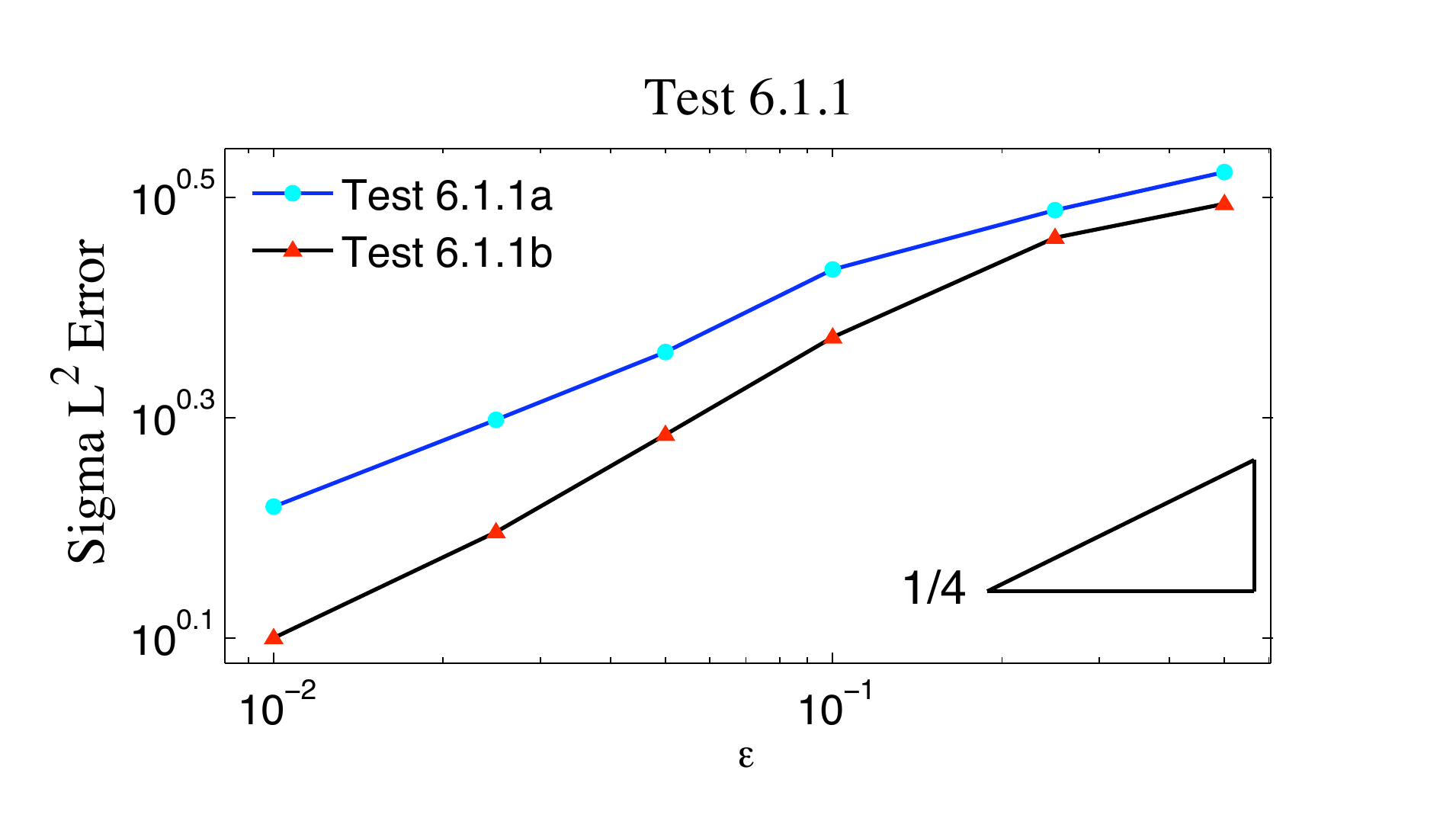} 
\caption{Test 6.1.1.  Error $\|u-\ueh\|_\ho$ (top) and
$\|\sigma-\seh\|_\lt$ (bottom) w.r.t. $\eps$ ($h=0.02$).}
\label{Test611Figure2}
\end{figure}

\subsubsection*{Test 6.1.2}

The purpose of this test is to calculate the rate of convergence of 
$\|\ue-\ueh\|$ for fixed $\eps$=0.001 in various norms.  We solve problem 
\eqref{mixedfemagain1}--\eqref{mixedfemagain2} using the linear
Lagrange element ($k=1$), but with the boundary condition 
$\seh\nu\cdot\nu\big|_{\p\Ome}=\eps$
replaced by $\seh\nu\cdot \nu\big|_{\p\Ome}=\phi^\eps$.  We use 
the following test functions and data:
\begin{alignat*}{3}
&{\rm (a)}\ &&\ue=x^2_1+x_2^2+x_3^2,\qquad 
&&f^\eps=8,\\
& &&g^\eps=x_1^2+x_2^2+x_3^2,\qquad\qquad &&\phi^\eps=2,\\
&{\rm (b)}\ &&\ue=x_1^4+x_2^2+x_3^6,\qquad &&f^\eps=720x_1^2x_3^4-\eps 8640x_3^2,\\
& &&g^\eps=x_1^4+x_2^2+x_3^6,\qquad 
&&\phi^\eps=12x_1^2\nu^2_1+2\nu_2^2+30x_3^4\nu_3^2.
\end{alignat*}

After computing the solution, we list the errors in Table \ref{Test612Table}
and plot the results in Figure \ref{Test612Figure}.
We note that the mixed finite element 
theory in the preceding sections was only developed for $k\ge 2$.  However, 
our numerical experiments also indicate that the method works for the case
$k=1$.  Indeed, the tests indicate the following rates of convergence:
\begin{align*}
\|\ue-\ueh\|_\lt=O\left(h^2\right),\qquad\qquad \|\ue-\ueh\|_\ho=O\left(h\right).
\end{align*}
 
\begin{table}[htbp]
\caption{Test 6.1.2.  Error of $\|\ue-\ueh\|$ w.r.t. $h$ ($\eps$=0.001, linear
Lagrange element)} \label{Test612Table}
\centering
\begin{tabular}{ccccc}
&{\scriptsize$h$} & 
{\scriptsize $\|\ue-\ueh\|_\lt$} &  {\scriptsize $\|\ue-\ueh\|_\ho$} & 
{\scriptsize $\|\se-\seh\|_\lt$}\\
\noalign{\smallskip}\hline\noalign{\smallskip}
Test 6.1.2a
&1.75E--01	&4.65E--02\blank	&2.46E--01\blank	&7.57E--01\blank\\
&1.25E--01	&2.25E--02(2.16)	&1.72E--01(1.07)	&8.75E--01(-0.43)\\
&7.50E--02	&7.95E--03(2.03)	&1.04E--01(0.99)	&8.39E--01(0.08)\\
&6.00E--02	&5.13E--03(1.97)	&8.07E--02(1.12)	&6.61E--01(1.07)\\
&4.00E--02	&1.97E--03(2.35)	&5.28E--02(1.05)	&5.85E--01(0.30)\\
&2.00E--02	&1.13E--03(0.80)	&4.17E--02(0.34)	&5.28E--01(0.15)\\
\noalign{\smallskip}\hline\noalign{\smallskip}
Test 6.1.2b
&1.75E--01	&1.04E--01\blank	&8.72E--01\blank	&3.91E+00\blank\\
&1.25E--01	&5.46E--02(1.92)	&6.80E--01(0.74)	&3.92E+00(-0.01)\\
&7.50E--02	&1.97E--02(1.99)	&4.26E--01(0.92)	&3.75E+00(0.09)\\
&6.00E--02	&1.30E--02(1.85)	&3.40E--01(1.01)	&3.33E+00(0.53)\\
&4.00E--02	&7.57E--03(1.34)	&2.29E--01(0.97)	&3.25E+00(0.06)\\
&2.00E--02	&8.43E--03(-0.16)	&1.85E--01(0.31)	&3.04E+00(0.09)
\end{tabular}
\end{table}

\begin{figure}[htbp]
\centering
\includegraphics[scale=0.5]{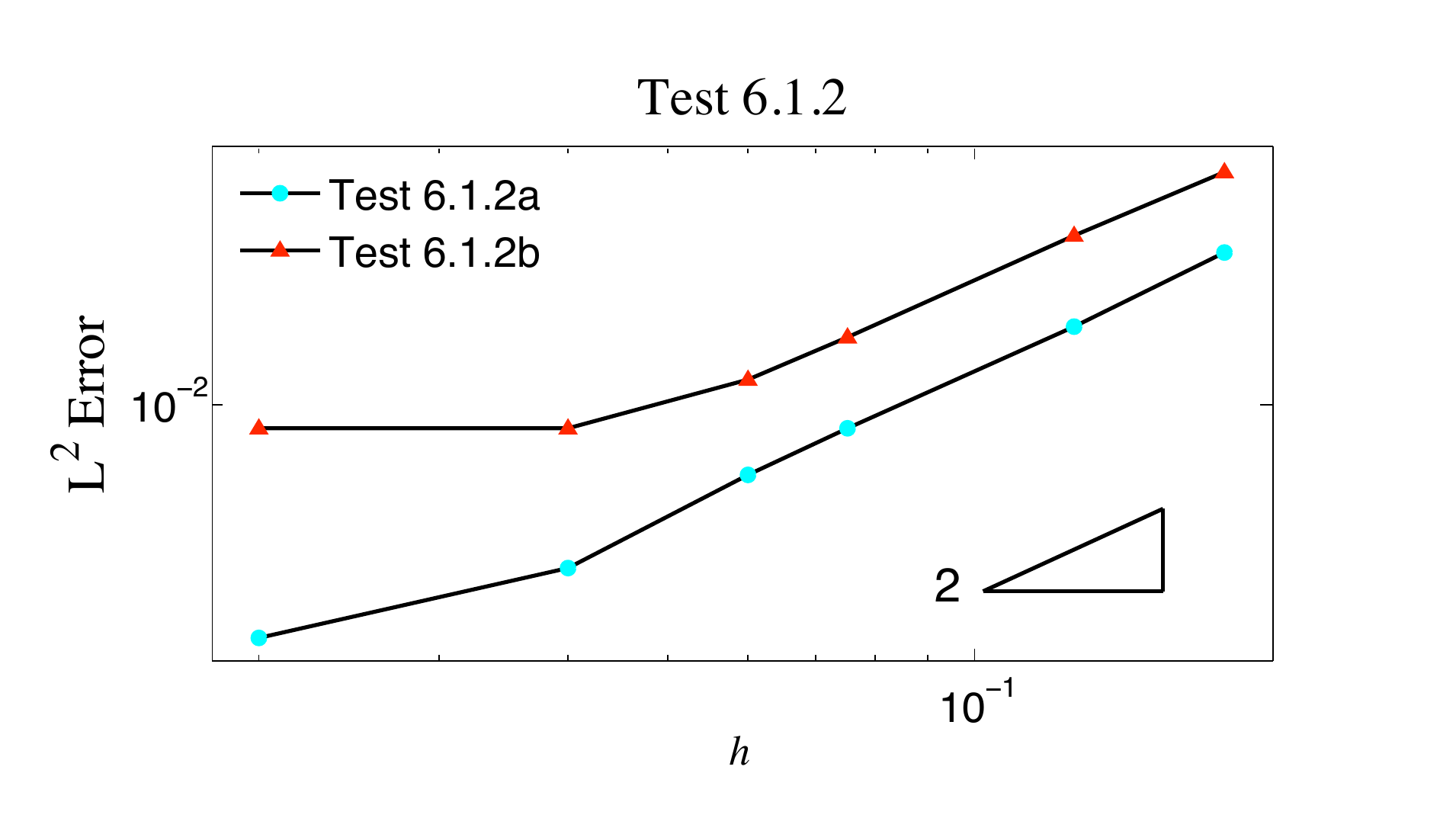}\\
\includegraphics[scale=0.5]{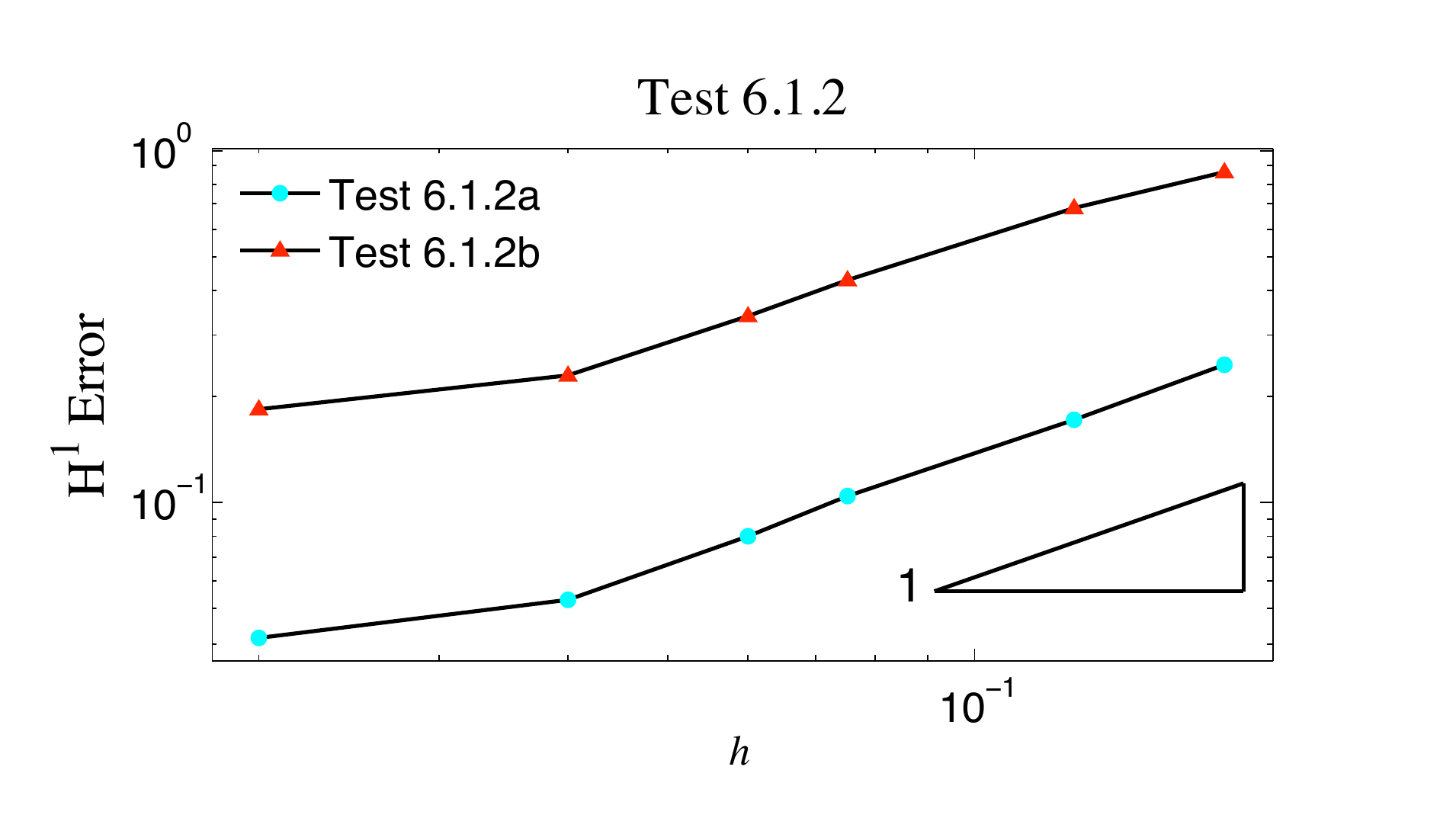}\\ 
\includegraphics[scale=0.5]{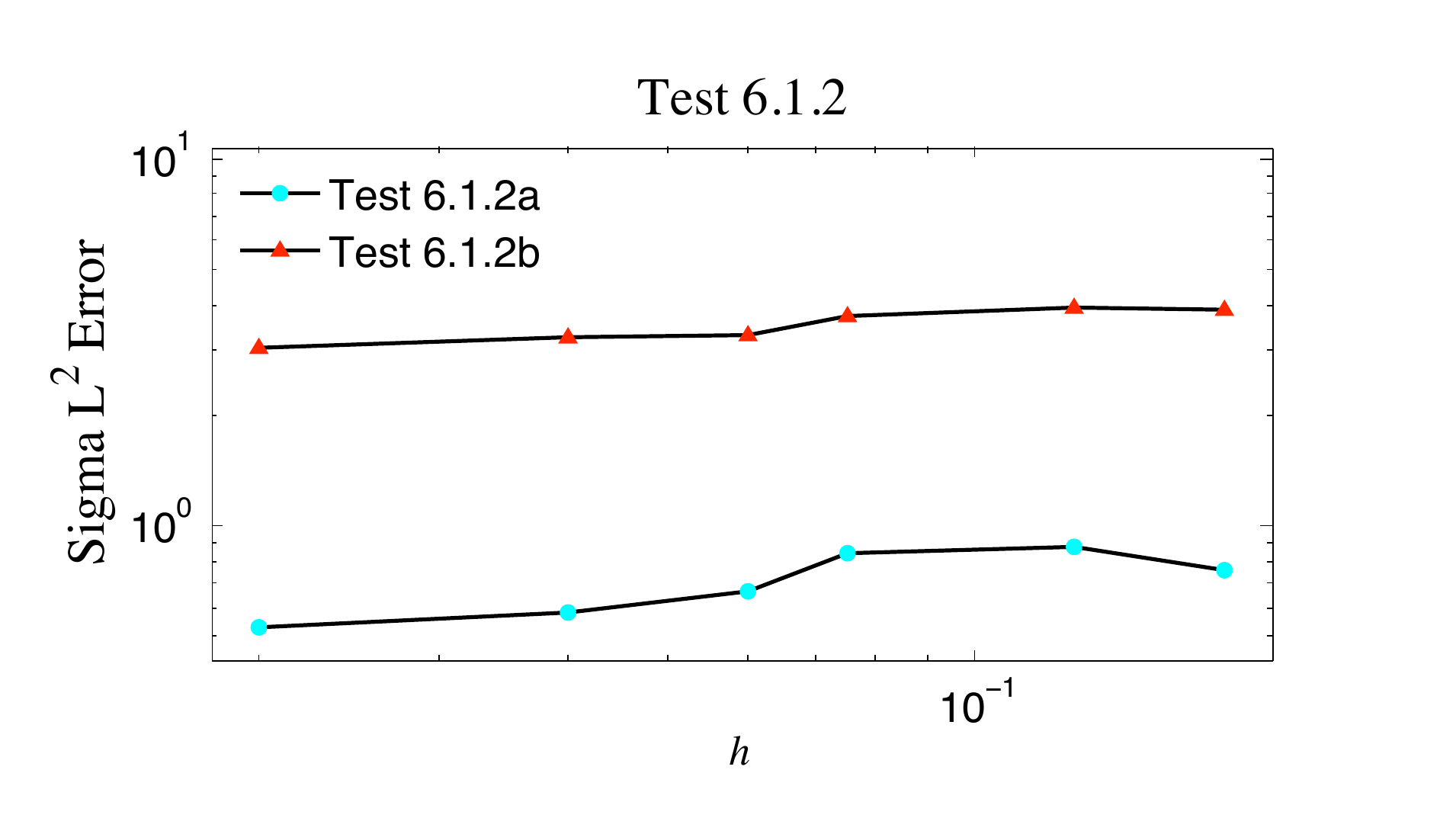} 
\caption{Test 6.1.2.  Error $\|\ue-\ueh\|_{L^2}$ (top), 
$\|\ue-\ueh\|_\ho$ (middle), and $\|\se-\seh\|_\lt$ (bottom)
w.r.t. $h$ ($\eps=0.001$).}
\label{Test612Figure}
\end{figure}

\subsubsection*{Test 6.1.3}

The purpose of this test is to calculate the error
$\|\ue-\ueh\|$ in various norms using a fixed $h-\eps$ 
relation.  We solve the finite element method \eqref{MAagainfem1}
in two dimensions with $V^h$ denoting the 
Argyris finite element space of degree five \cite{Ciarlet78}, 
and replace the boundary condition $\Del \ueh\big|_{\p\Ome}=\eps$ by
 $\Del \ueh\big|_{\p\Ome}=\phi^\eps$.  We use the 
following test function and data:
\begin{alignat*}{2}
&\ue=-\sqrt{r^\eps-(x^2_1+x^2_2)},\quad 
&&f^\eps=\frac{r^\eps}{\(r^\eps-(x_1^2+x_2^2)\)^2}\\
& &&\hspace{0.3in}
-\eps \frac{(x^2_1+x_2^2)(8r^\eps-x_1^2-x_2^2)+8(r^\eps)^2}
{\(r-(x_1^2+x_2^2)\)^\frac72}, \\
&g^\eps =-\sqrt{r^\eps-(x^2_1+x^2_2)},\quad
&&\phi^\eps=\frac{2r^\eps-(x_1^2+x_2^2)}{\(r^\eps-(x^2_1+x_2^2)\)^\frac32},\\
&r^\eps=2+\eps.
\end{alignat*}
On the domain $\Ome=(0,1)^2$, $\ue\in C^\infty(\Ome)$ for any $\eps>0$.  
However, the limiting function 
\[
u:=\lim_{\eps\to 0^+} \ue=-\sqrt{2-(x_1^2+x_2^2)}
\]
is not smooth, and in fact, there only holds $u\in W^{1,p}$ where $p\in [1,4)$
(cf. \cite{Dean_Glowinski_a}).  

We solve \eqref{MAagainfem1} using the following four $h-\eps$ relations:
\begin{alignat*}{2}
&h=2\eps^\frac32,\qquad\qquad &&h=\eps,\\
&h=0.5\eps^\frac12,\qquad\qquad &&h=0.5\eps^\frac14.
\end{alignat*}
We list the errors of the computed solution in Table \ref{Test613Table}
and plot the results in Figures \ref{Test613Figure1}--\ref{Test613Figure2}.

Since $\|\ue\|_\hl\to \infty$ for any $\ell\ge 2$ as $\eps\to 0^+$, we suspect that
a stringent $h-\eps$ relation will be needed in order for the method to converge 
in view of the error estimates \eqref{MAmainthmline1}--\eqref{MAmainthmline2}.
This supposition is verified  by the numerical tests, as the method
does not converge in any norm using the relation $h=0.5\eps^\frac14$.  Furthermore,
we observe that the method does not converge in the $H^2$-norm for any 
$h-\eps$ relations used in the experiments.  This behavior is expected since
the limiting solution $u$ is not in this space.
We plot the error of the computed solution in Figure \ref{Test613Picture}
with parameters $\eps=h=0.04$.  As seen from the picture, the 
error is concentrated at the singularity of $u$.

\begin{table}[htbp]
\caption{Test 6.1.3.  Error of $\|\ue-\ueh\|$ with $h-\eps$ relation}
\label{Test613Table}
\centering
\begin{tabular}{lcccccc}
&{\scriptsize$\eps$} & {\scriptsize $h$} & {\scriptsize$\|\ue-\ueh\|_{L^\infty}$} & 
{\scriptsize $\|\ue-\ueh\|_\lt$} &  {\scriptsize $\|\ue-\ueh\|_\ho$} & 
{\scriptsize $\|\ue-\ueh\|_\lt$}\\
\noalign{\smallskip}\hline\noalign{\smallskip}
$h=2\eps^\frac32$
&2.00E--01	&1.79E--01	&3.94E--02	&2.02E--02	&9.41E--02	&6.32E--01\\
&1.00E--01	&6.32E--02	&4.11E--02	&2.10E--02	&1.01E--01	&7.55E--01\\
&5.00E--02	&2.24E--02	&3.45E--02	&1.76E--02	&8.85E--02	&7.84E--01\\
&4.00E--02	&1.60E--02	&3.12E--02	&1.59E--02	&8.17E--02	&8.17E--01\\
\noalign{\smallskip}\hline\noalign{\smallskip}
$h=\eps$
&2.00E--01	&2.00E--01	&3.96E--02	&2.03E--02	&9.54E--02	&8.79E--01\\
&1.00E--01	&1.00E--01	&4.12E--02	&2.11E--02	&1.02E--01	&1.11E+00\\
&5.00E--02	&5.00E--02	&3.45E--02	&1.76E--02	&8.89E--02	&1.34E+00\\
&4.00E--02	&4.00E--02	&3.13E--02	&1.59E--02	&8.23E--02	&1.73E+00\\
&2.50E--02	&2.50E--02	&2.38E--02	&1.21E--02	&6.58E--02	&1.88E+00\\
&1.25E--02	&1.25E--02	&1.40E--02	&7.11E--03	&4.26E--02	&2.41E+00\\
\noalign{\smallskip}\hline\noalign{\smallskip}
$h=0.5\eps^\frac12$
&2.00E--01	&2.24E--01	&3.95E--02	&2.02E--02	&9.45E--02	&6.70E--01\\
&1.00E--01	&1.58E--01	&4.14E--02	&2.12E--02	&1.03E--01	&1.18E+00\\
&5.00E--02	&1.12E--01	&3.63E--02	&1.84E--02	&9.92E--02	&2.74E+00\\
&4.00E--02	&1.00E--01	&3.33E--02	&1.69E--02	&9.53E--02	&3.31E+00\\
&2.50E--02	&7.91E--02	&2.67E--02	&1.33E--02	&8.64E--02	&4.63E+00\\
&1.25E--02	&5.59E--02	&1.90E--02	&8.14E--03	&7.25E--02	&6.39E+00\\
&6.25E--03	&3.95E--02	&1.96E--02	&4.47E--03	&6.91E--02	&1.10E+01\\
\noalign{\smallskip}\hline\noalign{\smallskip}
$h=0.5\eps^\frac14$
&2.00E--01	&3.34E--01	&4.04E--02	&2.08E--02	&1.02E--01	&1.18E+00\\
&1.00E--01	&2.81E--01	&4.32E--02	&2.21E--02	&1.14E--01	&1.62E+00\\
&5.00E--02	&2.36E--01	&4.17E--02	&2.09E--02	&1.26E--01	&2.79E+00\\
&4.00E--02	&2.24E--01	&4.40E--02	&2.14E--02	&1.42E--01	&3.50E+00\\
&2.50E--02	&1.99E--01	&5.89E--02	&2.49E--02	&1.96E--01	&6.03E+00\\
&1.25E--02	&1.67E--01	&6.15E--02	&2.10E--02	&2.03E--01	&7.33E+00\\
\end{tabular}
\end{table}

\begin{figure}[htbp]
\centering
\includegraphics[scale=0.2]{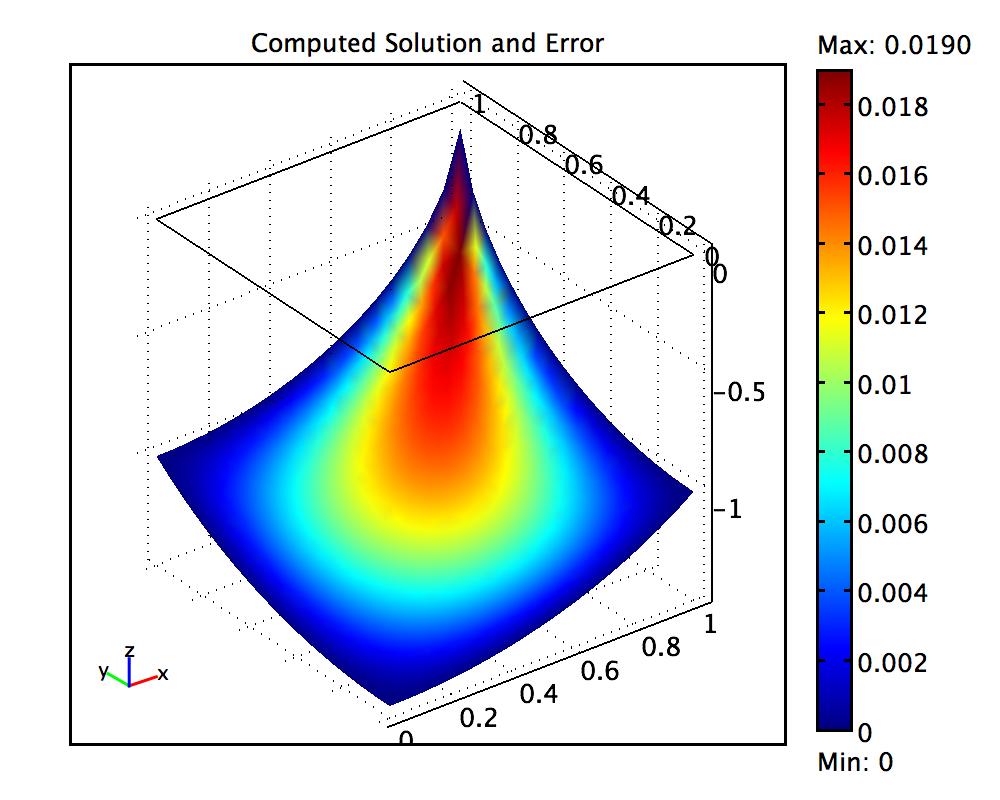}\\
\caption{Test 6.1.3.  Error of computed solution with $\eps=0.04$, $h=0.04$.}
\label{Test613Picture}
\end{figure}

\begin{figure}[htbp]
\centering
\includegraphics[scale=0.5]{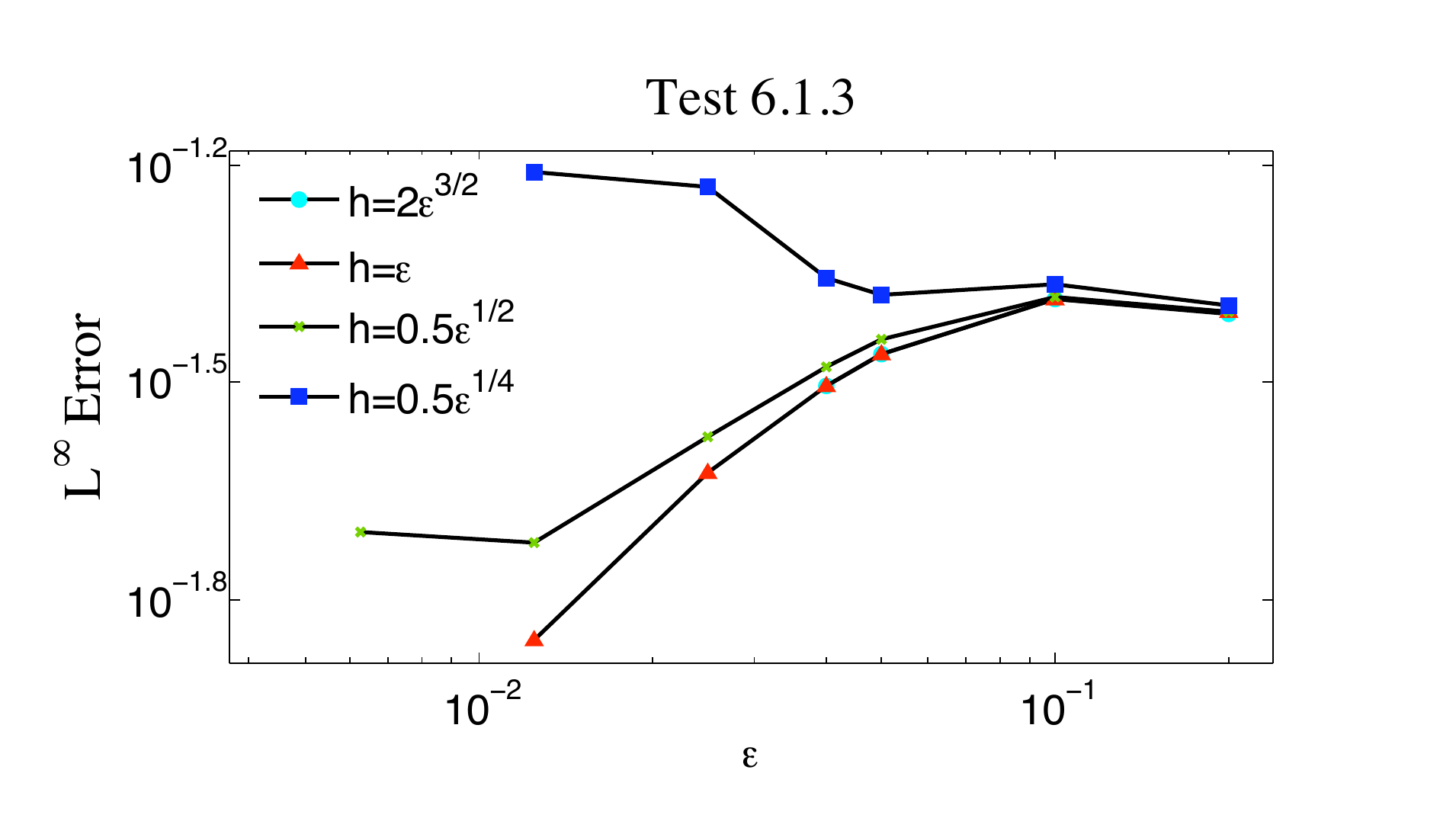}\\
\vspace*{0.5cm}
\includegraphics[scale=0.5]{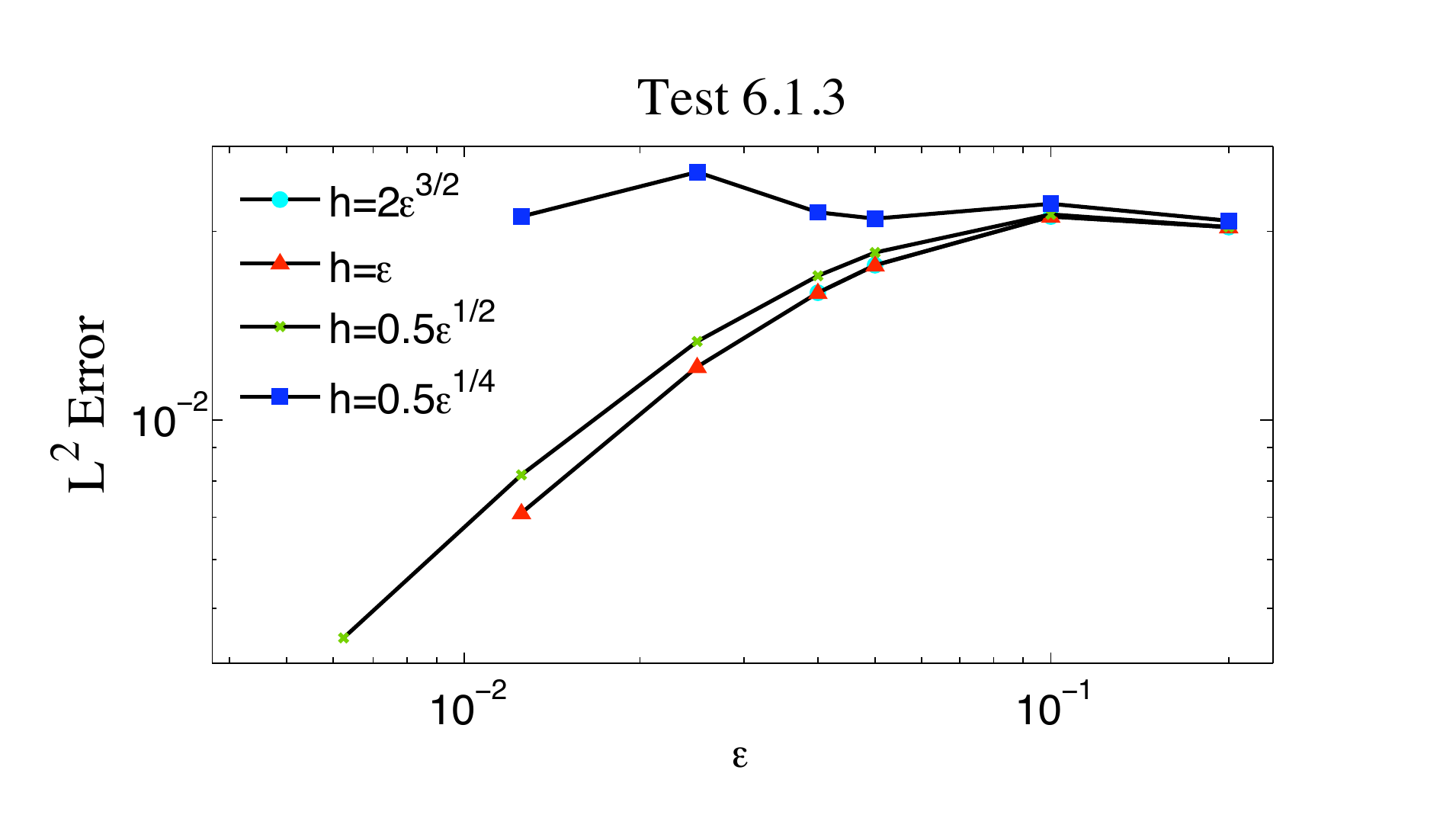} 
\caption{Test 6.1.3.  Error $\|\ue-\ueh\|_{L^\infty}$ (top) and 
$\|\ue-\ueh\|_\lt$ (bottom) with various $h-\eps$ relations.}
\label{Test613Figure1}
\end{figure}

\begin{figure}[htbp]
\centering
\includegraphics[scale=0.5]{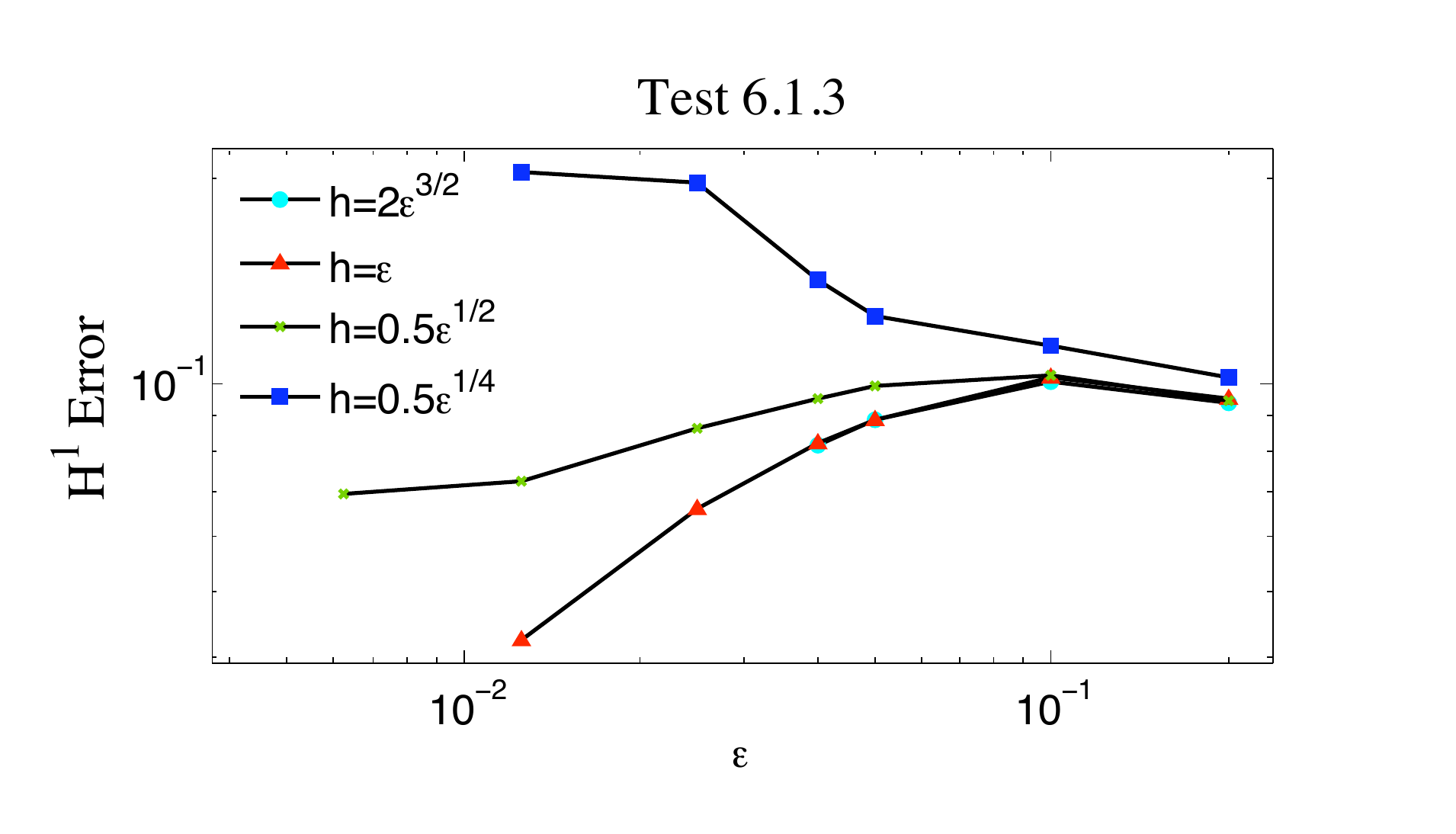}\\
\vspace*{0.5cm}
\includegraphics[scale=0.5]{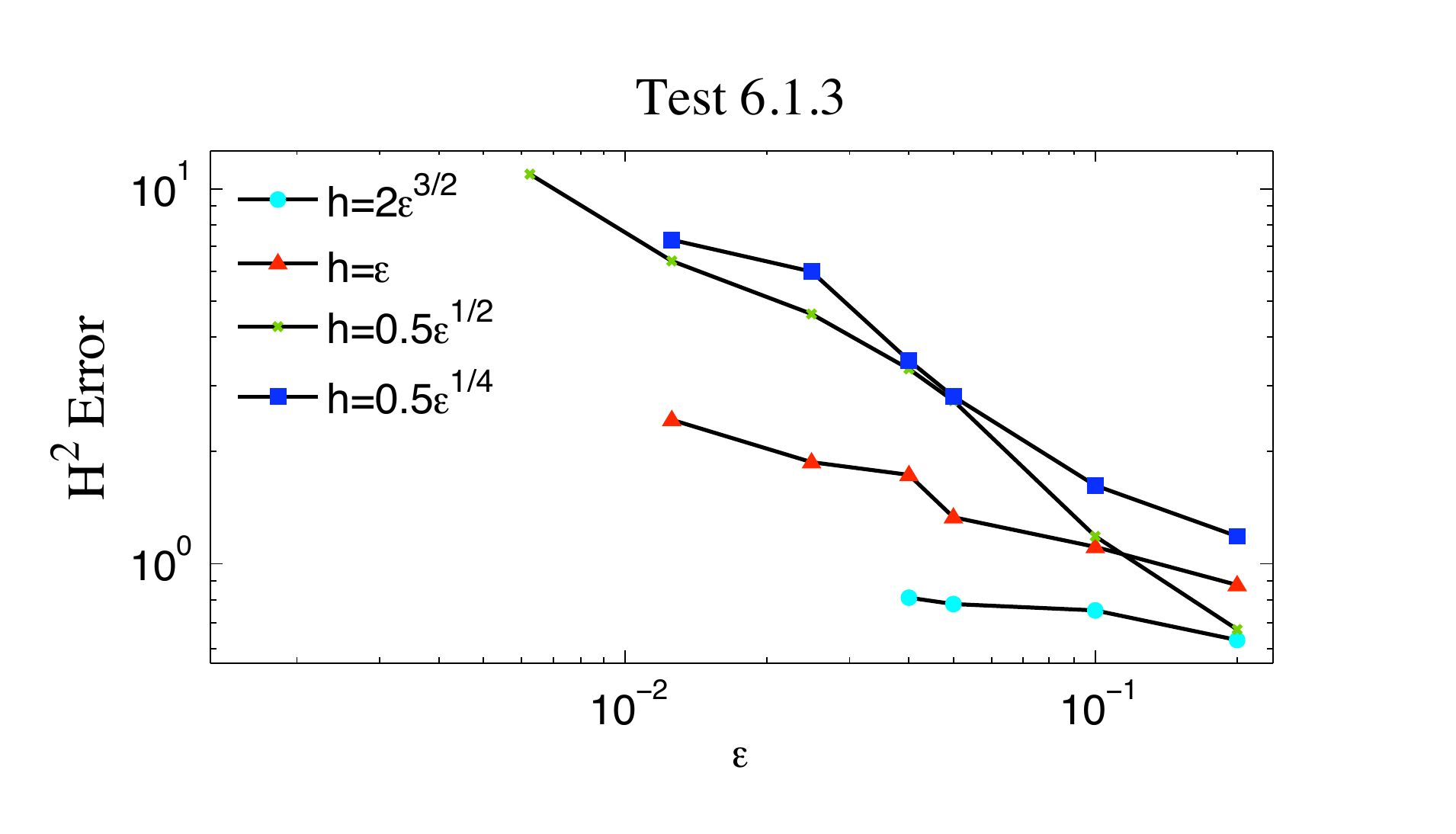} 
\caption{Test 6.1.3.  Error $\|\ue-\ueh\|_\htw$ (top) and 
$\|\ue-\ueh\|_\htw$ (bottom) with various $h-\eps$ relations.}
\label{Test613Figure2}
\end{figure}

\subsection*{Test 6.1.4}
For our last test, we numerically back up the theoretical results given 
in Chapter \ref{chapter-3}, that is, we compute the vanishing moment 
approximation \eqref{moment1}--\eqref{moment3}$_1$ in the radial 
symmetric case. To this end, we solve 
\eqref{moment1_radial}--\eqref{moment4_radial}  in the domain $\Ome=(0,1)$.
We use the Hermite cubic finite element to construct our finite 
element space, and we use the following data:
\begin{align*}
f = (1+r^2)e^{n r^2/2},\qquad g(1) = e^{\frac12}.
\end{align*}
It can be readily checked that the exact solution is $u = e^{r^2/2}$.

We plot the computed solution and corresponding error 
in Figure \ref{Test614Figure1} with parameters
$n=4,\ \eps = 10^{-1}, h=4.0\times 10^{-3}$. We also plot the computed 
Laplacian, $\Del \ue:=\ue_{rr}+\frac{2}{r}\ue_r$, as well.
As shown by the pictures, the vanishing moment methodology accurately 
captures the convex solution in higher dimensions. Also, as expected, 
the Laplacian of $\ue$ is strictly positive 
(cf. Theorem \ref{convexity_thm1}).

Next, we plot both $u_r$ and $u_{rr}$ in two and four dimensions 
in Figures \ref{Test614Figure2}--\ref{Test614Figure3}
with $\eps$-values, $10^{-1}, 10^{-3}, 10^{-5}$.  Recall that the Hessian 
matrix of $\ue$ only has two distinct eigenvalues
$\ue_{rr}$ and $\frac1r\ue_{r}$.
As seen in Figure \ref{Test614Figure2}, $\ue_r$ is positive 
for all $\eps$-values and for both dimensions $n=2$ and $n=4$.
This result is in accordance with Corollary \ref{existence_cor}.  
Finally, Figure \ref{Test614Figure3} shows that
$\ue_{rr}$ is strictly positive except for a small $\eps$-neighborhood 
of the boundary, which agrees with
the theoretical results established in Theorem \ref{convexity_thm}.

\begin{figure}[htbp]
\centering
\includegraphics[scale=0.15]{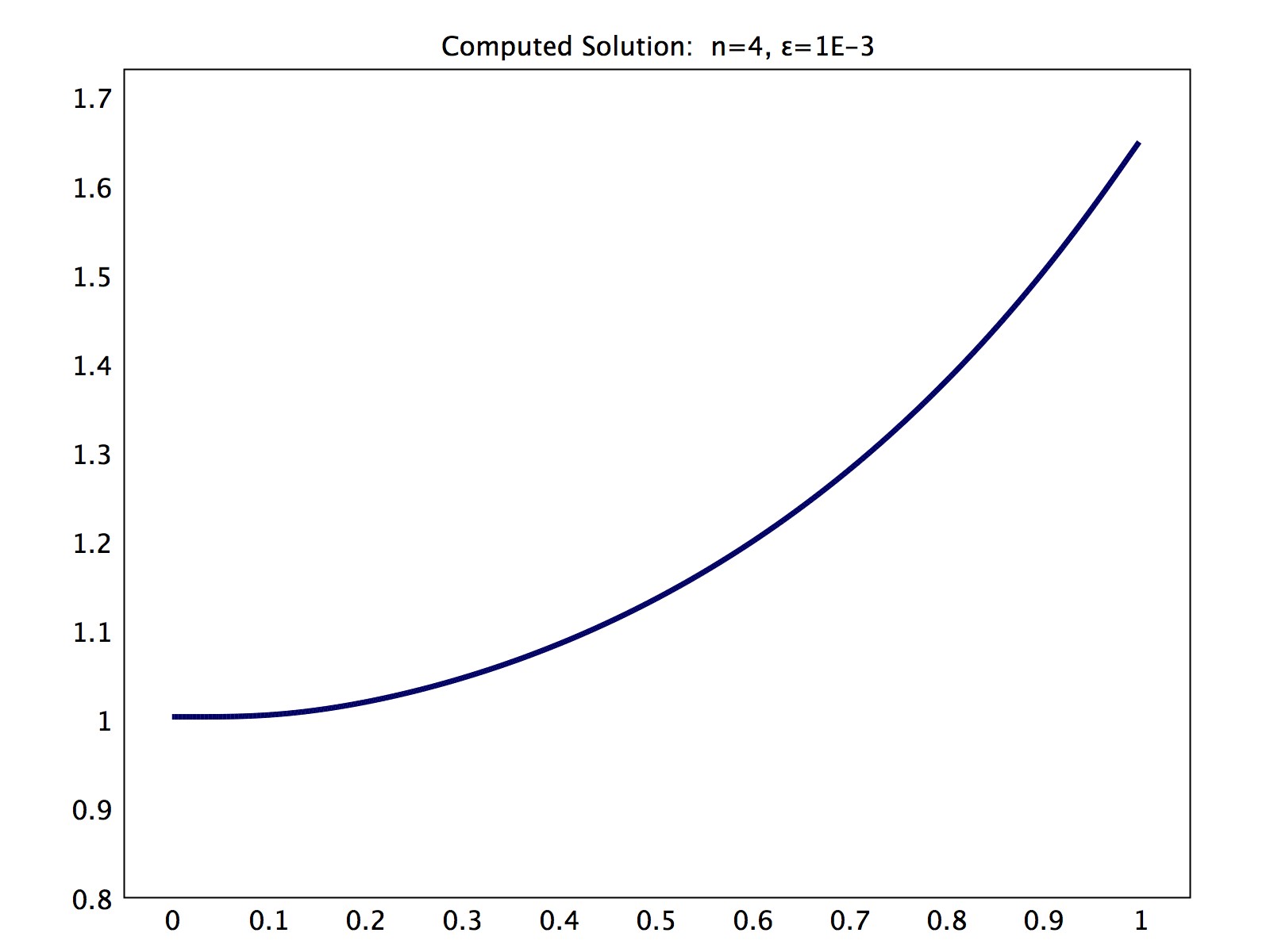}\\
\includegraphics[scale=0.15]{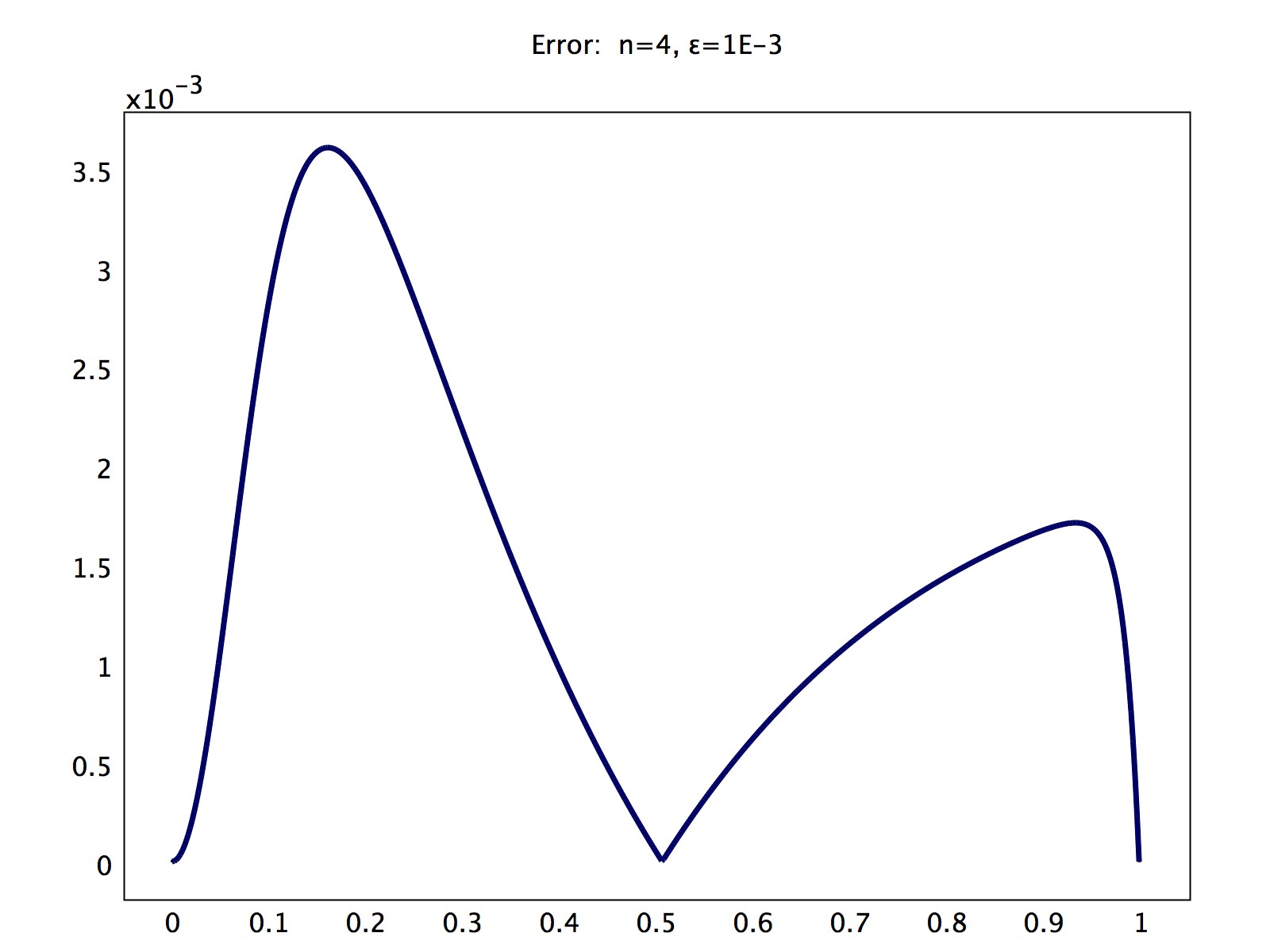}\\
\includegraphics[scale=0.15]{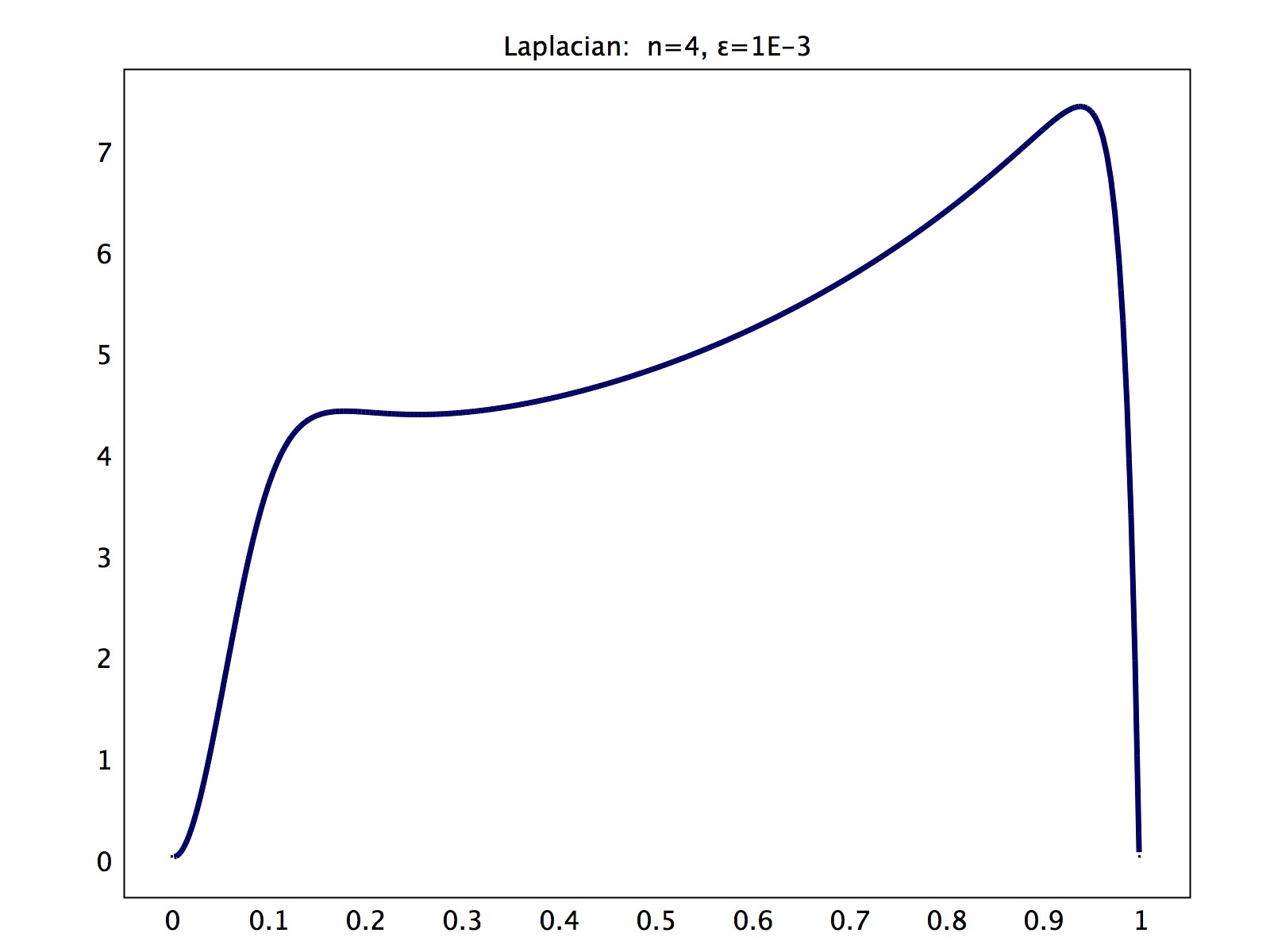} 
\caption{Test 6.1.4.  Computed solution of 
\eqref{moment1_radial}--\eqref{moment3_radial} (top), error (middle),
and computed Laplacian (bottom) with $n=4, \eps = 10^{-1}, h=4\times 10^{-3}$.}
\label{Test614Figure1}
\end{figure}

\begin{figure}[htbp]
\centering
\includegraphics[scale=0.15]{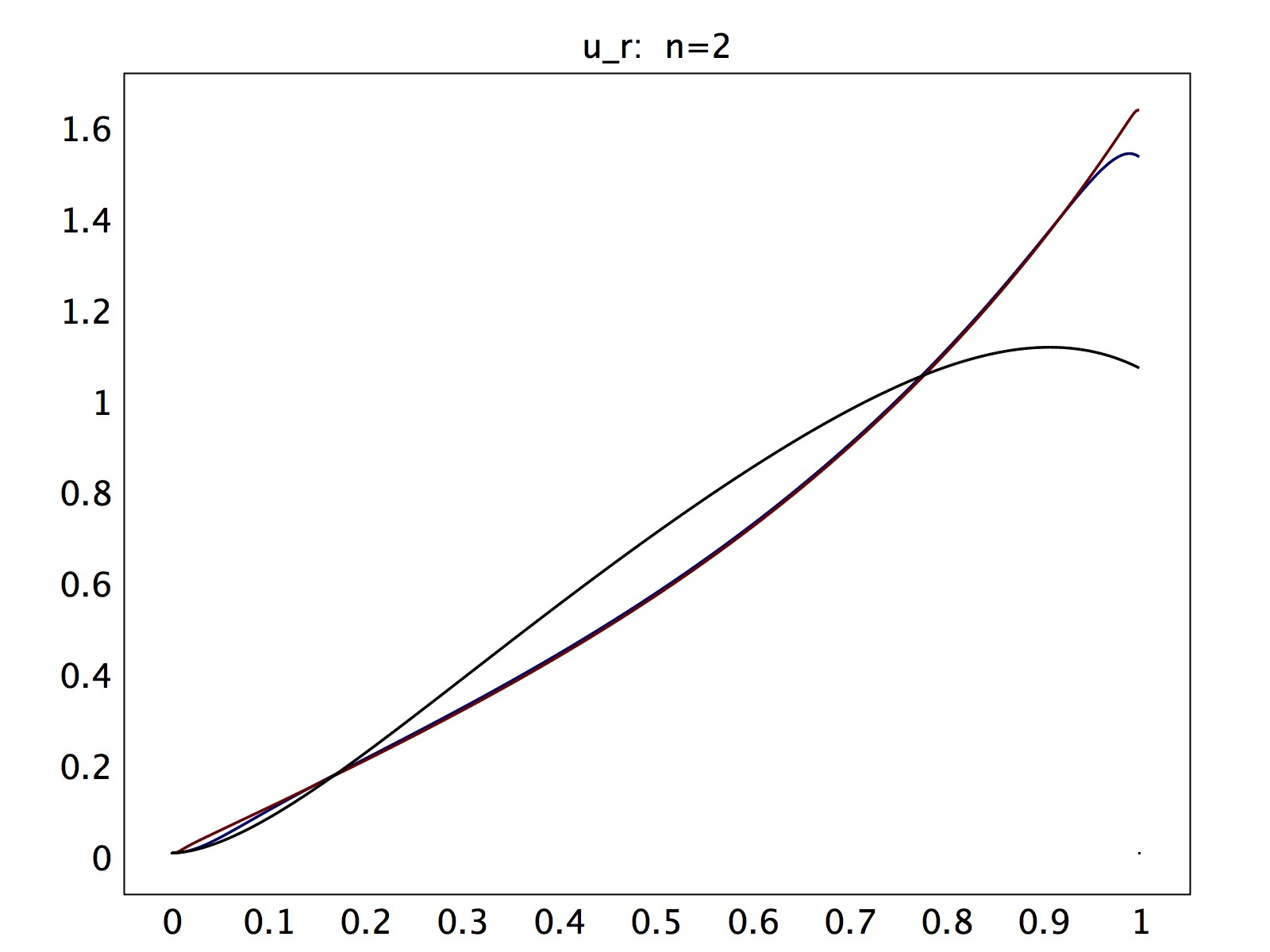}\\
\includegraphics[scale=0.15]{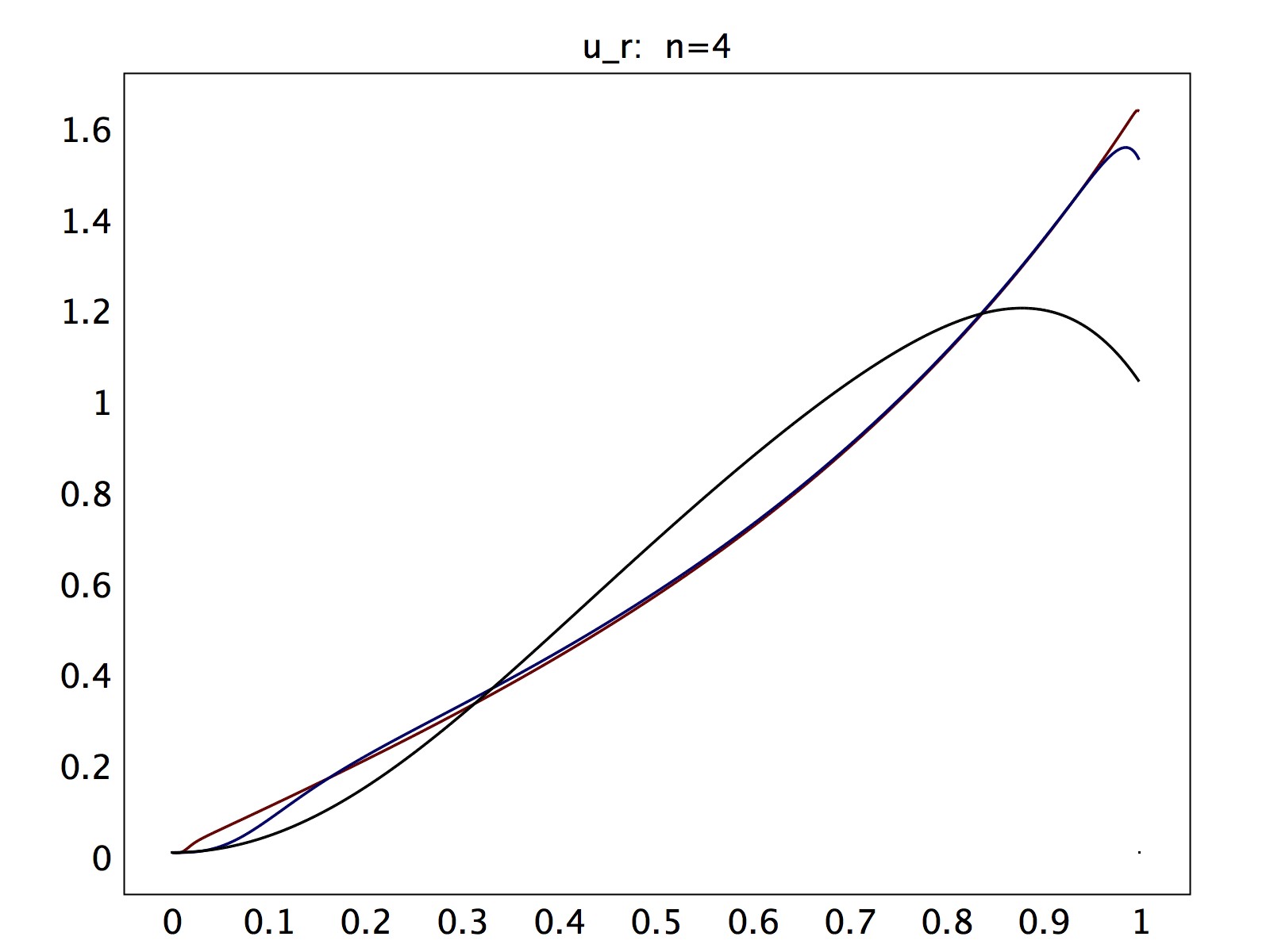}
\caption{Test 6.1.4.  Computed $u_{r}$ of 
\eqref{moment1_radial}--\eqref{moment3_radial} for $n=2$ (top), and
$n=4$ (bottom) with $\eps=10^{-1}$ (black), $\eps=10^{-3}$ (blue), and 
$\eps = 10^{-5}$ (red) ($h=4\times 10^{-3}$).}
\label{Test614Figure2}
\end{figure}

\begin{figure}[htbp]
\centering
\includegraphics[scale=0.15]{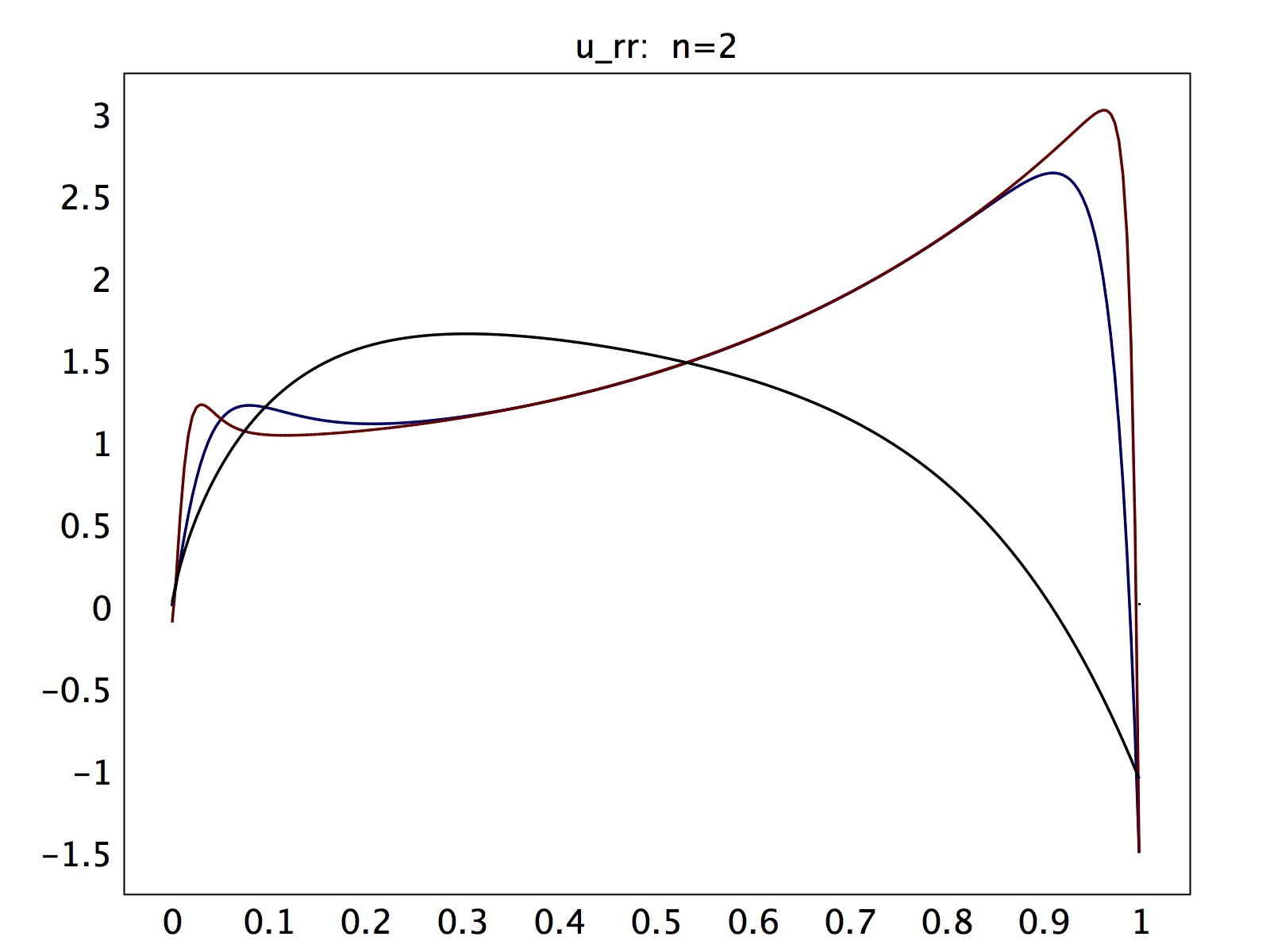}\\
\includegraphics[scale=0.15]{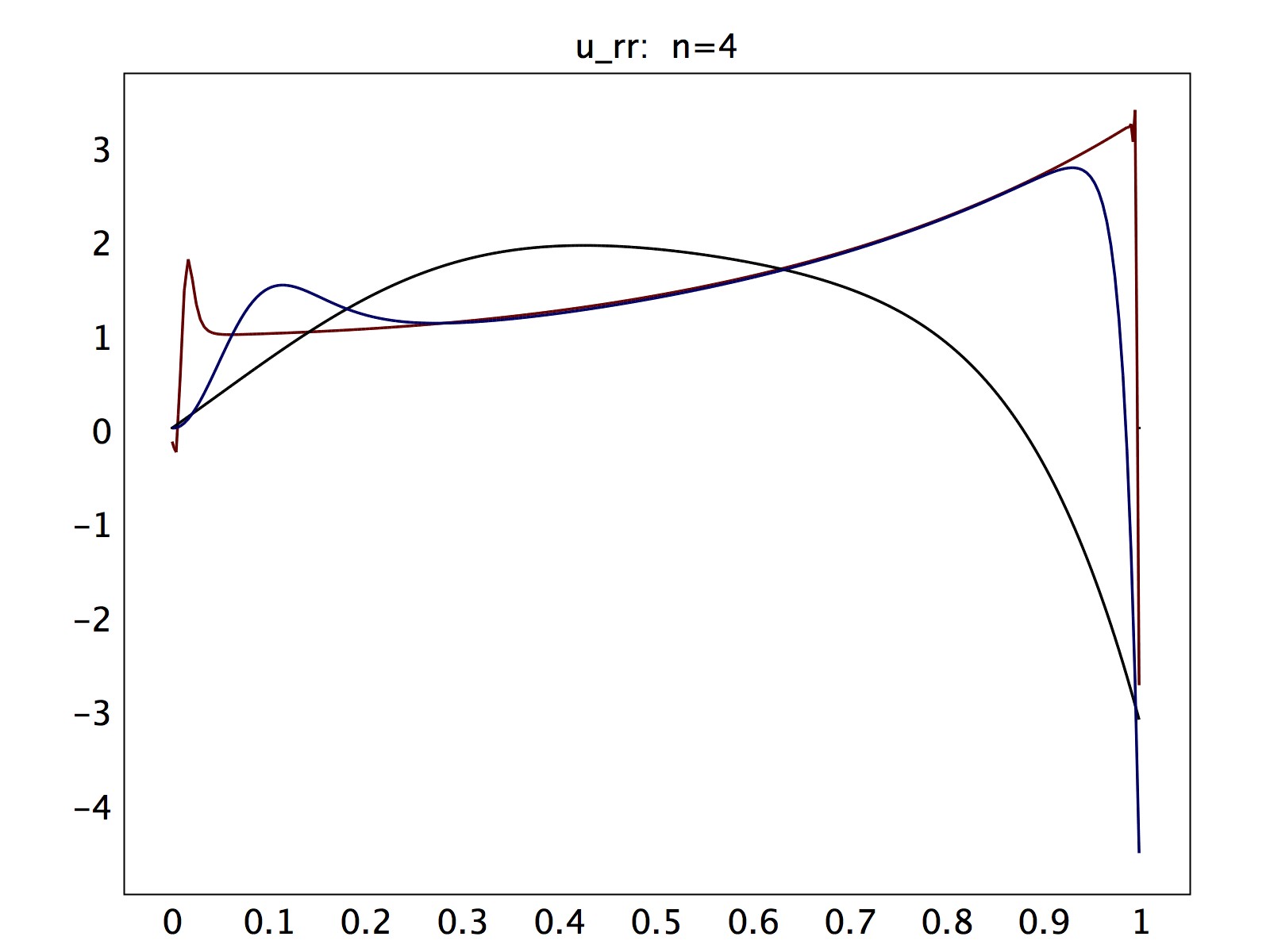}
\caption{Test 6.1.4. Computed $u_{rr}$ of 
\eqref{moment1_radial}--\eqref{moment3_radial} for $n=2$ (top), and
$n=4$ (bottom) with $\eps=10^{-1}$ (black), $\eps=10^{-3}$ (blue), 
and $\eps = 10^{-5}$ (red) ($h=4\times 10^{-3}$).}
\label{Test614Figure3}
\end{figure}

%% file: chapter6b.tex
\section{The equation of prescribed Gauss curvature}\label{chapter-6-sec-3}
Let $\Ome\subset \mathbf{R}^n$ be a bounded domain and
$g\in C^0(\p\Ome)$. For a given constant $\mck>0$, the simplest
version of the famous Minkowski problem (cf. \cite{Guan95,
Gilbarg_Trudinger01}) asks
to find a function $u$ whose graph (as a manifold) has
the constant Gauss curvature $\mck$ and $u$ takes the boundary
value $g$ on $\p\Ome$.  The Gauss curvature of the
graph of $u$ is given by the formula
\begin{align*}
\frac{\det(D^2 u)}{\(1+|\nab u|^2\)^{\frac{n+2}{2}}},
\end{align*}
and therefore, if such a function exists, it must satisfy
\begin{alignat}{2}
\label{absGauss1}\det(D^2u)&=\mck\(1+|\nab u|^2\)^{\frac{n+2}2}
\qquad &&\text{in }\Ome,\\
\label{absGauss2}u&=g\qquad &&\text{on }\p\Ome.
\end{alignat}
The equation \eqref{absGauss1}, which is called {\em
the equation of prescribed Gauss curvature},
is a fully nonlinear Monge-Amp\`ere-type equation.

It is known \cite{Guan95} that there exists a constant $\mck^*>0$ 
such that for each $\mck\in [0,\mck^*)$, 
problem \eqref{absGauss1}--\eqref{absGauss2} has a unique
convex viscosity solution.  
Theoretically, it is very difficult to give an accurate estimate
for the upper bound $\mck^*$.  
This then calls for help from accurate numerical methods.
Indeed, the methodology and analysis of the vanishing moment
method works very well for solving this problem and for estimating $\mck^*$.

Unlike the Monge-Amp\`ere equation considered in the previous section,
we have some leeway in defining $F(D^2 u,\nab u,u,x)$.  For reasons
that will be evident later (cf.~Remark \ref{Gaussaltremark}), we set
\begin{align}\label{GaussdefF}
F(D^2u,\nab u,u,x)&=-\frac{\det(D^2u)}{(1+|\nab u|^2)^{\frac{n+2}2}} +\mck,
\end{align}
and therefore,
\begin{align*}
\Fp[v](w)&=-\frac{\cof(D^2 v):D^2 w}{(1+|\nab v|^2)^{\frac{n+2}2}}
+(n+2)\frac{\det(D^2 v)\nab v\cdot \nab w}{(1+|\nab v|^2)^{\frac{n+4}2}},\\
\Fp[\mu,v](\kappa,w)&=-\frac{\cof(\mu):\kappa }{(1+|\nab v|^2)^{\frac{n+2}2}}
+(n+2)\frac{\det(\mu)\nab v\cdot \nab w}{(1+|\nab v|^2)^{\frac{n+4}2}}.
\end{align*}  
Therefore, the vanishing moment approximation 
\eqref{moment1}--\eqref{moment3}$_1$ is
\begin{alignat}{3}
\label{absGaussMM1}-\eps \Del^2 \ue
+\frac{\det(D^2\ue)}{(1+|\nab \ue|^2)^{\frac{n+2}2}}
&=\mck\qquad &&\text{in }\Ome,\\
\label{absGaussMM2}\ue&=g \qquad &&\text{on }\p\Ome,\\
\label{absGaussMM3}\Del \ue&=\eps\qquad &&\text{on }\p\Ome,
\end{alignat}
and the linearization of
\[
G_\eps(\ue)=\eps\Del^2\ue -\frac{\det(D^2\ue)}{(1+|\nab \ue|^2)^{\frac{n+2}2}} +\mck
\]
at the solution $\ue$ is
\[
\Gp[\ue](v)=\eps\Del ^2 v-\frac{\Phi^\eps:D^2v}{(1+|\nab \ue|^2)^{\frac{n+2}2}}+(n+2)\frac{\det(D^2 \ue)\nab \ue\cdot \nab v}{(1+|\nab \ue|^2)^{\frac{n+4}2}},
\]
where $\Phi^\eps$ denotes the cofactor matrix of $D^2\ue$.

Numerical tests indicate that there exists a unique strictly convex
solution to \eqref{absGaussMM1}--\eqref{absGaussMM3} with $\eps>0$
(cf.~Subsection \ref{chapter-6-sec-2-subsec3}, and \cite{Feng2,Neilan_thesis}).
For the continuation of this section, we assume
that there exists a unique strictly convex
solution to \eqref{absGaussMM1}--\eqref{absGaussMM3}.  Furthermore,
since the high-order terms in the equation of prescribed Gauss curvature
are the same as the Monge-Amp\`ere equation,
we expect that the a priori bounds \eqref{absbounds}--\eqref{absboundsinterp}
hold for the solution $\ue$ of the vanishing moment approximation
\eqref{absGaussMM1}--\eqref{absGaussMM3}.

Before stating the finite element methods for 
\eqref{absGaussMM1}--\eqref{absGaussMM3}
and applying the analysis of Chapters \ref{chapter-4} and \ref{chapter-5} 
to these methods, we first prove the following identity.
\begin{lem}\label{Gaussidentlem}
For all $v,w\in H^1_0(\Ome)$
\begin{align}\label{Gaussidentity}
\bl \Fp[\ue](v),w\br 
=\left(\frac{\Phi^\eps \nab v}{(1+|\nab \ue|^2)^{\frac{n+2}2}},\nab w\right).
\end{align}
\end{lem}
\begin{proof}
Integrating by parts, we have
\begin{align*}
\bl \Fp[\ue](v),w\br 
&=\left(\frac{\Phi^\eps \nab v}{(1+|\nab \ue|^2)^{\frac{n+2}2}},\nab w\right)
-\frac{n+2}2 \left(\frac{\Phi^\eps \nab v\cdot \nab (|\nab \ue|^2)}{(1+|\nab \ue|^2)^{\frac{n+4}2}},w\right)\\
&\qquad +(n+2)\left(\frac{\det(D^2\ue)\nab \ue\cdot \nab v}{(1+|\nab \ue|^2)^{\frac{n+4}2}},w\right).
\end{align*}
Noting that $\Phi^\eps D^2 \ue= \det(D^2\ue) I_{n\times n}$, we conclude
\[
\Phi^\eps \nab v\cdot \nab (|\nab \ue|^2)
=\Phi^\eps \nab v\cdot 2 D^2\ue\nab \ue = 2\det(D^2\ue) (\nab \ue\cdot \nab v).
\]
From this identity, \eqref{Gaussidentity} immediately follows.
\end{proof}
Since $\ue$ is strictly convex, we arrive at the following corollary.
\begin{cor}\label{Gausscoerclem}
There exists a constant $C>0$ such that 
\begin{align}\label{Gaussiscoercive}
\bl \Fp[\ue](w),w\br \ge C\|w\|_\ho^2\qquad \forall w\in H^1_0(\Ome).
\end{align}
\end{cor}

\begin{remark}\label{Gaussaltremark}
It is now obvious why we choose \eqref{GaussdefF}
as the definition of $F$ opposed to the following choice:
\begin{align}\label{GaussaltF}
F(D^2u,\nab u,u,x)=-\det(D^2u)+\mck (1+|\nab u|^2)^{\frac{n+2}2}.
\end{align}
Indeed, if we chose \eqref{GaussaltF} instead of \eqref{GaussdefF} then
\begin{align*}
\Fp[\ue](w)=-\Phi^\eps:D^2 w + \mck(n+2) (1+|\nab \ue|^2)^{\frac{n}2}\nab \ue\cdot \nab w,
\end{align*}
and a simple calculation shows
\begin{align*}
\bl \Fp[\ue](w),w\br = (\Phi^\eps \nab w,\nab w)
& -\frac{\mck (n+2)}{2}\Bigl((1+|\nab \ue|^2)^\frac{n}2 \Del \ue \\
& +n (1+|\nab \ue|^2)^{\frac{n-2}2} \widetilde{\Del}_\infty \ue,w^2 \Bigr)
\quad\forall w\in H^1_0(\Ome),
\end{align*}
where $\widetilde{\Del}_\infty \ue:=D^2\ue \nab \ue\cdot \nab \ue$.\footnote{
Throughout this chapter we define $\widetilde{\Del}_\infty v:=D^2 v \nab 
v\cdot \nab v$ and $\Del_\infty v
:=\frac{D^2 v \nab v\cdot \nab v}{|\nab v|^2}$.
We note that both operators are referred to the infinity-Laplacian
in the literature \cite{Aronsson_Crandall_Juutinen04,Evans07}.
We shall use the latter definition in Section \ref{chapter-6-sec-4}.}
Since $\ue$ is strictly convex, both $\Del \ue$ 
and $\widetilde{\Del}_\infty \ue$ 
are positive terms, and therefore, 
the linearization of this choice of $F$ is not coercive. 

Nevertheless, the above choice is also valid since it is easy 
to check that $\Fp[\ue]$ satisfies the following G\"arding inequality
\begin{align*}
\bl \Fp[\ue](w),w\br \ge \wt{C}_1 \|w\|_\ho^2-\wt{C}_0\|w\|_\lt^2
\qquad \forall w\in H^1_0(\Ome),
\end{align*}
for some positive constants $\wt{C}_0=\wt{C}_0(\eps),\ \wt{C}_1=\wt{C}_1(\eps)$.
Here $\bl \cdot,\cdot\br $ denotes the dual pairing between $H^1_0(\Ome)$ 
and $H^{-1}(\Ome)$. In addition, other conditions of Assumption (A) 
also can be verified.  We refer the reader to \cite{Neilan_thesis} for 
a detailed exposition. 
\end{remark}

\subsection{Conforming finite element methods for the equation of 
prescribed Gauss curvature} \label{chapter-6-sec-2-subsec1}

The finite element method for \eqref{absGaussMM1}--\eqref{absGaussMM3}
is to find $\ueh\in V^h_g$ such that for any $v_h\in V^h_0$
\begin{align} \label{Gauss1eqn}
-\eps(\Del \ueh,\Del v_h)
+\left(\frac{\det(D^2\ueh)}{(1+|\nab \ueh|^2)^{\frac{n+2}2}},v_h\right)
=\mck \bigl(1,v_h\bigr)-\left\langle \eps^2,\normd{v_h}\right\rangle_{\p\Ome}.
\end{align}

The goal of this section is to apply the abstract framework of Chapter
\ref{chapter-4} toward the finite element method \eqref{Gauss1eqn}.
Specifically, our goal is to show that assumptions {\rm [A1]--[A5]}
hold, and as a consequence, we will obtain existence and uniqueness
of a solution to \eqref{Gauss1eqn}, as well as optimal order
estimates for the error $u^\eps-\ueh$.  We also pay close attention
on the constants $C_i$ and $L(h)$ and how they depend on 
the parameter $\eps$.  We summarize our results in the following theorem.

\begin{thm}\label{Gaussmainthm}
Let $\ue\in H^s(\Ome)$ be the solution to 
\eqref{absGaussMM1}--\eqref{absGaussMM3}
with $s\ge 3$ when $n=2$ and $s>3$ when $n=3$.  
Then for $h\le h_2$, there exists 
a unique solution to \eqref{Gauss1eqn}.  Furthermore, 
there exists positive constants $C_{7},\ C_{8}$
such that
\begin{align} \label{Gaussmainthmline1}
\|\ue-\ueh\|_\htw&\le C_{7}h^{\ell-2}\|\ue\|_{H^\ell},\\
\label{Gaussmainthmline2}
\|\ue-\ueh\|_\lt&\le C_{8}\Big(\eps^{-\frac12}h^{\ell}\|\ue\|_{H^\ell}
+C_7L(h)h^{2\ell-4}\|\ue\|_{H^\ell}^2\Big),
\end{align}
where
\begin{alignat*}{2}
&C_7=O\bigl(\eps^{\frac12(1-2n)}\bigr),
\qquad  &&C_{8}=O\bigl(C_7\eps^{-(n+2)}\bigr),\\
&L(h)=C\bigl(\eps^{-\frac16(n+2)}+h^{\frac32(2-n)}\bigr),
\qquad &&\ell={\rm min}\{s,k+1\},
\end{alignat*}
and $h_2$ is chosen such that
\begin{align*}
h_2\le 
C\left(\eps^{-\frac12(1+2n)}\|\ue\|_\hl L(h_2)\right)^{\frac{1}{2-\ell}}.
\end{align*}
\end{thm}

\begin{proof}
First, 
 \eqref{Gaussiscoercive} implies
that 
\begin{align} \label{GaussC0C1def}
\bl \Gp[\ue](v),v\br\ge C \eps \|v\|_\htw^2\qquad \forall v\in V_0,
\end{align}
and it follows that $\(\Gp[\ue]\)^*$ is an isomorphism from $V_0$ to $V_0^*$.

Next, for any $v,w\in V_0$, using \eqref{absbounds} and a Sobolev inequality, we have
\begin{align}
\label{GaussC2def}
\bl \Fp[\ue](v),w\br
&=\left(\frac{\Phi^\eps\nab v}{(1+|\nab \ue|^2)^{\frac{n+2}2}},\nab w\right)\\
&\nonum \le \|\Phi^\eps\|_{L^{\frac32}} \|\nab v\|_{L^6}\|\nab w\|_{L^6}\\
&\nonum \le C\|\Phi^\eps\|_{L^{\frac32}} \|v\|_\htw \|w\|_\htw,
\end{align}
and therefore,
\begin{align*}
\|\Fp[\ue]\|_{VV^*} 
= \sup_{v\in V_0}\sup_{w\in V_0} \frac{\bl \Fp[\ue](v),w\br}{\|v\|_\htw \|w\|_\htw}
\le C\|\Phi^\eps\|_{L^{\frac32}}.
\end{align*}
In view of Remark \ref{Remark44ii} and the estimates
\begin{align*}
\left\|\frac{\p F(\ue)}{\p r_{ij}} \right\|_{L^\infty}& = \left\|\Phi^\eps_{ij}
\Bigl/(1+|\nab \ue|^2)^{\frac{n+2}2}\right\|_{L^\infty} = O(\eps^{-1}),\\
\left\|\frac{\p F(\ue)}{\p p_i}\right\|_{L^\infty} 
&= (n+2)\left\|\det(D^2\ue)\frac{\p \ue}{\p x_i}
\Bigr/(1+|\nab \ue|^2)^{\frac{n+4}2}\right\|_{L^\infty}\\
&\le C\|\nab \ue\|_{L^\infty}\|\det(D^2 \ue)\|_{L^\infty}= O\left(\eps^{-n}\right),
\end{align*}
we conclude that if $v\in V_0$ is the solution to 
\begin{align}
\label{Gausslemline}\bigl\langle (G^\prime_\eps[\ue])^*(v),w\bigr\rangle
=\langle \varphi,w\rangle\qquad \forall w\in V_0,
\end{align}
for some $\varphi\in L^2(\Ome)$,
then 
\begin{align}
\label{GaussCRdef}
\|v\|_{H^3}\le C\eps^{-2}\|\varphi\|_\lt\qquad 
\|v\|_{H^4}\le C\eps^{-(n+2)}\|\varphi\|_\lt.
\end{align}

Thus, by \eqref{GaussC0C1def}--\eqref{GaussCRdef}, 
{\rm [A2]} holds with
\begin{alignat}{3}
\label{GaussConformline1}
&C_0\equiv 0,\qquad &&C_1=O(\eps),
\qquad &&C_2=O\bigl(\eps^{-\frac12}\bigr),\\
\nonum&p=4,\qquad &&C_R=O\bigl(\eps^{-(n+2)}\bigr),
\end{alignat}
and therefore (cf. Theorem \ref{abstractbound1thm})
\begin{alignat}{2}
\label{GaussConformline2}
&C_3 = O(\eps^{-1}),\qquad &&C_4=O(\eps^{-\frac32}),\\
&\nonum C_5=O(\eps^{-(4+n)}),\qquad &&h_0=1.
\end{alignat}

To confirm {\rm [A3]--[A4]}, we take
\[
Y=W^{2,\frac{3(n-1)}{2}}(\Ome),\qquad 
\|\cdot\|_Y=\|\cdot\|_{W^{2,\frac{3(n-1)}{2}}}^{n-1}.
\] 
%
We then have the following bound for any $v,z\in V_0,\ y\in Y$:
\begin{align*}
\left(\frac{\cof(D^2 y)\nab v}{(1+|\nab y|^2)^{\frac{n+2}2}},\nab z\right)
&\le  \|\cof(D^2 y)\|_{L^{\frac32}} \|\nab v\|_{L^6}\|\nab z\|_{L^6}\\
&\le C\|\cof(D^2 y)\|_{L^{\frac32}} \|v\|_\htw \|z\|_\htw\\
&\le C\|D^2 y\|_{L^{\frac{3(n-1)}{2}}}^{n-1}\|v\|_\htw\|z\|_\htw.
%
\end{align*}
It then follows that
\begin{align*}
&\sup_{y\in Y} \frac{\bnorm{\Fp[y]}_{VV^*}}{\|y\|_Y}
\le C,
\end{align*}
and thus, {\rm [A3]--[A4]} holds.  We also note from 
\eqref{absbounds}--\eqref{absboundsinterp} that
\begin{align}\label{GaussYdef}
\|\ue\|_Y=O\left(\eps^{\frac12(3-2n)}\right),
\qquad C_6=\left(\eps^{\frac12(1-2n)}\right).
\end{align}

To verify condition {\rm [A5]}, we first make the following calculation
for any $w,z\in V_0,\ v_h\in V^h_g$:
\begin{align}\label{GaussA5line}
\bl \bigl(\Fp[\ue] &-\Fp[v_h]\bigr)(w),z\br\\
\nonum
&=\left(\frac{\Phi^\eps \nab w}{(1+|\nab \ue|^2)^{\frac{n+2}2}}
-\frac{\cof(D^2 v_h)\nab w }{(1+|\nab v_h|^2)^{\frac{n+2}2}},\nab z\right)\\
&\nonum 
=\left(\frac{\bigl(\Phi^\eps-\cof(D^2v_h)\bigr) \nab w}{(1+|\nab v_h|^2)^{\frac{n+2}2}},\nab z\right)\\
&\qquad 
+\nonum\left(\frac{\Phi^\eps \nab w}{\bigl(1+|\nab \ue|^2\bigr)^{\frac{n+2}2}} 
-\frac{\Phi^\eps \nab w}{\bigl(1+|\nab v_h|^2\bigr)^{\frac{n+2}2}},\nab z\right).
\end{align}
Bounding the first term in \eqref{GaussA5line}, we use a Sobolev 
inequality to conclude
\begin{align}\label{GaussA5first}
\left(\frac{\bigl(\Phi^\eps-\cof(D^2v_h)\bigr) \nab w}{(1+|\nab v_h|^2)^{\frac{n+2}2}},\nab z\right)
\le C\|\Phi^\eps-\cof(D^2 v_h)\|_{L^\frac32} \|w\|_\htw \|z\|_\htw.
\end{align}
To bound the second term in \eqref{GaussA5line}, we first use the mean value theorem
\begin{align*}
\bigl(1+|\nab \ue|^2\bigr)^{-\frac{n+2}2} &-\bigl(1+|\nab v_h|^2\bigr)^{-\frac{n+2}2}
\\
&=-(n+2)\bigl(1+|\nab y_h|^2\bigr)^{-\frac{n+4}2}\nab y_h\cdot \nab (\ue-v_h),
\end{align*}
where $y_h = \ue+\gamma v_h$ for some $\gamma\in [0,1]$.
Therefore, for any $\del\in (0,1)$ and $v_h\in V^h_g$ 
with $\|\util -v_h\|_\htw \le \del$
\begin{align*}
&\left(\frac{\Phi^\eps \nab w}{\bigl(1+|\nab \ue|^2\bigr)^{\frac{n+2}2}} 
-\frac{\Phi^\eps \nab w}{\bigl(1+|\nab v_h|^2\bigr)^{\frac{n+2}2}},\nab z\right)\\
&\hskip 1in
 \le C\|\nab y_h\|_{L^6}\|\nab (\ue-v_h)\|_{L^6}\|\Phi^\eps \|_{L^3} \|\nab w\|_{L^6}\|\nab z\|_{L^6}\\
&\hskip 1in
 \le C\bigl(\|u\|_{W^{1,\infty}}+\del\bigr) \|\ue-v_h\|_\htw \|\Phi^\eps\|_{L^3}\|w\|_\htw \|z\|_\htw\\
&\hskip 1in
\le C\eps^{-\frac23}\|\ue-v_h\|_\htw \|w\|_\htw \|z\|_\htw.
\end{align*}
It then follows from this calculation and \eqref{GaussA5first} that in 
the two-dimensional case
\begin{align*}
\big\|\Fp[\ue]-\Fp[v_h]\bigr\|_{VV*}
\le C\eps^{-\frac23 }\|\ue-v_h\|_\htw=L(h)\|\ue-v_h\|_\htw,
\end{align*}
that is, condition {\rm [A5]} holds with $L(h)=C\eps^{-\frac23}$.

In the three-dimensional setting, using arguments similar 
to those for the Monge-Amp\`ere equation, we have
\begin{align*}
\big\|\Fp[\ue]-\Fp[v_h]\bigr\|_{VV*}
\le C\(\eps^{-\frac56}+h^{-1}\)\|\ue-v_h\|_\htw=L(h)\|\ue-v_h\|_\htw,
\end{align*}
and therefore {\rm [A5]} holds with $L(h)=C\(\eps^{-\frac56}+h^{-1}\)$.

Gathering up these results, and applying Theorem \ref{abstractmainthm}
with estimates \eqref{GaussConformline1}--\eqref{GaussYdef},
we conclude that there exists a unique solution to the finite element
method \eqref{Gauss1eqn} and that the error estimates 
\eqref{Gaussmainthmline1}--\eqref{Gaussmainthmline2} hold.
\end{proof}

\subsection{Mixed finite element methods for 
the equation of prescribed Gauss curvature}\label{chapter-6-sec-2-subsec2}

The mixed finite element method for \eqref{absGauss1}--\eqref{absGauss2}
is defined as follows:  find $(\seh,\ueh)\in W^h_\eps\times Q_g^h$
such that
\begin{alignat}{2}
\label{mixedGauss1}(\seh,\kappa_h)+b(\kappa_h,\ueh)
&=G(\kappa_h)\qquad && \forall\kappa_h\in W^h_0,\\
\label{mixedGauss2}b(\seh,z_h)-\eps^{-1}c(\seh,\ueh,z_h)
&=0\qquad && \forall z_h\in Q^h_0,
\end{alignat}
where
\begin{align*}
b(\kappa_h,\ueh)&=\(\Div(\kappa_h),\nab \ueh\),\\
c(\seh,\ueh,z_h)
&=\left(\mck-\frac{\det(\seh)}{(1+|\nab\ueh|^2)^{\frac{n+2}{2}}},z_h\right),
\end{align*}
and $G(\kappa_h)$ is defined by \eqref{Gdef}.

In this section, we apply the results of Chapter \ref{chapter-5}
to the mixed finite element method \eqref{mixedGauss1}--\eqref{mixedGauss2}.
Namely, we verify that conditions {\rm [B1]--[B6]} hold, and from these results,
we obtain existence and uniqueness of a solution to 
\eqref{mixedGauss1}--\eqref{mixedGauss2} as well as its error estimates.
We summarize our findings in the following theorem.

\begin{thm} \label{Gaussimixedmainthm}
Let $\ue\in H^s(\Ome)$ be the solution to 
\eqref{absGaussMM1}--\eqref{absGaussMM3}
with $s>3$ when $n=2$ and $s>5$ when $n=3$.  
Suppose $k\ge 3$ when $n=2$ and $k\ge 5$ when $n=3$.
Then for $h\le h_2$, there exists 
a unique solution $(\seh,\ueh)\in W^h_\eps\times Q^h_g$
to \eqref{mixedGauss1}--\eqref{mixedGauss1}.  
Furthermore, there hold the following error estimates:
\begin{align}\label{Gaussmixedmainthmline1}
&\ttbar{\se-\seh}{\ue-\ueh} \le K_8h^{\ell-2}\snorm,\\ \label{Gaussmixedmainthmline2}
&\|\ue-\ueh\|_\ho \le K_{R_1}\Bigl(K_9h^{\ell-1}\snorm+K_8^2R(h)h^{2\ell-4}\snorms\Bigr),
\end{align}
where
\begin{alignat*}{1}
&\ttbar{\mu}{v}=h\|\mu\|_\ho+\|\mu\|_\lt+\eps^{-\frac12}\|v\|_\ho,\\
&K_8 =O\bigl(\eps^{\frac14(28-19n)}\bigr),\qquad\quad
K_{9} =O\bigl(\eps^{\frac{76-49n}{12}}\bigr),\\
&R(h)=
C\left\{ 
\begin{array}{ll}
|\log h|(h^{-1}+\eps^{-2})& n=2,\\
\eps^{-1}h^{-1}+\eps^{-3}+h^{-3}& n=3,
\end{array}\right.\\
&\ell={\rm min}\{s,k+1\},
\end{alignat*}
and $h_2$ is chosen such that
\begin{align*}
h_2&=C{\rm min}\Bigl\{\eps^{\frac{7+4n}{6}},
\left(\eps^{\frac14(26-19n)}R(h_2)\snorm\right)^{\frac{1}{1-\ell}},
 \left(\eps^{-\frac12}R(h_2)\snorm\right)^{\frac{1}{2-\ell}}\Bigr\}.
\end{align*}
\end{thm}

\begin{proof}

First, using the same arguments as those used to show assumption {\rm [A2]} 
in Theorem \ref{Gaussmainthm}, we can conclude that {\rm [B2]} holds with
\begin{alignat}{3}
\label{GaussMixedline1}
&K_0\equiv 0,\qquad &&K_1=O(1),\qquad &&K_2=O\bigl(\eps^{-1}\bigr),\\
\nonum&p=4,\qquad &&K_{R_0}=O\bigl(\eps^{-(n+2)}\bigr),
\qquad &&K_{R_1}=O\bigl(\eps^{-2}\bigr),
\end{alignat}
and therefore, (cf.~Theorem \ref{mixedlinthm} and Lemma \ref{mixedlem52})
\begin{align}\label{GaussMixedline2}
K_4=O\bigl(\eps^{-\frac32}\bigr),\qquad
K_5=
O\bigl(\eps^{-\frac12 (2n+5)}\bigr),\qquad
K_7=O\bigl(\eps^{-\frac12}\bigr).
\end{align}

We now turn our attention to condition {\rm [B3]}.  
To show that this condition holds, we set $(n=2,3)$
\begin{alignat*}{1}
&X=\left[L^{n(n+\vepsi(3-n))}(\Ome)\right]^{n\times n},\qquad Y=W^{1,\infty}(\Ome),\\
&\|(\ome,y)\|_{X\times Y} = \|\ome\|_{L^{(n-1)(n+\vepsi(3-n))}}^{n-1}
+\|\ome\|_{L^{n(n+\vepsi(3-n))}}^n \|\nab y\|_{L^\infty}
\quad \forall \ome\in X,\ y\in Y.
%
\end{alignat*}
%
%
Then for any $\omega\in X,\ y\in Y,\ \chi\in W,$ and $v\in Q,\ z\in Q_0$, we have
\begin{align*}
&\bl\Fp[\omega,y](\chi,v),z\br
=-\left(\frac{\cof(\ome):\chi}{(1+|\nab y|^2)^{\frac{n+2}2}}, z\right)
+(n+2)\left(\frac{\det(\ome)\nab y\cdot\nab v}{(1+|\nab y|^2)^{\frac{n+4}2}},z\right)
\\
&\, \le C\Bigl( \|\cof(\ome)\|_{L^{n+\vepsi(3-n)}}\|\chi\|_{L^2}\|z\|_{H^1} 
+ \|\det(\ome)\|_{L^{n+\vepsi(3-n)}}\|\nab y\|_{L^\infty}
\| \nab v\|_\lt \|z\|_{H^1}\Bigr)\\
&\, \le C\Bigl( \|\ome\|_{L^{(n-1)(n+\vepsi(3-n))}}^{n-1}  \|\chi\|_\lt 
+ \|\ome\|_{L^{n(n+\vepsi(3-n))}}^n\|\nab y\|_{L^\infty}\| v\|_\ho \Bigr)\|z\|_\ho\\
&\,\le C\bigl(\|\ome\|_{L^{(n-1)(n+\vepsi(3-n))}}^{n-1}
+\|\ome\|_{L^{n(n+\vepsi(3-n))}}^{n}\|\nab y\|_{L^\infty}\bigr)\(\|\chi\|_\lt
+\|v\|_\ho\)\|z\|_\ho.
\end{align*}
It follows from this calculation that 
\begin{align*}
\norm{\Fp[\ome,y](\chi,v)}_{H^{-1}}\le C\|(\ome,y)\|_{X\times Y}\(\|\chi\|_\lt+\|v\|_\ho\),
\end{align*}
and therefore condition {\rm [B3]} holds.

To confirm {\rm [B4]}, we use \eqref{mixedmess} to 
conclude that if $\se\in \left[H^{s-2}(\Ome)\right]^{n\times n}$, then
for any $\gamma\in [0,1]$ and $\ell\in [3,{\rm min}\{s,k+1\}]$
\begin{align*}
&\Bigl\|\Bigl(\Pi^h\se-\gamma \se,\util-\gamma \ue\Bigr)\Bigr\|_{X\times Y}\\
&=\bigl\|\Pi^h\se-\gamma \se\bigr\|_{L^{(n-1)(n+\vepsi(3-n))}}^{n-1}
+\bigl\|\Pi^h\se-\gamma \se\|_{L^{n(n+\vepsi(3-n))}}^n
\bigl\|\nab\util-\gamma\nab\ue\bigr\|_{L^\infty}\\
&\le C\Bigl(\big\|\Pi^h\se\bigr\|_{L^{n(n+\vepsi(3-n))}}^n
+\|\se\|_{L^{n(n+\vepsi(3-n))}}^n\Bigr)\\
&\le C\Bigl(h^{\ell-\frac{5}3-\frac{n}2}\|\se\|_{H^{\ell-2}}
+\|\se\|_{L^{n(n+\vepsi(3-n))}}\Bigr)^n,
\end{align*}
For the two-dimensional case, we set $\ell=3$ and use \eqref{absboundsinterp} to get
\begin{align*}
\Bigl\|\Bigl(\Pi^h\se-\gamma \se,\util-\gamma \ue\Bigr)\Bigr\|_{X\times Y}
\le C\Bigl(h^\frac13\|\se\|_\ho+\|\se\|_{L^{4+2\vepsi}}\Bigr)^2
=O\left(\eps^{-2}\right).
\end{align*}

For the three-dimensional case, we set $\ell=\frac{19}6$ and 
use \eqref{absboundsh} and \eqref{absboundsinterp} to conclude 
\begin{align*}
\Bigl\|\Bigl(\Pi^h\se-\gamma \se,\util-\gamma \ue\Bigr)\Bigr\|_{X\times Y}
\le C\Bigl(\|\se\|_{H^{\frac76}}+\|\se\|_{L^9}\Bigr)^3
=O\left(\eps^{-\frac{27}4}\right).
\end{align*}
Combing these two estimates, we have
\begin{align}\label{GaussMixedline3}
K_3=O\left(\eps^{\frac14(30-19n)}\right),
\qquad K_6=O\left(\eps^{\frac14(28-19n)}\right).
\end{align}

%

As a first step to confirm {\rm [B5]}, we
note that for all $(\mu_h,v_h)\in W^h_\eps\times Q_g^h$
and $(\kappa_h,z_h)\in W^h\times Q^h,\ w_h\in Q^h$
\begin{align} &\label{GaussB5line1}
\Bigl\langle \(\Fp[\se,\ue]-\Fp[\mu_h,v_h]\)\(\kappa_h,z_h\),w_h\Bigr\rangle\\
\nonum& =\left(\left(\frac{\cof(\mu_h)}{(1+|\nab v_h|^2)^{\frac{n+2}2}}-\frac{\cof(\se)}{(1+|\nab \ue|^2)^{\frac{n+2}2}}\right):\kappa_h,w_h\right)\\
&\nonum\qquad +(n+2)\left(\left(\frac{\det(\se)\nab \ue}{(1+|\nab \ue|^2)^{\frac{n+4}2}}
-\frac{\det(\mu_h)\nab v_h}{(1+|\nab v_h|^2)^{\frac{n+4}2}}\right)\cdot \nab z_h,w_h\right).
\end{align}
To bound the first term in \eqref{GaussB5line1}, we use the mean value theorem 
to conclude
\begin{align*}
&\left(\left(\frac{\cof(\mu_h)}{(1+|\nab v_h|^2)^{\frac{n+2}2}}-\frac{\cof(\se)}{(1+|\nab \ue|^2)^{\frac{n+2}2}}\right):\kappa_h,w_h\right)\\
&\quad =\left(\frac{\bigl(\cof(\mu_h)-\cof(\se)\bigr):\kappa_h}{(1+|\nab v_h|^2)^{\frac{n+2}2}},w_h\right)
+\left(\frac{\cof(\se):\kappa_h}{(1+|\nab v_h|^2)^{\frac{n+2}2}},w_h\right)\\
&\qquad\quad -\left(\frac{\cof(\se):\kappa_h}{(1+|\nab \ue|^2)^{\frac{n+2}2}},w_h\right)\\
&\quad =\left(\frac{\bigl(\cof(\mu_h)-\cof(\se)\bigr):\kappa_h}{(1+|\nab v_h|^2)^{\frac{n+2}2}},w_h\right)
 -(n+2)\left(\frac{\nab y_h\cdot \nab (v_h-\ue)\cof(\se):\kappa_h}{\bigl(1+|\nab y_h|^2\bigr)^{\frac{n+4}2}},w_h\right),
\end{align*}
where $y_h=v_h+\gamma \ue$ for some $\gamma \in [0,1]$.
Therefore, by \eqref{absbounds} and the inverse inequality, we have
\begin{align}
\label{Gaussmixedline1}
&\left(\left(\frac{\cof(\mu_h)}{(1+|\nab v_h|^2)^{\frac{n+2}2}}-\frac{\cof(\se)}{(1+|\nab \ue|^2)^{\frac{n+2}2}}\right):\kappa_h,w_h\right)\\
&\, \le C\Bigl(\|\cof(\mu_h)-\cof(\se)\|_{L^2}
+\|\nab y_h\|_{L^\infty}\|\ue-v_h\|_\ho\|\cof(\se)\|_{L^\infty} \Bigr) \nonum \\
&\qquad \times  \|\kappa_h\|_\lt\|w_h\|_{L^\infty} \nonum\\
%
%
&\,
\le C|\log h|^{\frac{3-n}2} h^{1-\frac{n}2} \Bigl(\|\cof(\mu_h)-\cof(\se)\|_{L^2}
+\eps^{-1}\|\nab y_h\|_{L^\infty}\|\ue-v_h\|_\ho \Bigr) \nonum \\
& \qquad\times \ttbar{\kappa_h}{z_h}\|w_h\|_\ho. \nonum
\end{align}

If $v_h\in Q^h_g$ with $\|\util-v_h\|_\ho\le \del \in (0,1)$, then by the inverse inequality
\begin{align*}
\|\nab y_h\|_{L^\infty}
&\le \|\nab \ue\|_{L^\infty}+\|\nab \util \|_{L^\infty}+|\log h|^{\frac{3-n}2}h^{1-\frac{n}2} \|\nab (\util-v_h)\|_\ho\\
&\le C|\log h|^{\frac{3-n}2}h^{1-\frac{n}2}.
\end{align*}
Furthermore, from the mixed finite element analysis for the Monge-Amp\`ere equation, we have
\begin{align*}
\|\cof(\se)-\cof(\mu_h)\|_\lt \le C\(\eps^{(2-n)}+h^{\frac32(2-n)}\)\|\se-\mu_h\|_\lt.
\end{align*}
Using these two inequalities in \eqref{Gaussmixedline1}, we arrive at
\begin{align}\label{Gaussmixedline1bound}
&\left(\left(\frac{\cof(\mu_h)}{(1+|\nab v_h|^2)^{\frac{n+2}2}}
-\frac{\cof(\se)}{(1+|\nab \ue|^2)^{\frac{n+2}2}}\right):\kappa_h,w_h\right)\\
&\nonum
\quad\le C|\log h|^{\frac{3-n}2} h^{1-\frac{n}2} \Bigl(\(\eps^{(2-n)}+h^{\frac32(2-n)}\)\|\se-\mu_h\|_\lt\\
&\nonum
\qquad+\eps^{-1}|\log h|^{\frac{3-n}2} h^{1-\frac{n}2}\|\ue-v_h\|_\ho\Bigr)
\ttbar{\kappa_h}{z_h}\|w_h\|_\ho\\
&\nonum
\quad \le C\bigl(\eps^{-1}|\log h|\bigr)^{3-n}\bigl(\eps^{-1}h^{-\frac12}+h^{-2}\bigr)^{n-2}\Bigl(\|\se-\mu_h\|_\lt+\|\ue-v_h\|_\ho\Bigr).
\end{align}

Using a similar strategy to bound the second term in \eqref{GaussB5line1},
we add and subtract terms and use the mean value theorem to conclude
\begin{align}\label{Gaussmixedline2}
&\left(\left(\frac{\det(\se)\nab \ue}{(1+|\nab \ue|^2)^{\frac{n+4}2}}
-\frac{\det(\mu_h)\nab v_h}{(1+|\nab v_h|^2)^{\frac{n+4}2}}\right)\cdot \nab z_h,w_h\right)\\
&\nonum
\quad =\left(\frac{\(\det(\se)-\det(\mu_h)\)\nab \ue\cdot \nab z_h}{(1+|\nab v_h|^2)^{\frac{n+4}2}}+
\frac{\det(\mu_h)\(\nab \ue-\nab v_h\)\cdot \nab z_h}{(1+|\nab v_h|^2)^{\frac{n+4}2}},w_h\right)\\
\nonum
&\qquad+\left(\frac{\det(\se)\nab \ue\cdot \nab z_h}{(1+|\nab \ue|^2)^{\frac{n+4}2}}
-\frac{\det(\se)\nab \ue\cdot \nab z_h}{(1+|\nab v_h|^2)^{\frac{n+4}2}},w_h\right)\\
&\nonum
\quad  =\left(\frac{\(\cof(\xi_h):(\se-\mu_h)\)\nab \ue\cdot \nab z_h}{(1+|\nab v_h|^2)^{\frac{n+4}2}}+
\frac{\det(\mu_h)\(\nab \ue-\nab v_h\)\cdot \nab z_h}{(1+|\nab v_h|^2)^{\frac{n+4}2}},w_h\right)\\
\nonum
&\qquad-(n+4)\left(\frac{\det(\se)\(\nab \ue\cdot \nab z_h\)\(\nab x_h\cdot \nab (\ue-v_h)\)}{(1+|\nab x_h|^2)^{\frac{n+6}2}},w_h\right),
\end{align}
where $\xi_h=\se+\gamma_1 \mu_h,\ x_h=\ue+\gamma_2v_h$ for some $\gamma_1,\gamma_2\in [0,1]$.
Bounding the first term in \eqref{Gaussmixedline2}, we use 
the inverse inequality to conclude
\begin{align*}
&\left(\frac{\(\cof(\xi_h):(\se-\mu_h)\)\nab \ue\cdot \nab z_h}{(1+|\nab v_h|^2)^{\frac{n+4}2}},w_h\right)\\
&\qquad \le C \|\cof(\xi_h)\|_\lt \|\se-\mu_h\|_\lt \|\nab \ue\|_{L^\infty}\|\nab z_h\|_{L^\infty}\|w_h\|_{L^\infty}\\
&\qquad \le C |\log h|^{\frac{3-n}2} h^{1-n}\|\cof(\xi_h)\|_{L^2}\|\se-\mu_h\|_\lt \|z_h\|_\ho\|w_h\|_\ho.
\end{align*}
If $\|\stil-\mu_h\|_\lt \le \del \in (0,1)$, then 
by \eqref{absboundsinterp} and the inverse inequality
\begin{align*}
\|\cof(\xi_h)\|_\lt 
&\le \|\xi_h\|_{L^{2(n-1)}}^{n-1}\le C\|\se\|_{L^{2(n-1)}}^{n-1}
+\|\stil-\mu_h\|_{L^{2(n-1)}}^{n-1}\\
&\le C\(\eps^{\frac{3-2n}2}+h^{[\frac{n}{2(n-1)}-\frac{n}2](n-1)}
\|\stil-\mu_h\|_\lt\)\\
&=O\left(\eps^{\frac{3-2n}2} +h^{\frac{n(n-2)}{2}}\right)=O
\left(\eps^{\frac{3-2n}2} +h^{\frac32(2-n)}\right).
\end{align*}
Therefore,
\begin{align}\label{Gaussmixedline2bound1}
&\left(\frac{\(\cof(\xi_h):(\se-\mu_h)\)\nab \ue\cdot \nab z_h}{(1+|\nab v_h|^2)^{\frac{n+4}2}},w_h\right)\\
&\nonum \le C |\log h|^{\frac{3-n}2} h^{1-n}\(\eps^{\frac{3-2n}2}+h^{\frac32(2-n)}\) \|\se-\mu_h\|_\lt \ttbar{\kappa_h}{z_h}\|w_h\|_\ho.
\end{align}

Bounding the second term in \eqref{Gaussmixedline2}, we have
\begin{align*}
&\left(\frac{\det(\mu_h)\(\nab \ue-\nab v_h\)\cdot \nab z_h}{(1+|\nab v_h|^2)^{\frac{n+4}2}},w_h\right)\\
&\le \|\det(\mu_h)\|_\lt \|\nab \ue-\nab v_h\|_\lt \| \|\nab z_h\|_{L^\infty}\|w_h\|_{L^\infty}\\
&\le C |\log h|^{3-n} h^{2-n}\|\det(\mu_h)\|_\lt \|\ue-v_h\|_\ho \|z_h\|_\ho \|w_h\|_\ho.
\end{align*}
If $\|\stil-\mu_h\|_\lt \le \del\in (0,1)$, then 
by \eqref{absboundsinterp} and the inverse inequality, 
\begin{align*}
\|\det(\mu_h)\|_\lt \le C\|\mu_h\|_{L^{2n}}^n
&\le C\(\|\se\|_{L^{2n}}^n+\|\stil - \mu_h\|_{L^{2n}}^n\)\\
&=O\(\eps^{\frac{1-2n}2}+h^{\frac{n}2(1-n)}\)=O\(\eps^{\frac{1-2n}2} +h^{-2n+3}\).
\end{align*}
Therefore,
\begin{align}
\label{Gaussmixedline2bound2}
&\left(\frac{\det(\mu_h)\(\nab \ue-\nab v_h\)\cdot \nab z_h}{(1+|\nab v_h|^2)^{\frac{n+4}2}},w_h\right)\\
&\nonum \le  C|\log h|^{3-n}h^{2-n}\(\eps^{\frac{1-2n}2}+h^{-2n+3}\)\|\ue-v_h\|_\ho \ttbar{\kappa_h}{z_h} \|w_h\|_\ho.
\end{align}

Next, using similar arguments as above, we bound the third
term in \eqref{Gaussmixedline2} as follows:
\begin{align}\label{Gaussmixedline2bound3}
&\left(\frac{\det(\se)\(\nab \ue\cdot \nab z_h\)\(\nab x_h\cdot \nab (\ue-v_h)\)}{(1+|\nab x_h|^2)^{\frac{n+6}2}},w_h\right)\\
&\nonum\quad \le \|\det(\se)\|_{L^\infty}\|\nab \ue\|_{L^\infty}\|\nab z_h\|_\lt\|\nab x_h\|_{L^\infty}\|\nab (\ue-v_h)\|_\lt \|w_h\|_{L^\infty}\\
&\nonum\quad \le C|\log h|^{\frac{3-n}2} h^{1-\frac{n}2} \|\se\|_{L^\infty}^n\|\nab x_h\|_{L^\infty} \|\ue-v_h\|_\ho \ttbar{\kappa_h}{z_h}\|w_h\|_\ho\\
&\nonum\quad \le C|\log h|^{\frac{3-n}2} h^{1-\frac{n}2}\eps^{-n}\|\nab x_h\|_{L^\infty} \|\ue-v_h\|_\ho \ttbar{\kappa_h}{z_h}\|w_h\|_\ho\\
&\nonum\quad \le C|\log h|^{{3-n}} h^{2-n}\eps^{-n}\|\ue-v_h\|_\ho \ttbar{\kappa_h}{z_h}\|w_h\|_\ho.
\end{align}

Applying the bounds \eqref{Gaussmixedline2bound1}--\eqref{Gaussmixedline2bound3}
to \eqref{Gaussmixedline2}, we have
\begin{align*}
&\left(\left(\frac{\det(\se)\nab \ue}{(1+|\nab \ue|^2)^{\frac{n+4}2}}
-\frac{\det(\mu_h)\nab v_h}{(1+|\nab v_h|^2)^{\frac{n+4}2}}\right)\cdot \nab z_h,w_h\right)\\
&\le  C\Bigl( |\log h|^{\frac{3-n}2} h^{1-{n}}\(\eps^{\frac{3-2n}2}
+h^{\frac32(2-n)}\)\\
&\qquad+|\log h|^{3-n}h^{2-n} \(\eps^{\frac{1-2n}2} 
+h^{-2n+3}\)+|\log h|^{3-n}h^{2-n}\eps^{-n}\Bigr)\\
&\qquad \times \Bigl(\|\se-\mu_h\|_\lt
+\|\ue-v_h\|_\ho\Bigr)\ttbar{\kappa_h}{z_h} \|w_h\|_\ho\\
&\le C|\log h|^{3-n}\(\eps^{-n}+h^{-2n+3}\)\bigl(\|\se-\mu_h\|_\lt
+\|\ue-v_h\|_\ho\bigr)\ttbar{\kappa_h}{z_h} \|w_h\|_\ho.
\end{align*}

Finally, we combine this last inequality with 
\eqref{GaussB5line1}--\eqref{Gaussmixedline2} to get
\begin{align*}
\Bigl\langle \(\Fp[\se,\ue] &-\Fp[\mu_h,v_h]\)\(\kappa_h,z_h\),w_h\Bigr\rangle\\
&\nonum \le C\Bigl(
\eps^{-1}|\log h|^{3-n}h^{2-n}+|\log h|^{3-n}\(\eps^{-n}+h^{-2n+3}\)\Bigr)\\
&\qquad\times \bigl(\|\se-\mu_h\|_\lt 
+\|\ue-v_h\|_\ho\bigr)  \ttbar{\kappa_h}{z_h}\|w_h\|_\ho\\ 
&\le C|\log h|^{3-n}\(\eps^{-1}h^{2-n}+\eps^{-n}+h^{-2n+3}\)\\
&\qquad\times \(\|\se-\mu_h\|_\lt+\|\ue-v_h\|_\ho\)\ttbar{\kappa_h}{z_h} \|w_h\|_\ho.
\end{align*}
Therefore, condition {\rm [B5]} holds with $R(h) = C|\log h|(\eps^{-2} +h^{-1}\)$ 
in the two-dimensional case and $R(h) = C\(\eps^{-1}h^{-1}+\eps^{-3}+h^{-3}\)$ in 
the three-dimensional case.

To establish condition {\rm [B6]}, we use similar arguments to that of the mixed
finite element analysis of the Monge-Amp\`ere equation to conclude
\begin{align*}
\left\|\frac{\p F}{\p r_{ij}}\right\|_{L^\infty}
&\le \left\|\frac{\cof(\se)}{(1+|\nab \ue|^2)^{\frac{n+2}2}}\right\|_{L^\infty}
\le \|\cof(\se)\|_{L^\infty}=O(\eps^{-1}),\\
\left\|\frac{\p F}{\p r_{ij}}\right\|_{W^{1,\frac65}} 
&\le  C \left\|(1+|\nab \ue|^2)^{-\frac{n+2}2} \right\|_{W^{1,\infty}}
\|\cof(\se)\|_{W^{1,\frac65}}\\
&\le  C \|\ue\|_{W^{2,\infty}}\|\cof(\se)\|_{W^{1,\frac65}} = O\left(\eps^{-\frac23 (1+n)}\right).
\end{align*}
Therefore by Proposition \ref{B6prop}, condition {\rm [B6]} holds with
\[
\alpha = 1,\qquad K_G= C\eps^{-\frac23 (1+n)}.
\]

Wrapping things up, we apply Theorem \ref{mixedmainthm} and \ref{mixedmainthmimp}
to obtain existence and uniqueness of a solution $(\seh,\ueh)$ to the
mixed finite element method \eqref{mixedGauss1}--\eqref{mixedGauss2}.
The error estimates \eqref{Gaussmixedmainthmline1}--\eqref{Gaussmixedmainthmline2}
also follow from these results and by the definitions
\begin{alignat*}{2}
K_8=CK_7,\qquad K_9=K_8{\rm max}\{K_2,K_G\}.
\end{alignat*}

\end{proof}
 
\subsection{Numerical experiments and rates of convergence}
\label{chapter-6-sec-2-subsec3}

In this section, we provide several two-dimensional numerical experiments to
gauge the efficiency of the finite element methods developed in
the previous two subsections.  

\subsubsection*{Test 6.2.1}
In this test, we fix $h=0.01$ in order to study the behavior of
$\ue$.  Notably, we are interested 
whether $\|u-\ue\|\to 0$ as $\eps\to 0^+$.  To this end,
we solve the following problem:  find $\ueh\in V^h_g$ such that\footnote{
We note that it is easy to see the finite element methods 
and their convergence analyses of Section \ref{chapter-6-sec-2-subsec1} 
and \ref{chapter-6-sec-2-subsec2}
also apply to the case $f>0$ but $f\not\equiv 1$.}
\begin{align*}
&-\eps(\Del \ueh,\Del v_h)
+\left(\frac{\det(D^2\ueh)}{(1+|\nab \ue|^2)^{\frac{n+2}2}},v_h\right)
=(\mck f,v_h)-\left\langle
\eps^2,\normd{v_h}\right\rangle_{\p\Ome}.
\end{align*}
Here, we take $V^h$ to be the Argyris finite element space \cite{Ciarlet78}
of degree $k=5$ and set $\Ome=(0,1)^2$.  We use the
following test function and parameters:
\begin{alignat*}{2}
\text{(a) } 
&u=e^{\frac{x_1^2+x_2^2}{2}},\hspace{3.75cm}\mck=0.1,\\
&f=\frac{(1+x_1^2+x_2^2)e^{x_1^2+x_2^2}}{0.1(1+(x_1^2+x_2^2)e^{x_1^2+x_2^2})^2},\qquad
g=e^{\frac{x_1^2+x_2^2}{2}}. \\
\text{(b) } &u=\cos(\sqrt{x_1}\pi)+\cos(\sqrt{x_2}\pi),\qquad\quad \mck=0.025, \\
&f=\frac{\pi^2}{16}\frac{\(x_1^{-\frac32}\sin(\sqrt{x_1}\pi)-x_1^{-1}\pi\cos(\sqrt{x_1}\pi)\)
\(x_2^{-\frac32}\sin(\sqrt{x_2}\pi)-x_2^{-1}\pi\cos(\sqrt{x_2}\pi)\)}{
0.025\big(1+\frac{\pi^2}{4}\big(x_1^{-1}\sin^2(\sqrt{x_1}\pi)
+x_2^{-1}\sin^2(\sqrt{x_2}\pi)\big)\big)^2},\\
&g=\cos(\sqrt{x_1}\pi)+\cos(\sqrt{x_2}\pi).
\end{alignat*}

The computed solution, whose values are given in Table \ref{Test621Table}, 
is compared to the exact solution in Figure \ref{Test621Figure}.
As seen from Figure \ref{Test621Figure}, the behavior of
$\|u-\ueh\|$ behaves similarly to that of the Monge-Amp\`ere
equation, that is, we observe the following rates of convergence as $\eps\to 0^+$:
\begin{align*}
&\|u-\ueh\|_\lt\approx O(\eps),\quad
\|u-\ueh\|_{H^1}\approx O\bigl(\eps^{\frac34}\bigr),\quad
\|u-\ueh\|_\htw\approx O\bigl(\eps^\frac14\bigr).
\end{align*}
Since we have fixed $h$ very small, we expect that $\|u-\ue\|$ 
behaves similarly.

\begin{table}[tbp]
\begin{center}
\caption{Test 6.2.1: Error of $\|u-\ueh\|$ w.r.t. 
$\eps$ ($h=0.01$) and estimated rate of convergence} \label{Test621Table}
\begin{tabular}{lccccc}
&{\scriptsize $\eps$} & {\scriptsize$\|u-\ueh\|_\lt$(rate)}  
&{\scriptsize $\|u-\ueh\|_\ho$(rate)}
&{\scriptsize $\|u-\ueh\|_\htw$(rate)}\\
\noalign{\smallskip}\hline\noalign{\smallskip}
Test 6.2.1a
&1.0E--01	&6.12E--02\blank	&3.34E--01\blank	&3.04E+00\blank\\	
&5.0E--02	&4.27E--02(0.52)	&2.59E--01(0.37)	&2.80E+00(0.12)\\
&2.5E--02	&2.88E--02(0.57)	&1.97E--01(0.39)	&2.54E+00(0.14)\\
&1.0E--02	&1.64E--02(0.62)	&1.34E--01(0.42)	&2.20E+00(0.16)\\
&5.0E--03	&1.03E--02(0.66)	&9.75E--02(0.46)	&1.94E+00(0.18)\\
&2.5E--03	&6.35E--03(0.70)	&6.92E--02(0.49)	&1.70E+00(0.19)\\
&1.0E--03	&3.18E--03(0.75)	&4.24E--02(0.53)	&1.41E+00(0.21)\\
&5.0E--04	&1.82E--03(0.80)	&2.85E--02(0.58)	&1.21E+00(0.22)\\
\noalign{\smallskip}\hline\noalign{\smallskip}
Test 6.2.1b
&1.0E--01	&2.84E--02\blank	&1.95E--01\blank	&2.51E+00\blank\\	
&5.0E--02	&1.87E--02(0.60)	&1.47E--01(0.41)	&2.27E+00(0.15)\\
&2.5E--02	&1.20E--02(0.64)	&1.08E--01(0.44)	&2.02E+00(0.17)\\
&1.0E--02	&6.34E--03(0.70)	&6.92E--02(0.49)	&1.70E+00(0.19)\\
&5.0E--03	&3.78E--03(0.75)	&4.80E--02(0.53)	&1.47E+00(0.21)\\
&2.5E--03	&2.19E--03(0.79)	&3.24E--02(0.56)	&1.27E+00(0.22)\\
&1.0E--03	&1.02E--03(0.83)	&1.87E--02(0.60)	&1.03E+00(0.23)\\
&5.0E--04	&5.56E--04(0.87)	&1.20E--02(0.64)	&8.74E--01(0.24)
\end{tabular}
\end{center}
\end{table}

\begin{figure}[ht]
\begin{center}
\includegraphics[scale=0.5]{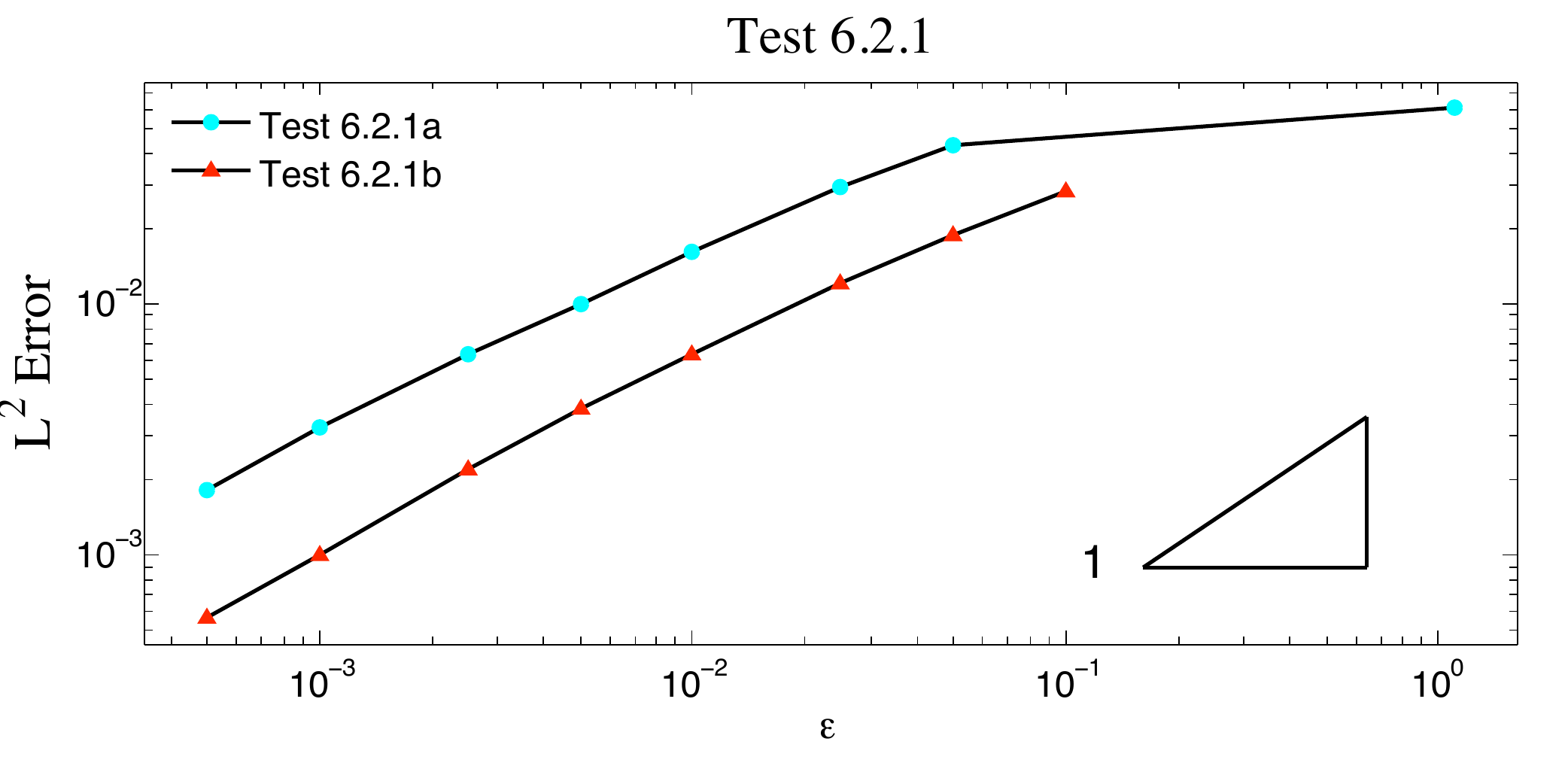}\\
\includegraphics[scale=0.5]{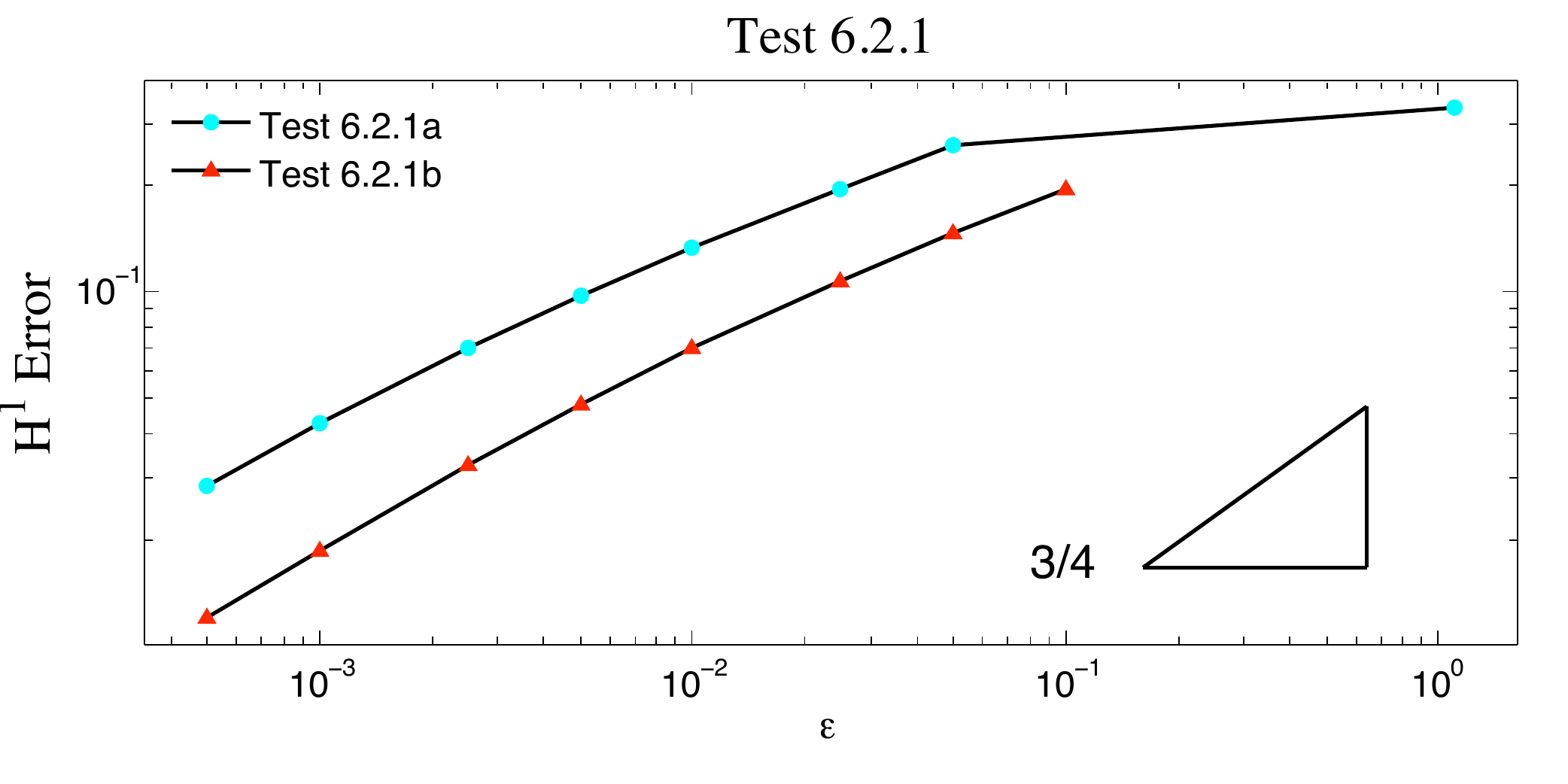}\\
\includegraphics[scale=0.5]{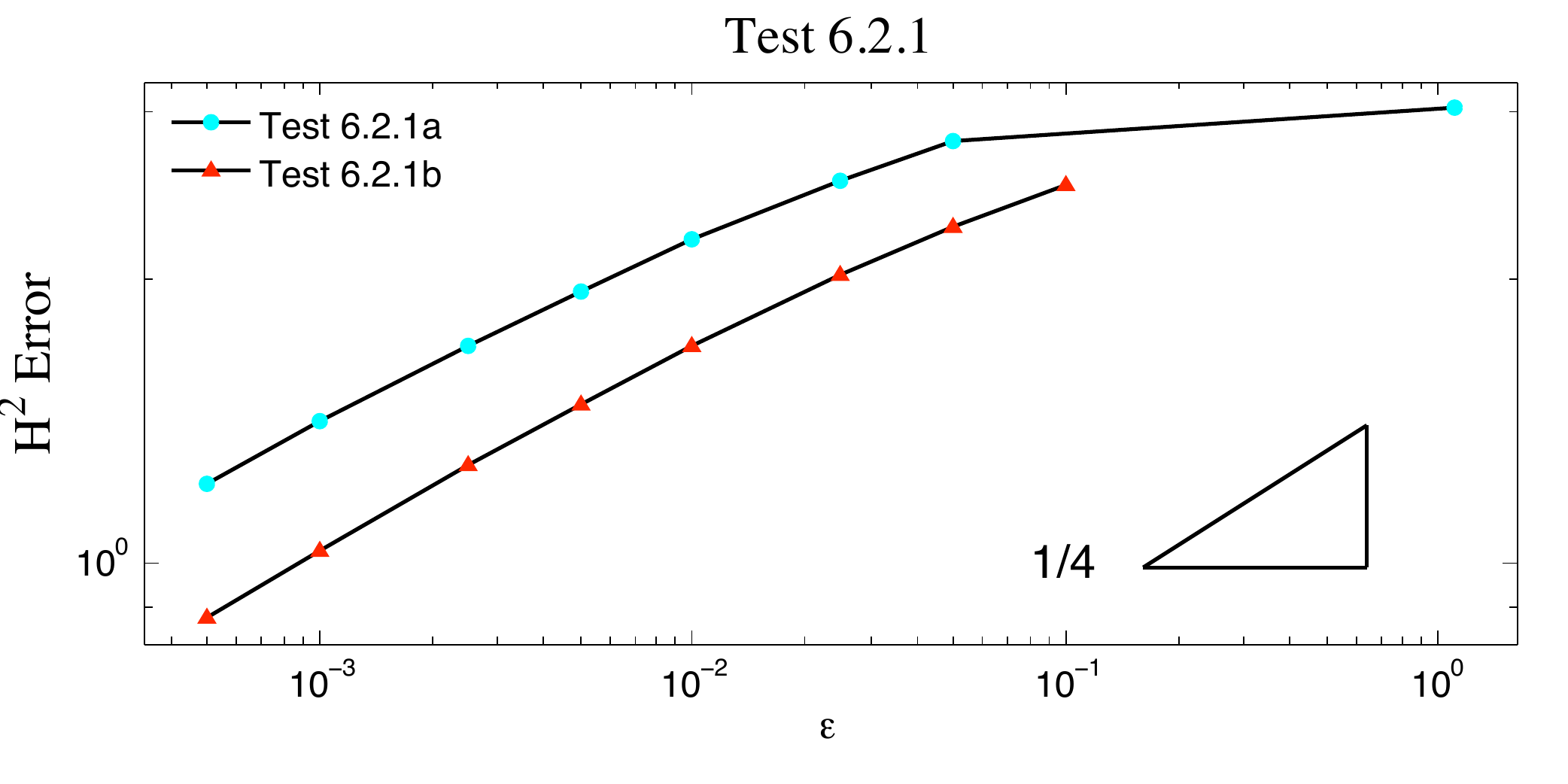}\\
\end{center}
\caption{Test 6.2.1.  Change of $\|u-\ueh\|$ w.r.t. $\eps$ ($h=0.01$)} 
\label{Test621Figure}
\end{figure}

\medskip
\subsubsection*{Test 6.2.2}
In this test, we calculate $\ueh$ using the Hermann-Miyoshi mixed finite 
element method developed in the previous subsection to calculate the rate 
of convergence of $\|\ue-\ueh\|$ with respect to $h$ 
for fixed $\eps$.  We also compare
the numerical tests with Theorem \ref{Gaussimixedmainthm}.
Since $\ue$ is generally not known, we solve the following problem 
(compare to \eqref{mixedGauss1}--\eqref{mixedGauss2}):
find $(\seh,\ueh)\in W^h_{\phi^\eps} \times Q^h_{g^\eps}$ such that 
\begin{alignat}{2} \label{Gaussaltm1}
&(\seh,\kappa_h)+\(\Div(\kappa_h),\nab \ueh\)
=G(\kappa_h)\qquad &&\forall \kappa_h\in W^h_0,\\
\label{Gaussaltm2}
&\(\Div(\seh),\nab z_h\)+\left(\frac{\det(\seh)}{(1+|\nab \ueh|^2)^2},z_h\right) 
=\bigl(\mck f^\eps,z_h\bigr)\qquad &&\forall z_h\in Q^h_0,
\end{alignat}
where
\begin{align*}
W^h_{\phi^\eps}:&=\left\{\mu_h\in W^h;\ D^2 \mu_h\nu\cdot \nu\big|_{\p\Ome}
=\phi^\eps\right\},\\
Q^h_{g^\eps}:&=\left\{v_h\in Q^h;\ v_h\big|_{\p\Ome}=g^\eps\right\}.
\end{align*}

We use the following test functions and data:
\begin{alignat*}{2}
\text{(a) } &\ue=e^{\frac{x_1^2+x_2^2}{2}}, 
&&\quad f^\eps=\frac{\(1+x_1^2+x_2^2\)e^{x_1^2+x_2^2}}
{0.1\(1+(x_1^2+x_2^2)e^{x_1^2+x_2^2}\)^2}\\
& &&\quad -\eps\(4(1+x_1^2+x_2^2)
+(2+x_1^2+x_2^2)^2\)e^{\frac{x_1^2+x_2^2}{2}},\\
&g^\eps=e^{\frac{x_1^2+x_2^2}{2}},
&&\quad \phi^\eps=e^{\frac{x^2_1+x^2_2}2}\Bigl((1+x_1^2)\nu_1^2+2x_1x_2\nu_1\nu_2
+(1+x_2^2)\nu_2^2\Bigr)\\ 
&\mck=0.1. &&\\
\text{(b) } &\ue=\frac18(x_1^2+x_2^2)^4, &&\quad \mck=0.1, 
\qquad g^\eps=\frac18 (x_1^2+x_2^2)^4, 
\end{alignat*}
\begin{align*}
&f^\eps=\frac{7\Bigl(6x_1^2x_2^2(x_1^8+x_2^8)+15x_1^4x_2^4(x_1^4+x_2^4)
+20x_1^6x_2^6+x_1^{12}+x_2^{12}\Bigr)}{0.1\Bigl(1+x_1^2(x_1^2+x_2^2)^6+x_2^2(x_2^2+x_1^2)^6\Bigr)^2} \\
&\quad\qquad -288\eps(x_1^2+x_2^2)^2, \\
&\phi^\eps=(7x_1^2+x_2^2)(x_1^2+x_2^2)^2\nu_1^2 
+12(x_1^2+x_2^2)^2x_1x_2\nu_1\nu_2  \\
&\quad\qquad +(7x_2^2+x_1^2)(x_1^2+x_2^2)^2\nu_2^2.
\end{align*}

We record the results in Table \ref{Test622Table} and plot the results
in Figure \ref{Test622Figure}.  The data clearly indicates the
following rates of convergence:
\begin{align*}
\|\ue-\ueh\|_\ho=O(h^2),\qquad
\|\ue-\ueh\|_\lt=O(h^3).
\end{align*}
These are exactly theoretical rates of convergence proved at the beginning of 
this section, indicating that our theoretical estimates for 
$\ue-\ueh$ are sharp. On the other hand, we note that the numerical rate is better
than the theoretical estimate for $\se-\seh$ which is expected
because the theoretical rate of convergence for $\se-\seh$ is clearly not optimal 
from the approximation point of view. This phenomenon also occurs 
when approximating the linear biharmonic equation by the Hermann-Miyoshi 
finite element method (cf. \cite{Falk_Osborn_1980}).

\begin{table}[tbp]
\begin{center}
\caption{Test 6.2.2: Error of $\|\ue-\ueh\|$ w.r.t. 
$h$ ($\eps=0.01$) and estimated rate of convergence} \label{Test622Table}
\begin{tabular}{lccccc}
&{\scriptsize $h$} & {\scriptsize$\|\ue-\ueh\|_\lt$(rate)}  
&{\scriptsize $\|\ue-\ueh\|_\ho$(rate)}
&{\scriptsize $\|\se-\seh\|_\lt$(rate)}\\
\noalign{\smallskip}\hline\noalign{\smallskip}
Test 6.2.2a
&2.00E--01&	2.04E--04\blank	&5.98E--03\blank	&4.40E--02\blank\\	
&1.00E--01&	2.60E--05(2.97)	&1.52E--03(1.98)	&1.68E--02(1.39)\\
&5.00E--02&	3.28E--06(2.98) 	&3.72E--04(2.03)	&6.07E--03(1.46)\\
&2.50E--02&	4.16E--07(2.98)	&9.25E--05(2.01)	&2.19E--03(1.47)\\
&1.25E--02&	5.24E--08(2.99)	&2.31E--05(2.00)	&7.87E--04(1.48)\\
\noalign{\smallskip}\hline\noalign{\smallskip}
Test 6.2.2b
&2.00E--01& 	2.05E--03\blank	&4.72E--02\blank	&3.64E--01\blank\\	
&1.00E--01&	2.77E--04(2.89)	&1.19E--02(1.99)	&1.46E--01(1.32)\\
&5.00E--02&	3.66E--05(2.92)	&2.89E--03(2.04)	&5.44E--02(1.42)\\
&2.50E--02&	4.72E--06(2.95)	&7.09E--04(2.03)	&1.97E--02(1.47)\\
&1.25E--02&	6.02E--07(2.97)	&1.76E--04(2.01)	&7.04E--03(1.48)\\
\end{tabular}
\end{center}
\end{table}

\begin{figure}[ht]\begin{center}
\includegraphics[scale=0.5]{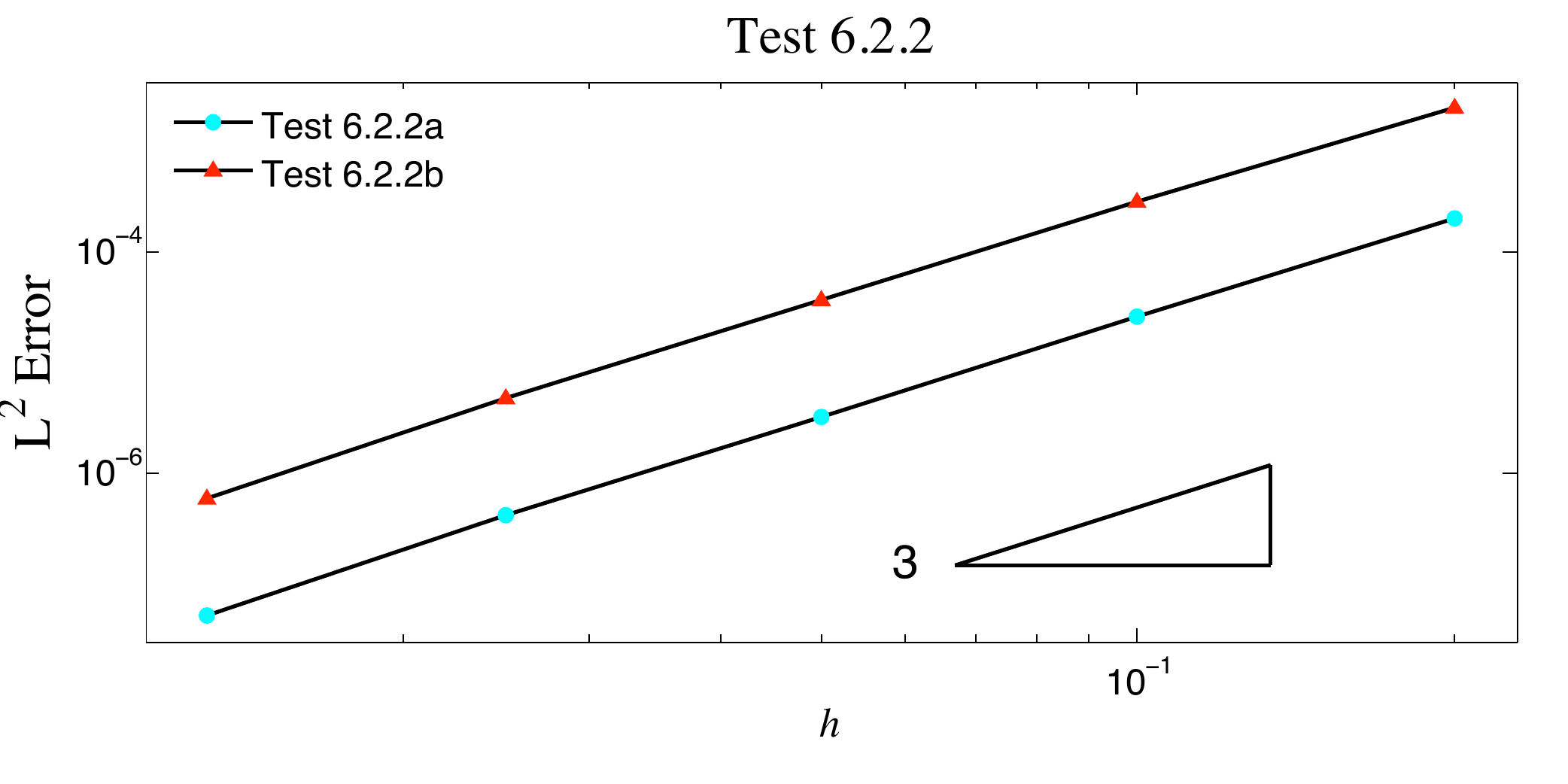}\\
\includegraphics[scale=0.5]{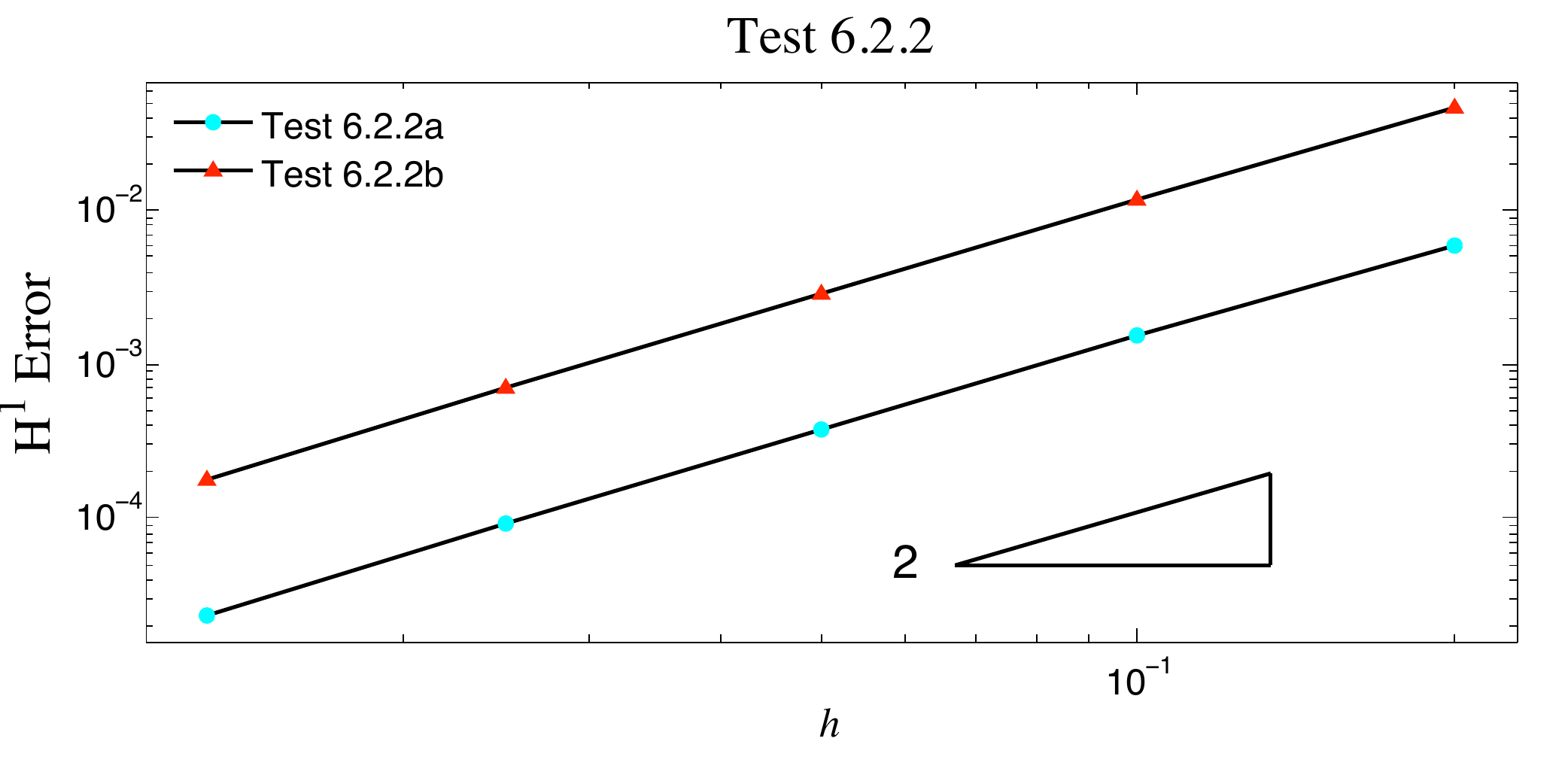}\\
\includegraphics[scale=0.5]{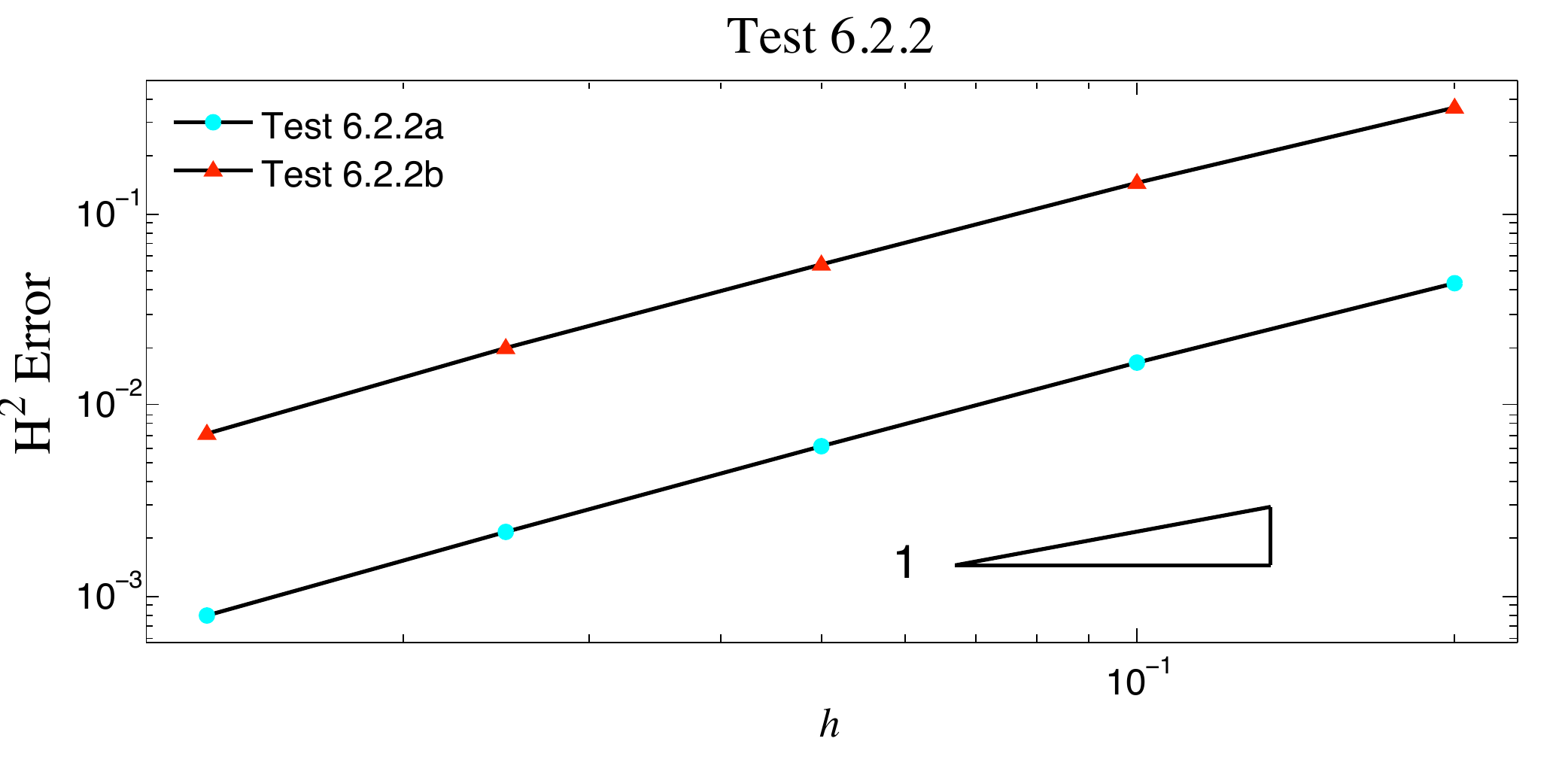}\\
\end{center}
\caption{{Test 6.2.2.  Change of $\|\ue-\ueh\|$ w.r.t.
$h$ ($\eps=0.01$)}} \label{Test622Figure}
\end{figure}

\medskip
\subsubsection*{Test 6.2.3}
For this test, we use our numerical method to approximate
$\mck^*$ and compare our results with those found in
\cite{Baginski_Whitaker96}, where the method of continuity 
(which was used to prove the existence of classical solutions
to the equation of prescribed Gauss curvature) was implemented at the
discrete level. We compute \eqref{absGaussMM1}--\eqref{absGaussMM3}
with the following Dirichlet boundary conditions and domains 
as used in \cite{Baginski_Whitaker96}:

\begin{alignat*}{3}
&\text{(a) }g&&=\sqrt{1-x_1^2-x_2^2},\qquad &&\Ome=(-0.57,0.57)^2.\\
&\text{(b) }g&&=1-x_1^2-x_2^2,\qquad &&\Ome=(-0.57,0.57)^2.\\
&\text{(c) }g&&=1-(x_1-0.075)^2-(x_2-0.015)^2,\qquad &&\Ome=(-0.57,0.57)^2.\\
&\text{(d) }g&&=\sqrt{1-x_1^2-x_2^2},\qquad &&\Ome=(-0.72,0.72)\times (-0.36,0.36).\\
&\text{(e) }g&&=1-x_1^2-x_2^2,\qquad &&\Ome=(-0.72,0.72)\times (-0.36,0.36).\\
&\text{(f) }g&&=1-(x_1-0.075)^2-(x_2-0.015)^2,\qquad &&\Ome=(-0.72,0.72)\times (-0.36,0.36).
\end{alignat*}

We remark that for the above choice of data, the solution of the prescribed
Gauss curvature equation is concave, and so
we set $\eps<0$ in order to approximate the solution
(see \cite{Feng2,Neilan_thesis} for further explanation).  
Table \ref{Test623Table} compares our results and those
of \cite{Baginski_Whitaker96}. 
Table \ref{Test623Table} shows that our numerical method
gives comparable values to those computed in \cite{Baginski_Whitaker96}.
Finally, we plot the computed solution of Test 6.3a for $\mck$-values $0,1$, and $2$
in Figure \ref{Test623Figure}.  We also compute and plot the corresponding convex
solution (with $g=-\sqrt{1-x_1^2-x^2_2}$ and $\eps=0.001$) for comparison.

\begin{table}
\begin{center}
\caption{Test 6.2.3. Computed $\mck^*$ with $\eps=-0.001$ and $h=0.031$} 
\label{Test623Table}
\begin{tabular}{lll}
& Computed $\mck^*$ & $\mck^*$ in \cite{Baginski_Whitaker96}\\
\noalign{\smallskip}\hline\noalign{\smallskip}
Test 6.2.3a  & 2.07 & 2.10\\
Test 6.2.3b  & 2.20 & 2.24\\
Test 6.2.3c  & 1.95 & 1.85\\
Test 6.2.3d  & 2.68 & 2.61\\
Test 6.2.3e  & 2.71 & 2.73\\
Test 6.2.3f   & 2.20 & 2.27\\
\end{tabular}
\end{center}
\end{table}

\begin{figure}[ht]
\begin{center}
\includegraphics[scale=0.055]{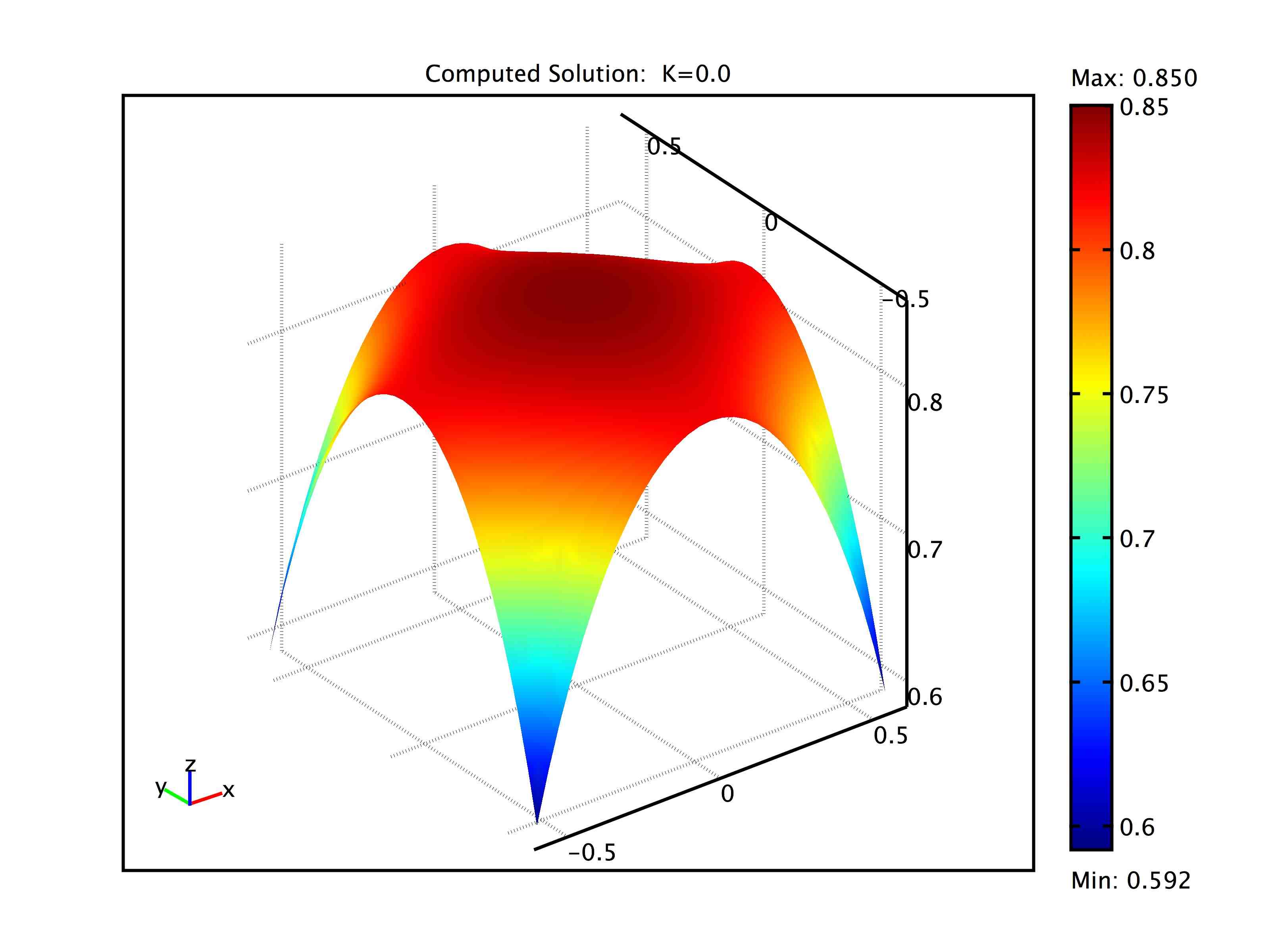}
\includegraphics[scale=0.055]{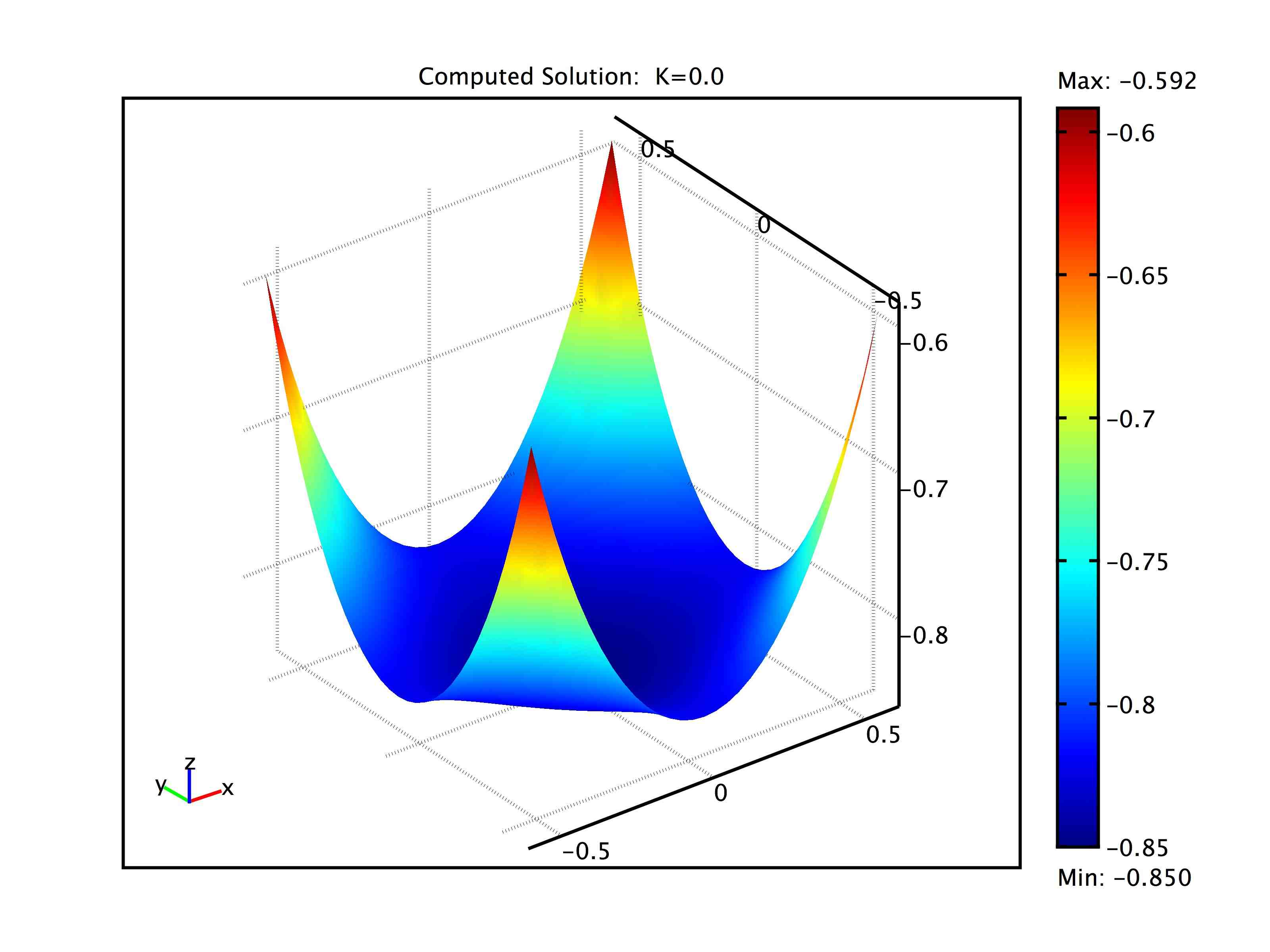}\\
\includegraphics[scale=0.055]{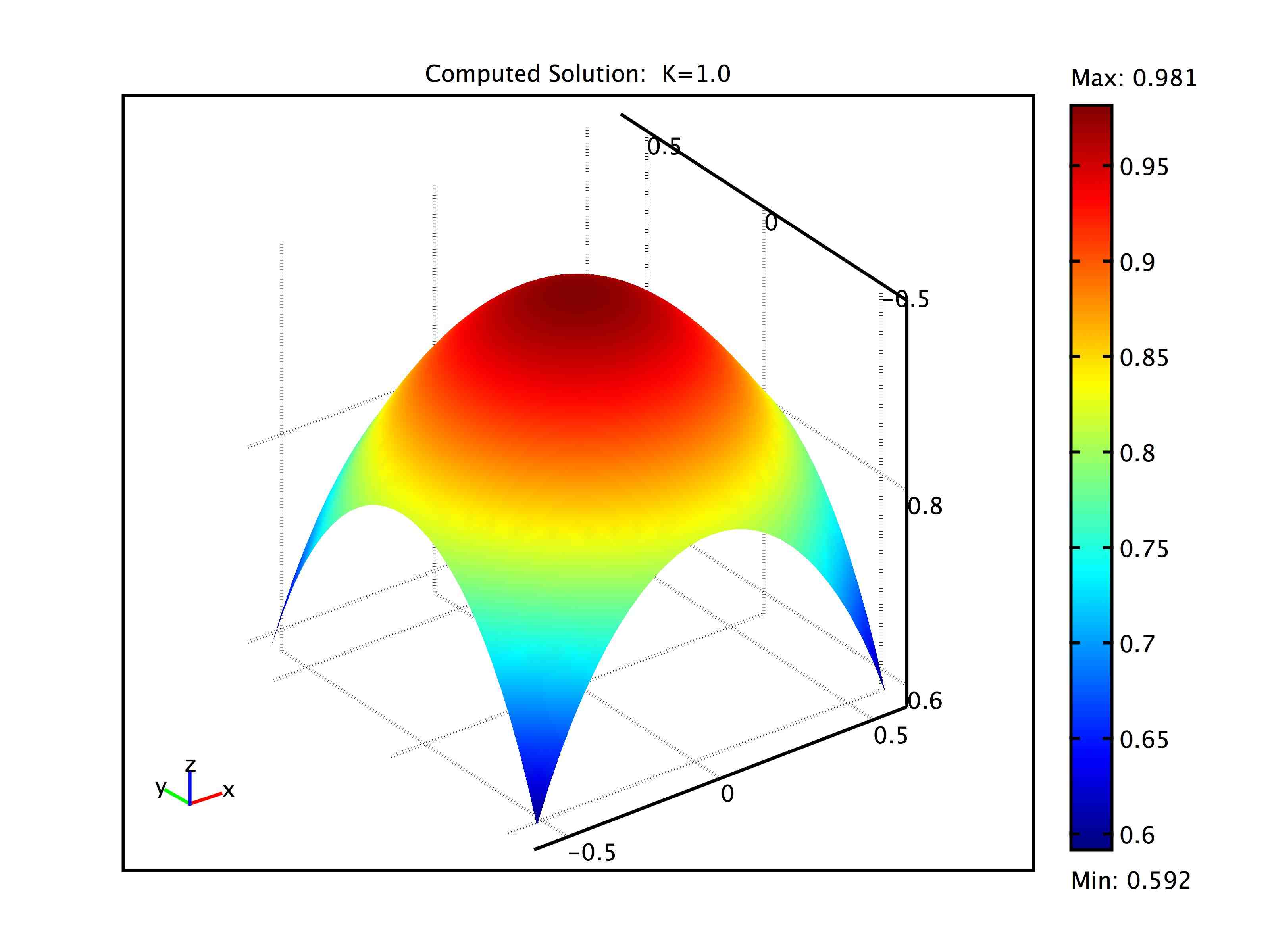}
\includegraphics[scale=0.055]{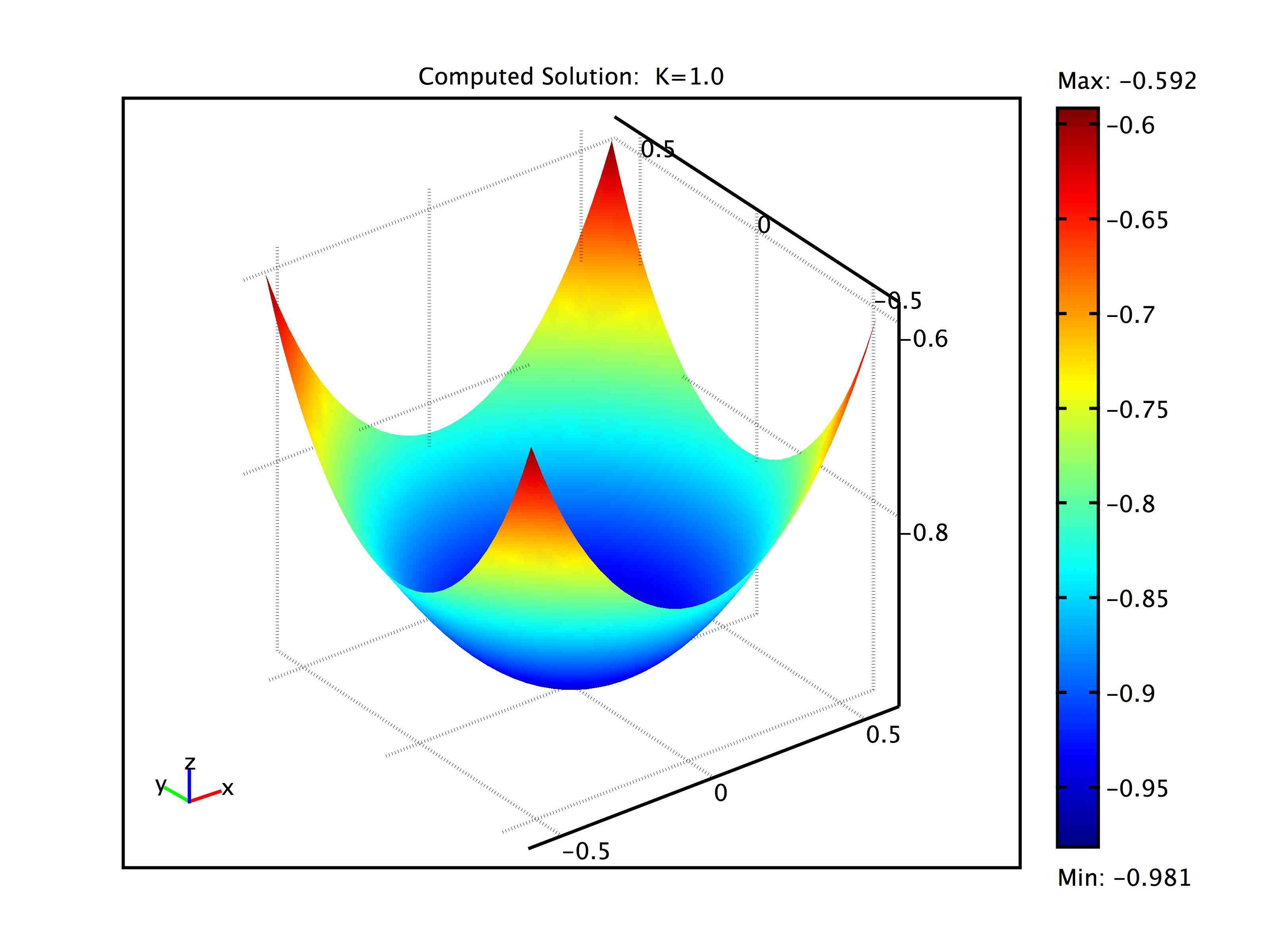}\\
\includegraphics[scale=0.055]{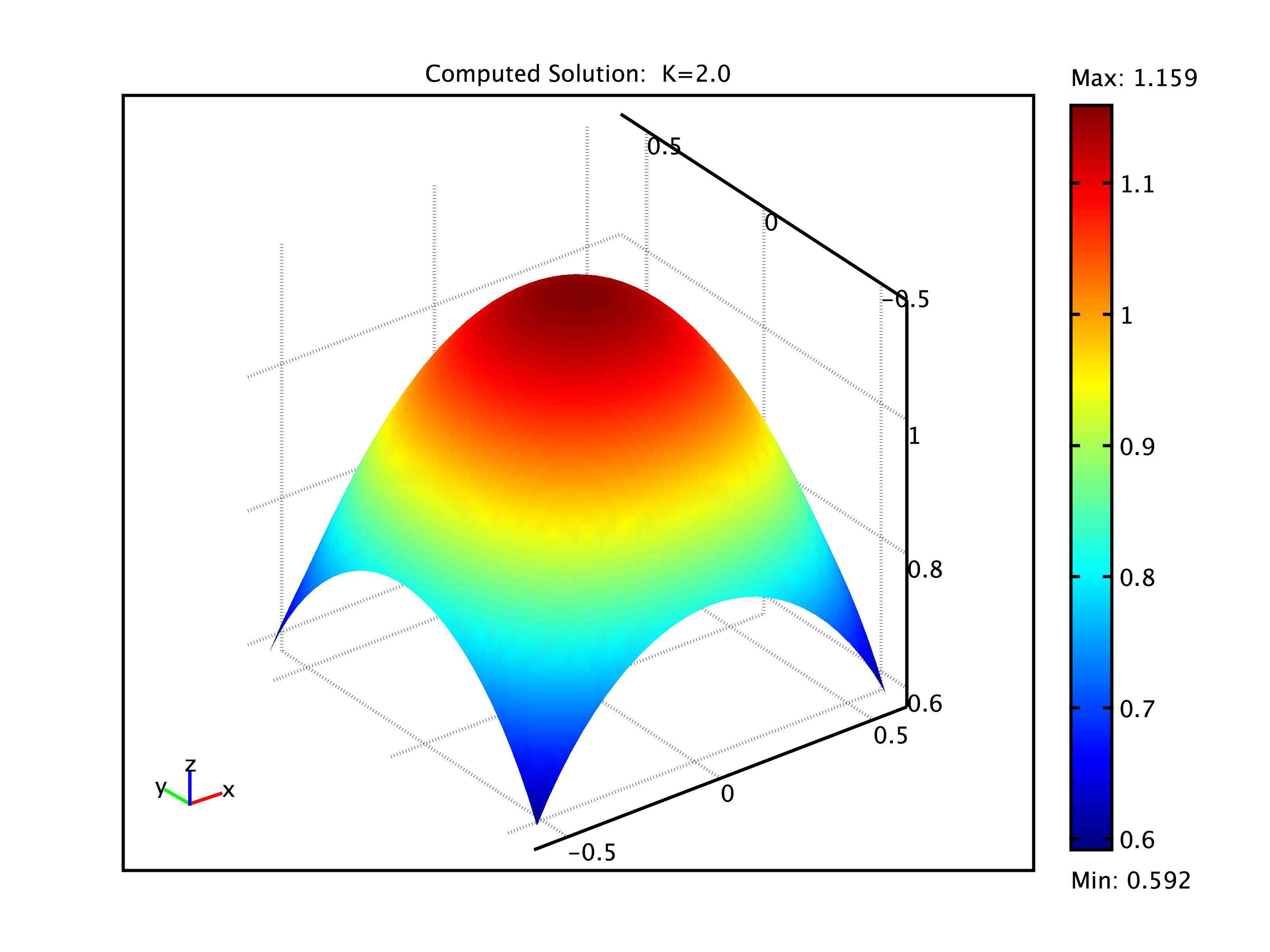}
\includegraphics[scale=0.055]{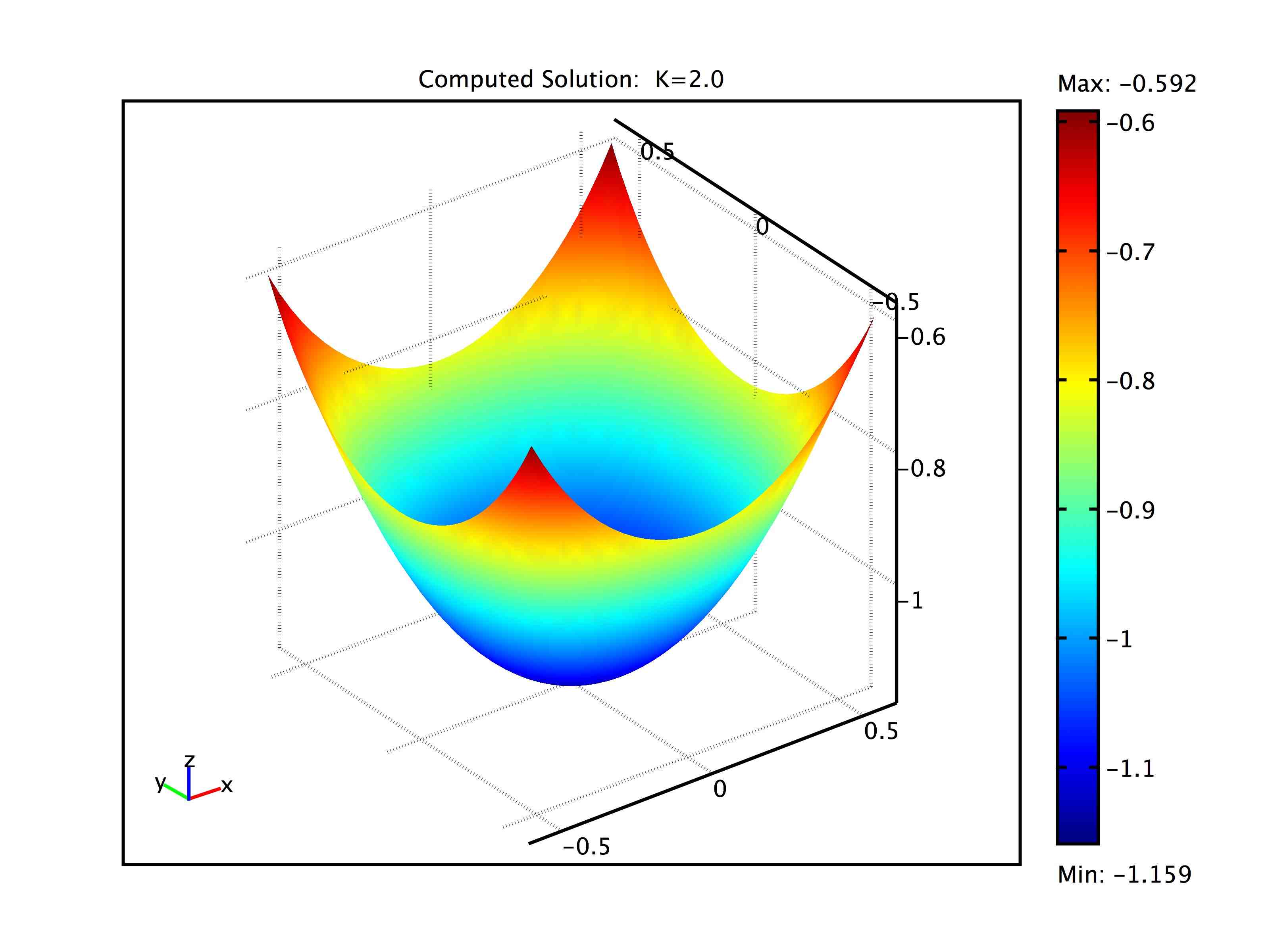}
\end{center}
\caption{Test 6.2.3a.  Computed concave solution (left)
and convex solution (right) with $\mck=0.0$ (top), $\mck=1$ (middle), 
and $\mck =2$ (bottom).  $h=0.025$ and $\eps=-0.001$ to compute the concave
solution, where as $h=0.025$ and $\eps=0.001$ to compute the convex solution.
\label{Test623Figure}}
\end{figure}

%% file: chapter6c.tex
\section{The infinity-Laplacian equation}\label{chapter-6-sec-4}
In this section, we consider finite element approximations of 
the infinity-Laplacian equation:
\begin{alignat}{2} \label{IL1}
\Del_\infty u&=0\qquad &&\text{in }\Ome,\\
\label{IL2}u&=g\qquad &&\text{on }\p\Ome,
\end{alignat}
where
\begin{align*}
\Del_\infty u:=\frac{D^2u\nab u\cdot\nab u}{|\nab u|^2}
=\frac{1}{|\nab u|^2}\sum_{i,j=1}^n \frac{\p^2 u}{\p x_i\p x_j}\frac{\p u}{\p x_i}\frac{\p u}{\p x_j},
\end{align*}
and $g\in C(\p\Ome)$. We note that unlike the PDEs considered
in the previous two sections, the infinity-Laplacian equation is not
fully nonlinear, but rather quasilinear.  Still, its non-divergence form,
degeneracy, and strong nonlinearlity in the first order derivatives makes
the PDE difficult to study and approximate 
(\cite{Bhattacharya91,Evans_Yu,Oberman04}).
In particular, the linearization of the operator $\Del_\infty$ gives
a degenerate linear differential operator which serves as a perfect
example for testing the mixed finite element theory developed 
in Section \ref{chapter-5-sec-5}.

\begin{remark}
As pointed out in Section \ref{chapter-6-sec-2}, both
$\widetilde{\Del}_\infty v:=D^2 v \nab v\cdot \nab v$ and $\Del_\infty v
:=\frac{D^2 v \nab v\cdot \nab v}{|\nab v|^2}$ are called  
the infinity-Laplacian in the literature 
\cite{Aronsson_Crandall_Juutinen04,Evans07} because they give the same
infinity-Laplacian equation. Here we adopt the latter definition 
for a reason which will be clear later (see Remark \ref{Remark613ii}). 
\end{remark}

The infinity-Laplacian equation \eqref{IL1} arises 
from the so-called ``absolute minimal problem'' which is stated as 
follows: {\em Given a continuous function $g:\p\Ome\mapsto \mathbf{R}$, 
find a function $u:\overline{\Ome}\mapsto\mathbf{R}$
such that for each $V\subset \Ome$ and each $v\in C(\overline{V})$
 $u\big|_{\p V}=v\big|_{\p V}$ implies ${\rm esssup}_V |\nab u|
\le {\rm esssup}_V |\nab v|$.}
The equation finds applications in image processing and many other 
fields, we refer the reader to two recent survey papers
\cite{Aronsson_Crandall_Juutinen04,Crandall_05} for detailed 
discussions on the latest developments on PDE analysis and 
applications of the infinity-Laplacian equation.
 

Like the equation of prescribed Gauss curvature, we have some
flexibility in defining $F(D^2 u,\nab u,u,x)$. One possibility
is to define $F(D^2 u,\nab u,u,x) := -\widetilde{\Del}_\infty u$, but 
this leads to difficulties in the linearization 
(see Remark \ref{Remark613ii}).  Here, we define
\begin{align}\label{ILFdef}
F(D^2 u,\nab u,u,x) := -\frac{\widetilde{\Del}_\infty u}{|\nab u|^2+\gamma}
=-\frac{D^2u\nab u\cdot\nab u}{|\nab u|^2+\gamma},
\end{align}
where $\gamma>0$ is a positive parameter that will be specified later.  
The reason for introducing $\gamma$ is to avoid dividing by zero
in the expression.

It is easy to check that
\begin{align*}
F^\prime[v](w)&=-\frac{D^2 w\nab v\cdot \nab v+2D^2 v\nab v\cdot \nab w}{|\nab v|^2 +\gamma}
+2\frac{\widetilde{\Del}_\infty v\nab v\cdot \nab w}{(|\nab v|^2+\gamma)^2},\\
F^\prime[\mu,v](\kappa,w)&=-\frac{\kappa \nab v\cdot \nab v
+2\mu\nab v\cdot \nab w}{|\nab v|^2 +\gamma}+2\frac{\(\mu\nab v\cdot \nab v\) \nab v\cdot \nab w}{(|\nab v|^2+\gamma)^2}.
\end{align*}
The vanishing moment approximation becomes
\begin{alignat}{2}
\label{ILVMM1}\eps\Del^2\ue
-\frac{\widetilde{\Del}_\infty \ue}{|\nab \ue|^2+\gamma}
&=0\qquad &&\text{in }\Ome,\\
\label{ILVMM2}\ue&=g\qquad &&\text{on }\p\Ome,\\
\label{ILVMM3}\Del\ue&=\eps\qquad &&\text{on }\p\Ome.
\end{alignat}
The linearization of 
\[
G_\eps(\ue)=\eps\Del^2\ue
-\frac{\widetilde{\Del}_\infty \ue}{|\nab \ue|^2+\gamma}
\]
at the solution $\ue$ is
\begin{align*}
G_\eps^\prime[\ue](v)=\eps\Del^2 v-\frac{D^2v\nab \ue\cdot\nab\ue 
+2D^2\ue\nab \ue\cdot \nab v}{|\nab \ue|^2+\gamma}
+2\frac{\widetilde{\Del}_\infty\ue\nab\ue\cdot\nab v}{(|\nab \ue|^2+\gamma)^2}.
\end{align*}

Numerical tests indicate that there exists a unique solution
to \eqref{ILVMM1}--\eqref{ILVMM3} (cf. Subsection \ref{Section-6.3.3} and \cite{Feng2}), 
and therefore, for the continuation of this section, we assume that there exists a unique
 solution to \eqref{ILVMM1}--\eqref{ILVMM3}.

Before formulating and analyzing finite element methods
for \eqref{ILVMM1}--\eqref{ILVMM3}, we first state
the following two identities.

\begin{lem}\label{ILfirstLemma}
Suppose that $n=2$.  Then there holds the following identity:
\begin{align*}
|\nab w|^2 \(|\Del w|^2 -|D^2 w|^2\)=\(\Del w\nab w-D^2 w\nab w\)\cdot \nab (|\nab w|^2).
\end{align*}
\end{lem}
The proof of of Lemma \ref{ILfirstLemma} is a straight-forward (and tedious) calculation, so we omit it.
Next, with the help of Lemma \ref{ILfirstLemma}, we are able to establish the following
identity.

\begin{lem}\label{ILGardinglem}
Suppose that $n=2$. 
Then for any $v\in H^1_0(\Ome)$, there holds
\begin{align}\label{ILGarding}
&\bl F^\prime[\ue](v),v\br
=\left\|\frac{\nab v\cdot \nab \ue}{\sqrt{|\nab \ue|^2+\gamma}}\right\|_\lt^2
-\gamma \left(\frac{\det(D^2\ue)}{(|\nab \ue|^2+\gamma)^2},v^2\right).
\end{align}
\end{lem}

\begin{proof}
Integrating by parts we get
\begin{align*}
&\left(\frac{D^2v\nab \ue\cdot \nab \ue}{|\nab \ue|^2+\gamma},v\right)\\
&\qquad=-\left(\frac{\nab v\cdot \nab \ue}{|\nab \ue|^2+\gamma},\nab v\cdot \nab \ue\right)
-\left(\frac{D^2\ue \nab \ue\cdot \nab v}{|\nab \ue|^2+\gamma},v\right)-\left(\frac{\nab v\cdot \nab \ue}{|\nab \ue|^2+\gamma},\Del \ue v\right)\\
&\qquad\qquad +\left(\frac{\nab v\cdot \nab \ue}{(|\nab \ue|^2+\gamma)^2},\nab (|\nab \ue|^2)\cdot \nab \ue v\right)\\
&\qquad=-\left(\frac{\nab v\cdot \nab \ue}{|\nab \ue|^2+\gamma},\nab v\cdot \nab \ue\right)
-\left(\frac{D^2\ue \nab \ue\cdot \nab v}{|\nab \ue|^2+\gamma},v\right)-\left(\frac{\nab v\cdot \nab \ue}{|\nab \ue|^2+\gamma},\Del \ue v\right)\\
&\qquad\qquad +2\left(\frac{\nab v\cdot \nab \ue}{(|\nab \ue|^2+\gamma)^2},
\widetilde{\Del}_\infty \ue v \right).
\end{align*}
Thus,
\begin{align*}
&\bl F^\prime[\ue](v),v\br\\
&\qquad=-\left(\frac{D^2v\nab \ue\cdot \nab \ue+2D^2 \ue\nab \ue\cdot \nab v}{|\nab \ue|^2+\gamma},v\right)
+2\left(\frac{\widetilde{\Del}_\infty \ue \nab \ue\cdot \nab v}{(|\nab \ue|^2+\gamma)^2},v\right)\\
&\qquad=\left(\frac{\nab v\cdot \nab \ue}{|\nab \ue|^2+\gamma},\nab v\cdot \nab \ue\right)
-\left(\frac{D^2\ue \nab \ue\cdot \nab v}{|\nab \ue|^2+\gamma},v\right)
+\left(\frac{\nab v\cdot \nab \ue}{|\nab \ue|^2+\gamma},\Del \ue v\right)\\
&\qquad=\left\|\frac{\nab v\cdot \nab \ue}{\sqrt{|\nab \ue|^2+\gamma}}\right\|_\lt^2
+\frac12 \left(\frac{\Del \ue\nab \ue-D^2 \ue\nab \ue}{|\nab \ue|^2+\gamma},\nab (v^2)\right)\\
&\qquad=\left\|\frac{\nab v\cdot \nab \ue}{\sqrt{|\nab \ue|^2+\gamma}}\right\|_\lt^2
-\frac12 \left(\Div\left(\frac{\Del \ue\nab \ue-D^2 \ue\nab \ue}{|\nab \ue|^2+\gamma}\right),v^2\right).
\end{align*}

Noting that 
\begin{align*}
&\Div\left(\frac{\Del \ue\nab \ue-D^2 \ue\nab \ue}{|\nab \ue|^2+\gamma}\right)\\
&\qquad=\frac{|\Del \ue|^2-|D^2 \ue|^2}{|\nab \ue|^2+\gamma}-\frac{\(\Del \ue\nab \ue
-D^2 \ue\nab \ue\)\cdot \nab (|\nab \ue|^2)}{(|\nab \ue|^2+\gamma)^2}\\
&\qquad=\frac{\(|\nab \ue|^2+\gamma\)\(|\Del \ue|^2-|D^2 \ue|^2\)
-\(\Del \ue\nab \ue-D^2 \ue\nab \ue\)\cdot \nab (|\nab \ue|^2)}{(|\nab \ue|^2+\gamma)^2},
\end{align*}
we have by Lemma \ref{ILfirstLemma},
\begin{align*}
\bl F^\prime[\ue](v),v\br
&=\left\|\frac{\nab v\cdot \nab \ue}{\sqrt{|\nab \ue|^2+\gamma}}\right\|_\lt^2
-\frac{\gamma}2 \left(\frac{ |\Del \ue|^2-|D^2 \ue|^2}{(|\nab \ue|^2+\gamma)^2},v^2\right)\\
&=\left\|\frac{\nab v\cdot \nab \ue}{\sqrt{|\nab \ue|^2+\gamma}}\right\|_\lt^2
-\gamma \left(\frac{\det(D^2\ue)}{(|\nab \ue|^2+\gamma)^2},v^2\right).
\end{align*}

\end{proof}

\begin{remark}
(a) \label{Remark613ii}
Unlike the two PDEs analyzed in the previous sections, the operator 
$F^\prime[u^\vepsi]$ is not uniformly elliptic,  that is, 
there does not exist constants $K_0,K_1>0$ such that
\begin{align*}
\bl F^\prime[\ue](v),v\br\ge K_1\|v\|_\ho-K_0\|v\|_\lt^2
\qquad \forall v\in H^1_0(\Ome).
\end{align*}
Thus, when constructing and analyzing mixed finite element 
methods for \eqref{ILVMM1}--\eqref{ILVMM3}, we must 
instead use the abstract analysis of Section \ref{chapter-5-sec-5}, 
which is developed exactly with such a case in mind.

(b) If we set $F(D^2 u,\nab u,u,x)=-\widetilde{\Del}_\infty u
:=D^2u\nab u\cdot\nab u$, then the linearization of $F$ would be
\begin{align*}
\Fp[v](w) = -D^2 w\nab v\cdot \nab v-2D^2 v\nab v\cdot \nab w,
\end{align*}
and it is an easy exercise to see that
\begin{align*}
\bl \Fp[\ue](v),v\br &=\bigl\|\nab v\cdot \nab \ue\bigr\|_\lt^2
-\frac12 \left(|\Del \ue|^2-|D^2 \ue|^2,v^2\right).
\end{align*}
Thus, the reason we use the definition \eqref{ILFdef} is so that we are 
able to control the zeroth order term in the linearization as shown 
in the following corollary. Nevertheless, numerical experiments
of \cite{Feng2,Neilan_thesis} indicate that the vanishing moment 
method with $F(D^2 u,\nab u,u,x)=-\widetilde{\Del}_\infty u$ also work
well for the infinity-Laplacian equation.
\end{remark}

\begin{cor}\label{ILGardingcor}
Suppose $n=2$.  Then there exists a constant $\gamma_0=\gamma_0(\eps)>0$, 
such that for $\gamma\in (0,\gamma_0]$, there holds
\begin{align}\label{ILGardingcorline}
\bl \Gp[\ue](v),v\br\ge C\eps \|v\|_\htw^2\qquad \forall v\in V_0.
\end{align}
\end{cor}
\begin{proof}
If $\|\nab \ue\|_{L^\infty}\neq 0$, then 
by \eqref{ILGarding}, we have
\begin{align*}
\bl \Gp[\ue](v),v\br
&\ge C\eps \|v\|_\htw^2+
\left\|\frac{\nab v\cdot \nab \ue}{\sqrt{|\nab \ue|^2+\gamma}}\right\|_\lt^2
-\gamma \left(\frac{\det(D^2\ue)}{(|\nab \ue|^2+\gamma)^2},v^2\right)\\
&\ge C\eps \|v\|_\htw^2-\gamma \left\|\frac{\det(D^2\ue)}{(|\nab \ue|^2+\gamma)^2}\right\|_{L^\infty} \|v\|_\lt^2\\
&\ge C\left(\eps -\gamma \frac{ \|\ue\|^2_{W^{2,\infty}}}{\|\nab \ue\|_{L^\infty}^4}\right)\|v\|_\htw^2.
\end{align*}
Choosing $\gamma_0 = \frac{\eps \|\nab \ue\|_{L^{\infty}}^4}{2\|\ue\|_{W^{2,\infty}}^2}$, we get the desired result.

On the other hand, if $\|\nab \ue\|_{L^\infty} =0$, then 
$\ue\equiv \mbox{const}$ and $F'[\ue]\equiv 0$, 
then we can choose $\gamma_0$ to be any positive number to obtain
\begin{align*}
\bl \Gp[\ue](v),v\br  =\eps \|\Del v\|_{L^2}^2.
\end{align*}

\end{proof}

\subsection{Conforming finite element methods
for the infinity-Laplacian equation}

The finite element method for \eqref{ILVMM1}--\eqref{ILVMM3}
is defined as finding $\ueh\in V^h_g$ such that
\begin{align}\label{ILFEM}
\eps(\Del \ueh,\Del v_h)-\left(\frac{\widetilde{\Del}_\infty \ueh}{|\nab \ueh|^2+\gamma},v_h\right)
&=\left\langle\eps^2,\normd{v_h}\right\rangle_{\p\Ome}
\qquad \forall v_h\in V^h_0,
\end{align}
where we assume that $\gamma\in (0,\gamma_0]$ for the rest of
this subsection so that the inequality \eqref{ILGardingcorline} holds.
Furthermore, we assume that $\|\nab \ue\|_{L^\infty}\ge 1$.
This assumption is not necessary in our analysis, but is does simplify our 
presentation (cf.~\eqref{quotientbound}).

The goal of this section is to apply the abstract framework of 
Chapter \ref{chapter-4} to the finite element method \eqref{ILFEM}.  
Specifically, we now show that conditions {\rm [A1]--[A5]} 
hold, which will then gives us the existence, uniqueness,
and error estimates of the solution to \eqref{ILFEM}.
Of particular interest is the constants' explicit dependence on 
$\eps$ in the error estimates.  
We summarize our findings in the following theorem.

\begin{thm}\label{ILConformingthm}
Suppose $n=2$, and let $\ue\in H^s(\Ome)$ be 
the solution to \eqref{ILVMM1}--\eqref{ILVMM3} with $s\ge 3$. 
Then there exists an $h_3=h_3(\eps)>0$ such that for $h\le h_3$,
\eqref{ILFEM} has a unique solution. Furthermore, there holds
the following error estimates:
\begin{align}\label{ILConformError1}
\|\ue-\ueh\|_\htw&\le C_7 h^{\ell-2}\|\ue\|_\hl,\\
\label{ILConformError2}
\|\ue-\ueh\|_\lt&\le C_8\Bigl(C_2 h^{\ell}\|\ue\|_\hl
+C_7 L(h)h^{2\ell-4}\|\ue\|_\hl \Bigr),
\end{align}
where
\begin{align*}
C_2&=|\ue|_{W^{2,\frac{q}{q-1}}},\qquad C_7&=C\eps^{-2}\gamma^{-\frac12}|\ue|_{W^{2,\infty}} 
|\ue|_{W^{2,\frac{q}{q-1}}},\qquad C_8=CC_7C_R.
\end{align*}
where $q$ is a number in the interval $(1,\infty)$, 
$C_R$ is defined by \eqref{ILConformline1},
$L(h)$ is defined by \eqref{ILLhdef},
$\ell = \min\{s,k+1\}$, and $k$ denotes the polynomial degree of the finite element space.
\end{thm}
\begin{proof}

First, Corollary \ref{ILGardingcor} implies that $\(\Gp[\ue]\)^*$ is an 
isomorphism from $V_0$ to $V_0^*$.

Next, we note that
\begin{align}\label{quotientbound}
\left\|\frac{|\nab \ue|^p}{(|\nab \ue|^2+\gamma)^m}\right\|_{L^\infty}
&\le \|\nab \ue\|_{L^\infty}^{p-2m}\le 1, \qquad 2m\ge p\ge 1.
\end{align}
Thus, for any $v,w\in V_0$, we have by using Sobolev inequalities for any $q\in (1,\infty)$
\begin{align*}
\bl \Fp[\ue](v),w\br
&=\left(\frac{\nab \ue\cdot \nab v}{|\nab \ue|^2+\gamma},\nab \ue\cdot \nab w\right)
-\left(\frac{D^2 \ue\nab \ue\cdot \nab v-\nab v\cdot \nab \ue \Del \ue}{|\nab \ue|^2+\gamma}
,w\right)\\
&\le \left\|\frac{|\nab \ue|^2}{|\nab \ue|^2+\gamma}\right\|_{L^\infty} \|\nab v\|_{L^2}\|\nab w\|_{L^2}
+\left\|\frac{|D^2 \ue\nab \ue|}{|\nab \ue|^2+\gamma}\right\|_{L^\frac{q}{q-1}} \|\nab v\|_{L^q}\|w\|_{L^\infty}\\
&\qquad +\left\|\frac{|\nab \ue|\Del \ue}{|\nab \ue|^2+\gamma}\right\|_{L^\frac{q}{q-1}} \|\nab v\|_{L^q}\|w\|_{L^\infty}\\
&\le C\Bigl(\|\nab v\|_{L^2}\|\nab w\|_{L^2}+\|D^2 \ue\|_{L^\frac{q}{q-1}}\|\nab v\|_{L^q}\|w\|_{L^\infty}\Bigr)\\
&\le C\|D^2 \ue\|_{L^\frac{q}{q-1}}\|v\|_\htw\|w\|_\htw.
\end{align*}

Next, by the standard PDE theory, if we assume that $\ue$ and $\p\Ome$
are sufficiently smooth, and if $v\in V_0$ solves
\[
\bl \Gp[\ue](v),w\br=(\varphi,w)\qquad \forall w\in V_0,
\]
where $\varphi$ is some $L^2(\Ome)$ function, then 
$v\in H^p(\Ome)$ for $p\geq 3$.  Furthermore, in view of Remark 
\ref{Remark44ii}, and the inequalities (which come from \eqref{quotientbound})
\begin{align*}
\left\|\frac{\p F(\ue)}{\p r_{ij}}\right\|_{L^\infty}
&=\left\|\frac{\frac{\p \ue}{\p x_i}\frac{\p \ue}{\p x_j}}{|\nab \ue|^2+\gamma}\right\|_{L^\infty}
\le \left\|\frac{|\nab \ue|^2}{|\nab \ue|^2+\gamma}\right\|_{L^\infty}\le C,\\
\left\|\frac{\p F(\ue)}{\p p_{i}}\right\|_{L^\infty}
&\le 2\left\|\frac{\(D^2\ue \nab \ue\)_i}{|\nab \ue|^2+\gamma}\right\|_{L^\infty}
+2\left\|\frac{\Del_\infty \ue (D^2\ue\nab \ue)_i}{(|\nab \ue|^2+\gamma)^2}\right\|_{L^\infty}\\
&\le C\left(\|D^2\ue\|_{L^\infty}\left\|\frac{|\nab \ue|}{|\nab \ue|^2+\gamma}\right\|_{L^\infty}
+\|D^2\ue\|_{L^\infty}^2 \left\|\frac{|\nab \ue|^3}{(|\nab \ue|^2+\gamma)^2}\right\|_{L^\infty}\right)\\
&\le C\|D^2\ue\|_{L^\infty}^2,
\end{align*}
we have that in the case $p=4$
\begin{align*}
\|v\|_{H^4}\le C\eps^{-2}\|D^2\ue\|_{L^\infty}^2\|\varphi\|_\lt.
\end{align*}
It then follows that condition {\rm [A2]} holds with
\begin{alignat}{2}
\label{ILConformline1}
&C_0=C\eps  ,\qquad &&C_1=C\eps,\qquad C_2 = C\|D^2\ue\|_{L^\frac{q}{q-1}},\\
&\nonum p=4,\qquad
&&C_R=C\eps^{-2}\|\ue\|_{W^{2,\infty}}^2.
\end{alignat}

It then follows from Theorem \ref{abstractbound1thm} that
\begin{align}
\label{ILConformline2}
C_4= CC_2\eps^{-1},
\qquad C_5=CC_2^2C_R\eps^{-1},
\qquad h_0=C\(C_2C_R\)^{-\frac{1}{2}}.
\end{align}

To confirm {\rm [A3]--[A4]}, we set
\begin{align}
\label{ILConformYdef}
Y= W^{2,\frac{q}{q-1}}(\Ome),\qquad \|\cdot\|_Y=\gamma^{-\frac12}|\cdot|_{W^{2,\frac{q-1}q}},
\end{align}
where $q\in (1,\infty)$.

Using a Sobolev inequality and the inequality \eqref{quotientbound}, 
we have for any $y\in Y,\ v,w\in V_0$
\begin{align*}
\bl \Fp[y](v),w\br
&=\left(\frac{\nab y\cdot \nab v}{|\nab y|^2+\gamma},\nab y\cdot \nab w\right)
-\left(\frac{D^2y\nab y\cdot \nab v-\nab v\cdot \nab y\Del y}{|\nab y|^2+\gamma},w\right)\\
&\le C\left(\left\|\frac{|\nab y|^2}{|\nab y|^2+\gamma}\right\|_{L^\infty} \|\nab v\|_{L^2}\|\nab w\|_{L^2}\right.\\
&\qquad\qquad+\left.\left\|\frac{|\nab y|}{|\nab y|^2+\gamma}\right\|_{L^\infty}\|D^2y\|_{L^{\frac{q}{q-1}}}\|\nab v\|_{L^q}\|w\|_{L^\infty}\right)\\
&\le C\gamma^{-\frac12}\|D^2 y\|_{L^\frac{q}{q-1}}\|v\|_\htw\|w\|_\htw.
\end{align*}
Here, we have used the fact that for $p\le 2m$ and $x\ge 0$, $\left|\frac{x^p}{(x^2+\gamma)^m}\right|\le C\gamma^{\frac{p-2m}2}$
for some constant that only depends on $p$ and $m$.

It then follows from this calculation that
\begin{align*}
\sup_{y\in Y} \frac{\left\|\Fp[y]\right\|_{VV^*}}{\|y\|_Y}\le C.
\end{align*}
Thus, {\rm [A3]--[A4]} holds.

To verify condition {\rm [A5]}, we first note for 
any $v_h\in V_g^h$ and $w\in V_0$
\begin{align}
&\label{ILA5}\(\Fp[\ue]-\Fp[v_h]\)(w)\\
&\nonum=\left(\frac{D^2 v_h\nab v_h}{|\nab v_h|^2+\gamma}-\frac{D^2 \ue\nab \ue}{|\nab \ue|^2+\gamma}\right)\cdot \nab w
 +\frac{D^2 w\nab v_h\cdot \nab v_h}{|\nab v_h|^2+\gamma}-\frac{D^2 w\nab \ue\cdot \nab \ue}{|\nab \ue|^2+\gamma}\\
&\nonum\quad+\frac{2\widetilde{\Del}_\infty \ue\nab \ue\cdot \nab w}{(|\nab \ue|^2+\gamma)^2}
-\frac{2\widetilde{\Del}_\infty v_h\nab v_h\cdot \nab w}{(|\nab v_h|^2+\gamma)^2}\\
&\nonum=\frac{\(D^2 v_h\nab v_h-D^2 \ue\nab \ue\)\cdot \nab w}{|\nab v_h|^2+\gamma}
-D^2 \ue\nab \ue\cdot \nab w \left(\frac{|\nab v_h|^2-|\nab \ue|^2}{(|\nab \ue|^2+\gamma)(|\nab v_h|^2+\gamma)}\right)\\
&\nonum\quad  +\frac{D^2 w\nab v_h\cdot \nab v_h-D^2 w\nab \ue\cdot \nab \ue}{|\nab v_h|^2+\gamma}
-D^2 w\nab \ue\cdot \nab \ue\left(\frac{|\nab v_h|^2-|\nab \ue|^2}{(|\nab \ue|^2+\gamma)(|\nab v_h|^2+\gamma)}\right)\\
&\nonum\quad+\frac{2\(\widetilde{\Del}_\infty \ue \nab \ue 
- \widetilde{\Del}_\infty v_h\nab v_h\)\cdot \nab w}{(|\nab v_h|^2+\gamma)^2}
+2\widetilde{\Del}_\infty \ue\nab \ue\cdot \nab w\left(\frac{1}{(|\nab \ue|^2+\gamma)^2}-\frac{1}{(|\nab v_h|^2+\gamma)^2}\right).
\end{align}

Bounding the second, fourth, and sixth term on the right-hand 
side of \eqref{ILA5}, we use Sobolev inequalities to conclude that 
for $\|\util-v_h\|_\htw\le \del\in (0,\frac12)$ and for any $q\in (1,\infty)$
\begin{align}\label{ILA5bound2}
&\left\|D^2 \ue\nab \ue\cdot \nab w \left(\frac{|\nab v_h|^2-|\nab \ue|^2}{(|\nab \ue|^2+\gamma)(|\nab v_h|^2+\gamma)}\right)\right\|_{L^1}\\
\nonum &\qquad=\left\|D^2 \ue\nab \ue\cdot \nab w \left(\frac{(\nab v_h-\nab \ue)\cdot (\nab v_h+\nab \ue)}{(|\nab \ue|^2+\gamma)(|\nab v_h|^2+\gamma)}\right)\right\|_{L^1}\\
&\nonum\qquad \le C\|D^2\ue\|_{L^\frac{q}{q-1}}\|\ue\|_\htw^2 \|\ue-v_h\|_\htw \|w\|_\htw.
\end{align}
Here, we have used that fact that if $\|\util-v_h\|_\htw\le \del\in (0,\frac12)$ and $\|\nab \ue\|_{L^\infty}\ge 1$, 
then $\|\nab v_h\|_{L^\infty}\ge C$ for some positive constant $C$ that is independent of $h$, $\eps$, and $\gamma$.

Similarly,
\begin{align}\label{ILA5bound4}
\left\|D^2 w\nab \ue\cdot \nab \ue\left(\frac{|\nab v_h|^2-|\nab \ue|^2}{(|\nab \ue|^2
+\gamma)(|\nab v_h|^2+\gamma)}\right)\right\|_{L^1}
%
& \le C \| \ue\|_\htw^3  \|\ue-v_h\|_\htw\|w\|_\htw,
\end{align}
and
\begin{align}
\label{ILA5bound6}
&\left\|\widetilde{\Del}_\infty \ue\nab \ue\cdot \nab w\left(\frac{1}{(|\nab \ue|^2+\gamma)^2}-\frac{1}{(|\nab v|^2+\gamma)^2}\right)\right\|_{L^1}\\
&\nonum=\left\|\widetilde{\Del}_\infty \ue\nab \ue\cdot \nab w\left(\frac{\(|\nab v_h|^2+|\nab \ue|^2
+2\gamma\)\(\nab \ue-\nab v_h\)\(\nab \ue+\nab v_h\)}{(|\nab \ue|^2+\gamma)^2(|\nab v|^2+\gamma)^2}\right)\right\|_{L^1}\\
&\nonum \le C \|D^2\ue\|_{L^\frac{q}{q-1}}  
\|\ue\|_\htw^6\|\ue-v_h\|_\htw\|w\|_\htw.
\end{align}

Bounding the first term in \eqref{ILA5}, we use similar techniques to conclude
\begin{align}\label{ILA5bound1}
\left\|\frac{\(D^2v_h\nab v_h-D^2\ue\nab \ue\)\cdot \nab w}{|\nab v_h|^2+\gamma}\right\|_{L^1}
&\le C\gamma^{-1}\Bigl(\|D^2v_h-D^2\ue)\nab \ue\cdot \nab w\|_{L^1}\\
&\nonum\qquad+\|D^2 v_h(\nab \ue-\nab v_h)\cdot \nab w\|_{L^1}\Bigr)\\
&\nonum\le C\|\ue\|_\htw\|\ue-v_h\|_\htw\|w\|_\htw.
\end{align}

To bound the third term in \eqref{ILA5}, we use the identity
\begin{align*}
D^2 w\nab v_h\cdot \nab v_h-D^2 w\nab \ue\cdot \nab \ue
&=D^2 w\(\nab v_h+\nab \ue\)\cdot \(\nab v_h-\nab \ue\),
\end{align*}
to obtain
\begin{align}\label{ILA5bound3}
&\left\|\frac{D^2 w\nab v_h\cdot \nab v_h-D^2 w\nab \ue\cdot \nab \ue}{|\nab v_h|^2+\gamma}\right\|_{L^1}
\le C\|\ue\|_\htw\|\ue-v_h\|_\htw\|w\|_\htw.
\end{align}

Next, we write
\begin{align*}
&\(\widetilde{\Del}_\infty \ue\nab \ue-\widetilde{\Del}_\infty v_h \nab v_h\)\cdot \nab w\\
&\qquad=\(\widetilde{\Del}_\infty \ue-\widetilde{\Del}_\infty v_h\)\nab v_h\cdot \nab w+\widetilde{\Del}_\infty \ue\(\nab \ue-\nab v_h\)\cdot \nab w\\
&\qquad = \Bigl(\(D^2 \ue-D^2v_h\)\nab v_h\cdot \nab v_h+D^2 \ue\nab \ue\cdot \nab \ue-D^2 \ue\nab v_h\cdot \nab v_h\Bigr)\nab v_h\cdot \nab w\\
&\qquad\qquad +\widetilde{\Del}_\infty \ue\(\nab \ue-\nab v_h\)\cdot \nab w\\
&\qquad = \Bigl(\(D^2 \ue-D^2v_h\)\nab v_h\cdot \nab v_h+D^2 \ue \(\nab \ue+\nab v_h\)\cdot \(\nab \ue-\nab v_h\)\Bigr)\nab v_h\cdot \nab w\\
&\qquad\qquad +\widetilde{\Del}_\infty \ue\(\nab \ue-\nab v_h\)\cdot \nab w,
\end{align*}
so that
\begin{align}\label{ILA5bound5}
\left\|\frac{\widetilde{\Del}_\infty \ue\nab \ue-\widetilde{\Del}_\infty v_h \nab v_h\)\cdot \nab w}{(|\nab v_h|^2+\gamma)^2}\right\|_{L^1}
& \le C\|D^2 \ue\|_{L^\frac{q}{q-1}}\|\ue\|_\htw^2\|\ue-v_h\|_\htw \|w\|_\htw.
\end{align}

Applying the bounds \eqref{ILA5bound2}--\eqref{ILA5bound5} to the identity
\eqref{ILA5}, we obtain
\begin{align*}
\big\|\Fp[\ue]-\Fp[v_h]\bigr\|_{VV^*}
&=\sup_{w\in V_0}\sup_{z\in V_0} \frac{\bl \(\Fp[\ue]-\Fp[v_h]\)(w),z\br}{\|w\|_\htw \|z\|_\htw}\\
&\le \sup_{w\in V_0}\sup_{z\in V_0} \frac{\big\|\Fp[\ue]-\Fp[v_h]\)(w)\bigr\|_{L^1}\|z\|_{L^\infty}}{\|w\|_\htw \|z\|_\htw}\\
&\le C\|D^2 \ue\|_{L^\frac{q}{q-1}}\|\ue\|_\htw^6 \|\ue-v_h\|_\htw.
\end{align*}

Hence {\rm [A5]} holds with 
\begin{align}\label{ILLhdef}
L(h) = 
& C\|D^2 \ue\|_{L^\frac{q}{q-1}}\|\ue\|_\htw^6 \|\ue-v_h\|_\htw
\end{align}
for any $q\in (1,\infty)$.

Gathering all of our results, existence and uniqueness of a solution
to the finite element method \eqref{ILFEM} and the error estimates
\eqref{ILConformError1}--\eqref{ILConformError2} follow from 
Theorem \ref{abstractmainthm} and the estimates
\eqref{ILConformline1}--\eqref{ILConformYdef}.
\end{proof}

\subsection{Mixed finite element methods
for the infinity-Laplacian equation}\label{chapter-6-IL}

As noted in the previous subsection, $F^\prime[u^\vepsi]$ is possibly 
degenerate, and therefore we need to resort to the abstract formulation 
and analysis of Section \ref{chapter-5-sec-5} for mixed finite element
approximations of the infinity-Laplacian equation.

The mixed finite element method for \eqref{ILVMM1}--\eqref{ILVMM3}
is then defined as follows:  find $(\set_h,\ueh)\in \wW_\eps^h\times Q_g^h$
such that
\begin{alignat}{2}
\label{ILmixedfem1}
(\set_h,\mu_h)+\wb(\mu_h,\ueh)&=G(\mu_h)\qquad &&\forall \mu_h\in W_0^h,\\
\label{ILmixedfem2}\wb(\set_h,v_h)-
\eps^{-1}\wc(\set_h,\ue_h,v_h)&=0\qquad &&\forall v_h\in Q_0^h,
\end{alignat}
where $\tau \in (0,\tau_0)$ ($\tau_0$ is defined in Lemma \ref{altinfsup})
\begin{align*}
\wb(\kappa_h,\ueh)&=\(\Div(\kappa_h),\nab \ueh\),\\
\wc(\seh,\ueh,z_h)&=\(\wF(\seh,\ueh),z_h\),\\
\wF(\set_h,\ueh)&=
-2\eps \tau \Del \ue-\eps n \tau^2\ue 
-\frac{\set_h \nab \ueh\cdot \nab \ueh}{|\nab \ueh|^2+\gamma}+\tau \frac{\ueh |\nab \ueh|^2}{|\nab \ueh|^2+\gamma}.
\end{align*}
We also recall that 
\[
\wW_\eps^h=\{\mu_h\in W^h;\ \mu_h\nu\cdot \nu\big|_{\p\Ome}=\eps+\tau g\}.
\]
The goal of this section is to apply the abstract analysis of Section 
\ref{chapter-5-sec-5} to the mixed method 
\eqref{ILmixedfem1}--\eqref{ILmixedfem2}.
We summarize our findings in the following theorem.

\begin{thm}
Let $\ue\in H^s(\Ome)$ be 
the solution to \eqref{ILVMM1}--\eqref{ILVMM3} and
let $\set=D^2\ue+\tau I_{n\times n}\ue$ with $\tau\in (0,\tau_0)$, 
where $\tau_0$ is defined in Lemma \ref{altinfsup}.
Then there exists $h_4=h_4(\vepsi)>0$ such that for $h\le h_4$ there
exists a unique solution to \eqref{ILmixedfem1}--\eqref{ILmixedfem2}.  
Furthermore, there holds the following error estimates:
\begin{align}
\label{ILVMM1estimate}
\ttbart{\set-\set_h}{\ue-\ueh}&\le \wt{K}_8h^{\ell-2}\|\ue\|_\hl,\\
\label{ILVMM2estimate}
\|\ue-\ue_h\|_\ho&\le K_{R_1}\Bigl(\wt{K}_9 h^{\ell-1} \|\ue\|_\hl +\wt{K}_8^2 R(h)h^{2\ell-4}\|\ue\|_\hl^2\Bigr),
\end{align}
where 
\begin{alignat*}{1}
\ttbart{\mu}{v}&=h\|\mu\|_\ho+\|\mu\|_\lt +\tau^\frac12 \|v\|_\ho,\\
\wt{K}_8&=C\wt{K}_3\eps^{-\frac12} \Bigl(\tau^{-\frac12}+\eps^{-\frac32}|\ue|_{W^{2,\infty}}^2\Bigr),\qquad
\wt{K}_9 = C\wt{K}_8K_G,\\
\ell&={\rm min}\{s,k+1\}.
\end{alignat*}
$\wt{K}_3$ is defined by \eqref{ILK3def}, $K_{R_1}$ is
defined by \eqref{ILMixedline1}, and $K_G$ is defined by \eqref{ILKGdef}.
\end{thm}

\begin{proof}
First, by \eqref{quotientbound} for any $v,z\in Q_0$ and for any 
$q\in (2,\infty)$
\begin{align*}
\bl \Fp[\se,\ue](D^2v,v),z\br
&=\left(\frac{\nab \ue\cdot \nab v}{(|\nab \ue|^2+\gamma)},\nab \ue\cdot \nab w\right)\\
&\qquad-\left(\frac{\se \nab \ue\cdot \nab v-\nab v\cdot \nab \ue{\rm tr}(\se)}{(|\nab \ue|^2+\gamma)},w\right)\\
&\le \left\|\frac{|\nab \ue|^2}{|\nab \ue|^2+\gamma}\right\|_{L^\infty} \|\nab v\|_\lt \|\nab w\|_\lt\\
&\qquad +\left\|\frac{|\nab \ue|}{|\nab \ue|^2+\gamma}\right\|_{L^\infty}\|\nab v\|_\lt \|\se\|_{L^\frac{2q}{q-2}}\|w\|_{L^q}\\
& \le C\|\se\|_{L^\frac{2q}{q-2}}\|v\|_\ho \|w\|_\ho.
\end{align*}
From this calculation, we conclude
\begin{align*}
\bnorm{\Fp[\se,\ue]}_{QQ^*}\le C\|\se\|_{L^\frac{2q}{q-2}}
\end{align*}
for some $q\in (2,\infty)$.

Therefore, using the same arguments as those used in the proof
of Theorem \ref{ILConformingthm}, we can conclude
that condition [$\wt{\rm B2}$] holds with
\begin{alignat}{2}
\label{ILMixedline1}
&K_0=C\eps , &&\quad
K_2=C\|\se\|_{L^\frac{2q}{q-2}},\\
&\nonum K_{R_0}=C\eps^{-2}|\ue|_{W^{2,\infty}}^2,\qquad
&&\quad p=4,\\
&\nonum K_{R_1}= C\eps^{-2}.  &&\quad
\end{alignat}

Next, to confirm {\rm [B3]--[B4]}, we set
\begin{alignat*}{1}
X&=\left[L^\frac{2q}{q-2}(\Ome)\right]^{n\times n},\qquad Y=W^{1,1}(\Ome),\\
\|(\ome,y)\|_{X\times Y}&=\gamma^{-\frac12}\|\ome\|_{L^\frac{2q}{q-2}}\qquad \forall \ome\in X,y\in Y.
\end{alignat*}
where $q$ is any number in the interval $(2,\infty)$
We then have for any
 $\ome\in X,\ y\in Y,\ \chi\in W,\ v\in Q,\ z\in Q_0$,
\begin{align*}
&\bigl\langle \Fp[\ome,y](\chi,v),z\bigr\rangle\\
&\qquad=-\left(\frac{\chi\nab y\cdot \nab y-2\ome \nab y\cdot v}{|\nab y|^2+\gamma},z\right)
+2\left(\frac{(\ome\nab y\cdot y)\nab y\cdot \nab v}{(|\nab y|^2+\gamma)^2},z\right)\\
&\qquad\le \left\|\frac{|\nab y|^2}{(|\nab y|^2+\gamma)}\right\|_{L^\infty} \|\chi\|_\lt \|z\|_\lt
+\left\|\frac{|\nab y|}{(|\nab y|^2+\del)}\right\|_{L^\infty}\|\nab v\|_\lt \|z\|_{L^q} \|\ome\|_{L^\frac{2q}{q-2}}\\
&\qquad\qquad +\left\|\frac{|\nab y|^3}{(|\nab y|^2+\gamma)^2}\right\|_{L^\infty}\|\nab v\|_\lt \|\ome\|_{L^\frac{2q}{q-2}}\|z\|_{L^q}\\
&\qquad \le C\gamma^{-\frac12}\|\ome\|_{L^\frac{2q}{q-2}}\(\|\chi\|_\lt+\|v\|_\ho\)\|z\|_\ho.
\end{align*}
It then follows that
\begin{align*}
\bigl\|\Fp[\ome,y](\chi,v)\|_{H^{-1}}\le C\|(\ome,y)\|_{X\times Y}\(\|\chi\|_\lt +\|v\|_\ho\),
\end{align*}
and thus, assumptions {\rm [B3]--[B4]} hold with
\begin{align}\label{ILK3def}
\left\|\(\Pi^h\se -\tau \se,\mci \ue-\tau \ue\)\right\|_{X\times Y}
=\gamma^{-\frac12}\left\|\Pi^h\se -\tau \se\right\|_{L^\frac{2q}{q-1}}= \wt{K}_3(\eps).
\end{align}

Next, for $(\mu_h,v_h)\in W^h_\eps\times Q^h_g$
with $\ttbar{\Pi^h\se-\mu_h}{\util-v_h}\le \del \in (0,\frac12)$ and $(\kappa_h,z_h)\in W^h\times Q^h$
\begin{align}
\label{ILB5line}
& \(\Fp[\se,\ue]-\Fp[\mu_h,v_h]\)(\kappa_h,z_h)\\
\nonum&=\frac{\kappa_h\nab v_h\cdot \nab v_h}{|\nab v_h|^2+\gamma}-\frac{\kappa_h\nab \ue\cdot \nab \ue}{|\nab \ue|^2+\gamma}
+2\left(\frac{\mu_h\nab v_h}{|\nab v_h|^2+\gamma}-\frac{\se\nab \ue}{|\nab \ue|^2+\gamma}\right)\cdot \nab z_h\\
\nonum&\qquad +2\left(\frac{(\se \nab \ue\cdot \nab \ue)\nab \ue\cdot \nab z_h}{(|\nab \ue|^2+\gamma)^2}
-\frac{(\mu_h\nab v_h\cdot \nab v_h)\nab v_h\cdot \nab z_h}{(|\nab v_h|^2+\gamma)^2}\right).
\end{align}

To bound the first term in \eqref{ILB5line}, 
we add and subtract terms to deduce
\begin{align*}
&\frac{\kappa_h\nab v_h\cdot \nab v_h}{|\nab v_h|^2+\gamma}-\frac{\kappa_h\nab \ue\cdot \nab \ue}{|\nab \ue|^2+\gamma}\\
&=\frac{\kappa_h\nab v_h\cdot \nab v_h-\kappa_h\nab \ue\cdot \nab \ue}{|\nab v_h|^2+\gamma}
+\kappa_h \nab \ue\cdot \nab \ue\left(\frac{1}{|\nab v_h|^2+\gamma}-\frac{1}{|\nab \ue|^2+\gamma}\right)\\
&=\frac{\kappa_h \(\nab v_h+\nab \ue\)\cdot \(\nab v_h-\nab \ue\)}{|\nab v_h|^2+\gamma}
+\kappa_h \nab \ue\cdot \nab \ue\left(\frac{\(\nab \ue-\nab v_h\)\(\nab \ue+\nab v_h\)}{\(|\nab v_h|^2+\gamma\)\(|\nab \ue|^2+\gamma\)}\right),
\end{align*}
and therefore by the inverse inequality, 
\begin{align}\label{ILB5bound1}
&\left\|\frac{\kappa_h\nab v_h\cdot \nab v_h}{|\nab v_h|^2+\gamma}-\frac{\kappa_h\nab \ue\cdot \nab \ue}{|\nab \ue|^2+\gamma}\right\|_{L^1}\\
\nonum&\le C\Bigl(\|\nab v_h+\nab \ue\|_\lt \|\nab v_h-\nab \ue\|_\lt\\
&\nonum\qquad+ \|\nab \ue\|_{L^\infty}^2\|\nab \ue-\nab v_h\|_\lt \|\nab \ue+\nab v_h\|_\lt\Bigr)\|\kappa_h\|_{L^\infty}\\
&\nonum\le Ch^{-1}\|\ue\|_\ho\|\nab \ue\|_{L^\infty}^2\|\ue-v_h\|_\ho \|\kappa_h\|_\lt.
\end{align}

Using a similar technique to bound the second term in \eqref{ILB5line}, 
we first write
\begin{align*}
&\left(\frac{\mu_h\nab v_h}{|\nab v_h|^2+\gamma}-\frac{\se\nab \ue}{|\nab \ue|^2+\gamma}\right)\cdot \nab z_h\\
&=\left(\frac{\mu_h\nab v_h-\se \nab \ue}{|\nab v_h|^2+\gamma}+\se\nab \ue\left(\frac{|\nab \ue|^2-|\nab v_h|^2}{\(|\nab v_h|^2+\gamma\)\(|\nab \ue|^2+\gamma\)}\right)\right)\cdot \nab z_h\\
&=\left(\frac{\(\mu_h-\se\)\nab v_h+\se(\nab v_h- \nab \ue\)}{|\nab v_h|^2+\gamma}+\se\nab \ue\left(\frac{\(\nab \ue+\nab v_h\)\(\nab \ue-\nab v_h\)}{\(|\nab v_h|^2+\gamma\)\(|\nab \ue|^2+\gamma\)}\right)\right)\cdot \nab z_h.
\end{align*}
It then follows that
\begin{align}\label{ILB5bound2}
&\left\|\left(\frac{\mu_h\nab v_h}{|\nab v_h|^2+\gamma}-\frac{\se\nab \ue}{|\nab \ue|^2+\gamma}\right)\cdot \nab z_h\right\|_{L^1}\\
\nonum
&\le C\Bigl( \|\mu_h-\se\|_\lt \|\nab v_h\|_{L^\infty} \|\nab z_h\|_\lt\\
&\qquad\nonum +\|\se\|_{L^\infty} \|\nab v_h-\nab \ue\|_\lt \|\nab z_h\|_\lt \\
\nonum
&\qquad\quad+\|\se\|_{L^\infty}\|\nab \ue\|_{L^\infty} \|\nab \ue+\nab v_h\|_\lt \|\nab \ue-\nab v_h\|_\lt \|\nab z_h\|_{L^\infty}\Big)\\
\nonum
&\le C\Bigl( h^{-1}\|\mu_h-\se\|_\lt \|\ue\|_\ho+\|\se\|_{L^\infty} \|\ue-v_h\|_\ho \\
\nonum
&\qquad + h^{-1}\|\se\|_{L^\infty}\|\nab \ue\|_{L^\infty}\|\ue\|_\ho\|\ue-v_h\|_\ho\Bigr)\|z_h\|_\ho\\
\nonum&\le C h^{-1}\|\se\|_{L^\infty}\|\nab \ue\|_{L^\infty}\|\ue\|_\ho\(\|\se-\mu_h\|_\lt +\|\ue-v_h\|_\ho\)\|z_h\|_\ho.
\end{align}

Next, we write
\begin{align*}
&\frac{(\se \nab \ue\cdot \nab \ue)\nab \ue\cdot \nab z_h}{(|\nab \ue|^2+\gamma)^2}
-\frac{(\mu_h\nab v_h\cdot \nab v_h)\nab v_h\cdot \nab z_h}{(|\nab v_h|^2+\gamma)^2}\\
&=\frac{(\se\nab \ue \cdot \nab \ue)\nab \ue\cdot \nab z_h-(\mu_h\nab v_h\cdot \nab v_h)\nab v_h\cdot \nab z_h}{(|\nab v_h|^2+\gamma)^2}\\
&\qquad+(\se\nab \ue\cdot \nab \ue)\nab \ue\cdot \nab z_h\left(\frac{1}{(|\nab \ue|^2+\gamma)^2}-\frac{1}{(|\nab v_h|^2+\gamma)^2}\right)
\end{align*}

Noting
\begin{align*}
&(\se \nab \ue\cdot \nab \ue)\nab \ue\cdot \nab z_h-(\mu_h\nab v_h\nab v_h)\nab v_h\cdot \nab z_h\\
&=\((\se-\mu_h)\nab \ue\cdot \nab \ue\)(\nab \ue\cdot \nab z_h)+\(\mu_h (\nab \ue-\nab v_h)\cdot (\nab \ue+\nab v_h)\)(\nab \ue\cdot \nab z_h)\\
&\qquad+(\mu_h\nab v_h\cdot \nab v_h)(\nab \ue-\nab v_h)\cdot \nab z_h,
\end{align*}
we conclude
\begin{align}
\label{ILB5bound3}
&\left\|\frac{(\se\nab \ue \cdot \nab \ue)\nab \ue\cdot \nab z_h-(\mu_h\nab v_h\cdot v_h)\nab v_h\cdot \nab z_h}{(|\nab v_h|^2+\gamma)^2}\right\|_{L^1}\\
\nonum&\le C\Bigl(\|\se-\mu_h\|_\lt \|\nab \ue\|_{L^\infty}^3 \|\nab z_h\|_\lt\\
\nonum&\qquad +\|\mu_h\|_\lt \|\nab \ue-\nab v_h\|_\lt \|\nab \ue+\nab v_h\|_{L^\infty}\|\nab \ue\|_{L^\infty}\|\nab z_h\|_{L^\infty}\\
 &\nonum\qquad+\|\mu_h\|_\lt \|\nab v_h\|_{L^\infty}^2 \|\nab \ue-\nab v_h\|_\lt \|\nab z_h\|_{L^\infty}\Bigr)\\
\nonum&\le C\Bigl(\|\nab \ue\|_{L^\infty}^3 +h^{-2}\|\se\|_\lt \|\ue\|_\ho \|\nab \ue\|_{L^\infty}\\
&\qquad \qquad \qquad\nonum + h^{-3}\|\se\|_\lt \|\ue\|_\ho^2\Bigr)\|\ue-v_h\|_\ho\|z_h\|_\ho.  
\end{align}

We also have
\begin{align*}
&(\se\nab \ue\cdot \nab \ue)\nab \ue\cdot \nab z_h\left(\frac{1}{(|\nab \ue|^2+\gamma)^2}-\frac{1}{(|\nab v_h|^2+\gamma)^2}\right)\\
&=(\se\nab \ue\cdot \nab \ue)\nab \ue\cdot \nab z_h\left(\frac{\(|\nab v_h|^2+|\nab \ue|^2+\gamma\)\(\nab v_h-\nab \ue\)\(\nab v_h+\nab \ue\)}
{(|\nab v_h|^2+\gamma)^2(|\nab \ue|^2+\gamma)^2}\right),
\end{align*}
and therefore,
\begin{align}
\label{ILB5bound4}
&\left\|(\se\nab \ue\cdot \nab \ue)\nab \ue\cdot \nab z_h\left(\frac{1}{(|\nab \ue|^2+\gamma)^2}-\frac{1}{(|\nab v_h|^2+\gamma)^2}\right)\right\|_{L^1}\\
\nonum& \qquad\le C\Bigl( \|\se\|_{L^\infty}\|\nab \ue\|_{L^\infty}^3 \|\nab z_h\|_{L^\infty}\||\nab v_h|^2\\
&\nonum\qquad \qquad \qquad +|\nab \ue|^2+\gamma\|_{L^4}\|\nab v_h-\nab \ue\|_\lt \|\nab v_h+\nab \ue\|_{L^4}\Bigr)\\
\nonum& \qquad\le C h^{-3}\|\se\|_{L^\infty}\|\nab \ue\|_{L^\infty}^3
\|\ue\|_\htw^3\|\ue-v_h\|_\ho\|z_h\|_\ho.
\end{align}

Combining \eqref{ILB5bound1}--\eqref{ILB5bound4}, we have
\begin{align*}
\bigl\|\Fp[\se,\ue]-\Fp[\mu_h,v_h]\bigr\|_{L^1}
&\le C\Bigl( h^{-3}\|\se\|_{L^\infty}\|\nab \ue\|_{L^\infty}^3\|\ue\|_\htw^3\Bigr)\\
&\qquad\times \(\|\se-\mu_h\|_\lt +\|\ue-v_h\|_\ho\)\(\|\kappa_h\|_\lt +\|z_h\|_\ho\).
\end{align*}

It then follows from the inverse inequality, that
\begin{align*}
&\sup_{w\in Q^h} \frac{\Bl \(\Fp[\se,\ue]-\Fp[\mu_h,v_h]\)\(\kappa_h,z_h\),w_h\Br}{\|w_h\|_\ho}\\
&\qquad\le C|\log h|^\frac12 \Bigl(h^{-3}\|\se\|_{L^\infty}\|\nab \ue\|_{L^\infty}^3
\|\ue\|_\htw^3\Bigr)\\
&\qquad\qquad \qquad \qquad\times \(\|\se-\mu_h\|_\lt +\|\ue-v_h\|_\ho\)\(\|\kappa_h\|_\lt +\|z_h\|_\ho\),
\end{align*}
and therefore condition {\rm [B5]} holds with
\begin{align*}
R(h) = C|\log h|^\frac12 \Bigl( h^{-3}\|\se\|_{L^\infty}\|\nab \ue\|_{L^\infty}^3
\|\ue\|_\htw^3\Bigr).
\end{align*}

Finally, we confirm assumption {\rm [B6]}. 
First, we note that
\begin{align*}
\frac{\p F(\se,\ue)}{\p r_{ij}} = \frac{\frac{\p \ue}{\p x_i}\frac{\p \ue}{\p x_j}}{|\nab \ue|^2+\gamma},
\end{align*}
and
\begin{align*}
\frac{\p}{\p x_k} \left(\frac{\p F(\se,\ue)}{\p r_{ij}}\right)
&=\frac{\frac{\p^2 \ue}{\p x_i\p x_k}\frac{\p \ue}{\p x_j}+\frac{\p \ue}{\p x_i}\frac{\p^2 \ue}{\p x_j \p x_k}}{|\nab \ue|^2+\gamma}
 -2\frac{\frac{\p \ue}{\p x_i}\frac{\p \ue}{\p x_j} (D^2\ue\nab \ue)_k}{(|\nab \ue|^2+\gamma)^2},
\end{align*}
and therefore by \eqref{quotientbound}
\begin{align*}
\max_{1\le i,j\le 2}\left\|\frac{\p F(\se,\ue)}{\p r_{ij}} \right\|_{L^\infty}&\le C,\\
\max_{1\le i,j\le 2}\left\|\frac{\p F(\se,\ue)}{\p r_{ij}}\right\|_{W^{1,\frac65}}&\le 
C\Bigl(\|\ue\|_{W^{2,\frac65}}+\|\nab \ue\|_{L^\infty} \|D^2\ue\|_{L^\frac65}\Bigr)\\
&\le C\|\nab \ue\|_{L^\infty}\|\ue\|_{W^{2,\frac65}},
\end{align*}
and therefore by Proposition \ref{B6prop}, condition {\rm [B6]} holds with
\begin{align}\label{ILKGdef}
\alpha=1,\qquad K_G = C\|\nab \ue\|_{L^\infty}\|\ue\|_{W^{2,\frac65}}.
\end{align}

Finally, we apply Theorem \ref{altmainmixedthm} to obtain
existence and uniqueness of a solution $(\set_h,\ueh)$
to the mixed finite element method \eqref{ILVMM1}--\eqref{ILVMM3} 
as well as the estimates \eqref{ILVMM1estimate}--\eqref{ILVMM2estimate}.

\end{proof}

\subsection{Numerical experiments and rates of convergence}\label{Section-6.3.3}

\subsubsection*{Test 6.3.1}
In this test, we numerically solve the infinity-Laplacian equation
using the Argyris element of degree $k=5$ for
fixed $h=0.015$ while varying $\eps$.  The purpose of these experiments is to 
estimate the rate of convergence of $\|u-\ue\|$ in various norms, where $u$ is 
the viscosity solution of \eqref{IL1}--\eqref{IL2}.  To this end, we solve the
following finite element method (compare to \eqref{ILFEM}):  find $\ueh\in V^h_g$
such that
\begin{align}
\eps(\Del \ueh,\Del v_h)-\left(\frac{\widetilde{\Del}_\infty \ueh}{|\nab \ueh|^2+\gamma},v_h\right)&=(f,v_h)
+\left\langle \eps^2,\normd{v_h}\right\rangle_{\p\Ome}
\qquad \forall v_h\in V^h_0.
\end{align}
We set $\Ome = (-0.5,0.5)^2$, $\gamma = \eps^2$, and use the following two test functions:
\begin{alignat*}{2}
&(a)\ u=x_1^{4/3}-x_2^{4/3},\quad &&f=0,\\
&(b)\ u=x_1^2+x_2^2,\quad&&f=\frac{8(x^2+y^2)}{4x^2+4y^2+\gamma}.
\end{alignat*}

We note that the second test function is smooth, 
but the first does not belong to $C^2(\Ome)$ since its second
derivatives have singularities at $x_1=0$ and $x_2=0$.
After computing the solution for different $\eps$-values,
we list the errors in Table \ref{Test631Table} with their
estimated rate of convergence and plot the results in Figure
\ref{Test631Figure1}. The numerical experiments 
indicate the following rates of convergence as $\eps\to 0^+$:
\begin{align*}
\|u-\ueh\|_{L^2}\approx O\left(\eps^\frac23\right),
\qquad \|u-\ueh\|_{H^1}\approx O\left(\eps^\frac13\right),
\qquad \|u-\ueh\|_\htw \approx O\left(\eps^\frac16\right).
\end{align*}
Since we have fixed $h$ small, we expect that
$\|u-\ue\|$ has similar rates of convergence.

\begin{figure}[ht]
\centerline{
\includegraphics[scale=0.115]{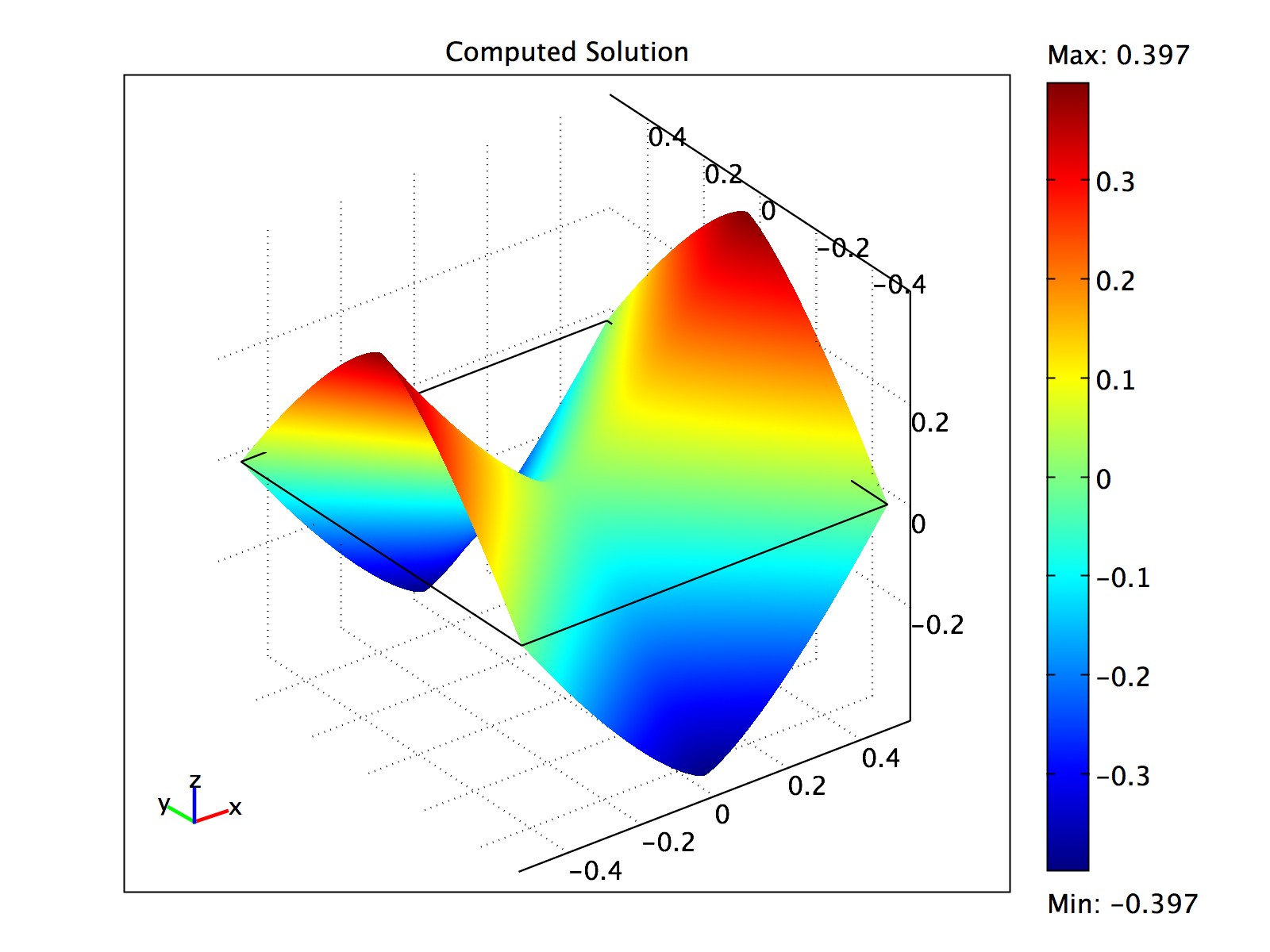}\,
\includegraphics[scale=0.115]{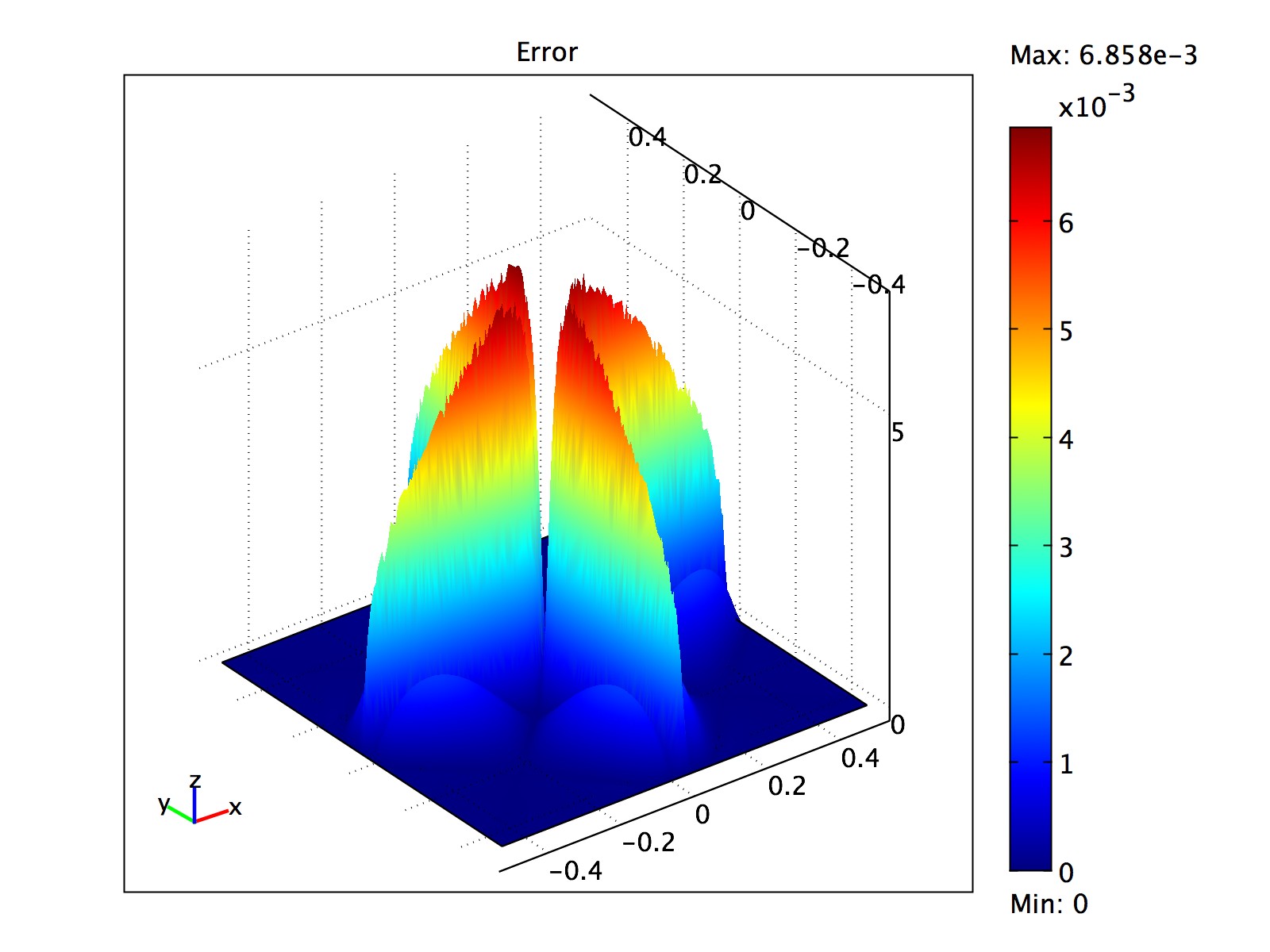}
}
\caption{{\small Test 6.3.1a.  Computed solution (left) and its error 
(right) with $\eps=0.001$ and $h=0.015$.}} \label{Test631Picture}
\end{figure}

\begin{table}[htbp]
\caption{Test 6.3.1. Error of $\|u-u^\eps_h\|$ w.r.t $\eps$ ($h=0.015$)} 
\label{Test631Table}
\centering
\begin{tabular}{lclll}
& {\scriptsize$\eps$}& 
{\scriptsize $\|u-\ueh\|_\lt$(rate)} &  {\scriptsize $\|u-\ueh\|_\ho$(rate)} & 
{\scriptsize $\|u-\ueh\|_\htw$(rate)}\\
\noalign{\smallskip}\hline\noalign{\smallskip}
Test 6.3.1a
&1.0E--03		&2.15E--03\blank	&3.90E--02\blank &\Blank\\
&5.0E--04		&1.52E--03(0.50)	&3.29E--02(0.25) &\Blank\\
&2.5E--04		&1.06E--03(0.52)	&2.75E--02(0.26) &\Blank\\
&1.0E--04		&6.54E--04(0.53)	&2.15E--02(0.27) &\Blank\\
&5.0E--05		&4.51E--04(0.54)	&1.77E--02(0.28) &\Blank\\
&2.5E--05		&3.09E--04(0.54)	&1.45E--02(0.29) &\Blank\\
&1.0E--05	&1.88E--04(0.55)	&1.10E--02(0.30) &\Blank\\
\noalign{\smallskip}\hline\noalign{\smallskip}
Test 6.3.1b
&1.0E--03	&1.01E--02\blank	&7.20E--02\blank	&1.36E+00\blank\\
&5.0E--04	&6.02E--03(0.75)	&5.10E--02(0.50)	&1.21E+00(0.16)\\
&2.5E--04	&3.61E--03(0.74)	&3.70E--02(0.46)	&1.08E+00(0.16)\\
&1.0E--04	&1.86E--03(0.72)	&2.50E--02(0.43)	&9.36E--01(0.15)\\
&5.0E--05	&1.15E--03(0.70)	&1.89E--02(0.41)	&8.44E--01(0.15)\\
&2.5E--05	&7.14E--04(0.68)	&1.44E--02(0.39)	&7.63E--01(0.15)\\
&1.0E--05	&3.87E--04(0.67)	&1.01E--02(0.39)	&6.70E--01(0.14)
\end{tabular}
\end{table}

\begin{figure}[ht]
\includegraphics[scale=0.48]{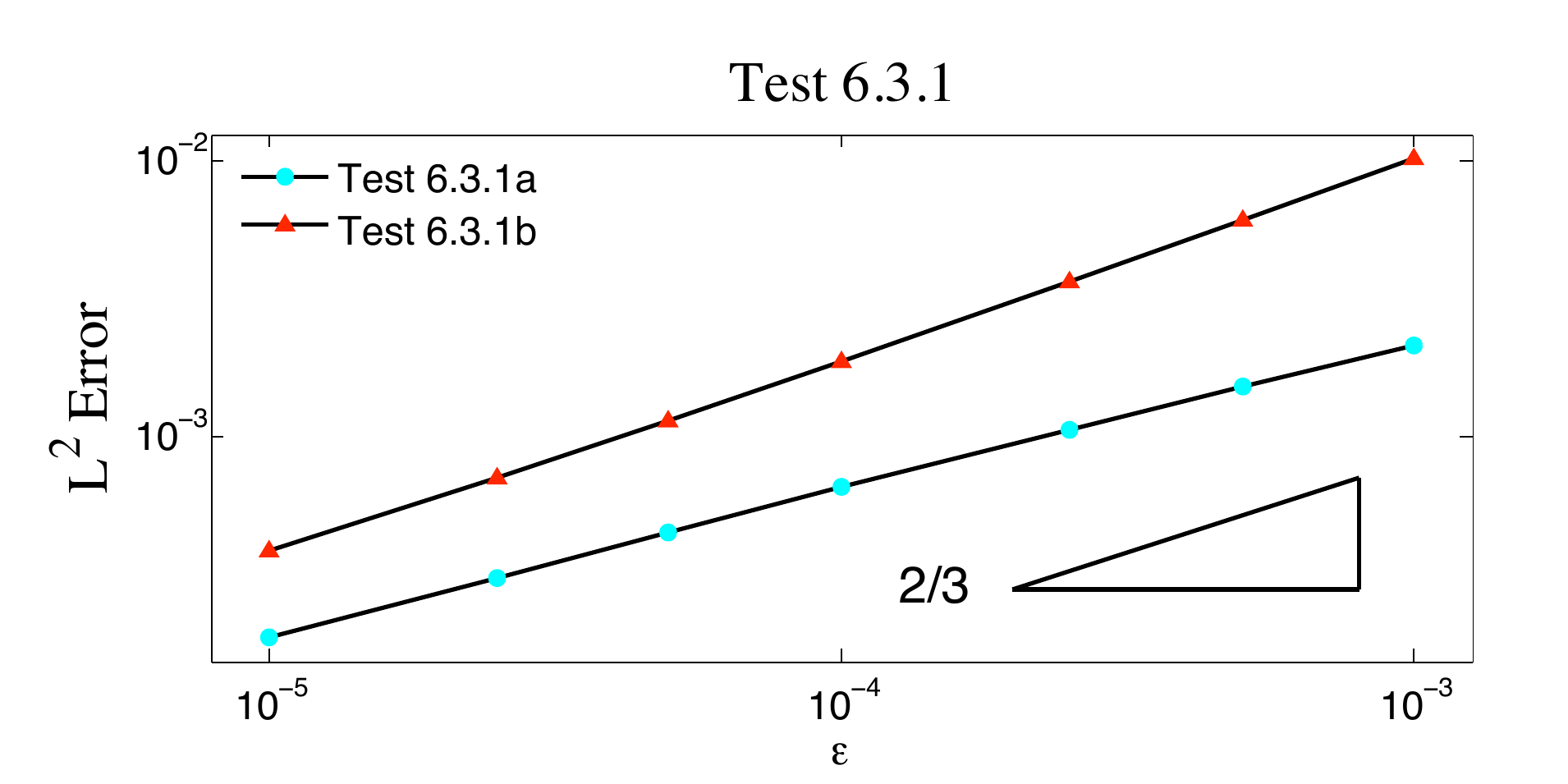}\\
\includegraphics[scale=0.5]{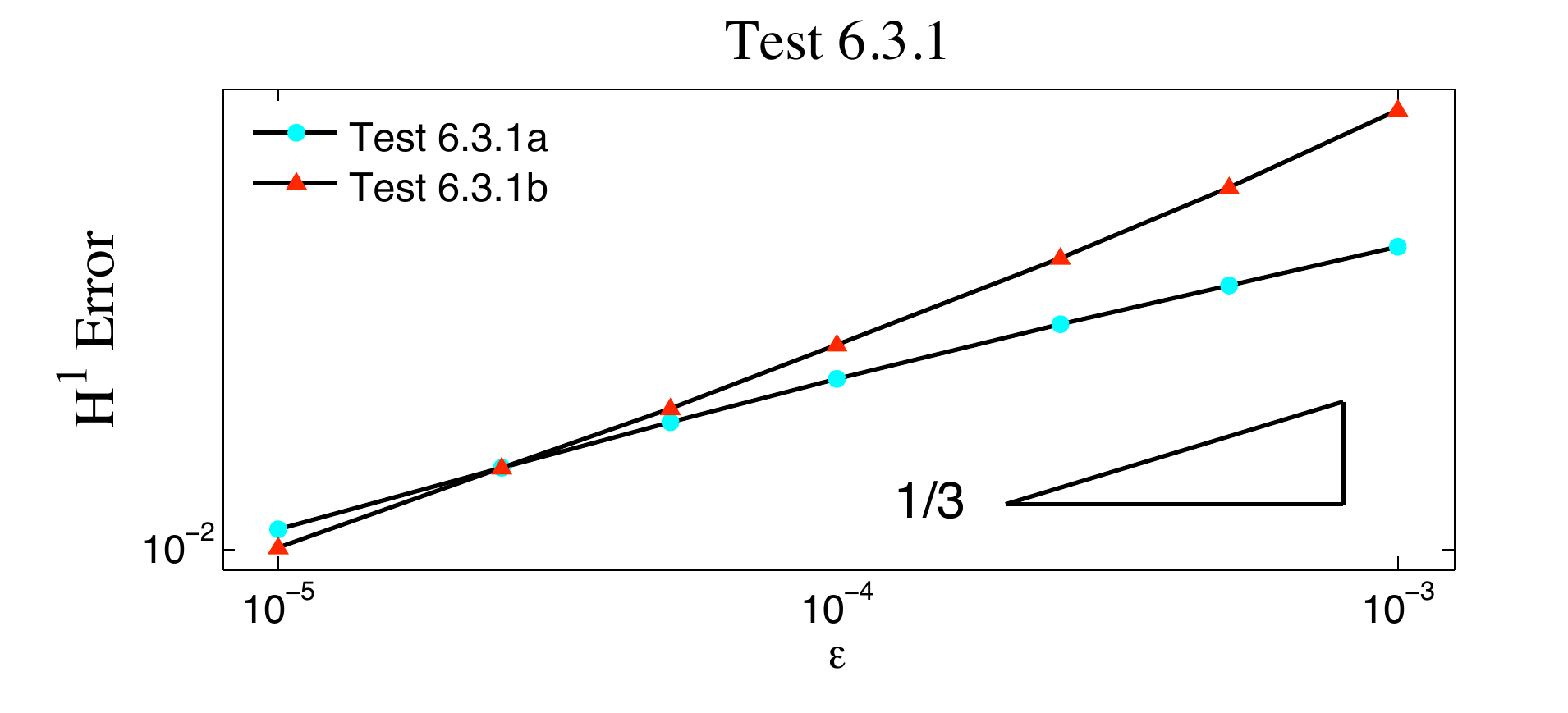}\\
\includegraphics[scale=0.5]{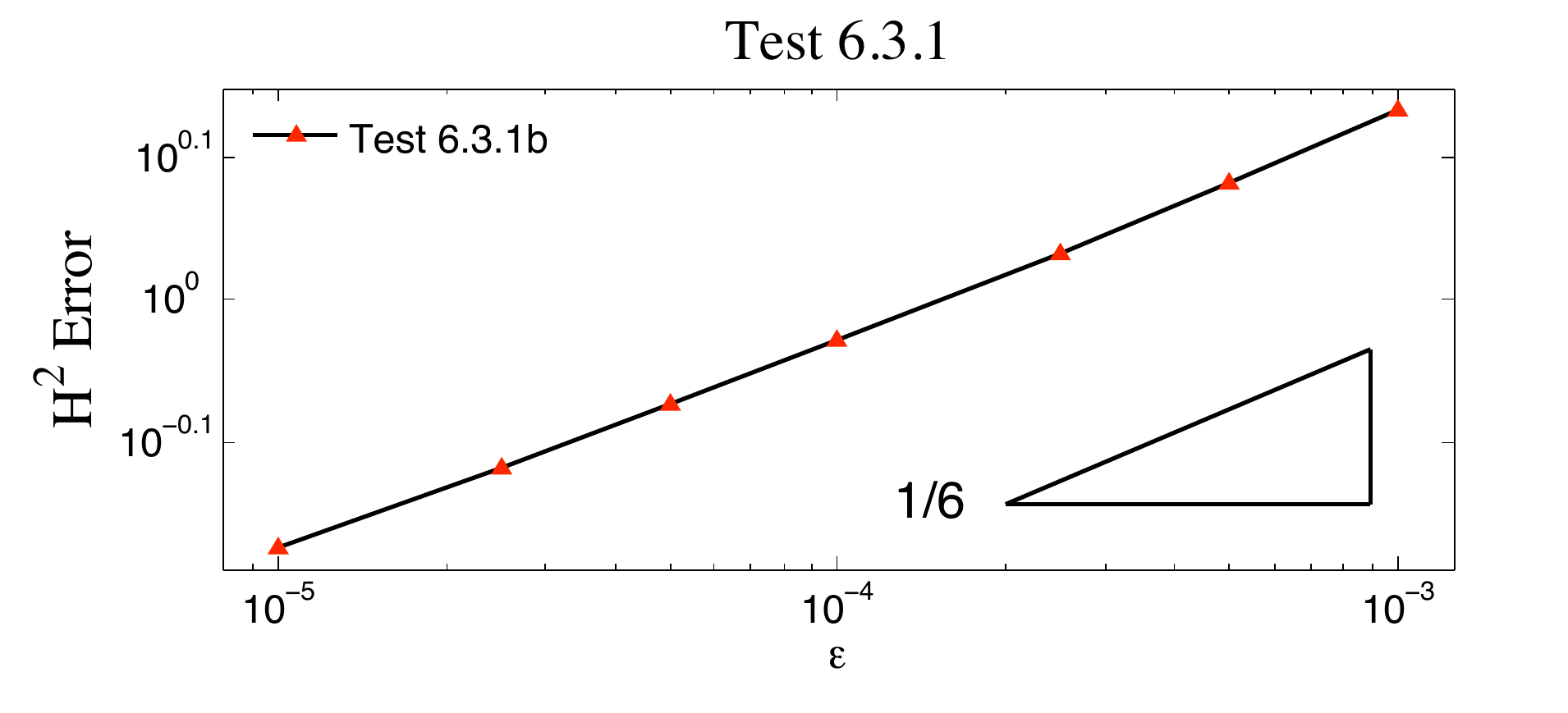}
\caption{Test 6.3.1. Error $\|u-\ueh\|_\lt$ (top), $\|u-\ueh\|_\ho$ (middle),
and $\|u-\ueh\|_\htw$ (bottom) w.r.t. $\eps$ ($h=0.015$).}
\label{Test631Figure1}
\end{figure}

\subsubsection*{Test 6.3.2}
For our last test, we verify the theoretical results derived
in Section \ref{chapter-6-IL}.  To this end, we solve the following 
problem:  find $(\set_h,\ueh)\in \wW^h_{\phi^\eps}$ such that
\begin{alignat}{2}
\label{Test632fem1}
(\set_h,\mu_h)+\wb(\mu_h,\ueh)
&=G(\mu_h)\qquad &&\forall \mu_h\in W_0^h,\\
\label{Test632fem2}
\wb(\set_h,v_h)-\eps^{-1}\wc(\set_h,\ue_h,v_h)
&=(f^\eps,v_h)\qquad &&\forall v_h\in Q_0^h,
\end{alignat}
where
\begin{align*}
\wW^h_{\phi^\eps}:=\bigl\{\mu_h\in W^h;\ 
\mu_h\nu\cdot \nu\big|_{\p\Ome}=\phi^\eps+\tau g \bigr\}.
\end{align*}
We use the following test function:
\begin{alignat*}{1}
&u^\eps=\cos(x_1)-\cos(x_2),
\qquad \phi^\eps=\nu_2^2\cos(x_2)-\nu_1^2\cos(x_1),\\
&f^\eps=\eps\(\cos(x_1)-\cos(x_2)\)+\frac{\cos(x_1)\sin^2(x_1)-\cos(x_2)\sin^2(x_2)}{\sin^2(x_1)+\sin^2(x_2)+\gamma}.
\end{alignat*}

We compute \eqref{Test632fem1}--\eqref{Test632fem2}
for fixed $\eps=0.01$, while varying $h$ with $\Ome = (-0.5,0.5)^2$ and $\gamma = \eps^2=$1E--4.  We list the error
of the computed solution in Table \ref{Test632Table} for both
$\tau=0$ and $\tau=1$.
As expected, for the 
case $\tau=1$, we observe the following rates of convergence:
\begin{align*}
\|\ue-\ueh\|_\lt=O(h^3),\quad \|\ue-\ueh\|_\ho=O(h^2),\quad \|\set-\set_h\|_\lt=O(h).
\end{align*}
We also observe that the same rates of convergence appear to hold for
the case $\tau=0$, although our theoretical results of 
Section \ref{chapter-6-IL} do not cover this case.

\begin{table}[htbp]
\caption{Test 6.3.2. Error of $\|\ue-\ueh\|$ w.r.t $h$ ($\eps=0.01$)} 
\label{Test632Table}
\centering
\begin{tabular}{cclll}
  &  {\scriptsize$h$} & 
  {\scriptsize $\|\ue-\ueh\|_\lt$(rate)} &  {\scriptsize $\|\ue-\ueh\|_\ho$(rate)} & 
  {\scriptsize $\|\set-\set_h\|_\lt$(rate)}\\
  \noalign{\smallskip}\hline\noalign{\smallskip}
  $\tau=0$
&2.0E--01	&8.74E--06\blank	&3.92E--04\blank	&9.93E--03\blank\\	
&1.0E--01	&1.14E--06(2.94)	&1.03E--04(1.93)	&4.11E--03(1.27)\\
&5.0E--02	&1.38E--07(3.05)	&2.53E--05(2.02)	&1.54E--03(1.41)\\
&2.5E--02	&1.67E--08(3.05)	&6.22E--06(2.03)	&5.08E--04(1.60)\\
&1.0E--02	&1.33E--09(2.76)	&9.98E--07(2.00)	&1.29E--04(1.49)\\
\noalign{\smallskip}\hline\noalign{\smallskip}
$\tau=1$
&2.0E--01	&8.74E--06\blank	&3.92E--04\blank	&9.93E--03\blank\\	
&1.0E--01	&1.14E--06(2.94)	&1.03E--04(1.93)	&4.11E--03(1.27)\\
&5.0E--02	&1.38E--07(3.05)	&2.53E--05(2.02)	&1.54E--03(1.41)\\
&2.5E--02	&1.66E--08(3.05)	&6.22E--06(2.03)	&5.08E--04(1.60)\\
&1.0E--02	&1.22E--09(2.85)	&9.98E--07(2.00)	&1.29E--04(1.49)
\end{tabular}
\end{table}



%% file: chapter7.tex
\chapter{Concluding Comments}\label{chapter-7}

In this final chapter, we give some concluding comments about 
the vanishing moment method and its finite element and mixed finite
element approximations for fully nonlinear second order PDEs.  
In particular, we point out some main issues accompanying 
with the methodology.  

We recall that the vanishing moment method and the notion of moment 
solutions are exactly in the same spirit as  the vanishing viscosity method
and the original notion of viscosity solutions proposed by
M. Crandall and P. L. Lions in \cite{Crandall_Lions83} for the
Hamilton-Jacobi equations, which is based on the idea of approximating 
a fully nonlinear PDE by a family of quasilinear higher order PDEs. 
The vanishing moment method then allows one to reliably
compute the viscosity solutions of fully nonlinear second order PDEs,
in particular, using Galerkin-type methods and existing 
numerical methods and computer software (with slight 
modifications), a task which had been impracticable before. 
As a by-product, the vanishing moment method reveals some insights 
for the understanding of viscosity solutions, and the notion of
moment solutions might also provide a logical and natural 
generalization/extension for the notion of viscosity solution,
especially, in the cases where there is no theory or
the existing viscosity solution theory fails (e.g. the Monge-Amp\`ere 
equations of hyperbolic type \cite{Chang_Gursky_Yang02}
and systems of fully nonlinear second order PDEs.)

\section{Boundary layers}\label{chapter-7.1}

As pointed out in Chapter \ref{chapter-2}, in order to approximate 
a second order PDE by a quasilinear fourth order PDE, we must 
impose an extra boundary condition such as those given in
\eqref{moment3}. Because the extra boundary condition is artificial, 
it is expected that a ``boundary layer" ought be introduced
in an $\vepsi$-neighborhood of $\p\Ome$. For example, 
in the case that $\Del u^\vepsi=\vepsi$ is used as the extra boundary condition 
on $\p\Ome$, and since we do not know a priori the true value $\Del u$ 
on $\p\Ome$ (note that $\Del u$ may not even exist if the viscosity 
solution $u$ is not differentiable),  $\Del u^\vepsi$ and $\Del u$ 
take different values on $\p\Ome$ in general, and the discrepancy 
between $\Del u^\vepsi$ and $\Del u$ could be large although
this can only occur in a very small region (i.e., an $\vepsi$-neighborhood 
of $\p\Ome$).

Since the convergence of $u^\vepsi$ to $u$ as
$\vepsi\searrow 0^+$ is only expected and proved in 
low order norms (cf. Chapters \ref{chapter-2} and \ref{chapter-3}), 
the error $\Del u^\vepsi-\Del u$ in an $\vepsi$-neighborhood 
of $\p\Ome$ does not cause any problem for the convergence.
Our numerical experiments do confirm this conclusion.
Moreover, as expected, our numerical experiments also 
confirm that $\|u^\vepsi-u\|_{H^2}$ does not converge 
in general (cf. Test 6.1.3).
On the other hand, a closer look at the error of computed 
solution in Figure \ref{Test613Picture} shows that the error
is concentrated in an $\vepsi$-neighborhood of $\p\Ome$ and 
at the singularity of the solution $u$. 

To improve the accuracy and efficiency of the vanishing moment 
method, we propose the following simple {\em iterative surgical strategy},
which consists of three steps. 

\medskip
{\em Step 1:} Solve numerically \eqref{moment1}--\eqref{moment3}$_1$ 
as before for a fixed (small) $\vepsi>0$. 

{\em Step 2:} Find $\Del u^\vepsi_h$ on the inner boundary of the 
$\vepsi$-neighborhood of $\p\Ome$, and extend the function to $\p\Ome$ by 
any (convenient) method. We denote the extended function by $c_\vepsi$. 

{\em Step 3:} Solve numerically \eqref{moment1}--\eqref{moment3}$_1$
again with $\Del u^\vepsi|_{\p\Ome}=\vepsi$ being replaced by 
$\Del u^\vepsi|_{\p\Ome}=c_\vepsi$.

\begin{remark}
(a) $c_\vepsi$ can be obtained by an interpolation technique, or by
doing a ray tracing along the normal on $\p\Ome$, or simply by
letting $c_\vepsi$ be the maximum value (a constant) of $\Del u^\vepsi_h$ 
on the inner boundary of the $\vepsi$-neighborhood of $\p\Ome$.

(b) Clearly, {\em Step 2} and {\em Step 3} can be repeated, although
one iteration is often sufficient in practice (see numerical 
experiment below).

(c) The above {\em iterative surgical strategy} is a ``predictor-corrector"
type strategy, where the prediction and correction are 
done on $\Del u^\vepsi_h|_{\p\Ome}$.

(d) To make the algorithm more efficient, the solution computed in Step 1 
can be used as an initial guess for the nonlinear solver in Step 3.
\end{remark}

As a numerical example for the {\em iterative surgical strategy}, 
we solve the Monge-Amp\`ere equation 
using the conforming finite element method developed and
analyzed in Section \ref{chapter-6-1-2},  that is, we
numerically solve \eqref{moment1} with 
$F(D^2 \ue,\nab \ue,\ue,x)=f(x)-\det(D^2\ue)$
using the finite element method
\eqref{MAagainfem1}.  
Here, we use fifth degree Argyris elements
to construct the finite element space, and set $\Ome=(0,1)^2$,
$f=(1+x_1^2+x_2^2)e^{x_1^2+x_2^2}$, so that the exact solution 
is $u=e^{(x^2_1+x_2^2)/2}$.

In Step 2, we extend $\Del u^\vepsi_h$ in the neighborhood of $\p\Ome$, 
to $\p\Ome$ by linear interpolation to construct $c_\eps$.
After performing Steps 1--3, we repeat
Steps 2 and 3 four more times to determine whether
repeated iterations make a significant impact on the error.
We use the parameters $\eps=0.01$ and $h=0.01$ for all computations.

After computing the solutions in Step 1 and 3,
we record the errors in Tables \ref{TableBL1}--\ref{TableBL2}.
We also plot the cross-section of the computed Laplacian $\Delta \ueh$
at $x_2=0.8$ in Figure \ref{BLFig} after each iteration.
Tables \ref{TableBL1}--\ref{TableBL2} clearly indicate
the iterative surgical strategy decreases the error
at each step.  In fact, the error in every norm is 
decreased by nearly  a factor of  ten by
performing Steps 1--3 just once.  However, the 
error decreases only modestly after repeated 
iterations and has no impact on the $L^2$ and $H^1$
errors after two iterations.  
Figure \ref{BLFig} also indicates that the boundary layer
is greatly reduced after the first iteration, 
and improves modestly after each subsequent iteration.

\begin{table}[htbp]
   \centering
   \begin{tabular}{lllll}
   iteration \#  & $\|u-u^{\eps }_h\|_\lt$ & $\|u-u^{\eps}_h\|_{H^1}$ & $\|u-u^{\eps}_h\|_{H^2}$\\
      \noalign{\smallskip}\hline\noalign{\smallskip}
0		&1.48E--02	&1.00E--01	&1.79E+00\\
1		&1.88E--03	&2.11E--02	&4.63E--01\\
2		&1.51E--03	&1.23E--02	&2.53E--01\\
3	 	&2.15E--03	&1.18E--02	&1.77E--01\\
4	&2.51E--03	&1.24E--02	&1.42E--01 \\
       \noalign{\smallskip}\noalign{\smallskip}
   \end{tabular}
 \caption{Errors of $u-u^\eps_h$ using the 
iterative surgical strategy  ($\eps=0.01, h=0.01$).} \label{TableBL1}
 \end{table}
 
 \begin{table}[htbp]
 \centering
   \begin{tabular}{lllll}
     iteration \# & $\|u-u^{\eps}_h\|_{L^\infty}$ & $\|u-u^{\eps}_h\|_{W^{1,\infty}}$ & $\|u-u^{\eps}_h\|_{W^{2,\infty}}$ &\\
           \noalign{\smallskip}\hline\noalign{\smallskip}
0& 2.02E--02	&4.25E--01&	2.93E+01&\\
1 & 3.94E--03 	&9.50E--02&	6.14E+00&\\
2 & 3.52E--03 	&5.06E--02&	3.83E+00&\\
3 & 4.66E--03 	&3.93E--02&	2.79E+00&\\
4 & 5.21E--03 &	3.31E--02&	2.50E+00&\\
       \noalign{\smallskip}\noalign{\smallskip}
   \end{tabular}
    \caption{Pointwise errors 
of $u-u^\eps_h$ using the iterative surgical strategy ($\eps=0.01, h=0.01$).}
\label{TableBL2}
\end{table}

\begin{figure}[h]
\centering
\includegraphics[width=4in]{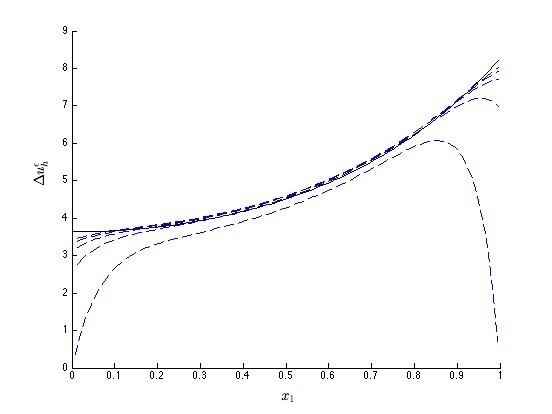} 
\caption{Cross-section plot of $\Delta u^\eps_h$ at $x_2=0.8$.  
Black solid line is exact solution, and dotted blue
lines are the computed solutions for iterations $0,1,2,3,$ and 4.}
\label{BLFig}
\end{figure}

\section{Nonlinear solvers}\label{chapter-7.2}

After problem \eqref{moment1}--\eqref{moment3}$_1$ is discretized, 
we obtain the (strong) nonlinear algebraic system \eqref{momentfem} or 
\eqref{mixedfem1}--\eqref{mixedfem2} or \eqref{altfem1}--\eqref{altfem2}
to solve. To this end, one has to use one or another iterative  
methods to do the job. In all numerical experiments given in 
Chapter \ref{chapter-6}, we use the ILU preconditioned
Newton iterative method as our nonlinear solver. Since Newton's method
often requires an accurate starting value to ensure convergence, hence
generating a good starting value for Newton's method is also an important 
issue here. So far we have used two strategies for the purpose in our 
numerical experiments in \cite{Feng2,Feng3,Feng4} and 
in Chapter \ref{chapter-6}. The first strategy is to 
use a fixed point iteration to generate a starting value for
Newton's method. However, this strategy may not always work 
although its success rate is pretty high.
The second strategy, which is more involved, 
is the following ``multi-resolution" or ``homotopy" strategy: first
compute a numerical solution using a relatively large $\vepsi$, then
use the computed solution as a starting value for the Newton method
at a finer resolution $\vepsi$. The process may need to be iterated
in $\vepsi$ for more than one step. Our experiences tell that 
$1-3$ steps should be enough to generate a good starting value
for Newton's method at the finest resolution $\vepsi$ at which 
one wants to compute a solution.

It is expected that for $3$-d simulations and for time-dependent
fully nonlinear PDEs (see Section \ref{chapter-7.3} below),  
more efficient fast solvers are required. It is well-known that
the key to this is to use better preconditioners for the linear
problem inside each Newton iteration because solving (large)
linear systems inside each Newton iteration costs most
of the total CPU time for executing the Newton's method. 
One plausible approach, which will be pursued in a
future work, is to use more sophisticated multigrid 
or Schwarz (or domain decomposition) preconditioners 
(cf. \cite{Toselli_Widlund05}) to replace the ILU preconditioner. 
With help of the better preconditioners, Krylov subspace methods 
\cite{Saad03} can be employed as the linear solver inside each 
Newton iteration. Put all pieces together, we arrive at a global 
nonlinear iterative solver which can be called the
Newton-Schwarz/Multigrid-Krylov method (cf. \cite{Keyes02}).

\section{Open problems}\label{chapter-7.3}

As the vanishing moment method was introduced very recently, 
there are many open questions concerning with the method. 
The foremost one is to generalize the convergence results of
Chapter \ref{chapter-3} to the general problem 
\eqref{moment1}--\eqref{moment3} under some reasonable structure
conditions on the nonlinear differential operator $F$. 
The convergence rate is probably hard to get 
unless the viscosity solution of the limiting problem
\eqref{generalPDEa}--\eqref{generalPDEb}
is sufficiently regular (cf. Theorem \ref{convergence_rate_thm3}).

Another interesting but completely open problem is 
to develop a vanishing moment method for fully nonlinear
second order parabolic PDEs. Unlike the situation for 
quasilinear PDEs, going from fully nonlinear second order 
elliptic PDEs to fully nonlinear second order 
parabolic PDEs is far from straightforward. 
One reason for this is that there are several 
different legitimate parabolic generalizations 
for equation \eqref{generalPDEa} 
(cf. \cite{Lieberman96,Wang92a,Wang92b}).
Two best known fully nonlinear second order parabolic PDEs are 
\begin{align}
\label{parabolic1}
F(D^2u, \nabla u, u, x, t)-u_t=0, \\
\label{parabolic2}
-u_t\, \mbox{det}(D^2 u)=f(\nabla u, u,x,t) \geq 0.
\end{align}
Extensive viscosity solution theories have been developed
for both equations (cf. \cite{Gutierrez_Huang01,Lieberman96,Wang92a,Wang92b}
and the references therein). However, to the best of our knowledge,
no numerical work has been reported for these equations in the literature.

Formulation of the vanishing moment method for \eqref{parabolic1}
is straightforward (see \cite{Feng2}). By adopting the method of 
lines approach, generalizations of the finite element and 
mixed finite element methods of Chapter \ref{chapter-4} and 
\ref{chapter-5} should be standard. However, the convergence 
analysis of any implicit scheme is expected to be hard, in particular,
establishing error estimates which depend on $\vepsi^{-1}$ {\em
polynomially} instead of {\em exponentially} will be very challenging.  
Furthermore, we note that numerically solving equation \eqref{parabolic2} 
using the vanishing moment method is expected to be difficult. In fact, 
it is not clear how to formulate the method for \eqref{parabolic2}.

Finally, another interesting open question is to explore the feasibility of
extending the notion of moment solutions and the vanishing moment method
to degenerate, non-elliptic, and systems of fully nonlinear 
second order PDEs (cf. \cite{Benamou_Brenier00,Caffarelli03,Feng5,
Chang_Gursky_Yang02,McCann_Oberman04}).

%% file: VMM.bbl
\begin{thebibliography}{99}
\bibliographystyle{amsalpha}

\bibitem{Agmon}S. Agmon, {\em Lectures on Elliptic Boundary
Value Problems}, Van Nostrand Mathemetical Studies, Princeton, NJ, 
1965.

\bibitem{Aleksandrov61} A. D. Aleksandrov, {\em Certain estimates for the
Dirichlet problem},  Soviet Math.  Dokl., 1:1151-1154, 1961.

\bibitem{Aronsson_Crandall_Juutinen04}
G.~Aronsson, M.~G. Crandall, and P.~Juutinen,
{\em A tour of the theory of absolutely minimizing functions},
Bull. Amer. Math. Soc. (N.S.), 41(4):439--505 (electronic), 2004.

\bibitem{Baginski_Whitaker96}
F.~E. Baginski and N.~Whitaker,
{\em Numerical solutions of boundary value problems for {${\mathcal K}$}-surfaces
in {${\bf R}\sp 3$}}, Numer. Methods for PDEs, 12(4):525--546, 1996.

\bibitem{Bakelman86}
I. J. Bakelman, {Generalized elliptic solutions of the Dirichlet problem 
for $n$-dimensional Monge-Amp\`ere equations}, in Nonlinear Functional 
Analysis and its Applications, Part 1 (Berkeley, Calif., 1983), 
Proc. Sympos. Pure Math., 45:73--102, 1986.

\bibitem{Barles_Souganidis91}
G.~Barles and P.~E. Souganidis,
{\em Convergence of approximation schemes for fully nonlinear second order 
equations}, Asymptotic Anal., 4(3):271--283, 1991.

\bibitem{Barles_Jakobsen05}
G. Barles and E.~R. Jakobsen, {\em Error bounds for monotone
approximation schemes for Hamilton-Jacobi-Bellman equations},
SIAM J. Numer. Anal., 43(2):540--558, 2005.

\bibitem{Bhattacharya91}
T. Bhattacharya, E. DiBenedetto, and J. Manfredi, 
{\em Limits as $p\to \infty$ of $\Del_p u_p=f$ and
related extremal problems,} Some topics in nonlinear PDEs
(Turin, 1989). Rend. Sem. Mat. Univ. Politec. Torino 1989.

\bibitem{Barth}T. Barth and J. Sethian, {\em Numerical 
schemes for the Hamilton-Jacobi
and level set equations on triangulated domains}, 
J. Comput. Phys. 145(1):1--40, 1998.

\bibitem{Benamou_Brenier98}
J.~Benamou and Y.~Brenier. {\em
Weak existence for the semigeostrophic equations formulated as a
coupled {M}onge-{A}mp\'ere/transport problem},
SIAM. J. Appl. Math., 58:1450--1461, 1998.

\bibitem{Benamou_Brenier00}
J.-D. Benamou and Y.~Brenier, {\em A computational fluid mechanics
solution to the {M}onge-{K}antorovich mass transfer problem}, Numer.
Math., 84(3):375--393, 2000.

\bibitem{Bernardi_Maday97}
C.~Bernardi and Y.~Maday,
{\em Spectral methods}, In {\em Handbook of numerical analysis, Vol. V},
Handb. Numer. Anal., V, pages 209--485. North-Holland, Amsterdam, 1997.

\bibitem{Brenner} S. C. Brenner and L. R. Scott, {\em The Mathematical
Theory of Finite Element Methods}, third edition, Springer, 2008.

\bibitem{Bryson_Levy03}S. Bryson and D. Levy, {\em High-order central
WENO schemes for multidimension Hamilton-Jacobi equations,}  SIAM
J. Numer. Anal. 41(4):1339--1369, 2003.

\bibitem{Bohmer08} 
K. B\"ohmer, {\em On  finite element methods for fully nonlinear 
elliptic equations of second order}, SIAM J. Numer. Anal., 46(3):1212--1249, 
2008.

\bibitem{Caffarelli_Nirenberg_Spruck84}
L. Caffarelli, L. Nerenberg, and J. Spruck, {\em The Dirichlet
problem for nonlinear second-order elliptic equations}, I. Monge-Amp\`ere
equation.  Commun. Pure Appl. Math., 37(3):369--402, 1984.

\bibitem{Caffarelli03}
L. Caffarelli, {\em The Monge-Amp\`ere equation and optimal transportation, 
an elementary review,} In Optimal Transportation and Appications, 
Martina Franca, 2001. Lecture Notes in Math., vol. 1813, Springer, Berlin, 2003.

\bibitem{Caffarelli_Cabre95} L.~A. Caffarelli and X.~Cabr{\'e},
{\em Fully nonlinear elliptic equations}, volume~43 of {\em American
Mathematical Society Colloquium Publications}. American Mathematical
Society, Providence, RI, 1995.

\bibitem{Caffarelli_Milman99} L.~A. Caffarelli and M. Milman, {\em
Monge {A}mp\`ere Equation: Applications to Geometry and
Optimization}, \textsl{Contemporary Mathematics}, American
Mathematical Society, Providence, RI, 1999.

\bibitem{Chang_Gursky_Yang02}
S.-Y.~A. Chang, M.~J. Gursky, and P.~C. Yang, {\em 
An equation of {M}onge-{A}mp\`ere type in conformal geometry, and
four-manifolds of positive {R}icci curvature},
Ann. of Math., 155(3):709--787, 2002.

\bibitem{Cheng_Yau77} S. Y. Cheng and S. T. Yau, {\em On the regularity of
the Monge-Amp\`ere equation $\det(\partial^2 u/\partial x_i \partial
x_j)=F(x,u)$}, Comm. Pure Appl. Math., 30(1):41-68, 1977.

\bibitem{Ciarlet78} P. G. Ciarlet, {\em The Finite Element Method for Elliptic
Problems.} North-Holland, Amsterdam, 1978.

\bibitem{Cockburn03}B. Cockburn, {\em Continuous dependence and 
error estimation for viscosity methods}, Acta Numer., 12:127--180, 2003.

\bibitem{Crandall_Lions83} M.~G. Crandall and P.-L. Lions,
{\em Viscosity solutions of {H}amilton-{J}acobi equations}, Trans.
Amer. Math. Soc., 277(1):1--42, 1983.

\bibitem{Crandall_Evans_Lions84}
M.~G. Crandall, L.~C. Evans, and P.-L. Lions,
{\em Some properties of viscosity solutions of Hamilton-Jacobi
equations}, Trans. Am. Math. Soc., 282(2):487--502, 1984.

\bibitem{Crandall_Ishii_Lions92}
M.~G. Crandall, H.~Ishii, and P.-L. Lions, {\em User's guide to
viscosity solutions of second order partial differential equations},
Bull. Amer. Math. Soc. (N.S.), 27(1):1--67, 1992.

\bibitem{Crandall_Lions96}M.~G. Crandall and P.-L. Lions, {\em Convergent
difference schemes for nonlinear parabolic equations and mean curvature
motion}, Numer. Math., 75(1):17--41, 1996.

\bibitem{Crandall_05} M.~G. Crandall, 
{\em A visit with the $\infty-$Laplace equation}, Lecture Notes in Mathematics,
1927:75--122, 2008.

\bibitem{Dean_Glowinski_a} E.~J. Dean and R.~Glowinski,
{\em Numerical solution of the two-dimensional elliptic Monge-Amp\`ere
equation with Dirichlet boundary conditions:  an augmented Lagrangian approach
}, C. R. Math. Acad. Sci. Paris, 339(12):887--892, 2004.

\bibitem{Dean_Glowinski_b} E.~J. Dean and R.~Glowinski,
{\em On the numerical solution of a two-dimensional Pucci's equation 
with Dirichlet boundary conditions:  a least-squares approach},
C. R. Math. Acad. Sci. Paris, 341(6), 375--380, 2005.

\bibitem{Dean_Glowinski_c} E.~J. Dean and R.~Glowinski,
{\em Numerical methods for fully nonlinear elliptic equations of
the {M}onge-{A}mp\`ere type}, Comput. Methods Appl. Mech. Engrg.,
195(13-16):1344--1386, 2006.

\bibitem{evans} L. C. Evans, {\em Partial Differential Equations},
volume 19 of \textsl{Graduate Studies in Mathematics},  American
Mathematical Society, Providence, RI, 1998.

\bibitem{Evans_Yu} L. C. Evans and Y. Yu, {\em 
Various properties of solutions of the infinity-Laplacian equation},
Comm. PDEs, 30:1401--1428, 2005.

\bibitem{Evans07} L. C. Evans, {\em The $1$-{L}aplacian, the 
{$\infty$}-{L}aplacian and differential games}, Contemp. Math., 
446:245--254, 2007.

\bibitem{Falk_Osborn_1980}
R.~S.~Falk and J.~E.~Osborn,
Error estimates for mixed methods, R. A. I. R. O. Anal. Numer., 
14:249--277, 1980.



\bibitem{Feng1} X. Feng, {\em Convergence of the vanishing moment method
for the Monge-Amp\`ere equation}, preprint. 

\bibitem{Feng2} X. Feng and M. Neilan, {\em Vanishing moment method and
moment solutions for second order fully nonlinear partial
differential equations}, J. Scient. Comp., 38(1):74--98, 2009.

\bibitem{Feng3} X. Feng and M. Neilan, {\em Mixed finite element methods
for the fully nonlinear Monge-Amp\`ere equation based on the vanishing
moment method}, SIAM J. Numer. Anal., 47(2):1226--1250, 2009.

\bibitem{Feng4}X. Feng and M. Neilan, {\em Error Analysis of Galerkin
approximations of the fully nonlinear Monge-Amp\`ere equation},
J. Sciet. Comp., 47:303--327, 2011.

\bibitem{Feng5}X. Feng and M. Neilan, {\em A modified characteristic
finite element method for a fully nonlinear formulation of the semigeostrophic
flow equations}, SIAM J. Numer. Anal., 47(4):2952-2981, 2009.

\bibitem{Fleming_Soner06}
W.~H. Fleming and H.~M. Soner,
{\em Controlled {M}arkov Processes and Viscosity Solutions},
vol. 25 of {\em Stochastic Modelling and Applied Probability}.
Springer, New York, second edition, 2006.

\bibitem{Gilbarg_Trudinger01} D. Gilbarg and N. S. Trudinger,  \textsl{Elliptic
Partial Differential Equations of Second Order},  \textsl{Classics
in Mathematics},  Springer-Verlang, Berlin, 2001.  Reprint of the
1998 edition.

\bibitem{Glowinski09} R.~Glowinski,
{\em Numerical methods for fully nonlinear elliptic equations},
In Proceedings of 6th {I}nternational {C}ongress on {I}ndustrial
and {A}pplied {M}athematics, R. Jeltsch and G. Wanner, editors, pages
155--192, 2009.

\bibitem{Gutierrez01} C. E. Gutierrez, {\em The Monge-Amp\`ere Equation},
volume 44 of \textsl{Progress in Nonlinear Differential Equations
and Their Applications},  Birkhauser, Boston, MA, 2001.

\bibitem{Gutierrez_Huang01}
C.~E. Guti{\'e}rrez and Q.~Huang,
{\em {$W\sp {2,p}$} estimates for the parabolic {M}onge-{A}mp\`ere equation},
Arch. Ration. Mech. Anal., 159(2):137--177, 2001.

\bibitem{Grisvard85} P. Grisvard {\em Elliptic Problems in Nonsmooth Domains},
Pitman (Advanced Publishing Program), Boston, MA, 1985.

\bibitem{Guan95}B. Guan, {\em On the existence and regularity of hypersurfaces
of prescribed Gauss curvature with boundary}, Indiana Univ. 
Math. J. 44(1):21--241, 1995.

\bibitem{HUT01} J.-B. Hiriart-Urruty and C. Lemar\'echal,
{\em Fundamentals of Convex Analysis}, Springer, 2001.

\bibitem{Ishii89} H. Ishii, {\em On uniqueness and existence of viscosity solutions
of fully nonlinear second order PDE's},  Comm. Pure Appl. Math.,
42:14--45, 1989.

\bibitem{Jakobsen03}
E.~R. Jakobsen, {\em On the rate of convergence of approximation schemes
for Bellman equations associated with optimal stopping time problems},
Math. Models Methods Appl. Sci., 13(5):613--644, 2003.

\bibitem{Jensen88} R. Jensen, {\em The maximum principle for  viscosity solutions
of fully nonlinear second order partial differential equations},
Arch. Rational Mech. and Anal., 101:1--27, 1988.

\bibitem{Keyes02} D. Keyes, {\em Terascale implicit methods for 
partial differential equations}, Contemp. Math. (AMS), 306:29--84, 2002.

\bibitem{Krylov05}
N.~V. Krylov, {\em The rate of convergence of finite difference 
approximations for Bellman equations with Lipschitz coefficients}, 
Appl. Math. Optim., 52(3):365-399, 2005.

\bibitem{LU} O. A. Ladyzhenskaya and N. N. Ural'tseva, {\em Linear and
Quasilinear Elliptic Equations}, Academic Press, New York, 1968.

\bibitem{Lieberman96}
G.~M. Lieberman, {\em Second Order Parabolic Differential Equations},
World Scientific Publishing Co. Inc., River Edge, NJ, 1996.

\bibitem{Lin_Tadmor00}
C.-T. Lin, and E. Tadmor., {\em High-resolution nonoscillatory central 
schemes for Hamilton-Jacobi equations}, SIAM J. Sci. Comput. 21(6):2163--2186, 
2000.

\bibitem{Loeper_Rapetti05}
G.~Loeper and F.~Rapeti,
{\em Numerical solution of the {M}onge-{A}mp\`ere equation by a {N}ewton
  algorithm}
C. R. Math. Acad. Sci. Paris, 340:319--324, 2005.

\bibitem{McCann_Oberman04}
R.~J. McCann and A.~M. Oberman.
{\em Exact semi-geostrophic flows in an elliptical ocean basin},
Nonlinearity, 17(5):1891--1922, 2004.

\bibitem{Monn86}
D. Monn, {\em Regularity of the complex Monge-Amp\`ere equation
for radially symmetric functions of the unit ball},
Math. Ann., 275:501--511, 1986.

\bibitem{Suli03} I. Mozolevski and E. S\"uli, {\em A priori error analysis
for the $hp$-version of the discontinuous Galerkin finite element
method for the biharmonic equation}, Comput. Meth. Appl. Math.
3:596--607, 2003.

\bibitem{Neilan_thesis} M. Neilan, {\em Numerical Methods for Fully Nonlinear
Second Order Partial Differential Equations},  Ph.D. Dissertation, The
University of Tennessee, 2009.

\bibitem{Neilan09}
M. Neilan, {\em A nonconforming Morley finite element method for the 
Monge-Amp\`ere equation}, Numer. Math., 115(3):371--394, 2010.

\bibitem{Oberman04}
A.~M. Oberman, {\em A convergent difference scheme for 
the infinity Laplacian:  construction of absolutely minimizing
Lipschitz extensions}, Math. Comp., 75(251):1217--1230, 2004.

\bibitem{Oberman07}
A.~M. Oberman, {\em Wide stencil finite difference schemes for
elliptic monge-amp\'ere equation and functions of the eigenvalues
of the hessian}, Discrete Contin. Dyn. Syst. B, 10:271--293, 2008. 

\bibitem{Oliker_Prussner88}
V.~I. Oliker and L.~D. Prussner,
{\em  On the numerical solution of the equation {$(\partial\sp 2z/\partial
 x\sp 2)(\partial\sp 2z/\partial y\sp 2)-((\partial\sp 2z/\partial x\partial
 y))\sp 2=f$} and its discretizations. {I}.},
Numer. Math., 54(3):271--293, 1988.

\bibitem{Osher1}S. Osher and J. Sethian, {\em Fronts propagating 
with curvature-dependent speed:  algorithms based on Hamilton-Jacobi
formulations}, J. Comput. Phys. 79(1):12--49, 1988.

\bibitem{Osher2}S. Osher and S.~W. Shu, {\em High-order essentially 
nonoscillatory schemes for Hamilton-Jacobi equations}, 
SIAM J. Numer. Anal., 28(4):907--922, 1991. 

\bibitem{Osher3}S. Osher and R. Fedkiw, {\em Level 
Set Methods and Dynamic Implicit Surfaces}, 
Springer-Verlag, New York, 2003.

\bibitem{Protter_Weinberger67} M. H. Protter and H. F. Weinberger,
{\em Maximum Priciples in Differential Equations},
Prentice-Hall, 1967.

\bibitem{Rios_Sawyer08}
C. Rios and E. T. Sawyer, {\em Smoothness of radial solutions to 
Monge-Amp\`ere equations}, Proc. of AMS, 137:1373--1379, 2008.

\bibitem{Saad03}
Y.~Saad, {\em Iterative methods for sparse linear systems},
Society for Industrial and Applied Mathematics, Philadelphia, PA, 
second edition, 2003.

\bibitem{Sethianbook}J. A. Sethian, {\em Level Set 
Methods and Fast Marching Methods.  Evolving
Interfaces in Computational Geometry, Fluid Mechanics, 
Computer Vision, and Materials Science.}
Second edition.  Cambridge University Press, Cambridge, 1999.


\bibitem{Tai_Wagner00} T. Nilssen, X. -C. Tai, and R. Wagner, {\em A robust
nonconfirming $H^2$ element}, Math. Comp., 70:489--505, 2000.

\bibitem{Toselli_Widlund05}
A. Toselli and O. Widlund, {\em Domain Decomposition Methods}, Springer, 2005.

\bibitem{Wang92a}
L.~Wang,
{\em On the regularity theory of fully nonlinear parabolic equations {I}},
Comm. Pure Appl. Math., 45:27--76, 1992.

\bibitem{Wang92b}
L.~Wang, 
{\em On the regularity theory of fully nonlinear parabolic equations II},
Comm. Pure Appl. Math., 45:141Ð178 (1992)

\bibitem{Wang_Shi_Xu07} M. Wang, Z. Shi, and J. Xu, {\em A new class of
Zienkiewicz-type nonconforming elements in any dimensions}, Numer.
Math, 106:335--347, 2007.

\bibitem{Wang_Xu07} M. Wang and J. Xu, {\em Some tetrahedron nonconforming elements
for fourth order elliptic equations},  Math. Comp., 76:1--18, 2007.

\bibitem{Zhang_Shu02}
Y.-T. Zhang and C.-W. Shu, {\em High-order WENO schemes for 
Hamilton-Jacobi equations on triangular meshes},
SIAM J. Sci. Comput. 24(3), 1005--1030, 2002.

\bibitem{Zhao05}
H. Zhao, {\em A fast sweeping method for Eikonal equations},
Math. Comput. 74(250), 603--627, 2005.


\end{thebibliography}
